\documentclass[11pt,english]{article}
\usepackage{color}
\usepackage[T1]{fontenc}
\usepackage[latin9]{inputenc}
\usepackage[a4paper]{geometry}
\usepackage{textcomp}
\usepackage{amsthm}
\usepackage{amsmath}
\usepackage{amssymb}
\usepackage{esint}
\usepackage{bbm}
\usepackage{hyperref}
\usepackage{babel}
\usepackage{amscd}
\usepackage{amsfonts}
\usepackage{amsthm}
\usepackage{mathrsfs}
\usepackage{dsfont}
\usepackage{mathtools}
\usepackage{enumerate}
\usepackage{bm}
\usepackage{bbm}
\usepackage{authblk}
\usepackage[inline]{enumitem}

\usepackage[english=american]{csquotes}

\usepackage{cite}

\numberwithin{equation}{section}

\newcommand{\scp}[2]{\langle #1 , #2 \rangle}
\newcommand{\norm}[1]{\lVert #1 \rVert}

\newcommand{\abs}[1]{\left| #1 \right|}

\DeclareMathOperator{\Opdef}{Op}
\newcommand{\OpN}[1]{\Opdef_{\hbar,#1}}

\DeclareMathOperator{\Opt}{Op_{\hbar,t}}
\DeclareMathOperator{\Op}{Op_\hbar}
\DeclareMathOperator{\OpW}{Op_\hbar^\mathrm{w}}
\DeclareMathOperator{\dist}{dist}
\DeclareMathOperator{\sgn}{sgn}
\DeclareMathOperator{\Tr}{\mathrm{Tr}}
\DeclareMathOperator{\supp}{\mathrm{supp}}
\DeclareMathOperator{\re}{Re}
\DeclareMathOperator{\im}{Im}

\DeclareMathOperator{\spec}{\mathrm{spec}}
\DeclareMathOperator{\Vol}{Vol}

\newcommand{\R}{\mathbb{R}}
\newcommand{\C}{\mathbb{C}}
\newcommand{\N}{\mathbb{N}}
\newcommand{\Z}{\mathbb{Z}}

\newtheorem{thm}{Theorem}[section]
\newtheorem{lemma}[thm]{Lemma}
\newtheorem{prop}[thm]{Proposition}
\newtheorem{corollary}[thm] {Corollary}
\newtheorem{notation}[thm]{Notation}

\theoremstyle{definition}
\newtheorem{definition}[thm] {Definition}
\newtheorem{assumption}[thm]{Assumption}
\newtheorem{remark}[thm]{Remark}
\newtheorem{obs}[thm]{Observation}
\newtheorem{example}[thm]{Example}

\begin{document}

\title{Optimal semiclassical spectral asymptotics for differential operators with non-smooth
  coefficients}

% author one information
\author{S{\o}ren Mikkelsen}
   \affil{\small{School of Mathematics, University of Bristol\\ Bristol
BS8 1UG, United Kingdom} }

\maketitle

\begin{abstract}
We consider differential operators defined as Friedrichs extensions of quadratic forms with non-smooth coefficients. We prove a two term optimal asymptotic for the Riesz means of these operators and thereby also reprove an optimal Weyl law under certain regularity conditions. The methods used are then extended to consider more general admissible operators perturbed by a rough differential operator and obtaining optimal spectral asymptotics again under certain regularity conditions. For the Weyl law we assume the coefficients are differentiable with H\"older continuous derivatives and for the Riesz means we assume the coefficients are two times differentiable with H\"older continuous derivatives.
\end{abstract}

 \tableofcontents

% \listoffixmes

%%%%%%%%%%%%%%%%%%%%%%%%%%%%%%%%%%%%%%%%%%%%%%%%%%%%%%%%%%%%%

\section{Introduction}
 In this paper we will study (uniformly) elliptic self-adjoint differential operators $A(\hbar)$ acting in $L^2(\R^d)$ defined as Friedrichs extensions of sesquilinear forms given by
  \begin{equation}\label{int_gen_sf}
   \mathcal{A}_\hbar[\varphi,\psi] =  \sum_{\abs{\alpha},\abs{\beta}\leq m}  \int_{\R^d} a_{\alpha\beta}(x) (\hbar D_x)^\beta\varphi(x) \overline{(\hbar D_x)^\alpha\psi(x)} \, dx, \qquad \varphi,\psi \in \mathcal{D}(   \mathcal{A}_\hbar),
  \end{equation}
 where $\hbar>0$ is the semiclassical parameter, $\mathcal{D}(   \mathcal{A}_\hbar)$ is the associated form domain and for $\alpha\in\N_0^d$, we have used the standard notation 
\begin{equation*}
  (\hbar D_x)^\alpha = \prod_{j=1}^d (-i\hbar\partial_{x_j})^{\alpha_j}.
\end{equation*}
Our exact assumptions on the functions $a_{\alpha\beta}$ will be stated later. For $\gamma$ in $[0,1]$ we will analyse the asymptotics as $\hbar$ tends to zero of the traces
\begin{equation}
	\Tr[(A(\hbar))_{-}^\gamma ],
\end{equation}
where we have used the notation $(x)_{-} = \max(0,-x)$. In the case $\gamma=0$ we use the convention that $(x)_{-}^0 = \boldsymbol{1}_{(-\infty,0]}(x)$, where $\boldsymbol{1}_{(-\infty,0]}$ is the characteristic function for the set $(-\infty,0]$. There are no technical obstructions in considering higher values of $\gamma$, but more regularity of the coefficients will be needed to obtain sharp remainder estimates.   

These kinds of traces are in the literature denoted as Riesz means, when $\gamma>0$, and counting functions when $\gamma=0$. Both objects have been studied extensively in the literature in different variations. For background on the Weyl law (asymptotic expansions of the counting function) and the connection to physics we refer the reader to the surveys  \cite{weyl_hist,MR3556544,MR510064}.

In the case of smooth coefficients and with some regularity conditions it was first
proven in \cite{MR724029} by Helffer and Robert that an estimate of the following type
\begin{equation}\label{Helffer_Robert_res}
 \big| \Tr[\boldsymbol{1}_{(-\infty,\lambda_0]}(A(\hbar))] - \frac{1}{(2\pi\hbar)^d} \int_{\R^{2d}}  \boldsymbol{1}_{(-\infty,\lambda_0]}(a_0(x,p)) \, dxdp \big| \leq C\hbar^{1-d},
\end{equation}
is true, where
\begin{equation}\label{int_princ_sym}
  a_0(x,p) =  \sum_{\abs{\alpha},\abs{\beta}\leq m} a_{\alpha\beta}(x)p^{\alpha+\beta},
\end{equation}
and $\lambda_0$ is a number such that there exists a strictly larger number $\lambda$ for which
$a_0^{-1}((-\infty,\lambda])$ is compact. Moreover, $\lambda_0$ is assumed to be non-critical for
$a_0(x,p)$. A number $\lambda_0$ is a non-critical value when
\begin{equation*}
  |\nabla a_0(x,p)| \geq c>0 \quad\text{for all } (x,p) \in a_0^{-1}(\{\lambda_0\}).
\end{equation*}
We will from here on always use that $a_0(x,p)$ is given by \eqref{int_princ_sym} when we are working with a sesquilinear form $ \mathcal{A}_\hbar$ given by \eqref{int_gen_sf}.
Actually they proved \eqref{Helffer_Robert_res} for a larger class of
pseudo-differential operators. The high energy analog was proven by
 H\"{o}rmander in \cite{MR609014}, where the operator was defined to act in $L^2(M)$, where $M$ is
a smooth compact manifold without a boundary. For the Riesz means an asymptotic expansion was obtained by Helffer and Robert in \cite{MR1061661} under the same conditions as above and again for this larger class of pseudo-differential operators. For the high energy case they where studied by H\"{o}rmander in \cite{MR0257589}.

When considering the assumptions above one question immediately arise:
\begin{enumerate}
\item\label{Main_question_1} What happens if the coefficients are not smooth? Can such results
  still be proven with optimal errors?
%
%
%\item\label{Main_question_2} What happens if a non-critical condition is not assumed? Can such results still be proven with optimal errors?
%
\end{enumerate}
A considerable body of literature have been devoted to this questions. Most prominent are the works of Ivrii \cite{MR1974450,MR1631419,ivrii2019microlocal1,MR1240575,MR1974451,MR1807155,MR2179891,MR3556544} and Zielinski \cite{MR1612880,MR1635856,MR1620550,MR1736710,MR2105486,MR2245259,MR2343462,MR2335576,MR2952218}. A more detailed review of the literature will be given in  section~\ref{sec.lit} after presenting the result obtained in this work. 
\subsection{Assumptions and main results I}

Before we state our assumptions and results on the Riesz means we need the following definition to clarify terminology.  
\begin{definition}
  For $k$ in $\N$ and $\mu$ in $(0,1]$ we denote by $C^{k,\mu}(\R^d)$
  the subspace of $C^{k}(\R^d)$ defined by
  \begin{equation*}
  	\begin{aligned}
    C^{k,\mu}(\R^d) = \big\{ f\in C^{k}(\R^d) \, \big| \, \exists C>0:   |\partial_x^{\alpha} f(x) - \partial_x^{\alpha} f(y&) |  \leq C |x-y|^{\mu} 
    \\
    & \forall \alpha \in\N^d \text{ with } \abs{\alpha}=k \text{ and } \forall x,y\in\R^d \big\}.
    \end{aligned}
  \end{equation*}
In the case $\mu=0$ we use the convention
  \begin{equation*}
    C^{k,0}(\R^d) = C^{k}(\R^d). 
  \end{equation*}
\end{definition}
We will for both the Weyl law and the Riesz means make the following assumption on the sesquilinear form.
\begin{assumption}[Tempered variation model]\label{ses_assump_gen}
 Let  $(k,\mu)$ be numbers in $\N\times[0,1]$ and $\mathcal{A}_\hbar$ be a sesquilinear form given by
   \begin{equation*}
   \mathcal{A}_\hbar[\varphi,\psi] =  \sum_{\abs{\alpha},\abs{\beta}\leq m}  \int_{\R^d} a_{\alpha\beta}(x) (\hbar D_x)^\beta\varphi(x) \overline{(\hbar D_x)^\alpha\psi(x)} \, dx, \qquad \varphi,\psi \in \mathcal{D}(   \mathcal{A}_\hbar).
  \end{equation*}
 Assume the coefficients $a_{\alpha\beta}(x)$ are in $C^{k,\mu}(\R^d)$ and satisfy that $a_{\alpha\beta}(x) = \overline{a_{\beta\alpha}(x)}$ for all multi-indices $\alpha$ and $\beta$. For all multi-indices $\alpha$ and $\beta$ assume that
  \begin{enumerate}[label={$(\roman*)$}]
  \item\label{ass.symbol_1} There is a $\zeta_0>0$ such that $\min_{x\in\R^d}(a_{\alpha\beta}(x))> - \zeta_0$.
  \item\label{ass.symbol_2} There is a $\zeta_1>\zeta_0$ and $C_1,M>0$ such that
    \begin{equation*}
      a_{\alpha\beta}(x) + \zeta_1 \leq C_1(a_{\alpha\beta}(y)+\zeta_1)(1+|x-y|)^M,
    \end{equation*}
    for all $x,y$ in $\R^d$.
  \item\label{ass.symbol_3} For all $\delta$ in $\N_0^d$ with $|\delta|\leq k$ there is a $c_\delta>0$ such that
    \begin{equation*}
      \abs{\partial_{x}^\delta a_{\alpha\beta}(x)} \leq  c_\delta(a_{\alpha\beta}(x) + \zeta_1).
    \end{equation*}
  \end{enumerate}
 Suppose there exists a constant $C_2$ such that
  \begin{equation}\label{ellip_assumption}
    \sum_{\abs{\alpha}=\abs{\beta}= m} a_{\alpha\beta}(x) p^{\alpha+\beta} \geq C_2 \abs{p}^{2m},
  \end{equation}
  for all $(x,p)$ in $\R^d_x\times\R_p^d$.
\end{assumption}
The assumption that the coefficients are in $C^{k,\mu}(\R^d)$ is that we need some regularity to get optimal asymptotic results. The assumption $a_{\alpha\beta}(x) = \overline{a_{\beta\alpha}(x)}$  is made to ensure the form is symmetric. The assumptions \ref{ass.symbol_1}--\ref{ass.symbol_3}  ensure that the coefficients are of tempered variation. Assumption~\ref{ass.symbol_2} is the condition that for all $\alpha$ and $\beta$ $a_{\alpha\beta}(x) + \zeta_1$ is a tempered weight. Note that due to the semiclassical structure of the problem these assumptions are needed to be true for all $\alpha$ and $\beta$ and not just $|\alpha|=|\beta|=m$.  The ``ellipticity'' assumption \ref{ellip_assumption} will help ensure the existence of eigenvalues.  

We can now state our two main results. Firstly we have that   
\begin{thm}\label{B.Weyl_law_thm_irr_cof}
Let $A(\hbar)$ be the Friedrichs extension of a sesquilinear form $\mathcal{A}_\hbar$ which satisfies Assumption~\ref{ses_assump_gen} with the numbers $(1,\mu)$ where $\mu>0$. Suppose there is a $\nu>0$ such that the set $a_0^{-1}((-\infty,\nu])$ is compact and there is $c>0$ such that
  \begin{equation}\label{non_critical_main_thm_1}
    |\nabla_p a_0(x,p)| \geq c \quad\text{for all } (x,p)\in a_0^{-1}(\{0\}).
  \end{equation}
 Then we have
  \begin{equation*}
   \big |\Tr[\boldsymbol{1}_{(-\infty,0]}(A(\hbar))] - \frac{1}{(2\pi\hbar)^d} \int_{\R^{2d}} \boldsymbol{1}_{(-\infty,0]}( a_{0}(x,p)) \,dx dp \big| \leq C \hbar^{1-d},
  \end{equation*}
  for all sufficiently small $\hbar$.
\end{thm}
\begin{remark}
	The assumption that there is a $\nu>0$ such that the set $a_0^{-1}((-\infty,\nu])$ is compact is needed to ensure that we only have pure point spectrum in $(-\infty,0]$. Due to the ellipticity assumption this is really an assumption on the coefficients $a_{\alpha\beta}(x)$.    
\end{remark}
Secondly for the Reisz means with $\gamma$ in $(0,1]$ we have
\begin{thm}\label{B.Riesz_means_thm_irr_cof}
Assume $\gamma$ is in $(0,1]$ and let $A(\hbar)$ be the Friedrichs extension of a sesquilinear form $\mathcal{A}_\hbar$ which satisfies Assumption~\ref{ses_assump_gen} with the numbers $(2,\mu)$. If $\gamma=1$ we suppose $\mu>0$ and if $\gamma<1$ we may suppose $\mu=0$. Lastly, suppose there is a $\nu>0$ such that the set $a_0^{-1}((-\infty,\nu])$ is compact and there is $c>0$ such that
  \begin{equation}\label{non_critical_main_thm}
    |\nabla_p a_0(x,p)| \geq c \quad\text{for all } (x,p)\in a_0^{-1}(\{0\}).
  \end{equation}
  Then we have
  \begin{equation*}
   \big |\Tr[(A(\hbar))^\gamma_{-}] - \frac{1}{(2\pi\hbar)^d} \int_{\R^{2d}} ( a_{0}(x,p))^{\gamma}_{-} +\hbar \gamma a_1(x,p) ( a_{0}(x,p))^{\gamma-1}_{-} \,dx dp \big| \leq C \hbar^{1+\gamma-d},
  \end{equation*}
  for all sufficiently small $\hbar$ where the symbol $a_1(x,p)$ is defined as
  \begin{equation}
	a_{1}(x,p) =  i \sum_{|\alpha|, |\beta|\leq m}\sum_{j=1}^d \frac{\beta_j-\alpha_j}{2} \partial_{x_j} a_{\alpha\beta} (x) p^{\alpha+\beta-\eta_j}, 
\end{equation}
  where $\eta_j$ is the multi index with all entries zero except the j'th which is one. 
\end{thm}
\begin{remark}
In the case where $\gamma \leq \frac{1}{3}$ we may assume the coefficients are in $C^{1,\mu}(\R^d)$, with $\mu\geq \frac{2\gamma}{1-\gamma}$ and still obtain sharp estimates. Further details are given in Remark~\ref{Estimate_error_reisz_1}. 

There are no technical obstruction in considering larger values of $\gamma$. For a $\gamma>1$ we will need the coefficients to have $\lceil \gamma \rceil$ derivatives if $\gamma$ is not an integer in general. If $\gamma$ is an integer we will need the coefficients to be in $C^{\gamma,\mu}(\R^d)$ for some $\mu>0$. The expansions will also consist of  $\lceil \gamma \rceil +1$ terms, where these terms can be calculated explicitly.
\end{remark}
\begin{example}
Let $A(x)$ be a $d\times d$--dimensional symmetric matrix and suppose that $A(x)$ is positive definite for all $x$ in $\R^d$. We assume that each entry in $a_{ij}(x)$ of $A(x)$  are in $C^{1,\mu}(\R^d)$, where $\mu>0$, and satisfies assumption \ref{ass.symbol_1}--\ref{ass.symbol_3} from Assumption~\ref{ses_assump_gen}. We will further assume that $\det(A(x)) \geq c(1+ |x|^2)^{d+1}$ for all $x$. Then one could be interested in the number of eigenvalues less than or equal to $E^2$ for the second order differential operator  with quadratic form
\begin{equation}
	\langle  A(x)(-i\nabla_x) \varphi, (-i\nabla_x)\varphi \rangle_{L^2(\R^d)}.
\end{equation}
In this example we will denote this number by $\mathcal{N}(E^2)$. This is equivalent to counting the number of eigenvalues less than or equal to zero of the operator with quadratic form  
\begin{equation}
	\langle  A(x)(-iE^{-1}\nabla_x) \varphi, (-iE^{-1}\nabla_x)\varphi \rangle_{L^2(\R^d)} -\langle   \varphi, \varphi \rangle_{L^2(\R^d)} .
\end{equation}
Treating $E^{-1}$ as the semiclassical parameter. It is a simple calculation to check that all assumptions for Theorem~\ref{B.Weyl_law_thm_irr_cof} are satisfied. This gives us that
  \begin{equation*}
   \big |\mathcal{N}(E^2) - \frac{1}{(2\pi)^d} \int_{\R^{2d}} \boldsymbol{1}_{(-\infty,E^2]}( \langle A(x)p,p\rangle) \,dx dp \big| \leq C E^{1-d},
  \end{equation*}
for $E$ suffciently large. Note that the phase space integral has this more familiar form
 \begin{equation*}
    \frac{1}{(2\pi)^d} \int_{\R^{2d}} \boldsymbol{1}_{(-\infty,E^2]}( \langle A(x)p,p\rangle) \,dx dp =   E^d \frac{\Vol_d(B(0,1))}{(2\pi)^d} \int_{\R^d} \frac{1}{\sqrt{\det(A(x))}} \, dx,
  \end{equation*}
where $\Vol_d(B(0,1))$ is the $d$-dimensional volume of the unit ball.
\end{example}
\subsection{Assumptions and main results II}
The main motivation and the novel results, to the best of my knowledge, in this work are the following results, where we consider spectral asymptotic of admissible operators perturbed by differential operators defined as Friedrichs extensions of sesquilinear forms with non-smooth coefficients. We will need some regularity conditions on the admissible operator, which will be the same as the assumption made in \cite[Chapter 3]{MR897108}. We will here recall them:

\begin{assumption}\label{admis_op_assump}
 For $\hbar$ in $(0,\hbar_0]$ let $B(\hbar)$ be an admissible operator with tempered weight $m$ and symbol
 \begin{equation}
 	b_\hbar(x,p) = \sum_{j\geq0} \hbar^j b_j(x,p).
 \end{equation}
 Assume that
  \begin{enumerate}[label={$(\roman*)$}]
  \item\label{ass.op_1} The operator $B(\hbar)$ is symmetric on $\mathcal{S}(\R^d)$ for all  $\hbar$ in $(0,\hbar_0]$.
  
  \item\label{ass.op_2} There is a $\zeta_0>0$ such that 
  \begin{equation*}
  \min_{(x,p)\in\R^d\times\R^d}(b_0(x,p))> - \zeta_0.
\end{equation*}	%
  \item\label{ass.op_3} There is a $\zeta_1>\zeta_0$ such that $b_0+\zeta_1$ is a tempered weight and $b_j$ is in $\Gamma_{0,1}^{b_0+\zeta_1} \left(\R^{2d}\right)$.
  \end{enumerate}
\end{assumption}
The definition of a tempered weight is recalled in Section~\ref{Sec:def_rough_op}, where the definition of the space $\Gamma_{0,1}^{b_0+\zeta_1} \left(\R^{2d}\right)$ is also given. We can now state the second set of main theorems.
\begin{thm}\label{B.Weyl_law_thm_adop_plus_irr_cof}
Let $B(\hbar)$ be an admissible operator satisfing Assumption~\ref{admis_op_assump} with $\hbar$ in $(0,\hbar_0]$ and $\mathcal{A}_\hbar$ be a sesquilinear form satisfying Assumption~\ref{ses_assump_gen} with the numbers $(1,\mu)$ with $\mu>0$. 
Define the operator $\tilde{B}(\hbar)$ as the Friedrichs extension of the form sum $B(\hbar) + \mathcal{A}_\hbar$ and let
\begin{equation}
	\tilde{b}_0(x,p) = b_0(x,p) +  \sum_{\abs{\alpha},\abs{\beta}\leq m} a_{\alpha\beta}(x)p^{\alpha+\beta}.
\end{equation}
  We suppose there is a $\nu$ such that $\tilde{b}_0^{-1}((-\infty,\nu])$ is compact and there is $c>0$ such that
  \begin{equation}\label{non_critical_main_thm_2}
    |\nabla_p \tilde{b}_0(x,p)| \geq c \quad\text{for all } (x,p)\in \tilde{b}_0^{-1}(\{0\}).
  \end{equation}
  Then we have
  \begin{equation*}
   \big |\Tr[\boldsymbol{1}_{(-\infty,0]}(\tilde{B}(\hbar))] - \frac{1}{(2\pi\hbar)^d} \int_{\R^{2d}} \boldsymbol{1}_{(-\infty,0]}( \tilde{b}_0(x,p)) \,dx dp \big| \leq C \hbar^{1-d},
  \end{equation*}
  for all sufficiently small $\hbar$.
\end{thm}
Next we give the result for Riesz means for this type of operator.
\begin{thm}\label{B.Riesz_means_thm_adop_plus_irr_cof}
Assume $\gamma$ is in $(0,1]$. 
Let $B(\hbar)$ be an admissible operator satisfing Assumption~\ref{admis_op_assump} with $\hbar$ in $(0,\hbar_0]$ and $\mathcal{A}_\hbar$ be a sesquilinear form satisfying Assumption~\ref{ses_assump_gen} with the numbers $(2,\mu)$. If $\gamma=1$ we suppose $\mu>0$ and if $\gamma<1$ we may suppose $\mu=0$. 
  
Define the operator $\tilde{B}(\hbar)$ as the Friedrichs extension of the form sum $B(\hbar) + \mathcal{A}_\hbar$ and let
\begin{equation}
	\tilde{b}_0(x,p) = b_0(x,p) +  \sum_{\abs{\alpha},\abs{\beta}\leq m} a_{\alpha\beta}(x)p^{\alpha+\beta}.
\end{equation}
  We suppose there is a $\nu$ such that $\tilde{b}_0^{-1}((-\infty,\nu])$ is compact and there is $c>0$ such that
  \begin{equation}\label{non_critical_main_thm_3}
    |\nabla_p \tilde{b}_0(x,p)| \geq c \quad\text{for all } (x,p)\in \tilde{b}_0^{-1}(\{0\}),
  \end{equation}
Then we have
  \begin{equation*}
   \big |\Tr[(\tilde{B}(\hbar))^\gamma_{-}] - \frac{1}{(2\pi\hbar)^d} \int_{\R^{2d}} ( \tilde{b}_{0}(x,p))^{\gamma}_{-} +\hbar \gamma \tilde{b}_1(x,p) ( \tilde{b}_{0}(x,p))^{\gamma-1}_{-} \,dx dp \big| \leq C \hbar^{1+\gamma-d},
  \end{equation*}
  for all sufficiently small $\hbar$ where the symbol $\tilde{b}_1(x,p)$ is defined as
  \begin{equation}
	\tilde{b}_{1}(x,p) = b_1(x,p)+  i \sum_{|\alpha|, |\beta|\leq m}\sum_{j=1}^d \frac{\beta_j-\alpha_j}{2} \partial_{x_j} a_{\alpha\beta} (x) p^{\alpha+\beta-\eta_j}.
\end{equation}
\end{thm}
\subsection{Previous work}\label{sec.lit}
The first results with an optimal Weyl law without full regularity was
proven in the papers
\cite{MR1736710,MR1612880,MR1620550,MR1635856} by Zielinski. In these papers
Zielinski obtained an optimal Weyl law under the assumption that
the coefficients are differentiable with Lipschitz continuous first
derivative.  Zielinski did not in those papers consider the
semiclassical setting. These results were generalised by Ivrii in the
semiclassical setting in \cite{MR1807155}. Here the coefficients are
assumed to be differentiable and with a H\"older continuous first
derivative. This was further generalised by Bronstein and Ivrii
in \cite{MR1974450}, where they reduced the assumptions further by
assuming the first derivative to have modulus of continuity
$\mathcal{O}(|\log(x-y)|^{-1})$. All these papers considered
differential operators acting in $L^2(M)$, where $M$ is
a compact manifold both with and without a boundary. In \cite{MR2105486} Zielinski  considers the semiclassical setting with differential operators acting in $L^2(\R^d)$ and proves an optimal Weyl Law under the assumption that all coefficients are one time differentiable with a H\"older continuous derivative. Moreover, it is assumed that the coefficients and the derivatives is bounded. However, it is remarked that it is possible to consider unbounded coefficients in a framework of tempered variation models. As we have seen above this is indeed the case and thereby proven this minor technical generalisation.   

Another question one could ask is can such results be obtained without a non-critical condition? This question have also been studied in the literature and it is possible for
Schr\"{o}dinger operators to prove optimal Weyl laws without a non-critical condition by a multiscale argument  see
\cite{MR1974450,MR1631419,ivrii2019microlocal1,MR1240575}, this
approach is also described in \cite{MR1343781}.  This multiscale argument can
be seen as a discreet approach and a continuous version have been
proved and used in \cite{MR2013804}.  The essence of this approach is
to localise and locally introduce a non-critical condition by
unitary conjugation. Then by an optimal Weyl law with a non-critical
condition one obtain the right asymptotics locally. The last step is
to average out the localisations.  Ivrii has also considered
multiscale analysis for higher order differential operators but to
treat these cases extra assumptions on the Hessian of the principal
symbol is needed see \cite{ivrii2019microlocal1,MR1974451}. There is
also another approach by Zielinski see \cite{MR2245259,MR2343462,MR2335576},
where he proves optimal Weyl laws without a non-critical condition, but
with an extra assumption on a specific phase space volume.

In the case of Riesz means sharp remainder estimates for the asymptotic expansions has been obtained by Bronstein and Ivrii in \cite{MR1974450}, for differential operators acting in $L^2(M)$, where $M$ is a compact closed manifold. However, it is mentioned in a remark that the results of the paper needs to be combined with the methods in section 4.3 in \cite{MR1631419} to obtain the reminder estimate for the Riesz means. Result on Riesz means can also be found in  \cite[Chapter 4]{ivrii2019microlocal1} by Ivrii. Again this is for differential operators acting in acting in $L^2(M)$, where $M$ is a compact closed manifold. In chapter 11 of \cite{ivrii2019microlocal1} unbounded domains are also considered but in the large eigenvalue/high energy case. In this chapter all results are stated for ``smooth'' coefficients. However, it is remarked in the start that the case of non-smooth coefficients are left to the reader.    
The results given here on the Reisz means for differential operators associated to quadratic forms are a small generalisation of previous known results. 

%The second set of main results, where we obtain sharp spectral asymptotic for admissible operators perturbed by differential operator with non smooth coefficients, have, to the best of my knowledge, not been discussed in the existing literature.    
%
%
%
%
%
%%%%%%%%%%%%%%%%%%%%%%%%%%%%%%%%%%%%%%%%%%%%%%%%%%%%%%%%%%%%%
\section{Preliminaries and notation}

\subsection{Notation and definitions}

We will here mainly set up the non standard notation used and some
definitions. We will in the following use the notation
\begin{equation*}
  \lambda(x) = (1+\abs{x}^2)^{\frac{1}{2}},
\end{equation*}
for $x$ in $\R^d$ and not the usual bracket notation. Moreover, for
more vectors $x_1,x_2,x_3$ from $\R^d$ we will use the convention
\begin{equation*}
  \lambda(x_1,x_2,x_3) = (1+\abs{x_1}^2+\abs{x_2}^2+\abs{x_3}^2)^{\frac{1}{2}},
\end{equation*}
and similar in the case of $2$ or even more vectors. For the natural numbers we use the following conventions
\begin{equation*}
  \N= \{1,2,3,\dots \} \quad\text{and}\quad \N_0 = \{0\} \cup \N.
\end{equation*}
When working with the Fourier transform we will use the following semiclassical
version for $\hbar>0$
\begin{equation*}
  \mathcal{F}_\hbar[ \varphi ](p) \coloneqq \int_{\R^d} e^{-i\hbar^{-1} \langle x,p \rangle}\varphi(x) \, dx,
\end{equation*}
and with inverse given by
\begin{equation*}
  \mathcal{F}_\hbar^{-1}[\psi] (x) \coloneqq \frac{1}{(2\pi\hbar)^d} \int_{\R^d} e^{i\hbar^{-1} \langle x,p \rangle}\psi(p) \, dp,
\end{equation*}
where $\varphi$ and $\psi$ are elements of $\mathcal{S}(\R^d)$.

We will by $\mathcal{L}(\mathcal{B}_1,\mathcal{B}_2)$ denote the linear bounded operators from the space $\mathcal{B}_1$ into $\mathcal{B}_1$ and $\mathcal{L}(\mathcal{B}_1)$ denotes the linear bounded operators from the space $\mathcal{B}_1$ into  itself. For an operator $A$ acting in a Hilbert space we will denote the spectrum of $A$ by
\begin{equation*}
	\spec(A).
\end{equation*}

As we later will need stationary phase asymptotics, we will for sake of completeness recall it here.
\begin{prop}\label{B.quad_stationary_phase}
  Let $B$ be a invertible, symmetric real $n\times n$ matrix and
  $(u,v)\rightarrow a(u,v;\hbar)$ be a function in
  $C^\infty(\R_u^{d} \times \R_v^{n})$ for all $\hbar$ in
  $(0,\hbar_0]$. We suppose $v\rightarrow a(u,v;\hbar)$ has compact
  support for all $u$ in $\R_u^{d}$ and $\hbar$ in
  $(0,\hbar_0]$. Moreover we let
  \begin{equation*}
    I(u;a,B,\hbar) = \int_{\R^{n}} e^{\frac{i}{2\hbar} \langle B v,v\rangle} a(u,v;\hbar) \, dv.
  \end{equation*}
  Then for each $N$ in $\mathbb{N}$ we have
  \begin{equation*}
    I(u;a,B,\hbar)  = (2\pi\hbar)^{\frac{n}{2}}
    \frac{e^{i\frac{\pi}{4} \sgn(B)}}{\abs{\det(B)}^{\frac12}}
    \sum_{j=0}^N \frac{\hbar^j}{j!} \Big(\frac{\langle B^{-1}
      D_v,D_v\rangle}{2i}\Big)^j a(u,v;\hbar) \Big|_{v=0} +
    \hbar^{N+1} R_{N+1}(u;\hbar),
  \end{equation*}
  where $\sgn(B)$ is the difference between the number of positive and
  negative eigenvalues of $B$. Moreover there exists a constant $c_n$
  only depending on the dimension such the error term $R_{N+1}$
  satisfies the bound
  \begin{equation*}
    \abs{R_{N+1}(u;\hbar)} \leq c_n \bigg\lVert \frac{\langle B^{-1} D_v,D_v\rangle^{N+1}}{(N+1)!}  a(u,v;\hbar)  \bigg\rVert_{H^{[\frac{n}{2}]+1}(\R_v^n)},
  \end{equation*}
  where $ [\tfrac{n}{2}]$ is the integer part of $\tfrac{n}{2}$ and
  $\norm{\cdot}_{H^{[\tfrac{n}{2}]+1}(\R_v^n)}$ is the Sobolev norm.
\end{prop}
A proof of the proposition can be found in e.g. \cite{MR897108} or
\cite{MR2952218}. For the sesquilinear forms we consider we will use the following terminology
\begin{definition}\label{B.def_non_critical}
  For a  sesquilinear form $\mathcal{A}_\hbar$ given by
   \begin{equation*}
   \mathcal{A}_\hbar[\varphi,\psi] =  \sum_{\abs{\alpha},\abs{\beta}\leq m}  \int_{\R^d} a_{\alpha\beta}(x) (\hbar D_x)^\beta\varphi(x) \overline{(\hbar D_x)^\alpha\psi(x)} \, dx, \qquad \varphi,\psi \in \mathcal{D}(   \mathcal{A}_\hbar),
  \end{equation*}
 we call a number $E$ in $\R$ non-critical if there exists $c>0$ such that
  \begin{equation*}
    |\nabla_p a_0(x,p)| \geq c \quad\text{for all } (x,p)\in a_0^{-1}(\{E\}),
  \end{equation*}
  where
  \begin{equation*}
    a_0(x,p) =  \sum_{\abs{\alpha},\abs{\beta}\leq m} a_{\alpha\beta}(x)p^{\alpha+\beta}.
  \end{equation*}
\end{definition}

\begin{definition}\label{B.def_quad_elliptic}
  A  sesquilinear form $\mathcal{A}_\hbar$ given by
   \begin{equation*}
   \mathcal{A}_\hbar[\varphi,\psi] =  \sum_{\abs{\alpha},\abs{\beta}\leq m}  \int_{\R^d} a_{\alpha\beta}(x) (\hbar D_x)^\beta\varphi(x) \overline{(\hbar D_x)^\alpha\psi(x)} \, dx, \qquad \varphi,\psi \in \mathcal{D}(   \mathcal{A}_\hbar),
  \end{equation*}
is called elliptic if there exists a strictly positive constant $C$ such that
  \begin{equation}
    \sum_{\abs{\alpha}=\abs{\beta}= m} a_{\alpha\beta}(x) p^{\alpha+\beta} \geq C \abs{p}^{2m},
  \end{equation}
 for all $(x,p)$ in $\R^d_x\times\R_p^d$.
\end{definition}

\subsection{Approximation of quadratic forms}\label{B.approx_op_section}
 We will here construct our approximating (framing) quadratic forms and prove that some properties of the original quadratic form can be inherited by the approximations. The construction is similar to the ones used in \cite{MR1974450,MR2179891,MR1631419,MR1974451,MR2105486}  to construct their framing operators. The first example of approximation non smooth coefficients in this way appeared in \cite{MR518297}, to the authors knowledge. The crucial part in this construction is  Proposition~\ref{B.smoothning_of_func}, for which a proof can be found in \cite{MR1974450, ivrii2019microlocal1}. 
\begin{prop}\label{B.smoothning_of_func}
  Let $f$ be in $C^{k,\mu}(\R^d)$ for a $\mu$ in $[0,1]$. Then for
  every $\varepsilon >0$ there exists a function $f_\varepsilon$ in
  $C^\infty(\R^d)$ such that
  \begin{equation}\label{B.smoothning_of_func_bounds}
    \begin{aligned}
      \abs{\partial_x^\alpha f_\varepsilon(x) -\partial_x^\alpha f(x)
      } \leq{}& C_\alpha \varepsilon^{k+\mu-\abs{\alpha}} \qquad
      \abs{\alpha}\leq k,
      \\
      \abs{\partial_x^\alpha f_\varepsilon(x)} \leq{}& C_\alpha
      \varepsilon^{k+\mu-\abs{\alpha}} \qquad \abs{\alpha}\geq k+1,
    \end{aligned}
  \end{equation}
where the constants is independent of $\varepsilon$. 
\end{prop}
The function $f_\varepsilon$ is a smoothing (mollification) of
$f$. Usually this is done by convolution with a compactly supported
smooth function. However here one uses a Schwartz function in the
convolution in order to ensure the stated error terms. The convolution
with a compactly supported smooth function will in most cases \enquote{only}
give an error of order $\varepsilon$.
The ideas used to construct the approximating or framing quadratic forms are of the same type as the ideas used to construct the framing operators in
\cite{MR2105486}.
\begin{prop}\label{B.construc_framing_op}
Let $\mathcal{A}_\hbar$ be an elliptic sesquilinear form given by
   \begin{equation*}
   \mathcal{A}_\hbar[\varphi,\psi] =  \sum_{\abs{\alpha},\abs{\beta}\leq m}  \int_{\R^d} a_{\alpha\beta}(x) (\hbar D_x)^\beta\varphi(x) \overline{(\hbar D_x)^\alpha\psi(x)} \, dx, \qquad \varphi,\psi \in \mathcal{D}(   \mathcal{A}_\hbar).
  \end{equation*}
 Assume the coefficients $a_{\alpha\beta}(x)$ are in $C^{k,\mu}(\R^d)$, for a pair $(k,\mu)$ in $\N\times[0,1]$, bounded from below and satisfies that $a_{\alpha\beta}(x) = \overline{a_{\beta\alpha}(x)}$ for all multi indices $\alpha$ and $\beta$. Then for all $\varepsilon>0$ there exists a pair of sesquilinear forms  $\mathcal{A}_{\hbar,\varepsilon}^{+}$ and  $\mathcal{A}_{\hbar,\varepsilon}^{-}$ such that
 \begin{equation}
 	\mathcal{A}_{\hbar,\varepsilon}^{-}[\varphi,\varphi]\leq \mathcal{A}_{\hbar}[\varphi,\varphi]\leq\mathcal{A}_{\hbar,\varepsilon}^{+}[\varphi,\varphi]
 \end{equation}
  for all $\varphi$ in $\mathcal{D}(   \mathcal{A}_\hbar)$. Moreover, for all sufficiently small $\varepsilon$ the sesquilinear forms  $\mathcal{A}_{\hbar,\varepsilon}^{\pm}$ will be elliptic. If  $E$ is a non-critical value of $ \mathcal{A}_{\hbar}$ then $E$ will also be non-critical for  $\mathcal{A}_{\hbar,\varepsilon}^{\pm}$ for all $\varepsilon$ sufficiently small. The sesquilinear forms  $\mathcal{A}_{\hbar,\varepsilon}^{\pm}$ is explicit given by
 \begin{equation}\label{B.def_approx_quad_form}
 	\begin{aligned}
   \mathcal{A}_{\hbar,\varepsilon}^{\pm}[\varphi,\psi] = {}& \sum_{\abs{\alpha},\abs{\beta}\leq m}  \int_{\R^d} a_{\alpha\beta}^\varepsilon (x) (\hbar D_x)^\beta\varphi(x) \overline{(\hbar D_x)^\alpha\psi(x)} \, dx 
   \\
   & \pm C_1 \varepsilon^{k+\mu}  \sum_{\abs{\alpha}\leq m}  \int_{\R^d} (\hbar D_x)^\alpha\varphi(x) \overline{(\hbar D_x)^\alpha\psi(x)} \, dx, \qquad \varphi,\psi \in \mathcal{D}(   \mathcal{A}_\hbar),
   	\end{aligned}
  \end{equation}
 where $a_{\alpha\beta}^\varepsilon(x)$  are the smoothed functions of $a_{\alpha\beta}(x)$ according to Proposition~\ref{B.smoothning_of_func} and $C_1$ is some positive constant.
\end{prop}
\begin{proof}
  We start by considering the form $ \mathcal{A}_{\hbar,\varepsilon}$  given by
 \begin{equation*}
   \mathcal{A}_{\hbar,\varepsilon}[\varphi,\psi] =  \sum_{\abs{\alpha},\abs{\beta}\leq m}  \int_{\R^d} a_{\alpha\beta}^\varepsilon (x) (\hbar D_x)^\beta\varphi(x) \overline{(\hbar D_x)^\alpha\psi(x)} \, dx, 
  \end{equation*}
  where we have replaced the coefficients of $ \mathcal{A}_{\hbar}$ with smooth functions made according to
  Proposition~\ref{B.smoothning_of_func}. For $\varphi$ in
  $\mathcal{D}( \mathcal{A}_{\hbar}) \cap \mathcal{D}( \mathcal{A}_{\hbar,\varepsilon})$ we
  have by Cauchy-Schwarz inequality
  \begin{equation}\label{B.diff_framing_1}
    \begin{aligned}
      |  \mathcal{A}_{\hbar}[\varphi,\varphi]
      -  \mathcal{A}_{\hbar,\varepsilon}[\varphi,\varphi] | \leq {}&
      \sum_{\abs{\alpha},\abs{\beta}\leq m}| \langle 
      (a_{\alpha\beta}-a_{\alpha\beta}^\varepsilon)(\hbar D)^\beta
      \varphi, (\hbar D)^\alpha \varphi \rangle |
      \\
      \leq{}& \sum_{\abs{\alpha},\abs{\beta}\leq m}
      \frac{1}{2\varepsilon^{k+\mu}}
      \norm{(a_{\alpha\beta}-a_{\alpha\beta}^\varepsilon)(\hbar
        D)^\beta \varphi}_{L^2(\R^d)}^2 +
      \frac{\varepsilon^{k+\mu}}{2} \norm{(\hbar D)^\alpha
        \varphi}_{L^2(\R^d)}^2
      \\
      \leq{}& c \varepsilon^{k+\mu} \sum_{\abs{\alpha}\leq m} \langle
      (\hbar D)^{\alpha} \varphi , (\hbar D)^{\alpha} \varphi\rangle,
    \end{aligned}
  \end{equation}
  where we in the last inequality have used Proposition~\ref{B.smoothning_of_func}. From this inequality we get that  $\mathcal{D}( \mathcal{A}_{\hbar})= \mathcal{D}( \mathcal{A}_{\hbar,\varepsilon})$. We recognise the last bound in \eqref{B.diff_framing_1} as the
  quadratic form associated to $(I-\hbar^2\Delta)^m$. Hence for
  sufficiently choice of constant $C_1$ we can choose the approximating forms $ \mathcal{A}_{\hbar,\varepsilon}^{\pm}$ to be given by \eqref{B.def_approx_quad_form} such that 
   \begin{equation}
 	\mathcal{A}_{\hbar,\varepsilon}^{-}[\varphi,\varphi]\leq \mathcal{A}_{\hbar}[\varphi,\varphi]\leq\mathcal{A}_{\hbar,\varepsilon}^{+}[\varphi,\varphi]
 \end{equation}
  for all $\varphi$ in $\mathcal{D}(   \mathcal{A}_\hbar)$.
  
 In order to show that we can choose the forms $ \mathcal{A}_{\hbar,\varepsilon}^{\pm}$ elliptic we note that
  \begin{equation}\label{B.diff_framing_2}
    \begin{aligned}
      \sum_{\abs{\alpha}=\abs{\beta}= m}& a_{\alpha\beta}^\varepsilon
      (x) p^{\alpha+\beta} \pm C_1\varepsilon^{k+\mu}\abs{p}^{2m}
      \\
      ={}&\sum_{\abs{\alpha}=\abs{\beta}= m}
      (a_{\alpha\beta}^\varepsilon (x) -a_{\alpha\beta} (x) )
      p^{\alpha+\beta} +\sum_{\abs{\alpha}=\abs{\beta}= m}
      a_{\alpha\beta} (x) p^{\alpha+\beta} \pm
     C_1 \varepsilon^{k+\mu}\abs{p}^{2m}
      \\
      \geq& (C-(C_1-c)\varepsilon^{k+\mu}) \abs{p}^{2m}
      \geq \tilde{C} \abs{p}^{2m},
    \end{aligned}
  \end{equation}
  for sufficiently small $\varepsilon$ and all $(x,p)$ in
  $\R_x^d\times \R_p^d$. This gives us that both forms $ \mathcal{A}_{\hbar,\varepsilon}^{\pm}$ are elliptic.  

  For the last part we assume $E$ is a non-critical value for the
  form $ \mathcal{A}_{\hbar}$. That is there exist a $c>0$ such that
  \begin{equation*}
    |\nabla_p a_0(x,p)| \geq c \quad\text{for all } (x,p)\in a_0^{-1}(\{E\}),
  \end{equation*}
  where
  \begin{equation*}
    a_0(x,p) =  \sum_{\abs{\alpha},\abs{\beta}\leq m} a_{\alpha\beta}(x)p^{\alpha+\beta}.
  \end{equation*}
  In order to prove that $E$ is a non-critical value for the framing
  forms we need to find an expression for
  $a_{\varepsilon,0}^{-1}(\{E\})$ for the framing operators, where we
  have omitted the $+$ and $-$ in the notation. By the ellipticity we
  can in the following calculations with out loss of generality assume
  $p$ belongs to a bounded set. We have
  \begin{equation}\label{B.diff_framing_3}
    \begin{aligned}
      a_{\varepsilon,0}(x,p)
      = \sum_{\abs{\alpha},\abs{\beta}\leq m}
      (a_{\alpha\beta}^\varepsilon(x) -
      a_{\alpha\beta}(x))p^{\alpha+\beta} \pm C_1
      \varepsilon^{k+\mu}(1+p^2)^m +
      \sum_{\abs{\alpha},\abs{\beta}\leq m}
      a_{\alpha\beta}(x)p^{\alpha+\beta}.
    \end{aligned}
  \end{equation}
  Since we can assume $p$ to be in a compact set we have that
  \begin{equation*}
    \Big| \sum_{\abs{\alpha},\abs{\beta}\leq m} (a_{\alpha\beta}^\varepsilon(x) - a_{\alpha\beta}(x))p^{\alpha+\beta}  \pm C_1 \varepsilon^{k+\mu}(1+p^2)^m \Big| \leq C \varepsilon^{k+\mu}. 
  \end{equation*}
  This combined with \eqref{B.diff_framing_3} implies the inclusion
  \begin{equation*}
    \{ (x,p) \in\R^{2d} \, |\,  a_{\varepsilon,0}(x,p) =E \} \subseteq \{ (x,p) \in\R^{2d} \, |\, | a_{0}(x,p) -E | \leq C \varepsilon^{k+\mu} \}. 
  \end{equation*}
  Hence for a sufficiently small $\varepsilon$ we have the inclusion
  \begin{equation}\label{B.diff_framing_4}
    \{ (x,p) \in\R^{2d} \, |\,  a_{\varepsilon,0}(x,p) =E \} \subseteq \{ (x,p) \in\R^{2d} \, |\,  \abs{\nabla_p a_{0}(x,p)} \geq \frac{c}{2} \}. 
  \end{equation}
  by continuity. For a point $(x,p)$ in
  $\{ (x,p) \in\R^{2d} \, |\, a_{\varepsilon,0}(x,p) =E \}$ we have
  \begin{equation}\label{B.diff_framing_5}
    \nabla_p a_{\varepsilon,0}(x,p)
      = \sum_{\abs{\alpha},\abs{\beta}\leq m}
      (a_{\alpha\beta}^\varepsilon(x) - a_{\alpha\beta}(x)) \nabla_p
      p^{\alpha+\beta} \pm C_1 \varepsilon^{k+\mu} \nabla_p (1+p^2)^m
      + \nabla_p a_{0}(x,p).
  \end{equation}
  Again since we can assume $p$ to be contained in a compact set we
  have
  \begin{equation*}
    \Big| \sum_{\abs{\alpha},\abs{\beta}\leq m} (a_{\alpha\beta}^\varepsilon(x) - a_{\alpha\beta}(x)) \nabla_p p^{\alpha+\beta}  \pm C_1 \varepsilon^{k+\mu}  \nabla_p (1+p^2)^m \Big| \leq C \varepsilon^{k+\mu}. 
  \end{equation*}
  Combining this with \eqref{B.diff_framing_4} and
  \eqref{B.diff_framing_5} we get
  \begin{equation}
    |\nabla_p a_{\varepsilon,0}(x,p)| \geq  |\nabla_p a_{0}(x,p)| - C \varepsilon^{k+\mu} \geq \frac{c}{2} -   C \varepsilon^{k+\mu} \geq \frac{c}{4}.
  \end{equation}
  where the last inequality is for $\varepsilon$ sufficiently
  small. This inequality proves $E$ is also a non-critical value of
  the framing forms.
\end{proof}
The framing forms constructed in the previous proposition are
forms with smooth coefficients. But when we take derivatives of
these coefficients we start to get negative powers of $\varepsilon$
from some point. Hence we can not associate a ``classic''  pseudo-differential
operator to the form. We will instead  in the following sections consider a ``rough'' theory for pseudo-differential
operators. Here we will  see that it is in fact possible to
verify most of the results from classic theory of pseudo-differential
operators in this rough theory. After this has been developed we will return to these
framing forms.

As a last remark of this section note that there is not a unique way to construct these framing forms.

%%%%%%%%%%%%%%%%%%%%%%%%%%%%%%%%%%%%%%%%%%%%%%%%%%%%%%%%%%%%%

\section{Definitions and properties of rough pseudo-differential
  operators}\label{Sec:def_rough_op}
In this section we will inspired by the approximation results in the
previous section define a class of pseudo-differential operators with
rough symbols. We will further state and prove some of the properties for these
operators. The definitions are very
similar to the definitions in the monograph \cite{MR897108} and we will see that the properties of these operators can be deduced from the results in the monograph \cite{MR897108}. 

\subsection{Definition of rough operators and rough symbols}
We start for the sake of completeness by recalling the definition of a tempered weight
function.
\begin{definition}
  A tempered weight function on $\R^D$ is a continuous function
  \begin{equation*}
    m:\R^D \rightarrow [0,\infty[, 
  \end{equation*}
  for which there exists positive constants $C_0$, $N_0$ such that for
  all points $x_1$ in $\R^D$ the estimate
  \begin{equation*}
    m(x) \leq C_0 m(x_1) \lambda(x_1-x)^{N_0},
  \end{equation*}
  holds for all points $x$ in $\R^D$.
\end{definition}
For our purpose here we will consider the cases where $D=2d$ or
$D=3d$.  These
types of functions are in the literature sometimes called order
functions. This is the case in the monographs
\cite{MR1735654,MR2952218}. But we have chosen the name tempered
weights to align with the terminology in the monographs
\cite{MR897108,MR2304165}. We can now define the symbols we will be working with.
\begin{definition}[Rough symbol]\label{B.def_rough_symbol}
  Let $\Omega \subseteq \R_x^d \times \R_p^d \times \R_y^d$ be open,
  $\rho$ be in $[0,1]$, $\varepsilon>0$, $\tau$ be in $\Z$ and $m$ a
  tempered weight function on $\R_x^d \times \R_p^d \times \R_y^d$. We
  call a function $a_\varepsilon$ a rough symbol of regularity $\tau$
  with weights $(m,\rho, \varepsilon)$ if $a_\varepsilon$ is in
  $C^{\infty}(\Omega)$ and satisfies that
  \begin{equation}\label{B.symbolb_est_def}
  	\begin{aligned}
  |\partial_x^\alpha\partial_p^\beta \partial_y^\gamma
    a_\varepsilon(x,p,y)|
    \leq
      C_{\alpha\beta\gamma}
      \varepsilon^{\min(0,\tau-\abs{\alpha}-\abs{\gamma})} m (x,p,y)
      \lambda(x,p,y)^{-\rho (\abs{\alpha}+\abs{\beta}+\abs{\gamma})} 
    \end{aligned}
  \end{equation}
  for all $(x,p,y)$ in $\Omega$ and $\alpha$, $\beta$, $\gamma$ in
  $\N^d$, where the constants $C_{\alpha\beta\gamma}$'s do not depend on $\varepsilon$. The space of these functions is denoted
  $\Gamma_{\rho,\varepsilon}^{m,\tau} \left(\Omega\right)$.
\end{definition}
\begin{remark}
The space $\Gamma_{\rho,\varepsilon}^{m,\tau} \left(\Omega\right)$ can be turned into a Fr\'{e}chet space with semi norms associated to the estimates in \eqref{B.symbolb_est_def}.
It is important to note that the semi norms on  $\Gamma_{\rho,\varepsilon}^{m,\tau} \left(\Omega\right)$ should be chosen weighted such that the norms associated to a set of numbers $\alpha,\beta,\gamma$ will be bounded by the constant $C_{\alpha\beta\gamma}$ and hence independent of $\varepsilon$.

  If $\varepsilon$ is equal to $1$, then these symbols are the same as
  the symbols defined in the monograph~\cite{MR897108}
  (Definition~{II-$10$}).  We will always assume $\varepsilon\leq1$ as we are interested in the cases of very small $\varepsilon$.
  
  We will later call a function $a_\varepsilon(x,p)$ or $b_\varepsilon(p,y)$ a
  rough symbol if it satisfies the above definition in the two
  variables $x$ and $p$ or $p$ and $y$. This more general definition
  is made in order to define the different forms of quantisation and
  the interpolation between them.
  If we say a symbol of regularity $\tau$ with tempered weight $m$ we implicit assume that $\rho=0$. This type of rough symbols is contained in the class of rough
  symbols consider in \cite[Section 2.3 and
  4.6]{ivrii2019microlocal1}.
\end{remark}
\begin{remark}\label{B.assumption_epsilon}
	We will later assume that a rough symbol is a tempered weight. When this is done we will implicit assume that the constants from the definition of a tempered weight is independent of $\varepsilon$. This is an important assumption since we need the estimates we make to be uniform for $\hbar$ in $(0,\hbar_0]$ with $\hbar_0>0$ sufficiently small and then for a choice of $\delta$ in $(0,1)$ we need the estimates to be uniform for $\varepsilon$ in $[\hbar^{1-\delta},1]$.
	
	Essentially the constants will be uniform for both  $\hbar$ in $(0,\hbar_0]$ and $\varepsilon$ in $(0,1]$, but if $\varepsilon\leq\hbar$ then the estimates will diverge in the semiclassical parameter. Hence we will assume the lower bound on $\varepsilon$.
The assumption that $\varepsilon\geq\hbar^{1-\delta}$ is in \cite{ivrii2019microlocal1,MR1631419} called a microlocal uncertainty principal. In \cite{ivrii2019microlocal1,MR1631419} there is two parameter instead of just one. This other parameter can be used to scale in the $p$-variable.    
\end{remark}

As we are interested in asymptotic expansions in the semiclassical
parameter we will define $\hbar$-$\varepsilon$-admissible symbols,
which is the symbols depending on the semiclassical parameter $\hbar$
for which we can make an expansion in $\hbar$.
\begin{definition}\label{B.def_admis_sym}
  With the notation from Definition~\ref{B.def_rough_symbol}. We call a symbol $a_\varepsilon(\hbar)$
  $\hbar$-$\varepsilon$-admissible of regularity $\tau$ with weights
  $(m,\rho, \varepsilon)$ in $\Omega$, if for fixed $\varepsilon$ and a $\hbar_0>0$ the map that takes $\hbar$
  into $a_\varepsilon(\hbar)$ is smooth from $(0,\hbar_0]$ into
  $\Gamma_{\rho,\varepsilon}^{m,\tau} \left(\Omega\right)$ such that
  there exists a $N_0$ in $\N$ such for all $N\geq N_0$ we have
  \begin{equation*}
    a_\varepsilon(x,p,y;\hbar) = a_{\varepsilon,0}(x,p,y) + \hbar a_{\varepsilon,1}(x,p,y) + \dots + \hbar^N a_{\varepsilon,N}(x,p,y) + \hbar^{N+1} r_N(x,p,y;\hbar),
  \end{equation*}
  where $a_{\varepsilon,j}$ is in
  $\Gamma_{\rho,\varepsilon,-2j}^{m,\tau_j} \left(\Omega\right)$ with
  the notation $\tau_j=\tau-j$ and $r_N$ is a symbol satisfying the
  bounds
  \begin{equation*} \begin{aligned}
    \hbar^{N+1} |\partial_x^\alpha& \partial_p^\beta \partial_y^\gamma
    r_N(x,p,y;\hbar)|
    \\
    \leq{}& C_{\alpha\beta\gamma} \hbar^{\kappa_1(N)}
    \varepsilon^{-\abs{\alpha}-\abs{\gamma}
    }m(x,y,p) \lambda(x,y,p)^{-\rho(\kappa_2(N)+
      \abs{\alpha}+\abs{\beta} +\abs{\gamma})},
   \end{aligned} \end{equation*}
  where $\kappa_1$ is a positive strictly increasing function and
  $\kappa_2$ is non-decreasing function. For $k$ in $\Z$
  $\Gamma_{\rho,\varepsilon,k}^{m,\tau} \left(\Omega\right)$ is the space
  of rough symbols of regularity $\tau$ with weights
  $(m(1+|(x,y,p)|)^{k\rho},\rho, \varepsilon)$.
\end{definition}
\begin{remark}
  We will also use the terminology $\hbar$-$\varepsilon$-admissible
  for symbols in two variables, where the definition is the same just
  in two variables. This definition is slightly different to the
  \enquote{usual} definition of an $\hbar$-admissible symbol
  \cite[Definition~{II-$11$}]{MR897108}. One
  difference is in the error term. 
  %Here is the fist sign of error terms getting small in the semiclassical parameter but not as fast as in the non-rough case. 
  The functions $\kappa_1$ and $\kappa_2$
  will in most cases be dependent on the tempered weight function through the constants in the definition of a tempered weight, the
  regularity $\tau$ and the dimension $d$.  It should be noted that
  the function $\kappa_2$ might be constant negative.
\end{remark}
We will now define the pseudo-differential operators associated to the
rough symbols. We will call them rough pseudo-differential operators.
\begin{definition}\label{B.pseudo-differential-operator-def}
  Let $m$ be a tempered weight function on
  $\R^d_x\times\R^d_p\times\R_y^d$, $\rho$ in $[0,1]$, $\varepsilon>0$
  and $\tau$ in $\Z$. For a rough symbol $a_\varepsilon$ in
  $\Gamma_{\rho,\varepsilon}^{m,\tau}(\R^d_x\times\R^d_p\times\R_y^d)$
  we associate the operator $\Op(a_\varepsilon)$ defined by
  \begin{equation*}
    \Op(a_\varepsilon)\psi(x) = \frac{1}{(2\pi \hbar)^{d}} \int_{\R^{2d}} e^{i \hbar^{-1} \langle{x-y},{p}\rangle } 
    a_\varepsilon( x,p,y) \psi(y) \, d y \, d p,
  \end{equation*}
  for $\psi$ in $\mathcal{S}(\R^d)$.
\end{definition}
\begin{remark}
  With the notation from
  Definition~\ref{B.pseudo-differential-operator-def}. We remark that
  the integral in the definition of $\Op(a_\varepsilon)\psi(x)$ shall
  be considered as an oscillating integral. By applying the techniques
  for oscillating integrals it can be proven that $\Op(a_\varepsilon)$
  is a continuous linear operator from $\mathcal{S}(\R^d)$ into
  itself. The proof of this is analogous to the proof in
  \cite{MR897108} in the non-rough case. By duality it can
  also be defined as an operator from $\mathcal{S}'(\R^d)$ into
  $\mathcal{S}'(\R^d)$. \end{remark}
\begin{definition}
  We call an operator $A_\varepsilon(\hbar)$ from
  $\mathcal{L}(\mathcal{S}(\R^d),L^2(\R^d))$
  $\hbar$-$\varepsilon$-admissible of regularity $\tau \geq0$ with
  tempered weight $m$ if for fixed $\varepsilon$ and a $\hbar_0>0$ the map
  \begin{equation*}
    A_\varepsilon: (0,\hbar_0] \rightarrow\mathcal{L}(\mathcal{S}(\R^d),L^2(\R^d))
  \end{equation*}
  is smooth and there exists a sequence $a_{\varepsilon,j}$ in
  $\Gamma_{0,\varepsilon}^{m,{\tau_j}}(\R^d_x\times\R^d_p\times\R_y^d)$,
  where $\tau_0=\tau$ and $\tau_{j+1}= \tau_j-1$ and a sequence $R_N$
  in $\mathcal{L}(L^2(\R^d))$ such that for $N\geq N_0$, $N_0$
  sufficient large,
  \begin{equation}\label{B:op_for_prin_sym}
    A_\varepsilon(\hbar)=\sum_{j=0}^N \hbar^j  \Op( a_{\varepsilon,j}) + \hbar^{N+1}  R_N(\varepsilon,\hbar),
  \end{equation}
  and
  \begin{equation*}
    \hbar^{N+1}  \norm{R_N(\varepsilon,\hbar)}_{\mathcal{L}(L^2(\R^d))} \leq \hbar^{\kappa(N)} C_N ,
  \end{equation*}
  for a strictly positive increasing function $\kappa$.
\end{definition}
\begin{remark}
  By the results in Theorem~\ref{B.cal-val-thm} we have that if the
  tempered weight function $m$ is in $L^\infty(\R^d)$. Then for a
  $\hbar$-$\varepsilon$-admissible symbol $a_\varepsilon(\hbar)$ of
  regularity $\tau\geq0$ with tempered weight $m$ the operator
  $A_\varepsilon(\hbar)=\Op(a_\varepsilon(\hbar))$ is a
  $\hbar$-$\varepsilon$-admissible operator of regularity $\tau$.
\end{remark}
\begin{remark}
When we have an operator $  A_\varepsilon(\hbar)$  with an expansion 
  \begin{equation*}
    A_\varepsilon(\hbar)=\sum_{j\geq0} \hbar^j  \Op( a_{\varepsilon,j}),
  \end{equation*}
  where the sum is understood as a formal sum and in the sense that for all $N$ sufficiently large there exists $R_N$ in $\mathcal{L}(L^2(\R^d))$ such that the operator is of the same form as in \eqref{B:op_for_prin_sym}. Then we call the symbol $a_{\varepsilon,0}$ the principal symbol and the symbol $a_{\varepsilon,1}$ the subprincipal symbol. 
\end{remark}

\begin{definition}
  Let $A_\varepsilon(\hbar)$ be a $\hbar$-$\varepsilon$-admissible of
  regularity $\tau$ with tempered weight $m$. For any $t$ in $[0,1]$
  we call all $\hbar$-$\varepsilon$-admissible symbols
  $b_\varepsilon(\hbar)$ in
  $\Gamma_{0,\varepsilon}^{m,\tau}(\R^d_x\times\R^d_p)$ such,
  \begin{equation*}
    A_\varepsilon(\hbar) \psi(x) =  \frac{1}{(2\pi \hbar)^{d}} \int_{\R^{2d}} e^{i \hbar^{-1} \langle{x-y},{p}\rangle } b_\varepsilon( (1-t)x +ty,p;\hbar) \psi(y) \, d y \, d p,
  \end{equation*}
  for all $\psi\in\mathcal{S}(\R^d)$ and all $\hbar\in]0,\hbar_0]$, where $\hbar_0$ is a strictly positive number,
  rough $t$-$\varepsilon$-symbols of regularity $\tau$ associated to
  $A_\varepsilon(\hbar)$.
\end{definition}
\begin{notation}
  In general for a symbol $b_\varepsilon(\hbar)$ in
  $\Gamma_{\rho,\varepsilon}^{m,\tau}(\R^d_x\times\R^d_p)$ and $\psi$
  in $\mathcal{S}(\R^d)$ we will use the notation
  \begin{equation*}
    \Opt(b_\varepsilon)\psi(x) = \frac{1}{(2\pi\hbar)^d} \int_{\R^{2d}}  e^{i\hbar^{-1} \langle x-y,p\rangle} b_\varepsilon( (1-t)x +ty,p;\hbar) \psi(y) \, d y \, d p
  \end{equation*}
  We have the special case of Weyl quantisation when $t=\frac12$,
  which is the one we will work the most with. In this case we write
  \begin{equation*}
    \OpN{\frac12}(b_\varepsilon) = \OpW(b_\varepsilon).
  \end{equation*}
\end{notation}
For some application we will need stronger assumptions than
$\hbar$-$\varepsilon$-admissibility of our operators. The operators
satisfying these stronger assumptions will be called strongly
$\hbar$-$\varepsilon$-admissible operators with some regularity. As an
example we could consider a symbol $a_\varepsilon(x,p)$ in
$\Gamma_{\rho,\varepsilon}^{m,\tau} (\R_x^{d}\times\R^d_p)$. For this
symbol define
$\tilde{a}_\varepsilon(x,p,y) =a_\varepsilon(tx+(1-t)y,p) $ and ask if
this symbol is in
$\Gamma_{\rho,\varepsilon}^{\tilde{m},\tau}
(\R_x^{d}\times\R^d_p\times\R^d_y)$, where
$\tilde{m}(x,p,y)=m(tx+(1-t)y,p)$. The answer will not in general be
positive. Hence in general we can not ensure decay in the variables
$(x,p,y)$ when viewing a function of $(x,p)$ as a function of
$(x,p,y)$. With this in mind we define a new class of symbols and
strongly $\hbar$-$\varepsilon$-admissible operators.
\begin{definition}
  A symbol $a_\varepsilon$ belongs to the class
  $\tilde{\Gamma}_{\rho,\varepsilon}^{m,\tau}
  (\R_x^{d}\times\R^d_p\times\R^d_y)$ if $a_\varepsilon$ is in
  $\Gamma_{0,\varepsilon}^{m,\tau} (\R_x^{d}\times\R^d_p\times\R^d_y)$
  and there exists a positive $\nu$ such that
  \begin{equation*}
    a_\varepsilon \in \Gamma_{\rho,\varepsilon}^{m,\tau} (\Omega_\nu),
  \end{equation*}
  where $\Omega_\nu=\{(x,p,y)\in \R^{3d} \,|\, \abs{x-y}<\nu\}$.
\end{definition}
\begin{definition}
  We call the family of operators
  $A_\varepsilon(\hbar)=\Op(a_\varepsilon(\hbar))$ strongly
  $\hbar$-$\varepsilon$-admissible of regularity $\tau$ if
  $a_\varepsilon(\hbar)$ is an $\hbar$-$\varepsilon$-admissible symbol
  of regularity $\tau$ with respect to the weights $(m,0,\varepsilon)$
  on $ \R_x^{d}\times\R^d_p\times\R^d_y$ and the weights
  $(m,\rho,\varepsilon)$ on
  $\Omega_\nu=\{(x,p,y)\in \R^{3d} \,|\, \abs{x-y}<\nu\}$ for a
  positive $\nu$.
\end{definition}
\begin{remark}
  We note that a strongly $\hbar$-$\varepsilon$-admissible
  operator is also $\hbar$-$\varepsilon$-admissible but as a
  consequence of the definition the error term of a strongly
  $\hbar$-$\varepsilon$-admissible operator will be a
  pseudo-differential operator and not just a bounded operator as for
  the $\hbar$-$\varepsilon$-admissible operators.
\end{remark}

Before we start proving/stating results about these operators we make the following observation.

\begin{obs}\label{obs_connection_symbol}
	 Let $m$ be a tempered weight function on
  $\R^d_x\times\R^d_p\times\R_y^d$, $\rho$ in $[0,1]$, $0<\varepsilon\leq1$
  and $\tau$ in $\N_0$. Consider  a rough symbol $a_\varepsilon$ in
  $\Gamma_{\rho,\varepsilon}^{m,\tau}(\R^d_x\times\R^d_p\times\R_y^d)$. We suppose that there is a $\delta$ in $(0,1)$ such $\varepsilon \geq \hbar^{1-\delta}$ and consider the operator $\Op(a_\varepsilon)$ associated with $a_\varepsilon$. We define the unitary dilation operator $\mathcal{U}_{\varepsilon}$ by
  \begin{equation*}
  	\mathcal{U}_{\varepsilon} f(x) = \varepsilon^{d/2} f(\varepsilon x),
  \end{equation*}
  for $f$ in $L^2(\R^d)$. We observe that with this operator we obtain the following equality
    \begin{equation*}
  	\mathcal{U}_{\varepsilon} \Op(a_\varepsilon(x,p,y)) \mathcal{U}_{\varepsilon}^{*} = \Opdef_{\hbar^{\delta}}(a_\varepsilon(\varepsilon x, \tfrac{\hbar^{1-\delta}}{\varepsilon} p,\varepsilon y)).
  \end{equation*}
 Set
 \begin{equation*}
 	a_{\varepsilon}^{\#}(x,p,y) = a_\varepsilon(\varepsilon x, \tfrac{\hbar^{1-\delta}}{\varepsilon} p,\varepsilon y).
 \end{equation*}
 Since we have that $a_\varepsilon$ is in $\Gamma_{\rho,\varepsilon}^{m,\tau}(\R^d_x\times\R^d_p\times\R_y^d)$ we get for all $\alpha, \beta, \gamma$ in $\N^d_0$ that
   \begin{equation*}
  	\begin{aligned}
 	\MoveEqLeft  |\partial_x^\alpha\partial_p^\beta \partial_y^\gamma a_{\varepsilon}^{\#}|
 	\\
    &\leq
      C_{\alpha\beta\gamma} \varepsilon^{\min(0,\tau-\abs{\alpha}-\abs{\gamma})+\abs{\alpha}+\abs{\gamma}} \left(\tfrac{\hbar^{1-\delta}}{\varepsilon}\right)^{\abs{\beta}}  m (\varepsilon x, \tfrac{\hbar^{1-\delta}}{\varepsilon} p,\varepsilon y)\lambda(\varepsilon x, \tfrac{\hbar^{1-\delta}}{\varepsilon} p,\varepsilon y)^{-\rho (\abs{\alpha}+\abs{\beta}+\abs{\gamma})} 
      \\
      &\leq
      C_{\alpha\beta\gamma}C_0 \varepsilon^{\min(0,\tau-\abs{\alpha}-\abs{\gamma})+\abs{\alpha}+\abs{\gamma}} \left(\tfrac{\hbar^{1-\delta}}{\varepsilon}\right)^{\abs{\beta}}  m ( x,  p, y)\lambda((1-\varepsilon) x,(1- \tfrac{\hbar^{1-\delta}}{\varepsilon}) p,(1-\varepsilon) y)^{N_0} 
        \\
        &\leq  \tilde{C}_{\alpha\beta\gamma}   \tilde{m} (x,p,y), 
    \end{aligned}
  \end{equation*}
 where we have used that  $  \hbar^{1-\delta} \leq \varepsilon\leq1$ and used the notation $\tilde{m} = m\lambda^{N_0}$. This shows that $a_{\varepsilon}^{\#}$ is in $\Gamma_{0,1}^{\tilde{m},0}(\R^d_x\times\R^d_p\times\R_y^d)$ by definition of the set. As mentioned earlier this is the type of symbols usually considered in the literature.
\end{obs}
In what follows we will use this observation to establish the symbolic calculus for the rough pseudo-differential operators. We will here mainly focus on the  results we will need later in the proofs of our main theorems.  

We will now prove a connection between operators with symbols in the
class
$\tilde{\Gamma}_{\rho,\varepsilon}^{m,\tau}
(\R_x^{d}\times\R^d_p\times\R^d_y)$ and $t$-quantised operators.
\begin{thm}\label{B.connection_quantisations}
  Let $a_\varepsilon$ be a symbol in
  $\tilde{\Gamma}_{\rho,\varepsilon}^{m,\tau}
  (\R_x^{d}\times\R^d_p\times\R^d_y)$ of regularity $\tau\geq0$ with
  weights $(m,\rho,\varepsilon)$ and
  \begin{equation*}
    A_\varepsilon(\hbar) \psi(x) = \frac{1}{(2\pi\hbar)^d} \int_{\R^{2d}} e^{i\hbar^{-1}\langle x-y , p \rangle} a_\varepsilon(x,p,y) \psi(y) \, dy \,dp.
  \end{equation*}	
  We suppose there is a $\delta$ in $(0,1)$ such
  $\varepsilon \geq \hbar^{1-\delta}$. Then for every $t$ in $[0,1]$
  we can associate a unique $t$-$\varepsilon$-symbol $b_t$ of
  regularity $\tau$ with weights $(\tilde{m},\rho,\varepsilon)$, where
  $\tilde{m}(x,p)=m(x,x,p)$. The $t$-$\varepsilon$-symbol $b_t$ is
  defined by the oscillating integral
  \begin{equation*}
    b_t(x,p,\hbar) = \frac{1}{(2\pi\hbar)^d}  \int_{\R^{2d}}  e^{i\hbar^{-1}\langle u,q \rangle} a_\varepsilon (x+tu,p+q,x-(1-t)u) \, dq \,du 
  \end{equation*}
  and symbol $b_t$ has the following asymptotic expansion
  \begin{equation*}
    b_t(x,p;\hbar) = \sum_{j=0}^N \hbar^j a_{\varepsilon,j}(x,p) + \hbar^{N+1} r_{\varepsilon,N+1}(x,p;\hbar),
  \end{equation*}
  where
  \begin{equation*}
    a_{\varepsilon,j} (x,p) =  \frac{(-i)^j}{j!} \langle D_u , D_p\rangle^j a_\varepsilon(x+tu, p,x-(1-t)u) \Big|_{u=0},
  \end{equation*}
  and the error term satisfies that
  \begin{equation*}
    \hbar^{N+1} |\partial_x^\alpha\partial_p^\beta r_{N+1}(x,p,\hbar) | 
    \leq   C_{d,N,\alpha,\beta}  \hbar^{N +1} \varepsilon^{-(\tau-N-2-d-\abs{\alpha})_{-}}  m(x,p,x)  \lambda(x,p)^{\rho N_0 } ,
  \end{equation*}
  for all $\alpha$ and $\beta$ in $\N^d$. In particular we have that
  \begin{equation*} \begin{aligned}
    a_{\varepsilon,0}(x,p) &= a_\varepsilon(x,p,x)
    \\
    a_{\varepsilon,1} (x,p) &= (1-t) (\nabla_y D_p
    a_\varepsilon)(x,p,x) - t (\nabla_x D_p a_\varepsilon) (x,p,x).
   \end{aligned} \end{equation*}
\end{thm}
\begin{remark}
  It can be noted
  that in order for the error term not to diverge, when
  the semiclassical parameter tends to zero, one needs to take $N$
  such that
  \begin{equation*}
    \tau -1-d + \delta(N+2+d) \geq 0.
  \end{equation*}	
  If the symbol is a polynomial in one of the variables or both then
  the asymptotic expansion will be exact and a finite sum. This is in
  particular the case when \enquote{ordinary} differential operators are
  considered.
\end{remark}
\begin{proof}
	Let $\mathcal{U}_{\varepsilon}$ be the unitary dilation operator as defined in Observation~\ref{obs_connection_symbol} and define the operator $ A_\varepsilon^{\#}(\hbar)$ by conjugation with $\mathcal{U}_{\varepsilon}$. This new operator is defined by 
  \begin{equation*}
   A_\varepsilon^{\#}(\hbar) \psi(x) = \frac{1}{(2\pi\hbar^{\delta})^d} \int_{\R^{2d}} e^{i\hbar^{-\delta}\langle x-y , p \rangle} a_\varepsilon^{\#}(x,p,y) \psi(y) \, dy \,dp,
  \end{equation*}	
where $\psi$ is a Schwartz function.	From Observation~\ref{obs_connection_symbol} we have that $a_\varepsilon^{\#}$ is in $\tilde{\Gamma}_{0,1}^{m,0}(\R^d_x\times\R^d_p\times\R_y^d)$. So by \cite[Theorem II-27]{MR897108} the result is true for the operator $ A_\varepsilon^{\#}(\hbar)$. Due to the identity
  \begin{equation*}
   A_\varepsilon(\hbar)  = \mathcal{U}_{\varepsilon}^{*} A_\varepsilon^{\#}(\hbar) \mathcal{U}_{\varepsilon},
  \end{equation*}
what remains is to conjugate the terms in the representation of  $A_\varepsilon^{\#}(\hbar)$ by $\mathcal{U}_{\varepsilon}^{*}$. First we observe that
  \begin{equation}\label{proof_connection_t_sym_1}
  	\begin{aligned}
    \MoveEqLeft b_t^{\#}(\varepsilon^{-1} x, \tfrac{\varepsilon}{\hbar^{1-\delta}}  p;\hbar^{\delta}) 
    \\
    ={}& \frac{1}{(2\pi\hbar^{\delta})^d}  \int_{\R^{2d}}   e^{i\hbar^{-\delta}\langle u,q \rangle} a_\varepsilon^{\#} (\varepsilon^{-1} x+tu,\tfrac{\varepsilon}{\hbar^{1-\delta}}p+q,\varepsilon^{-1}x-(1-t)u) \, dq \,du 
    \\
     ={}& \frac{1}{(2\pi\hbar^{\delta})^d}  \int_{\R^{2d}}  e^{i\hbar^{-\delta}\langle u,q \rangle} a_\varepsilon (x+t \varepsilon u,p+ \tfrac{\hbar^{1-\delta}}{\varepsilon}q,x-(1-t) \varepsilon u) \, dq \,du 
      \\
     ={}& \frac{1}{(2\pi\hbar)^d}   \int_{\R^{2d}}  e^{i\hbar^{-1}\langle u,q \rangle} a_\varepsilon (x+t u,p+ q,x-(1-t)  u) \, dq \,du
     \\
     ={}&    b_t(x,p,\hbar).
    \end{aligned}
  \end{equation}
Next  by using the identity obtained in \eqref{proof_connection_t_sym_1} we get that
  \begin{equation}\label{proof_connection_t_sym_1.5}
	\begin{aligned}
  	 A_\varepsilon(\hbar)  = {}& \mathcal{U}_{\varepsilon}^{*} A_\varepsilon^{\#}(\hbar) \mathcal{U}_{\varepsilon} 
	 = \mathcal{U}_{\varepsilon}^{*}   \Opdef_{\hbar^{\delta},t}(b_t^{\#}) \mathcal{U}_{\varepsilon} 
	 \\
	  ={}&  \Opt (b_t^{\#}(\varepsilon^{-1} x, \tfrac{\varepsilon}{\hbar^{1-\delta}}  p;\hbar^{\delta}) ) = \Opt (b_t)
   	\end{aligned}
  \end{equation}
From  the  asymptotic expansion of $ b_t^{\#}(x,p;\hbar^\delta) $ obtained in \cite[Theorem II-27]{MR897108}  we have that 
  \begin{equation}\label{proof_connection_t_sym_2}
    b_t^{\#}(x,p;\hbar^\delta) = \sum_{j=0}^N (\hbar^\delta)^j a_{\varepsilon,j}^{\#}(x,p) + (\hbar^\delta)^{N+1} r_{\varepsilon,N+1}^{\#}(x,p;\hbar^{\delta}).
  \end{equation}
In order to arrive at the stated expansion of $ b_t(x,p,\hbar)$ we have to find an expression of $a_{\varepsilon,j}^{\#}(\varepsilon^{-1}x,\tfrac{\varepsilon}{\hbar^{1-\delta}} p)$ in terms of $a_\varepsilon$. By definition of $a_{\varepsilon,j}^{\#}$ we have that
 \begin{equation} \label{proof_connection_t_sym_3}
    a_{\varepsilon,j}^{\#} (\varepsilon^{-1}x, \tfrac{\varepsilon}{\hbar^{1-\delta}}p) =  \frac{(-i)^j}{j!} \hbar^{(1-\delta) j}   \langle D_u , D_p\rangle^j a_\varepsilon ( x+t u, p, x-(1-t)u) \Big|_{u=0}.
  \end{equation}
What remains is to prove that the error term follows the desired estimate. The error term is give by
  \begin{equation} \label{proof_connection_t_sym_4}
  r_{\varepsilon,N+1}(x,p;\hbar) = r_{\varepsilon,N+1}^{\#}(\varepsilon^{-1}x,\tfrac{\varepsilon}{\hbar^{1-\delta}}p;\hbar^{\delta}).
  \end{equation}
In order to see that this function satisfied the stated estimate one needs to go back to the proof of \cite[Theorem II-27]{MR897108} and consider the exact definition of the function. But from here it is a straight forward argument to see that that the desired bound is indeed true. Combining \eqref{proof_connection_t_sym_1}, \eqref{proof_connection_t_sym_1.5} \eqref{proof_connection_t_sym_2}, \eqref{proof_connection_t_sym_3} and \eqref{proof_connection_t_sym_4} we obtain the desired result.

Alternatively one can also follow the proof in \cite[Theorem II-27]{MR897108} and do the full stationary phase argument. This gives a slightly longer proof but one obtains the same results.
\end{proof}
From this Theorem we immediate obtain the following Corollary.
\begin{corollary}\label{B.connection_t_quantisations}
  Let $t_1$ be in $[0,1]$ and $b_{t_1}$ be a
  $t_1$-$\varepsilon$-symbol of regularity $\tau\geq0$ with weights
  $(m,\rho,\varepsilon)$ and suppose $\varepsilon\geq\hbar^{1-\delta}$
  for a $\delta$ in $(0,1)$. Let $A_\varepsilon(\hbar)$ be the
  associated operator acting on a Schwarzt function by the formula
  \begin{equation*}
    A_\varepsilon(\hbar) \psi(x) = \frac{1}{(2\pi\hbar)^d}  \int_{\R^{2d}} e^{i\hbar^{-1}\langle x-y, p\rangle} b_{t_1}((1-t_1)x+t_1y,p) \psi(y) \,dy \,dp.
  \end{equation*}
  Then for every $t_2$ in $[0,1]$ we can associate an admissible
  $t_2$-$\varepsilon$-symbol given by the expansion
  \begin{equation*}
    b_{t_2}(\hbar) = \sum_{j=0}^N \hbar^j b_{t_2,j}+  \hbar^{N+1} r_{\varepsilon,N+1}(x,p;\hbar),
  \end{equation*}
  where
  \begin{equation*}
    b_{t_2,j}(x,p) = \frac{( t_1-t_2)^j}{j!} (\nabla_x D_p)^j b_{t_1}(x,p),
  \end{equation*}
  and the error term satisfies that
  \begin{equation*}
    \hbar^{N+1} |\partial_x^\alpha\partial_p^\beta r_{N+1}(x,p,\hbar) | 
    \leq   C_{d,N,\alpha,\beta}  \hbar^{N +1} \varepsilon^{-(\tau-N-2-d-\abs{\alpha})_{-}}  m(x,p)  \lambda(x,p)^{\rho N_0 }, 
  \end{equation*}
  for all $\alpha$ and $\beta$ in $\N^d$, the number $N_0$ is the
  number connected to the tempered weight $m$.
\end{corollary}
This corollary can also be proven directly by considering the kernel
as an oscillating integral and the integrant as a function in the
variable $t_1$. To obtain the corollary do a Taylor expansion in $t_1$
at the point $t_2$, then preform integration by parts a number of times and
then one would recover the result.
\subsection{Composition of rough pseudo-differential operators}
With the rough pseudo-differential operators defined and the ability
to interpolate between the different quantisations our next aim is
results concerning composition of rough pseudo-differential
operators. This is done in the following theorem. We will here omit the proof as it is analogous to the proof of Theorem~\ref{B.connection_quantisations}. The idea is to conjugate the operators with $\mathcal{U}_\varepsilon$ use the results from e.g. \cite{MR897108} and then conjugate with $\mathcal{U}_\varepsilon^{*}$.
\begin{thm}\label{B.composition-t-symbol-thm}
  Let $A_\varepsilon(\hbar)$ and $B_\varepsilon(\hbar)$ be two
  t-quantised operators given by
  \begin{equation*}
    A_\varepsilon(\hbar) \psi(x) = \frac{1}{(2\pi\hbar)^d}  \int_{\R^{2d}}e^{i\hbar^{-1}\langle x-z, p\rangle} a_\varepsilon((1-t)x+tz,p) \psi(z) \,dz \,dp
  \end{equation*}
  and
  \begin{equation*}
    B_\varepsilon(\hbar) \psi(z) = \frac{1}{(2\pi\hbar)^d} \int_{\R^{2d}}e^{i\hbar^{-1}\langle z-y, q\rangle} b_\varepsilon((1-t)z+ty,q) \psi(y) \,dy \,dq,
  \end{equation*}
  where $a_\varepsilon$ and $b_\varepsilon$ be two rough symbols of
  regularity $\tau_1,\tau_2 \geq 0$ with weights
  $(m_1, \rho , \varepsilon)$ and $(m_2, \rho , \varepsilon)$
  respectively. We suppose there exists a number $\delta\in(0,1)$ such that
  $\varepsilon\geq\hbar^{1-\delta}$.  Then the operator
  $C_\varepsilon(\hbar) = A_\varepsilon(\hbar) \circ
  B_\varepsilon(\hbar)$ is strongly $\hbar$-$\varepsilon$-admissible
  and $C_ \varepsilon(\hbar) = \Opt(c_\varepsilon)$, where
  $c_\varepsilon$ is a rough admissible symbol of regularity
  $\tau = \min(\tau_1,\tau_2)$ with weights
  $(m_1m_2, \rho , \varepsilon)$. The symbol $c_\varepsilon$ satisfies
  the following: For every $N \geq N_\delta$ we have
  \begin{equation*}
    c_\varepsilon(\hbar) = \sum_{j=0}^N \hbar^j c_{\varepsilon,j} + \hbar^{N+1} r_{\varepsilon,N+1}(a_\varepsilon,b_\varepsilon;\hbar)
  \end{equation*}
  with
  \begin{equation*}
    c_{\varepsilon,j}(x,p) = \frac{(i\sigma(D_u,D_\mu;D_v,D_\nu))^j}{j!} [\tilde{a}_\varepsilon(x,p;u,v,\mu,\nu)\tilde{b}_\varepsilon(x,p;u,v,\mu,\nu)]\Big|_{\substack{u=v=0\\\mu=\nu=0}},
  \end{equation*}
  where
  \begin{equation*}
    \begin{aligned}
      \sigma(u,\mu;v,\nu) &= \langle v,\mu \rangle - \langle u, \nu
      \rangle
      \\
      \tilde{a}_\varepsilon(x,p;u,v,\mu,\nu) &= a_\varepsilon(x+tv
      +t(1-t)u,\nu+(1-t)\mu +p)
      \\
      \tilde{b}_\varepsilon(x,p;u,v,\mu,\nu) &= b_\varepsilon(x+(1-t)v
      - t(1-t)u,\nu-t\mu +p).
    \end{aligned}
  \end{equation*}
  Moreover the error term
  $r_{\varepsilon,N+1}(a_\varepsilon,b_\varepsilon;\hbar)$ satisfies
  that for every multi indices $\alpha,\beta$ in $\N^d$ there exists a
  constant $C(N,\alpha,\beta)$ independent of $a_\varepsilon$ and
  $b_\varepsilon$ and a natural number $M$ such that:
  \begin{equation*}
  	\begin{aligned}
    \hbar^{N+1}|\partial^\alpha_x\partial^\beta_p
    r_{\varepsilon,N+1}(a_\varepsilon,b_\varepsilon;x,p,\hbar)|
    \leq {}& C \varepsilon^{-\abs{\alpha}}\hbar^{\delta(\tau-N-2d-2)_{-}
      + \tau -2d-1}
    \mathcal{G}_{M,\tau}^{\alpha,\beta}(a_\varepsilon,m_1,b_\varepsilon,m_2)
    \\
    &\times m_1(x,\xi) m_2(x,\xi) \lambda(x,\xi)^{-\rho (\tilde{N}(M) + \abs{\alpha}+ \abs{\beta})},
   \end{aligned}
  \end{equation*}
  where
  \begin{equation*}
  	\begin{aligned}
\mathcal{G}_{M,\tau}^{\alpha,\beta}(a_\varepsilon,m_1,b_\varepsilon,m_2)
    = \sup_{{\substack{\abs{\gamma_1 +
            \gamma_2}+\abs{\eta_1 + \eta_2}\leq M \\ (x,\xi) \in
          \R^{2d}}}} & \varepsilon^{(\tau-M)_{-}+\abs{\alpha}}   \lambda(x,\xi)^{\rho(\abs{\gamma_1 + \gamma_2}+\abs{\eta_1 + \eta_2})}
         \\
         &\times  \frac{ \abs{\partial_x^\alpha\partial_\xi^\beta
        ( \partial_{x}^{\gamma_1} \partial_{\xi}^{\eta_1}a_\varepsilon(x,\xi) \partial_{x}^{\gamma_2} \partial_{\xi}^{\eta_2}b_\varepsilon(x,\xi))}}{m_1(x,\xi)
      m_2(x,\xi)}.
      \end{aligned}
  \end{equation*}
  The function $\tilde{N}(M)$ is also depending on the weights $m_1$,
  $m_2$ and the dimension $d$.
\end{thm}
\begin{remark}
  The number $N_\delta$ is explicit and it is the smallest number such
  that
  \begin{equation*}
    \delta(N_\delta +2d + 2-\tau) +\tau > 2d+1.
  \end{equation*}
  This restriction is made in order to ensure that the error term is
  estimated by the semiclassical parameter raised to a positive
  power.  The main difference between this result and the classical analog is that for this new class there is a minimum of terms
  in the expansion of the symbol for the composition in order to
  obtain an error that does not diverge as $\hbar\rightarrow0$.

  The form of the $c$'s we obtain in the theorem is sometimes expressed as
  \begin{equation*} \begin{aligned}
    c_\varepsilon(x,\xi;\hbar) = &e^{i\hbar
      \sigma(D_u,D_\mu;D_v,D_\nu)} [ a_\varepsilon(x+tv+t(1-t)u,\nu
    +(1- t)\mu + \xi)
    \\
    \hbox{}&\times b_\varepsilon(x+(1-t)v - t(1-t)u,\nu - t\mu + \xi)
    ] \Big|_{\substack{u=v=0 \\ \mu=\nu=0}}.
   \end{aligned} \end{equation*}
\end{remark}
\begin{remark}[Particular cases of
  Theorem~\ref{B.composition-t-symbol-thm}]
  We will see the 3 most important cases for this presentation of the composition for
  $t$-quantised operators. We suppose the assumptions of
  Theorem~\ref{B.composition-t-symbol-thm} is
  satisfied. 
  \paragraph{$\boldsymbol{t=0}$:} In this case the amplitude will be independent of $u$
    hence we have
    \begin{equation*}
      c_\varepsilon(x,p;\hbar) = e^{i\hbar \langle D_y,D_q,\rangle} [ a_\varepsilon(x,q) b_\varepsilon(y,p) ] \Big|_{\substack{y=x \\ p=q}}.
    \end{equation*}
    This gives the formula
    \begin{equation*}
      c_{\varepsilon,j}(x,p) = \sum_{\abs{\alpha}=j} \frac{1}{\alpha!} \partial_p^\alpha a_\varepsilon(x,p) D_x^\alpha b_\varepsilon(x,p).
    \end{equation*}
  \paragraph{$\boldsymbol{t=1}$:} This case is similar to the one above, except a change
    of signs. The composition formula is given by
    \begin{equation*}
      c_\varepsilon(x,p;\hbar) = e^{-i\hbar \langle D_y,D_q,\rangle} [ a_\varepsilon(y,p) b_\varepsilon(x,q) ] \Big|_{\substack{y=x \\ p=q}}.
    \end{equation*}
    This gives the formula
    \begin{equation*}
      c_{\varepsilon,j}(x,p) = (-1)^j \sum_{\abs{\alpha}=j} \frac{1}{\alpha!} D_x^\alpha a_\varepsilon(x,p) \partial_p^\alpha b_\varepsilon(x,p).
    \end{equation*}
  \paragraph {$\boldsymbol{t=}\tfrac{\boldsymbol{1}}{\boldsymbol{2}}$ (Weyl-quatisation):} 
  After a bit of additional work we can arrive at the usual formula in this case as well. That is we obtain that
    \begin{equation*}
      c_\varepsilon(x,p;\hbar) = e^{i\frac{\hbar}{2} \sigma(D_x,D_p;D_y,D_q)} [ a_\varepsilon(x,p) b_\varepsilon(y,q) ] \Big|_{\substack{y=x \\ p=q}}
    \end{equation*}
    with
    \begin{equation*}
      c_{\varepsilon,j}(x,p) = \Big(\frac{i}{2} \Big)^j \frac{1}{j!} [\sigma(D_x,D_p;D_y,D_q)]^j  a_\varepsilon(x,p) b_\varepsilon(y,q) \Big|_{\substack{y=x \\ p=q}}.
    \end{equation*}
    The last equation can be rewritten by some algebra to the classic
    formula
    \begin{equation*}
      c_{\varepsilon,j}(x,p) = \sum_{\abs{\alpha}+\abs{\beta}=j} \frac{1}{\alpha!\beta!}\Big(\frac{1}{2} \Big)^{\abs{\alpha}}\Big(-\frac{1}{2} \Big)^{\abs{\beta}} (\partial_p^\alpha D_x^\beta a_\varepsilon) (\partial_p^\beta D_x^\alpha b_\varepsilon)(x,p).
    \end{equation*}

In all three cases we can note that the symbols for the compositions of operators is the same as in the non-rough case.
\end{remark}
We now have composition of operators given by a single symbol. The
next result generalises the previous to composition of strongly
$\hbar$-$\varepsilon$-admissible operators. Moreover it verifies that
the strongly $\hbar$-$\varepsilon$-admissible operators form an
algebra. More precisely we have.
\begin{thm}\label{B.composition-weyl-thm}
  Let $A_\varepsilon(\hbar)$ and $B_\varepsilon(\hbar)$ be two
  strongly $\hbar$-$\varepsilon$-admissible operators of regularity
  $\tau_a\geq0$ and $\tau_b\geq0$. with weights
  $(m_1, \rho , \varepsilon)$ and $(m_2, \rho , \varepsilon)$
  respectively and of the form
  \begin{equation*}
    A_\varepsilon(\hbar) = \OpW(a_\varepsilon) \quad\text{and}\quad B_\varepsilon(\hbar) = \OpW(b_\varepsilon).
  \end{equation*}
  We suppose $\varepsilon \geq \hbar^{1-\delta}$ for a $\delta$ in
  $(0,1)$ and let $\tau=\min(\tau_a,\tau_b)$. Then is
  $C_\varepsilon(\hbar)=A_\varepsilon(\hbar)\circ
  B_\varepsilon(\hbar)$ a strongly $\hbar$-$\varepsilon$-admissible
  operators of regularity $\tau\geq0$ with weights
  $(m_1 m_2, \rho , \varepsilon)$. The symbol
  $c_\varepsilon(x,p;\hbar)$ of $C_\varepsilon(\hbar)$ has for
  $N\geq N_\delta$ the expansion
  \begin{equation*}
    c_\varepsilon(x,p;\hbar) = \sum_{j=0}^N \hbar^j c_{\varepsilon,j}(x,p) + \hbar^{N+1} \mathcal{R}_\varepsilon(a_\varepsilon(\hbar), b_\varepsilon(\hbar);\hbar),
  \end{equation*}
  where
  \begin{equation*}
    c_{\varepsilon,j}(x,p) = \sum_{\abs{\alpha}+\abs{\beta}+k+l=j} \frac{1}{\alpha!\beta!}\Big(\frac{1}{2} \Big)^{\abs{\alpha}}\Big(-\frac{1}{2} \Big)^{\abs{\beta}} (\partial_p^\alpha D_x^\beta a_{\varepsilon,k}) (\partial_p^\beta D_x^\alpha b_{\varepsilon,l})(x,p).
  \end{equation*}
  The symbols $a_{\varepsilon,k}$ and $b_{\varepsilon,l}$ are from the
  expansion of $a_\varepsilon$ and $b_\varepsilon$ respectively. Let
  \begin{equation*}
    a_\varepsilon(x,p) = \sum_{k=0}^N \hbar^j a_{\varepsilon,j}(x,p) + \hbar^{N+1} r_{\varepsilon,N+1}(a_\varepsilon,x,p;\hbar)
  \end{equation*}
  and equvalint for $b_\varepsilon(x,p)$. Then for every multi indices
  $\alpha$, $\beta$ there exists a constant $C(\alpha,\beta,N)$
  independent of $a_\varepsilon$ and $b_\varepsilon$ and an integer
  $M$ such that
  \begin{equation*}
  \begin{aligned}
    \hbar&^{N+1} |\partial_x^\alpha \partial_p^\beta
    \mathcal{R}_\varepsilon(a_\varepsilon(\hbar),
    b_\varepsilon(\hbar);x,p;\hbar) |
    \\
    &\leq C(\alpha,\beta,N) \hbar^{\delta(\tau-N-2d-2)_{-}+ \tau - 2d
      - 1} \varepsilon^{-\abs{\alpha}} m_1(x,p) m_2(x,p)
    \lambda(x,p)^{-\rho(\tilde{N}(M)+\abs{\alpha}+\abs{\beta})}
    \\
    &\phantom{\leq}{} \times \Big[ \sum_{j=0}^N \{
    \mathcal{G}^{\alpha,\beta}_{M,\tau}(a_{\varepsilon,j},m_1,
    r_{\varepsilon,N+1}(b_\varepsilon(\hbar)),m_2) +
    \mathcal{G}^{\alpha,\beta}_{M,\tau}(r_{\varepsilon,N+1}(a_{\varepsilon}(\hbar)),m_1,b_{\varepsilon,j},m_2)\}
    \\
    &\phantom{\leq}{} + \; \sum_{\mathclap{N\leq j+k \leq 2N}} \;
    \mathcal{G}^{\alpha,\beta}_{M,\tau}(a_{\varepsilon,j},m_1,b_{\varepsilon,k},m_2)
    +
    \mathcal{G}^{\alpha,\beta}_{M,\tau}(r_{\varepsilon,N+1}(a_{\varepsilon}(\hbar)),m_1,
    r_{\varepsilon,N+1}(b_\varepsilon(\hbar)),m_2) \Big],
    \end{aligned}
  \end{equation*}
  where
  \begin{equation*}
  	\begin{aligned}
    \mathcal{G}_{M,\tau}^{\alpha,\beta}(a_\varepsilon,m_1,b_\varepsilon,m_2)
    = \sup_{{\substack{\abs{\gamma_1 +
            \gamma_2}+\abs{\eta_1 + \eta_2}\leq M \\ (x,\xi) \in
          \R^{2d}}}} & \varepsilon^{(\tau-M)_{-}+\abs{\alpha}} \frac{
      \abs{\partial_x^\alpha\partial_\xi^\beta
        ( \partial_{x}^{\gamma_1} \partial_{\xi}^{\eta_1}a_\varepsilon(x,\xi) \partial_{x}^{\gamma_2} \partial_{\xi}^{\eta_2}b_\varepsilon(x,\xi))}}{m_1(x,\xi)
      m_2(x,\xi)}
    \\
    \hbox{}&\times \lambda(x,\xi)^{\rho(\abs{\gamma_1 +
        \gamma_2}+\abs{\eta_1 + \eta_2})}.
        \end{aligned}
  \end{equation*}
  The function $\tilde{N}(M)$ is also depending on the weights $m_1$,
  $m_2$ and the dimension $d$.
\end{thm}
The proof of this theorem is an application of
Theorem~\ref{B.composition-t-symbol-thm} a number of times and
recalling that the error operator of a strongly
$\hbar$-$\varepsilon$-admissible operator of some regularity is a
quantised pseudo-differential operator.
\subsection{Rough pseudo-differential operators as operators on $L^2(\R^d)$}
So far we have only considered operators acting on $\mathcal{S}(\R^d)$
or $\mathcal{S}'(\R^d)$. Hence they can be viewed as unbounded
operators acting in $L^2(\R^d)$ with domain $\mathcal{S}(\R^d)$. The
question is then when is this a bounded operator? The first theorem of
this section gives a criteria for when the operator can be extended to
a bounded operator. This theorem is a Calderon-Vaillancourt type
theorem and the proof uses the Calderon-Vaillancourt Theorem for the
non-rough pseudo-differential operators. We will not recall this
theorem but refer to \cite{MR897108,MR2952218,MR1735654}.
\begin{thm}\label{B.cal-val-thm}
  Let $a_\varepsilon$ be in
  $\Gamma_{0,\varepsilon}^{m,\tau}(\R^d_x\times\R_p^d)$, where $m$ is
  a bounded tempered weight function, $\tau\geq0$ and suppose there exists a
  $\delta$ in $(0,1)$ such that
  $\varepsilon\geq\hbar^{1-\delta}$. Then there exists a constant
  $C_d$ and an integer $k_d$ only depending on the dimension such that
  \begin{equation*}
    \norm{\OpW(a_\varepsilon)\psi}_{L^2(\R^d)} \leq C_d \sup_{\substack{\abs{\alpha},\abs{\beta}\leq k_d \\ (x,p)\in \R^{2d}}} \varepsilon^{\abs{\alpha}} \abs{\partial_x^\alpha \partial_p^\beta a_\varepsilon(x,p)} \norm{\psi}_{L^2(\R^d)},
  \end{equation*}
  for all $\psi$ in $\mathcal{S}(\R^d)$. Especially can
  $\OpW(a_\varepsilon)$ be extended to a bounded operator on
  $L^2(\R^d)$.
\end{thm}
\begin{proof}
Let $\mathcal{U}_\varepsilon$ be the unitary dilation operator as defined in Observation~\ref{obs_connection_symbol}. We have that 
  \begin{equation}\label{cal-vai-proof-1}
  \begin{aligned}
    \norm{\OpW(a_\varepsilon)\psi}_{L^2(\R^d)}  &=   \norm{\mathcal{U}_\varepsilon^{*}\mathcal{U}_\varepsilon\OpW(a_\varepsilon) \mathcal{U}_\varepsilon^{*}\mathcal{U}_\varepsilon\psi}_{L^2(\R^d)} 
    =   \norm{\OpW(a_\varepsilon^{\#}) \mathcal{U}_\varepsilon\psi}_{L^2(\R^d)} 
    \end{aligned}
  \end{equation}
By our assumptions and Observation~\ref{obs_connection_symbol} we note that the symbol $a_\varepsilon^{\#}$ satisfies the assumptions of the classical Calderon-Vaillancourt theorem. This gives us a constant $C_d$ and an integer $k_d$ only depending on the dimension such that
  \begin{equation}\label{cal-vai-proof-2}
  \begin{aligned}
    \norm{\OpW(a_\varepsilon^{\#}) \mathcal{U}_\varepsilon\psi}_{L^2(\R^d)} \leq C_d \sup_{\substack{\abs{\alpha},\abs{\beta}\leq k_d \\ (x,p)\in \R^{2d}}}  \abs{\partial_x^\alpha \partial_p^\beta a_\varepsilon^{\#}(x,p)} \norm{\psi}_{L^2(\R^d)},
    \end{aligned}
  \end{equation}
where we have used the unitarity of $\mathcal{U}_\varepsilon$. By the definition of $a_\varepsilon^{\#}$ we have that
  \begin{equation}\label{cal-vai-proof-3}
  \begin{aligned}
    \sup_{\substack{\abs{\alpha},\abs{\beta}\leq k_d \\ (x,p)\in \R^{2d}}}  \abs{\partial_x^\alpha \partial_p^\beta a_\varepsilon^{\#}(x,p)} ={}&  \sup_{\substack{\abs{\alpha},\abs{\beta}\leq k_d \\ (x,p)\in \R^{2d}}}  \abs{\partial_x^\alpha \partial_p^\beta a_\varepsilon(\varepsilon x, \tfrac{\hbar^{1-\delta}}{\varepsilon} p)}
    \leq \sup_{\substack{\abs{\alpha},\abs{\beta}\leq k_d \\ (x,p)\in \R^{2d}}} \varepsilon^{\abs{\alpha}}  \abs{\partial_x^\alpha \partial_p^\beta a_\varepsilon( x,  p)}
    \end{aligned}
  \end{equation}
Combining \eqref{cal-vai-proof-1}, \eqref{cal-vai-proof-2} and \eqref{cal-vai-proof-3} we arrive at the desired result. This completes the proof.  
\end{proof}
We can now give a criteria for the rough pseudo-differential operators
to be trace class. The criteria will be sufficient but not necessary.
Hence it does not provide a full characteristic for the set of rough
pseudo-differential operators which are trace class.
\begin{thm}\label{B.thm_est_tr}
  There exists a constant $C(d)$ only depending on the dimension such
\begin{equation*}
    \norm{\OpW(a_\varepsilon)}_{\Tr} \leq  \frac{ C(d)} {\hbar^d} \sum_{\abs{\alpha}+\abs{\beta}\leq 2d+2}  \varepsilon^{\abs{\alpha}}   \hbar^{\delta\abs{\beta}}  \int_{\R^{2d}} |\partial_x^\alpha \partial_p^\beta a_\varepsilon(x,p)|  \,dxdp.
  \end{equation*}
  for every $a_\varepsilon$ in
  $\Gamma_{0,\varepsilon}^{m,\tau}(\R^d_x\times\R_p^d)$ with
  $\tau\geq0$.
\end{thm}
\begin{proof}
Let $\mathcal{U}_\varepsilon$ be the unitary dilation operator as defined in Observation~\ref{obs_connection_symbol}. We have by the unitary invariance of the trace norm that 
  \begin{equation*}
    \norm{\OpW(a_\varepsilon)}_{\Tr} =   \norm{\mathcal{U}_\varepsilon \OpW(a_\varepsilon) \mathcal{U}_\varepsilon^{*}}_{\Tr} =  \norm{ \OpW(a_\varepsilon^{\#}) }_{\Tr}.
  \end{equation*}
  From our assumptions and Observation~\ref{obs_connection_symbol} we get that $a_\varepsilon^{\#}$ satisfies the assumption for \cite[Theorem II-49]{MR897108}. From this theorem we get the exsistence of a constant $C(d)$ only depending on the dimension such that
    \begin{equation}\label{B.thm_est_tr_eq_1}
     \norm{ \OpW(a_\varepsilon^{\#}) }_{\Tr} \leq  \frac{ C(d)} {\hbar^{\delta d}} \sum_{\abs{\alpha}+\abs{\beta}\leq 2d+2} \hbar^{\delta \abs{\beta}}  \int_{\R^{2d}} |\partial_x^\alpha \partial_p^\beta a_\varepsilon^{\#}(x,p)]  \,dxdp.
  \end{equation}
  By the definition of $\partial_x^\alpha \partial_p^\beta a_\varepsilon(\varepsilon x, \tfrac{\hbar^{1-\delta}}{\varepsilon} p)$ we have that
  \begin{equation}\label{B.thm_est_tr_eq_2}
  	\begin{aligned}
  	  \int_{\R^{2d}} |\partial_x^\alpha \partial_p^\beta a_\varepsilon^{\#}(x,p)|  \,dxdp &=   \int_{\R^{2d}}|\partial_x^\alpha \partial_p^\beta a_\varepsilon(\varepsilon x, \tfrac{\hbar^{1-\delta}}{\varepsilon} p)|  \,dxdp
	 \\
	&= \varepsilon^{\abs{\alpha}}  \left(\tfrac{\hbar^{1-\delta}}{\varepsilon}\right)^{\abs{\beta}} \int_{\R^{2d}} |[\partial_x^\alpha \partial_p^\beta a_\varepsilon](\varepsilon x, \tfrac{\hbar^{1-\delta}}{\varepsilon} p)|  \,dxdp
	\\
	&= \hbar^{(\delta-1)d} \varepsilon^{\abs{\alpha}}  \left(\tfrac{\hbar^{1-\delta}}{\varepsilon}\right)^{\abs{\beta}}  \int_{\R^{2d}} |[\partial_x^\alpha \partial_p^\beta a_\varepsilon](x,  p)|  \,dxdp
	 \end{aligned}
  \end{equation}
  combining \eqref{B.thm_est_tr_eq_1} and \eqref{B.thm_est_tr_eq_2} we get that
    \begin{equation*}
    \norm{\OpW(a_\varepsilon)}_{\Tr} \leq  \frac{ C(d)} {\hbar^d} \sum_{\abs{\alpha}+\abs{\beta}\leq 2d+2}  \varepsilon^{\abs{\alpha}}   \hbar^{\delta\abs{\beta}}  \int_{\R^{2d}} |\partial_x^\alpha \partial_p^\beta a_\varepsilon(x,p)|  \,dxdp.
  \end{equation*}
  This is the desired estimate and this concludes the proof.
\end{proof}
The previous theorem gives us a sufficient condition for the rough
pseudo-differential operators to be trace class. The next theorem
gives the form of the trace for the rough pseudo-differential
operators.
\begin{thm}\label{B.trace_formula}
  Let $a_\varepsilon$ be in
  $\Gamma_{0,\varepsilon}^{m,\tau}(\R^d_x\times\R_p^d)$ with
  $\tau\geq0$ and suppose
  $\partial_x^\alpha \partial_p^\beta a_\varepsilon(x,p)$ is an
  element of $L^1(\R^d_x\times\R_p^d)$ for all
  $\abs{\alpha}+\abs{\beta}\leq 2d+2$. Then $\OpW(a_\varepsilon)$ is
  trace class and
  \begin{equation*}
    \Tr(\OpW(a_\varepsilon))=\frac{1}{(2\pi\hbar)^d} \int_{\R^{2d}} a_\varepsilon(x,p) \,dxdp.
  \end{equation*}	
\end{thm}
\begin{proof}
	That our operator is trace class under the assumptions follows from Theorem~\ref{B.thm_est_tr}. To obtain the formula for the trace let $\mathcal{U}_\varepsilon$ be the unitary dilation operator as defined in Observation~\ref{obs_connection_symbol}. Then by \cite[Theorem~{II-$53$}]{MR897108} and Observation~\ref{obs_connection_symbol} we get that
\begin{equation*}
	\begin{aligned}
	\Tr(\OpW(a_\varepsilon)) &= \Tr(\mathcal{U}_\varepsilon\OpW(a_\varepsilon)\mathcal{U}_\varepsilon^{*}) = \Tr(\OpW(a_\varepsilon^{\#}))
	\\
	&= \frac{1}{(2\pi\hbar^\delta)^d} \int_{\R^{2d}} a_\varepsilon(\varepsilon x, \tfrac{\hbar^{1-\delta}}{\varepsilon} p) \,dxdp = \frac{1}{(2\pi\hbar)^d}  \int_{\R^{2d}} a_\varepsilon(x,p) \,dxdp.
	\end{aligned}
\end{equation*}
This identity concludes the proof.
\end{proof}
The last thing for this section is a sharp G{\aa}rdinger inequality, which we will need later.
\begin{thm}\label{B.sharp_gaar}
  Let $a_\varepsilon$ be a bounded rough symbol of regularity
  $\tau\geq0$ which satisfies
  \begin{equation*}
    a_\varepsilon (x,p) \geq 0 \quad\text{for all } (x,p) \in \R_x^d\times\R^d_p,
  \end{equation*}
  and suppose there exists $\delta\in(0,1)$ such that
  $\varepsilon>\hbar^{1-\delta}$. Then there exists a $C_0>0$ and
  $\hbar_0>0$ such
  \begin{equation*}
    \langle \OpW(a_\varepsilon) g,g \rangle \geq -\hbar^{\delta} C \norm{g}_{L^2(\R^d)}
  \end{equation*}
  for all $g$ in $L^2(\R^d)$ and $\hbar$ in $(0,\hbar_0]$.
\end{thm}
\begin{proof}
 Again let $\mathcal{U}_\varepsilon$ be the unitary dilation operator as defined in Observation~\ref{obs_connection_symbol}. After conjugation with this operator we will be able to use the \enquote{usual} semiclassical sharp G{\aa}rdinger
  inequality (see e.g. \cite[Theorem 4.32]{MR2952218}) with the semiclassical parameter
  $\hbar^\delta$. We have that
  \begin{equation}\label{B.sharp_gaar_eq_1}
    \begin{aligned}
      \langle \OpW(a_\varepsilon) g,g \rangle 
      &= \langle \mathcal{U}_\varepsilon\OpW(a_\varepsilon)\mathcal{U}_\varepsilon^{*}\mathcal{U}_\varepsilon g,\mathcal{U}_\varepsilon g \rangle 
      =  \langle \OpW(a_\varepsilon^{\#})\mathcal{U}_\varepsilon g,\mathcal{U}_\varepsilon g \rangle
      \\
      &\geq - C_0 \hbar^{\delta} \norm{\mathcal{U}_\varepsilon g}_{L^2(\R^d)} =  - C_0 \hbar^{\delta} \norm{g}_{L^2(\R^d)} 
    \end{aligned}
  \end{equation}
The existence of the numbers $C_0$ and $\hbar_0$ is ensured by the \enquote{usual} semiclassical sharp G{\aa}rdinger
  inequality. This is the desired estimate and this ends the proof.
\end{proof}
\section{Self-adjointness and functional calculus for rough pseudo-differential operator}
In this section we will establish a functional calculus for rough pseudo-differential operators. In order for this to be well defined we will also give a set of assumptions that ensures essential self-adjointness of the operators. We could as in the previous section use unitary conjugations to obtain non-rough operators and the use the results for this type of operators. We have chosen not to do this, since in the construction we will need to control polynomials in the derivatives of our symbols. So to make it transparent that the regularity of these polynomials are as desired we have chosen to do the full construction. The construction is based on the results of Helffer--Robert in \cite{MR724029} and this approach is also described in \cite{MR897108 }. The method they used is based on the Mellin transform, where we will use the Helffer-Sj{\"o}strand formula instead. The construction of a functional calculus using this formula can be found in the monographs \cite{MR2952218,MR1735654}. There is also a construction of the functional calculus using Fourier theory in \cite{ivrii2019microlocal1}, but we have not tried to adapt this to the case studied here.  
 
 \subsection{Essential self-adjointness of rough pseudo-differential operators}
 First we will give criteria for the operator to be lower semi-bounded
and essential self-adjoint.
\begin{assumption}\label{B.self_adj_assumption}
  Let $A_\varepsilon(\hbar)$ be a $\hbar$-$\varepsilon$-admissible
  operator of regularity $\tau$ and suppose that
\begin{enumerate}[label={$(\roman*)$}]
  \item\label{B.H.1} $A_\varepsilon(\hbar)$ is symmetric on
    $\mathcal{S}(\R^n)$ for all $\hbar$ in $]0,\hbar_0]$.
  \item\label{B.H.2} The principal symbol $a_{\varepsilon,0}$
    satisfies that
    \begin{equation*}
      \min_{(x,p) \in \R^{2n}} a_{\varepsilon,0}(x,p) = \zeta_0 > -\infty. 
    \end{equation*}
  \item\label{B.H.3} Let $\zeta_1 < \zeta_0$ and $\zeta_1 \leq
    0$. Then $a_{\varepsilon,0} - \zeta_1$ is a tempered weight
    function with constants independent of $\varepsilon$ and
    \begin{equation*}
      a_{\varepsilon,j} \in \Gamma_{0,\varepsilon}^{a_{\varepsilon,0} - \zeta_1,\tau-j} \left(\R_x^{d}\times\R^d_p\right),
    \end{equation*}
    for all $j$ in $\N$.
  \end{enumerate}
\end{assumption}
\begin{thm}\label{B.self_adjoint_thm_1}
  Let $A_\varepsilon(\hbar)$, for $\hbar$ in $(0,\hbar_0]$, be a
  $\hbar$-$\varepsilon$-admissible operator of regularity $\tau\geq 1$
  with tempered weight $m$ and symbol
  \begin{equation*}
    a_\varepsilon(\hbar) = \sum_{j\geq 0} \hbar^j  a_{\varepsilon,j}.
  \end{equation*}
  Suppose that $A_\varepsilon(h)$ satisfies
  Assumption~\ref{B.self_adj_assumption}. Then there exists $\hbar_1$
  in $(0,\hbar_0]$ such that for all $\hbar$ in $(0,\hbar_1]$
  $A_\varepsilon(\hbar)$ is essential self-adjoint and lower
  semi-bounded.
\end{thm}
\begin{proof}
  We let $t<\zeta_0$, where $\zeta_0$ is the number from
  Assumption~\ref{B.self_adj_assumption}. For this $t$ we define the
  symbol
  \begin{equation*}
    b_{\varepsilon,t}(x,p) = \frac{1}{a_{\varepsilon,0}(x,p) -t}.
  \end{equation*}
  By assumption we have that
  $b_{\varepsilon,t} \in \Gamma_{0,\varepsilon}^{(a_{\varepsilon,0} -
    \gamma_1)^{-1},\tau}(\R_x^d\times\R^d_p)$. For $N$ sufficiently
  large we get by assumption that 
  \begin{equation*}
  \begin{aligned}
    (A_\varepsilon(\hbar) - t) \OpW(b_{\varepsilon,t}) ={}& \sum_{j=0}^N
    \hbar^j \OpW( a_{\varepsilon,j}) \OpW(b_{\varepsilon,t}) + \hbar^{N+1} R_N(\hbar)\OpW(b_{\varepsilon,t})
    \\
    ={}& I + \hbar S_N(\varepsilon,\hbar).
    \end{aligned}
  \end{equation*}
The formula for composition of operators and the Calderon-Vaillancourt theorem gives us that the   operator $S_N$ satisfies the estimate
  \begin{equation}
  \sup_{\hbar\in(0,\hbar_0]}  \norm{S_N(\varepsilon,\hbar)}_{\mathcal{L}(L^2(\R^d))}<\infty
  \end{equation}
  We note that if
  $\hbar$ is chosen such that
  $\hbar\norm{S_N(\varepsilon,\hbar)}_{\mathcal{L}(L^2(\R^d))}<1$ then
  the operator $ I + \hbar S_N(\varepsilon,\hbar)$ will be
  invertible. 
  
  The Calderon-Vaillancourt theorem gives us that
  $ \OpW(b_{\varepsilon,t})$ is a bounded operator. This implies that
  the expression
  $ \OpW(b_{\varepsilon,t}) ( I + \hbar S_N(\varepsilon,\hbar))^{-1}$
  is a well defined bounded operator. Hence we have that the operator
  $(A_\varepsilon(\hbar) - t) $ maps its domain surjective onto all of
  $L^2(\R^d)$. By \cite[Proposition 3.11]{MR2953553} this implies that
  $A_\varepsilon(\hbar)$ is essential self-adjoint.

  Since we have for all $t<\zeta_0$ that
  $(A_\varepsilon(\hbar) - t) $ maps its domain surjective onto all of
  $L^2(\R^d)$ they are all in the resolvent set and hence the operator
  has to be lower semi-bounded.
\end{proof}
\subsection{The resolvent of a rough pseudo-differential operator}
A main part in the construction of the functional calculus is to prove that the resolvent of a $\hbar$-$\varepsilon$-admissible
  operator of regularity $\tau$, which satisfies Assumption~\ref{B.self_adj_assumption}, is an operator of the same type. this is the content of the following Theorem. 
\begin{thm}\label{B.approx_of_resolvent}
  Let $A(\hbar)$, for $\hbar$ in $(0,\hbar_0]$, be a
  $\hbar$-$\varepsilon$-admissible operator of regularity
  $\tau \geq 1$ with tempered weight $m$ and symbol
  \begin{equation*}
    a_\varepsilon(\hbar) = \sum_{j\geq 0} \hbar^j a_{\varepsilon,j}.
  \end{equation*}
  Suppose that $A_\varepsilon(h)$ satisfies
  Assumption~\ref{B.self_adj_assumption} with the numbers $\zeta_0$
  and $\zeta_1$. For $z$ in $\C\setminus[\zeta_1,\infty)$ we define
  the sequence of symbols
  \begin{equation}\label{B.res.symbolform_1}
    \begin{aligned}
      b_{\varepsilon,z,0} &= (a_{\varepsilon,0} - z)^{-1}
      \\
      b_{\varepsilon,z,j+1} &= - b_{\varepsilon,z,0} \cdot
      \sum_{\substack{l + \abs{\alpha} + \abs{\beta} + k = j+1\\ 0\leq
          l \leq j }} \frac{1}{\alpha ! \beta !}
      \frac{1}{2^{\abs{\alpha}}} \frac{1}{(-2)^{\abs{\beta}}}
      (\partial_p^{\alpha} D_x^{\beta} a_{\varepsilon,k})
      (\partial_p^{\beta} D_x^{\alpha} b_{\varepsilon,z,l}),
    \end{aligned}
  \end{equation}
  for $j\geq1$. Moreover we define
  \begin{equation*}
    B_{\varepsilon,z,M}(\hbar) = \sum_{j=0}^M \hbar^j b_{\varepsilon,z,j}.
  \end{equation*}
  Then for $N$ in $\N$ we have that
  \begin{equation}
    (A_\varepsilon(h) - z ) \OpW B_{\varepsilon,z,N} = I + h^{N+1} \Delta_{z,N+1}(h),
  \end{equation}
  with
  \begin{equation}
    \hbar^{N+1}\norm{ \Delta_{z,N+1}(h)}_{\mathcal{L}(L^2(\R^d))} \leq C \hbar^{\kappa(N)} \left( \frac{\abs{z}}{\dist(z,[\zeta_1,\infty))} \right)^{q(N)},
  \end{equation}
  where $\kappa$ is a positive strictly increasing function and $q(N)$
  is a positive integer depending on $N$. In particular we have for
  all $z$ in $\C\setminus[\zeta_1,\infty)$ and all $\hbar$ in
  $(0,\hbar_1]$ ($\hbar_1$ sufficient small and independent of $z$),
  that $(A_\varepsilon(h) - z)^{-1}$ is a
  $\hbar$-$\varepsilon$-admissible operator with respect to the
  tempered weight $(a_{\varepsilon,0}-\zeta_1)^{-1}$ and of
  regularity $\tau$ with symbol:
  \begin{equation}
    B_z(h) = \sum_{j\geq0} h^j b_{\varepsilon,z,j}.
  \end{equation}
\end{thm}
Before we can prove the theorem we will need some lemma's with the
same setting. It is in these lemma's we will see that the symbol for the resolvent has the same regularity. From these lemmas we can also find the explicit formula's for every symbol in the expansion.
\begin{lemma}\label{B.approx_of_resolvent_lemma_1}
  Let the setting be as in Theorem~\ref{B.approx_of_resolvent}. For
  every $j$ in $\mathbb{N}$ we have
  \begin{equation}\label{B.res.symbolform_2}
    b_{\varepsilon,z,j} = \sum_{k=1}^{2j-1} d_{\varepsilon, j ,k} b_{\varepsilon,z,0}^{k+1},
  \end{equation}	
  where $d_{\varepsilon, j ,k}$ are universal polynomials in
  $\partial_p^\alpha\partial_x^\beta a_{\varepsilon,l}$ for
  $\abs{\alpha}+\abs{\beta}+l\leq j$ and
  $d_{\varepsilon, j ,k} \in
  \Gamma_{0,\varepsilon}^{(a_0-\zeta_1)^k,\tau-j}$ for all $k$,
  $1\leq k \leq 2j-1$.  In particular we have that
  \begin{equation*}
    b_{\varepsilon,z,1}=-a_{\varepsilon,1} b_{\varepsilon,z,0}^2.
  \end{equation*}
\end{lemma}
In order to prove this Lemma we will be needing the following Lemma:
\begin{lemma}\label{B.approx_of_resolvent_lemma_2}
  Let the setting be as in
  Lemma~\ref{B.approx_of_resolvent_lemma_1}. For any $j$ and $k$ in
  $\mathbb{N}$ we let
  $d_{\varepsilon, j ,k} b_{\varepsilon,z,0}^{k+1}$ be one of the
  elements in the expansion of $b_{\varepsilon,z,j}$. Then for all
  multi indices $\alpha$ and $\beta$ it holds that
  \begin{equation*}
    \partial_p^\beta \partial_x^\alpha d_{\varepsilon, j ,k} b_{\varepsilon,z,0}^{k+1}  = \sum_{n=0}^{\abs{\alpha}+\abs{\beta}} \tilde{d}_{\varepsilon, j ,k, n, \alpha , \beta} b_{\varepsilon,z,0}^{k+1+n},
  \end{equation*}
  where $ \tilde{d}_{\varepsilon, j ,k, n, \alpha , \beta}$ are
  polynomials in
  $\partial_x^{\alpha'} \partial_p^{\beta'} a_{\varepsilon,k}$ with
  $\abs{\alpha'}+\abs{\beta'}+k \leq j +\abs{\alpha}+\abs{\beta}$ of
  degree at most $k+n$ and they are of regularity at least
  $\tau-j-\abs{\alpha}$. They are determined only by $\alpha$, $\beta$
  and $d_{\varepsilon, j ,k}$.
\end{lemma}
\begin{proof}
  The proof is an application of Theorem~\ref{B.Faa_di_bruno_x} (Fa\`a
  di {B}runo formula) and the Corollary~\ref{B.Faa_di_bruno_xp} to the
  formula. For our $\alpha,\beta$ we have by the Leibniz's formula that:
  \begin{equation*}
  \begin{aligned}
   \MoveEqLeft  \partial_p^\beta \partial_x^\alpha d_{\varepsilon, j ,k}
    b_{\varepsilon,z,0}^{k+1} 
    = \partial_p^\beta \big\{
    (\partial_x^{\alpha} d_{\varepsilon, j ,k})
    b_{\varepsilon,z,0}^{k+1} + \sum_{\substack{\abs{\gamma}=1\\
        \gamma\leq\alpha}}^{\abs{\alpha}}
    \binom{\alpha}{\gamma} \partial_x^{\alpha-\gamma} d_{\varepsilon,
      j ,k} \partial_x^\gamma b_{\varepsilon,z,0}^{k+1} \big\}
    \\
    ={} & (\partial_p^\beta \partial_x^{\alpha} d_{\varepsilon, j
      ,k})b_{\varepsilon,z,0}^{k+1} +
    \sum_{\substack{\abs{\zeta}=1\\\zeta\leq\beta}}^{\abs{\beta}}\binom{\beta}{\zeta} \partial_p^{\beta-\zeta} \partial_x^{\alpha}
    d_{\varepsilon, j ,k} \partial_p^\zeta b_{\varepsilon,z,0}^{k+1}
     + \sum_{\substack{\abs{\gamma}=1\\
        \gamma\leq\alpha}}^{\abs{\alpha}} \binom{\alpha}{\gamma}
    (\partial_p^\beta\partial_x^{\alpha-\gamma} d_{\varepsilon, j
      ,k}) \partial_x^\gamma b_{\varepsilon,z,0}^{k+1}
    \\
    & +
    \sum_{\substack{\abs{\zeta}=1\\\zeta\leq\beta}}^{\abs{\beta}}
    \sum_{\substack{\abs{\gamma}=1\\ \gamma\leq\alpha}}^{\abs{\alpha}}
    \binom{\beta}{\zeta} \binom{\alpha}{\gamma}
    (\partial_p^{\beta-\zeta}\partial_x^{\alpha-\gamma}
    d_{\varepsilon, j ,k}) \partial_p^\zeta \partial_x^\gamma
    b_{\varepsilon,z,0}^{k+1} .
    \end{aligned}
  \end{equation*}
  We will here consider each of the three sums separately for the
  first we get by the Fa\`a di {B}runo formula
  (Theorem~\ref{B.Faa_di_bruno_x})
  \begin{equation*}
  \begin{aligned}
    \MoveEqLeft \sum_{\substack{\abs{\zeta}=1\\\zeta\leq\beta}}^{\abs{\beta}}
    \binom{\beta}{\zeta} \partial_p^{\beta-\zeta} \partial_x^{\alpha}
    d_{\varepsilon, j ,k} \partial_p^\zeta b_{\varepsilon,z,0}^{k+1}
    \\
    & =
    \sum_{\substack{\abs{\zeta}=1\\\zeta\leq\beta}}^{\abs{\beta}}\binom{\beta}{\zeta} \partial_p^{\beta-\zeta} \partial_x^{\alpha}
    d_{\varepsilon, j ,k} \sum_{{n}=1}^{\abs{\zeta}} (-1)^{n}
    \frac{(k+n)!}{k!}  b_{\varepsilon,z,0}^{k+1+ n} \;\;
    \sum_{\mathclap{\substack{\zeta_1+\cdots+\zeta_n=\zeta\\\abs{\zeta_i}>0}}}
    \;\; c_{\zeta_1,\dots,\zeta_n} \partial_p^{\zeta_1}a_0
    \cdots \partial_p^{\zeta_n}a_0
    \\
    &= \sum_{n_\beta=1}^{\abs{\beta}} \Big\{
    \sum_{\substack{\abs{\zeta}\geq n_\beta \\ \zeta\leq
        \beta}}^{\abs{\beta}}
    c_{k,n_\beta,\beta,\zeta} \partial_p^{\beta-\zeta} \partial_x^{\alpha}
    d_{\varepsilon, j ,k}
    \sum_{\substack{\zeta_1+\cdots+\zeta_{n_\beta}=\zeta\\\abs{\zeta_i}>0}}
    c_{\zeta_1,\dots,\zeta_n} \partial_p^{\zeta_1}a_0
    \cdots \partial_p^{\zeta_{n_\beta}}a_0
    \Big\}b_{\varepsilon,z,0}^{k+1+ n_\beta}
    \\
    &= \sum_{n_\beta=1}^{\abs{\beta}}
    \tilde{d}_{\varepsilon,j,k,\alpha,\beta,n_\beta}
    b_{\varepsilon,z,0}^{k+1+ n_\beta}.
    \end{aligned}
  \end{equation*}
  This calculation gives that we have a polynomial structure, where the coefficients are 
  the polynomials $ \tilde{d}_{\varepsilon,j,k,\alpha,\beta,n_\beta}$, which themselfs  
  are polynomials in
  $\partial_x^{\alpha'} \partial_p^{\beta'} a_{\varepsilon,k}$ with
  $\abs{\alpha'}+\abs{\beta'}+k \leq j +\abs{\alpha}+\abs{\beta}$ and
  of regularity $\tau-j-\abs{\alpha}$.
  For the second sum we again use Fa\`a di {B}runo formula
  (Theorem~\ref{B.Faa_di_bruno_x}) and get
  \begin{equation*}
  \begin{aligned}
   \MoveEqLeft  \sum_{\substack{\abs{\gamma}=1\\
        \gamma\leq\alpha}}^{\abs{\alpha}} \binom{\alpha}{\gamma}
    (\partial_p^\beta\partial_x^{\alpha-\gamma} d_{\varepsilon, j
      ,k}) \partial_x^\gamma b_{\varepsilon,z,0}^{k+1}
    \\
    ={}&
    \sum_{\substack{\abs{\gamma}=1\\\gamma\leq\alpha}}^{\abs{\alpha}}\binom{\alpha}{\gamma} \partial_p^{\beta} \partial_x^{\alpha-\gamma}
    d_{\varepsilon, j ,k} \sum_{{n}=1}^{\abs{\gamma}} (-1)^{n}
    \frac{(k+n)!}{k!}  b_{\varepsilon,z,0}^{k+1+ n} \quad
    \sum_{\mathclap{\substack{\gamma_1+\cdots+\gamma_n=\gamma\\\abs{\gamma_i}>0}}}
    \;c_{\gamma_1,\dots,\gamma_n} \partial_x^{\gamma_1}a_0
    \cdots \partial_x^{\gamma_n}a_0
    \\
    ={}& \sum_{n_\alpha=1}^{\abs{\alpha}} \Big\{
    \sum_{\substack{\abs{\gamma}\geq n_\alpha \\ \gamma\leq
        \alpha}}^{\abs{\alpha}}
    c_{k,n_\alpha,\alpha,\gamma} \partial_p^{\beta} \partial_x^{\alpha-\gamma}
    d_{\varepsilon, j ,k} \quad
    \sum_{\mathclap{\substack{\gamma_1+\cdots+\gamma_{n_\alpha}=\gamma\\\abs{\gamma_i}>0}}}\;
    c_{\gamma_1,\dots,\gamma_n} \partial_x^{\gamma_1}a_0
    \cdots \partial_x^{\gamma_{n_\alpha}}a_0
    \Big\}b_{\varepsilon,z,0}^{k+1+ n_\alpha}
    \\
    ={}& \sum_{n_\alpha=1}^{\abs{\alpha}}
    \tilde{d}_{\varepsilon,j,k,\alpha,\beta,n_\alpha}
    b_{\varepsilon,z,0}^{k+1+ n_\alpha}.
    \end{aligned}
  \end{equation*}
 Again we have the polynomial structure,where the coefficients are 
  the polynomials $ \tilde{d}_{\varepsilon,j,k,\alpha,\beta,n_\alpha}$ again
  are polynomials in
  $\partial_x^{\alpha'} \partial_p^{\beta'} a_{\varepsilon,k}$ with
  $\abs{\alpha'}+\abs{\beta'}+k \leq j +\abs{\alpha}+\abs{\beta}$ and
  they are of at least regularity $\tau-j-\abs{\alpha}$.

  For the last sum we need a slightly modified version of the Fa\`a di
  {B}runo formula which is Corollary~\ref{B.Faa_di_bruno_xp}. If we
  use this we get that
%
%\begin{AllowDisplayBreaks}
  \begin{equation*}
  \begin{aligned}
   \MoveEqLeft  \sum_{\substack{\abs{\zeta}=1\\\zeta\leq\beta}}^{\abs{\beta}} 
    \sum_{\substack{\abs{\gamma}=1\\ \gamma\leq\alpha}}^{\abs{\alpha}}
    \binom{\beta}{\zeta} \binom{\alpha}{\gamma}
    (\partial_p^{\beta-\zeta}\partial_x^{\alpha-\gamma}
    d_{\varepsilon, j ,k}) \partial_p^\zeta \partial_x^\gamma
    b_{\varepsilon,z,0}^{k+1}
    \\
    ={}& \sum_{\substack{\abs{\zeta}=1\\\zeta\leq\beta}}^{\abs{\beta}}
    \sum_{\substack{\abs{\gamma}=1\\ \gamma\leq\alpha}}^{\abs{\alpha}}
    \binom{\beta}{\zeta} \binom{\alpha}{\gamma}
    (\partial_p^{\beta-\zeta}\partial_x^{\alpha-\gamma}
    d_{\varepsilon, j ,k}) \sum_{n=1}^{\abs{\zeta}+\abs{\gamma}} c_n
    b_{\varepsilon,z,0}^{k+1+n}
    \sum_{\mathcal{I}_n(\gamma,\zeta)}c_{\gamma_1\cdots\gamma_k}^{\zeta_1\cdots\zeta_k} \partial_p^{\zeta_1} \partial_x^{\gamma_1}
    a_{\varepsilon,0}
    \cdots \partial_p^{\zeta_n} \partial_x^{\gamma_n}
    a_{\varepsilon,0}
    \\
    ={}& \sum_{n=1}^{\abs{\alpha}+\abs{\beta}}
    \Big\{\sum_{\substack{\abs{\zeta} + \abs{\gamma} \geq n
        \\\zeta\leq\beta,\
        \gamma\leq\alpha}}^{\abs{\beta}+\abs{\alpha}}
    \sum_{\mathcal{I}_n(\gamma,\zeta)}
    c_{k,n,\alpha,\beta,\gamma,\zeta}
    (\partial_p^{\beta-\zeta}\partial_x^{\alpha-\gamma}
    d_{\varepsilon, j ,k})
   \partial_p^{\zeta_1} \partial_x^{\gamma_1} a_{\varepsilon,0}
    \cdots \partial_p^{\zeta_n} \partial_x^{\gamma_n}
    a_{\varepsilon,0} \Big\} b_{\varepsilon,z,0}^{k+1+n}
    \\
    ={}& \sum_{n=1}^{\abs{\alpha}+\abs{\beta}}
    \tilde{d}_{\varepsilon,j,k,n,\alpha,\beta}
    b_{\varepsilon,z,0}^{k+1+n},
    \end{aligned}
  \end{equation*}
 % \end{AllowDisplayBreaks}
%
  where $\tilde{d}_{\varepsilon,j,k,n,\alpha,\beta}$ are polynomials
  in $\partial_x^{\alpha'} \partial_p^{\beta'} a_{\varepsilon,k}$ with
  $\abs{\alpha'}+\abs{\beta'}+k \leq j +\abs{\alpha}+\abs{\beta}$ of
  degree at most $k+n$ and they are of regularity at least
  $\tau-j-\abs{\alpha}$. If we combine all of the above calculations
  we get the desired result:
  \begin{equation*}
    \partial_p^\beta \partial_x^\alpha d_{\varepsilon, j ,k} b_{\varepsilon,z,0}^{k+1}  = \sum_{n=0}^{\abs{\alpha}+\abs{\beta}} \tilde{d}_{\varepsilon, j ,k, n, \alpha , \beta} b_{\varepsilon,z,0}^{k+1+n},
  \end{equation*}
  where $ \tilde{d}_{\varepsilon, j ,k, n, \alpha , \beta}$ are
  polynomials in
  $\partial_x^{\alpha'} \partial_p^{\beta'} a_{\varepsilon,k}$ with
  $\abs{\alpha'}+\abs{\beta'}+k \leq j +\abs{\alpha}+\abs{\beta}$ of
  degree at most $k+n$ and they are of regularity at least
  $\tau-j-\abs{\alpha}$. The form of the polynomials is entirely
  determined by the multi-indices $\alpha,\beta$ and the polynomial
  $ d_{\varepsilon, j ,k} $.
\end{proof}
\begin{proof}[Proof of Lemma~\ref{B.approx_of_resolvent_lemma_1}]
  The proof will be induction in the parameter $j$. We start by
  considering the case $j=1$, where we by definition of
  $b_{\varepsilon,z,1}$ have
  \begin{equation*}
  \begin{aligned}
    b_{\varepsilon,z,1} &= - b_{\varepsilon,z,0} \cdot \sum_{
      \abs{\alpha} + \abs{\beta} + k = 1} \frac{1}{\alpha ! \beta !}
    \frac{1}{2^{\abs{\alpha}}} \frac{1}{(-2)^{\abs{\beta}}}
    (\partial_p^{\alpha} D_x^{\beta} a_{\varepsilon,k})
    (\partial_p^{\beta} D_x^{\alpha} b_{\varepsilon,z,0})
    \\
    &= - b_{\varepsilon,z,0} \big(
    a_{\varepsilon,1}b_{\varepsilon,z,0} -\frac{i}{2}
    \sum_{n=1}^d \partial_{p_n} a_{\varepsilon,0} \partial_{x_n}
    a_{\varepsilon,0}b_{\varepsilon,z,0}^2 - \partial_{x_n}
    a_{\varepsilon,0}\partial_{p_n} a_{\varepsilon,0}
    b_{\varepsilon,z,0}^2 \big)
    \\
    &= - a_{\varepsilon,1}b_{\varepsilon,z,0}^2.
    \end{aligned}
  \end{equation*}
  This calculation verifies the form of $b_{\varepsilon,z,1}$ stated
  in the lemma and that it has the form given by
  \eqref{B.res.symbolform_2} with
  $d_{\varepsilon,1,1}=-a_{\varepsilon,1}$ which is in the symbol
  class $\Gamma_{0,\varepsilon}^{(a_0-\zeta_1),\tau-1}$ by assumption.

  Assume the lemma to be correct for $b_{\varepsilon,z,j}$ and
  consider $b_{\varepsilon,z,j+1}$. By the definition of
  $b_{\varepsilon,z,j+1}$ and our assumption we have
  \begin{equation*}
  \begin{aligned}
    b_{\varepsilon,z,j+1} ={}& - b_{\varepsilon,z,0} \cdot
    \sum_{\substack{l + \abs{\alpha} + \abs{\beta} + k = j+1\\ 0\leq l
        \leq j }} \frac{1}{\alpha ! \beta !}
    \frac{1}{2^{\abs{\alpha}}} \frac{1}{(-2)^{\abs{\beta}}}
    (\partial_p^{\alpha} D_x^{\beta} a_{\varepsilon,k})
    (\partial_p^{\beta} D_x^{\alpha} b_{\varepsilon,z,l})
    \\
    ={}& - b_{\varepsilon,z,0} \Big\{\sum_{\abs{\alpha} + \abs{\beta} +
      k = j+1} c_{\alpha,\beta} (\partial_p^{\alpha} D_x^{\beta}
    a_{\varepsilon,k}) (\partial_p^{\beta} D_x^{\alpha}
    b_{\varepsilon,z,0})
    \\
    \hbox{}&\qquad+ \sum_{l=1}^j \sum_{m=1}^{2l-1}
    \sum_{l+\abs{\alpha} + \abs{\beta} + k = j+1} c_{\alpha,\beta}
    (\partial_p^{\alpha} D_x^{\beta} a_{\varepsilon,k})
    (\partial_p^{\beta} D_x^{\alpha}d_{\varepsilon, l ,m}
    b_{\varepsilon,z,0}^{m+1})\Big\}.
    \end{aligned}
  \end{equation*}
  We will consider each of the sums separately. To calculate the first
  sum we get by applying Corollary~\ref{B.Faa_di_bruno_xp}
  \begin{equation*}
  \begin{aligned}
    \MoveEqLeft \sum_{\abs{\alpha} + \abs{\beta} + k = j+1} c_{\alpha,\beta}
    (\partial_p^{\alpha} D_x^{\beta} a_{\varepsilon,k})
    (\partial_p^{\beta} D_x^{\alpha} b_{\varepsilon,z,0}) - c_{\alpha,\beta} a_{\varepsilon,j+1} b_{\varepsilon,z,0} 
    \\
    ={}&     \sum_{\substack{\abs{\alpha} + \abs{\beta} + k = j+1 \\
        \abs{\alpha} + \abs{\beta}\geq 1}} \Big\{c_{\alpha,\beta}
    (\partial_p^{\alpha} D_x^{\beta} a_{\varepsilon,k})
    \sum_{n=1}^{\abs{\alpha}+\abs{\beta}} c_n
    b_{\varepsilon,z,0}^{n+1}  \sum_{\mathcal{I}_n(\alpha,\beta)}c_{\alpha_1\cdots\alpha_n}^{\beta_1\cdots\beta_n} \partial_p^{\beta_1}
    D_x^{\alpha_1} a_{\varepsilon,0} \cdots \partial_p^{\beta_n}
    D_x^{\alpha_n} a_{\varepsilon,0}\Big\}
    \\
    ={}&
    \sum_{n=1}^{j+1} \Big\{ \sum_{\substack{\abs{\alpha} + \abs{\beta}
        + k = j+1 \\ \abs{\alpha} + \abs{\beta}\geq n}}
    \sum_{\mathcal{I}_n(\alpha,\beta)} c_{\alpha,\beta,n}
   (\partial_p^{\alpha} D_x^{\beta}
    a_{\varepsilon,k}) \partial_p^{\beta_1} D_x^{\alpha_1}
    a_{\varepsilon,0} \cdots \partial_p^{\beta_n} D_x^{\alpha_n}
    a_{\varepsilon,0} \Big\} b_{\varepsilon,z,0}^{n+1}
    \\
    ={}& \sum_{n=1}^{j+1} \tilde{d}_{\varepsilon,j,n,\alpha,\beta}
    b_{\varepsilon,z,0}^{n+1} - c_{\alpha,\beta} a_{\varepsilon,j+1} b_{\varepsilon,z,0} ,
    \end{aligned}
  \end{equation*}
  where $ \tilde{d}_{\varepsilon, j , n, \alpha , \beta}$ are
  polynomials in
  $\partial_x^{\alpha'} \partial_p^{\beta'} a_{\varepsilon,k}$ with
  $\abs{\alpha'}+\abs{\beta'}+k \leq j + 1$ of degree $n+1$ and of
  regularity at least $\tau-j-1$. The index set $\mathcal{I}_n(\alpha,\beta)$ is defined by
    \begin{equation*} \begin{aligned}
    \mathcal{I}_n(\alpha,\beta) =
    \{(\alpha_1,\dots,\alpha_n,&\beta_1,\dots,\beta_n) \in \N^{2nd} \
    \\
    & | \ \sum_{l=1}^n \alpha_l=\alpha, \, \sum_{l=1}^n \beta_l=\beta,
    \, \max(\abs{\alpha_l},\abs{\beta_l}) \geq1 \, \forall l \}.
   \end{aligned} \end{equation*}
  The form of the polynomials is
  determined by the multi indices $\alpha$ and $\beta$. Moreover we
  have that the polynomials
  $ \tilde{d}_{\varepsilon, j , n, \alpha , \beta}$ are elements of
  $\Gamma_{0,\varepsilon}^{(a_0-\zeta_1)^n,\tau-j-1}$.

  If we now consider the triple sum and apply
  Lemma~\ref{B.approx_of_resolvent_lemma_2} we get
  \begin{equation*}
  \begin{aligned}
    \MoveEqLeft \sum_{l=1}^j \sum_{m=1}^{2l-1} \sum_{l+\abs{\alpha} + \abs{\beta}
      + k = j+1} c_{\alpha,\beta} (\partial_p^{\alpha} D_x^{\beta}
    a_{\varepsilon,k}) (\partial_p^{\beta} D_x^{\alpha}d_{\varepsilon,
      l ,m} b_{\varepsilon,z,0}^{m+1})
    \\
    ={}& \sum_{l=1}^j \sum_{m=1}^{2l-1} \sum_{l+\abs{\alpha} +
      \abs{\beta} + k = j+1} c_{\alpha,\beta} (\partial_p^{\alpha}
    D_x^{\beta} a_{\varepsilon,k}) (-i)^{\abs{\alpha}} \sum_{n=0}^{
      \abs{\alpha}+\abs{\beta}} \tilde{d}_{\varepsilon, l ,m, n,
      \alpha,\beta} b_{\varepsilon,z,0}^{m+1+n}
    \\
    ={}& \sum_{m=1}^{2j-1} \sum_{l =\lceil \frac{m-1}{2} \rceil +1 }^j
    \sum_{l+\abs{\alpha} + \abs{\beta} + k = j+1} \sum_{n=0}^{
      \abs{\alpha}+\abs{\beta}} c_{\alpha,\beta} (\partial_p^{\alpha}
    D_x^{\beta} a_{\varepsilon,k}) (-i)^{\abs{\alpha}}
    \tilde{d}_{\varepsilon, l ,m, n, \alpha,\beta}
    b_{\varepsilon,z,0}^{m+1+n}
    \\
    ={}&\sum_{m=1}^{2j-1} \sum_{l =\lceil \frac{m-1}{2} \rceil +1 }^j
    \sum_{n=0}^{ j+1-l} \sum_{\abs{\alpha} + \abs{\beta} = n}
    c_{\alpha,\beta} (\partial_p^{\alpha} D_x^{\beta}
    a_{\varepsilon,j+1-l-n}) (-i)^{\abs{\alpha}}
    \tilde{d}_{\varepsilon, l ,m, n, \alpha,\beta}
    b_{\varepsilon,z,0}^{m+1+n},
    \end{aligned}
  \end{equation*}
  where the $\tilde{d}_{\varepsilon, l ,m, n \alpha,\beta}$ are the
  polynomials from Lemma~\ref{B.approx_of_resolvent_lemma_2}. The way
  we have expressed the sums ensures $k$ always is the fixed value
  $j+1-l-n$. From Lemma~\ref{B.approx_of_resolvent_lemma_2} we have
  that $\tilde{d}_{\varepsilon, l ,m, n, \alpha,\beta}$ are
  polynomials in
  $\partial_x^{\alpha'} \partial_p^{\beta'} a_{\varepsilon,m}$ with
  $\abs{\alpha'}+\abs{\beta'}+m \leq l+n\leq j+1$ of degree $n+m$ and
  with regularity at least $\tau-l-\abs{\alpha}$. Hence the factors
  $c_{\alpha,\beta} (\partial_p^{\alpha} D_x^{\beta}
  a_{\varepsilon,j+1-l-n}) (-i)^{\abs{\alpha}} \tilde{d}_{\varepsilon,
    l ,m, n, \alpha,\beta}$ will be polynomials in
  $\partial_x^{\alpha'} \partial_p^{\beta'} a_{\varepsilon,m}$ with
  $\abs{\alpha'}+\abs{\beta'}+m\leq j+1$ of degree $n+m+1$. The
  regularity of the terms will be at least
  \begin{equation*}
    \tau- l-\abs{\alpha} - (j+1-l-n) - \abs{\beta} = \tau-(j+1), 
  \end{equation*}
  where most terms will have more regularity. By rewriting and
  renaming some of the terms we get the following equality
  \begin{equation*}
    \sum_{l=1}^j \sum_{m=1}^{2l-1} \sum_{l+\abs{\alpha} + \abs{\beta} + k = j+1} c_{\alpha,\beta}  (\partial_p^{\alpha} D_x^{\beta} a_{\varepsilon,k}) (\partial_p^{\beta} D_x^{\alpha}d_{\varepsilon, l ,m} b_{\varepsilon,z,0}^{m+1})  = \sum_{n=1}^{2j}  \tilde{d}_{\varepsilon,j,n,\alpha,\beta}  b_{\varepsilon,z,0}^{n+1},
  \end{equation*}
  where $\tilde{d}_{\varepsilon,j,n,\alpha,\beta} $ again are
  polynomials in
  $\partial_x^{\alpha'} \partial_p^{\beta'} a_{\varepsilon,k}$ with
  $\abs{\alpha'}+\abs{\beta'}+k \leq j+1$ of degree $n+1$ of
  regularity at least $\tau-(j+1)$. By combing these calculation we
  arrive at the expression
  \begin{equation*}
  \begin{aligned}
    b_{\varepsilon,z,j+1} ={}& - b_{\varepsilon,z,0} \cdot
    \sum_{\substack{l + \abs{\alpha} + \abs{\beta} + k = j+1\\ 0\leq l
        \leq j }} \frac{1}{\alpha ! \beta !}
    \frac{1}{2^{\abs{\alpha}}} \frac{1}{(-2)^{\abs{\beta}}}
    (\partial_p^{\alpha} D_x^{\beta} a_{\varepsilon,k})
    (\partial_p^{\beta} D_x^{\alpha} b_{\varepsilon,z,l})
    \\
    ={}& - b_{\varepsilon,z,0} \Big\{ \sum_{n=0}^{j+1}
    \tilde{d}_{\varepsilon,j,n,\alpha,\beta} b_{\varepsilon,z,0}^{n+1}
    + \sum_{n=1}^{2j} \tilde{d}_{\varepsilon,j,n,\alpha,\beta}
    b_{\varepsilon,z,0}^{n+1}\Big\}
    \\
    ={}& \sum_{k=1}^{2j+1} d_{\varepsilon,j+1,k}
    b_{\varepsilon,z,0}^{k+1},
    \end{aligned}
  \end{equation*}
  where the polynomials $d_{\varepsilon,j+1,k}$ are universal
  polynomials in
  $\partial_x^{\alpha'} \partial_p^{\beta'} a_{\varepsilon,k}$ with
  $\abs{\alpha'}+\abs{\beta'}+k \leq j+1$ of degree $k$ and with
  regularity at least $\tau-j-1$. Hence they are elements of
  $\Gamma_{0,\varepsilon}^{(a_0-\zeta_1)^k,\tau-(j+1)}$.  This ends
  the proof.
\end{proof}
\begin{lemma}\label{B.approx_of_resolvent_lemma_3}
  Let the setting be as in Theorem~\ref{B.approx_of_resolvent}. For
  every $j$ in $\N$ and $\alpha$, $\beta$ in $\N^d$ there exists a
  number $c_{j,\alpha,\beta}>0$ such that
  \begin{equation*}
    \abs{\partial_p^\beta\partial_x^\alpha b_{\varepsilon,z,j}} \leq c_{j,\alpha,\beta} \varepsilon^{-(\tau-j-\abs{\alpha})_{-}}  (a_{\varepsilon,0}-\zeta_1)^{-1} \left( \frac{\abs{z-\zeta_1}}{\dist(z,[\zeta_1,\infty))}\right)^{2j+\abs{\alpha}+\abs{\beta}},
  \end{equation*}
  for all $z \in \C\setminus[\zeta_1,\infty)$ and all
  $(x,p) \in \R^d_x\times\R^d_p$.
\end{lemma}

\begin{proof}
  We start by considering the fraction
  $\frac{\abs{a_{\varepsilon,0}-\zeta_1}}{\abs{a_{\varepsilon,0}-z}}$
  and we will divide it into two cases according to the real part of
  $z$. If $\re(z)<\zeta_1$ then
  \begin{equation*}
    \frac{\abs{a_{\varepsilon,0}-\zeta_1}}{\abs{a_{\varepsilon,0}-z}} \leq 1 \leq  \frac{\abs{z-\zeta_1}}{\dist(z,[\zeta_1,\infty))}.
  \end{equation*}
  If instead $\re(z)\geq\zeta_1$ and $\abs{\im{z}}>0$ we have by the
  law of sines that
  \begin{equation*}
    \frac{\abs{a_{\varepsilon,0} - z}}{\sin(\phi_1)} = \frac{\abs{a_{\varepsilon,0}-\zeta_1}}{\sin(\phi_2)} \geq \abs{a_{\varepsilon,0}-\zeta_1},
  \end{equation*}
  where $\phi_1$ and $\phi_2$ are angles in the triangle with vertices $\zeta_1$, $a_{\varepsilon,0}$ and $z$.  We have in the above estimate used that $0< \sin(\phi_2)\leq 1$. If we apply this
  inequality and the law of sines again we arrive at the following
  expression
  \begin{equation*}
    \frac{\abs{a_{\varepsilon,0}-\zeta_1}}{\abs{a_{\varepsilon,0}-z}} \leq \frac{1}{\sin(\phi_1)} = \frac{\abs{z-\zeta_1}}{\abs{\im(z)}}.
  \end{equation*}	
  Combining these two cases we get the estimate
  \begin{equation}\label{B.res_aprox4.1}
    \frac{\abs{a_{\varepsilon,0}-\zeta_1}}{\abs{a_{\varepsilon,0}-z}} \leq \frac{\abs{z-\zeta_1}}{\dist(z,[\zeta_1,\infty))} \quad\quad\text{For all $z \in\C\setminus[\zeta_1,\infty)$.}
  \end{equation}
  If we now consider a
  given $b_{\varepsilon,z,j}$ and $\alpha,\beta$ in
  $\N^d$. Lemma~\ref{B.approx_of_resolvent_lemma_1} and
  Lemma~\ref{B.approx_of_resolvent_lemma_2} gives us
  \begin{equation*}
    \partial_p^\beta\partial_x^\alpha b_{\varepsilon,z,j} = \sum_{k=1}^{2j-1} \partial_p^\beta\partial_x^\alpha (d_{\varepsilon,j,k}  b_{\varepsilon,z,0}^{k+1}) = \sum_{k=1}^{2j-1}  \sum_{n=0}^{\abs{\alpha}+\abs{\beta}} \tilde{d}_{\varepsilon, j ,k, n, \alpha , \beta} b_{\varepsilon,z,0}^{k+1+n},
  \end{equation*}
  with $\tilde{d}_{\varepsilon,j,k,\alpha,\beta}$ in
  $\Gamma_{0,\varepsilon}^{(a_0-\zeta_1)^{k+n},\tau-j-\abs{\alpha}}$. By
  taking absolute value and applying \eqref{B.res_aprox4.1} we get
  that
  \begin{equation*}
  \begin{aligned}
    \abs{\partial_p^\beta\partial_x^\alpha b_{\varepsilon,z,j} } 
    \leq{}&
    \sum_{k=1}^{2j-1} \sum_{n=0}^{\abs{\alpha}+\abs{\beta}}
    \abs{\tilde{d}_{\varepsilon, j ,k, n, \alpha , \beta}
      b_{\varepsilon,z,0}^{k+1+n}}
    \\
    \leq{}& \abs{b_{\varepsilon,z,0}} \sum_{k=1}^{2j-1}
    \sum_{n=0}^{\abs{\alpha}+\abs{\beta}}
    \varepsilon^{-(\tau-j-\abs{\alpha})_{-}} c_{j,k,\alpha,\beta}
    \left(
      \frac{\abs{a_{\varepsilon,0}-\zeta_1}}{\abs{a_{\varepsilon,0}-z}}\right)^{k+n}
    \\
    \leq{}& c_{j,\alpha,\beta} \varepsilon^{-(\tau-j-\abs{\alpha})_{-}}
    (a_{\varepsilon,0}-\zeta_1)^{-1}
    \left(\frac{\abs{z-\zeta_1}}{\dist(z,[\zeta_1,\infty))}\right)^{2j+\abs{\alpha}+\abs{\beta}},
    \end{aligned}
  \end{equation*}
  where we have use that
  \begin{equation*}
    \abs{b_{\varepsilon,z,0}} = \frac{(a_{\varepsilon,0}-\zeta_1)}{|a_{\varepsilon,0}-z|(a_{\varepsilon,0}-\zeta_1)}  \leq \frac{1}{(a_{\varepsilon,0}-\zeta_1)}  \left(\frac{\abs{z-\zeta_1}}{\dist(z,[\zeta_1,\infty))}\right).
  \end{equation*} 
  We have now obtained the desired estimate and this ends the proof.
\end{proof}
\begin{proof}[Proof of Theorem~\ref{B.approx_of_resolvent}]
  By Lemma~\ref{B.approx_of_resolvent_lemma_1} the
  symbols $b_{\varepsilon,z,j}$ are in the class
  $\Gamma_{0,\varepsilon}^{(a_0-\zeta_1)^{-1},\tau-j}$ for every $j$
  in $\N$, where $b_{\varepsilon,z,j}$ is defined
  \eqref{B.res.symbolform_1}. Hence we have that
  \begin{equation*}
    B_{\varepsilon,z,N}(\hbar) = \sum_{j=0}^N \hbar^j b_{\varepsilon,z,j}.
  \end{equation*}
  is a well defined symbol for every $N$ in $\N$. Moreover as
  $(a_0-\zeta_1)^{-1}$ is a bounded function we have by
  Theorem~\ref{B.cal-val-thm} that $\OpW(B_{\varepsilon,z,N}(\hbar))$
  is a bounded operator. Now for $N$ sufficiently large we have by
  assumption
  \begin{equation*}
    A_\varepsilon(\hbar) -z = \OpW( a_{\varepsilon,0}-z)+ \sum_{k=1}^N \hbar^k \OpW( a_{\varepsilon,k}) + \hbar^{N+1}  R_N(\varepsilon,\hbar),
  \end{equation*}
  where the error term satisfies
  \begin{equation*}
    \hbar^{N+1}  \norm{R_N(\varepsilon,\hbar)}_{\mathcal{L}(L^2(\R^d))} \leq \hbar^{\kappa(N)} C_N 
  \end{equation*}
  for a positive strictly increasing function $\kappa$. If we consider
  the composition of $A_\varepsilon(\hbar)$ and
  $\OpW(B_{\varepsilon,z,N}(\hbar))$ we get
  \begin{equation}\label{B.res_con_1}
    \begin{aligned}
      \MoveEqLeft A_\varepsilon(\hbar)\OpW(B_{\varepsilon,z,N}(\hbar)) 
      \\
      &=
      \sum_{k=0}^N \hbar^{k} \OpW( a_{\varepsilon,k}) \sum_{j=0}^N
      \hbar^{j} \OpW(b_{\varepsilon,z,j})
      + \sum_{j=0}^N \hbar^{N+1+j} R_N(\varepsilon,\hbar)
      \OpW(b_{\varepsilon,z,j}).
    \end{aligned}
  \end{equation}
  If we consider the first part then this corresponds to a composition
  of two strongly $\hbar$-$\varepsilon$-admissible operators. As we
  want to apply Theorem~\ref{B.composition-weyl-thm} we need to ensure
  $N$ satisfies the inequality
  \begin{equation*}
    \delta(N +2d + 2-\tau) +\tau > 2d+1.
  \end{equation*}
  As this is the condition that ensures a positive power in front of
  the error term. If we assume $N$ satisfies the inequality. Then by
  Theorem~\ref{B.composition-weyl-thm} we have
  \begin{equation*}
  	\begin{aligned}
    \MoveEqLeft \sum_{k=0}^N \hbar^{k} \OpW( a_{\varepsilon,k}) \sum_{j=0}^N
    \hbar^{j} \OpW(b_{\varepsilon,z,j}) 
    \\
    &= \sum_{l=0}^N \hbar^l
    \OpW(c_{\varepsilon,l})
     + \hbar^{N+1}
    \OpW(\mathcal{R}_\varepsilon(a_\varepsilon(\hbar),
    B_\varepsilon,z,N(\hbar);\hbar)),
    \end{aligned}
  \end{equation*}
  where
  \begin{equation*}
    c_{\varepsilon,l}(x,p) = \sum_{\abs{\alpha}+\abs{\beta}+k+j=l} \frac{1}{\alpha!\beta!}\Big(\frac{1}{2} \Big)^{\abs{\alpha}}\Big(-\frac{1}{2} \Big)^{\abs{\beta}} (\partial_p^\alpha D_x^\beta a_{\varepsilon,k}) (\partial_p^\beta D_x^\alpha b_{\varepsilon,z,j})(x,p).
  \end{equation*}
  The error term satisfies for every multi indices $\alpha$, $\beta$
  in $\N^d$, that there exists a constant $C(\alpha,\beta,N)$ independent of
  $a_\varepsilon$ and $B_{\varepsilon,z,N}$ and an integer $M$ such
  that
  \begin{equation*}
  \begin{aligned}
   \hbar^{N+1} |\partial_x^\alpha \partial_p^\beta \mathcal{R}_\varepsilon
    (a_\varepsilon(\hbar), b_\varepsilon(\hbar);x,p;\hbar) | 
   & \leq
    C(\alpha,\beta,N) \hbar^{\delta(\tau-N-2d-2)_{-} + \tau -2d-1}
    \varepsilon^{-\abs{\alpha}}
    \\
    &\times \sum_{ j+k \leq 2N}
    \mathcal{G}^{\alpha,\beta}_{M,\tau-j-k}(a_{\varepsilon,k},(a_{\varepsilon,0}-\zeta_1),b_{\varepsilon,z,j},(a_{\varepsilon,0}-\zeta_1)^{-1}),
    \end{aligned}
  \end{equation*}
  where $\mathcal{G}^{\alpha,\beta}_{M,\tau-j-k}$ are as defined in
  Theorem~\ref{B.composition-weyl-thm}. By
  Lemma~\ref{B.approx_of_resolvent_lemma_3} we have for all
  $j+k\leq 2N$
  \begin{equation*}
  \begin{aligned}
    \MoveEqLeft \mathcal{G}^{\alpha,\beta}_{M,\tau-j-k}(a_{\varepsilon,k},(a_{\varepsilon,0}-\zeta_1),b_{\varepsilon,z,j},(a_{\varepsilon,0}-\zeta_1)^{-1})
    \\
    &= \sup_{\substack{\abs{\gamma_1 + \gamma_2}+\abs{\eta_1 +
          \eta_2}\leq M \\ (x,\xi) \in \R^{2d}}} \varepsilon^{(\tau -j
      - k -M)_{-}+\abs{\alpha}}
    \abs{\partial_x^\alpha\partial_\xi^\beta
      ( \partial_{x}^{\gamma_1} \partial_{\xi}^{\eta_1}a_{\varepsilon,k}(x,\xi) \partial_{x}^{\gamma_2} \partial_{\xi}^{\eta_2}b_{\varepsilon,z,j}(x,\xi))}
    \\
    &\leq C_{\alpha,\beta,M}\; \sup_{\mathclap{\substack{\abs{\gamma_1 +
          \gamma_2}+\abs{\eta_1 + \eta_2}\leq M \\ (x,\xi) \in
        \R^{2d}}}} \; \varepsilon^{(\tau -j - k
      -M)_{-}-(\tau-k-\gamma_1)_{-}-(\tau-j-\gamma_2)_{-}}\left( \frac{\abs{z-\zeta_1}}{\dist(z,[\zeta_1,\infty))}\right)^{2j+M+\abs{\alpha}+\abs{\beta}}
    \\
    &\leq C_{\alpha,\beta,M} \left(
      \frac{\abs{z-\zeta_1}}{\dist(z,[\zeta_1,\infty))}\right)^{2j+M+\abs{\alpha}+\abs{\beta}}.
      \end{aligned}
  \end{equation*}
  Now by Theorem~\ref{B.cal-val-thm}
 there exists a number $M_d$ such
  that
  \begin{equation*}
    \hbar^{N+1}\norm{\OpW(\mathcal{R}_\varepsilon(a_\varepsilon(\hbar), B_\varepsilon,z,N(\hbar);\hbar))}_{\mathcal{L}(L^2(\R^d))} \leq C \hbar^{\kappa(N)} \left( \frac{\abs{z-\zeta_1}}{\dist(z,[\zeta_1,\infty))}\right)^{M_d}.
  \end{equation*}
  If we now consider the symbols $c_{\varepsilon,l}(x,p)$ for
  $0\leq l \leq N$. For $l=0$ we have
  \begin{equation*}
    c_{\varepsilon,0}(x,p) = (a_{\varepsilon,0}(x,p)-z) b_{\varepsilon,z,0}(x,p) =1,
  \end{equation*}
  by definition of $b_{\varepsilon,z,0}(x,p) $. Now for
  $1\leq l \leq N$ we have
  \begin{equation*}
  	\begin{aligned}
    c_{\varepsilon,l} ={}& \sum_{\abs{\alpha}+\abs{\beta}+k+j=l}
    \frac{1}{\alpha!\beta!}\Big(\frac{1}{2}
    \Big)^{\abs{\alpha}}\Big(-\frac{1}{2} \Big)^{\abs{\beta}}
    (\partial_p^\alpha D_x^\beta a_{\varepsilon,k}) (\partial_p^\beta
    D_x^\alpha b_{\varepsilon,z,j})
    \\
    ={}& \sum_{\substack{\abs{\alpha}+\abs{\beta}+k+j=l\\0\leq j\leq
        l-1}} \frac{1}{\alpha!\beta!}\Big(\frac{1}{2}
    \Big)^{\abs{\alpha}}\Big(-\frac{1}{2} \Big)^{\abs{\beta}}
    (\partial_p^\alpha D_x^\beta a_{\varepsilon,k}) (\partial_p^\beta
    D_x^\alpha b_{\varepsilon,z,j}) +(a_{\varepsilon,0}-z)
    b_{\varepsilon,z,l}
    \\
    ={}& \sum_{\substack{\abs{\alpha}+\abs{\beta}+k+j=l\\0\leq j\leq
        l-1}} \frac{1}{\alpha!\beta!}\Big(\frac{1}{2}
    \Big)^{\abs{\alpha}}\Big(-\frac{1}{2} \Big)^{\abs{\beta}}
    (\partial_p^\alpha D_x^\beta a_{\varepsilon,k}) (\partial_p^\beta
    D_x^\alpha b_{\varepsilon,z,j}) 
    \\
    &-\sum_{\substack{\abs{\alpha} + \abs{\beta} + k +j = l\\
        0\leq j \leq l-1 }} \frac{1}{\alpha ! \beta !}
    \frac{1}{2^{\abs{\alpha}}} \frac{1}{(-2)^{\abs{\beta}}}
    (\partial_p^{\alpha} D_x^{\beta} a_{\varepsilon,k})
    (\partial_p^{\beta} D_x^{\alpha} b_{\varepsilon,z,j})
    \\
    ={}&0,
    \end{aligned}
  \end{equation*}
  by definition of $b_{\varepsilon,z,l}$. These two equalities implies that
  \begin{equation}\label{B.res_con_2}
    \sum_{k,j=0}^N \hbar^{k+j} \OpW( a_{\varepsilon,k})  \OpW(b_{\varepsilon,z,j})  =  I + \hbar^{N+1}\OpW(\mathcal{R}_\varepsilon(a_\varepsilon(\hbar), B_\varepsilon,z,N(\hbar);\hbar)).
  \end{equation}
  This was the first part of equation~\eqref{B.res_con_1}. If we now
  consider the second part of \eqref{B.res_con_1}:
  \begin{equation*}
    \sum_{j=0}^N \hbar^{N+1+j}  R_N(\varepsilon,\hbar) \OpW(b_{\varepsilon,z,j}).
  \end{equation*}
  By Theorem~\ref{B.cal-val-thm} and
  Lemma~\ref{B.approx_of_resolvent_lemma_3} there exist constants
  $M_d$ and $C$ such that
  \begin{equation*}
    \hbar^j \norm{\OpW(b_{\varepsilon,z,j})}_{\mathcal{L}(L^2(\R^d))} \leq C \left( \frac{\abs{z-\zeta_1}}{\dist(z,[\zeta_1,\infty))}\right)^{2j+M_d}
  \end{equation*}
  for all $j$ in $\{0,\dots,N\}$. Hence by assumption we have
  \begin{equation*}
    \sum_{j=0}^N \hbar^{N+1+j}  \norm{R_N(\varepsilon,\hbar) \OpW(b_{\varepsilon,z,j})}_{\mathcal{L}(L^2(\R^d))} \leq C \hbar^{\kappa(N)}  \left( \frac{\abs{z-\zeta_1}}{\dist(z,[\zeta_1,\infty))}\right)^{q(N)}.
  \end{equation*}
  Now by combining this with \eqref{B.res_con_1} and
  \eqref{B.res_con_2} we get
  \begin{equation*}
    (A_\varepsilon(h) - z ) \OpW B_{\varepsilon,z,N} = I + h^{N+1} \Delta_{z,N+1}(h)
  \end{equation*}
  with
  \begin{equation*}
    \hbar^{N+1}\norm{ \Delta_{z,N+1}(h)}_{\mathcal{L}(L^2(\R^d))} \leq C \hbar^{\kappa(N)} \left( \frac{\abs{z}}{\dist(z,[\zeta_1,\infty))} \right)^{q(N)},
  \end{equation*}
  where $\kappa$ is a positive strictly increasing function and $q(N)$
  is a positive integer depending on $N$. Which is the desired form and this ends the proof.
\end{proof}
\subsection{Functional calculus for rough pseudo-differential operators}
We are now almost ready to construct/prove a functional calculus for
operators satisfying Assumption~\ref{B.self_adj_assumption}. First we
need to settle some terminology and recall a theorem.
\begin{definition}\label{B.def_almost_analytic_ex}
  For a smooth function $f:\R\rightarrow\R$ we define the almost
  analytical extension $\tilde{f}:\C\rightarrow\C$ of $f$ by
  \begin{equation*}
    \tilde{f}(x+iy) = \left( \sum_{r=0}^n f^{(r)}(x)\frac{(iy)^r}{r!} \right) \sigma(x,y),
  \end{equation*}
  where $n \geq1$ and
  \begin{equation*}
    \sigma(x,y) = \omega\left(\tfrac{y}{\lambda(x)}\right),
  \end{equation*}
  for some smooth function $\omega$, defined on $\R$ such that
  $\omega(t)=1$ for $\abs{t}\leq1$ and $\omega(t)=0$ for
  $\abs{t}\geq2$. Moreover we will use the notation
  \begin{equation*}
  \begin{aligned}
    \bar{\partial} \tilde{f}(x+iy)& \coloneqq \frac{1}{2} \left(
      \frac{\partial \tilde{f}}{\partial x} + i \frac{\partial
        \tilde{f}}{\partial y} \right)
    \\
    & = \frac{1}{2} \left( \sum_{r=0}^n f^{(r)}(x)\frac{(iy)^r}{r!}
    \right) ( \sigma_x(x,y) + i \sigma_y(x,y)) + \frac{1}{2}
    f^{(n+1)}(x)\frac{(iy)^n}{n!} \sigma(x,y),
    \end{aligned}
  \end{equation*}
  where $\sigma_x$ and $\sigma_y$ are the partial derivatives of
  $\sigma$ with respect to $x$ and $y$ respectively.
\end{definition}
\begin{remark}\label{B.almost_analytic_ex_remark}
  The above choice is one way to define an almost analytic extension
  and it is not unique. Once an $n$ has been fixed the extension has
  the property that
  \begin{equation*}
    |\bar{\partial} \tilde{f}(x+iy)| =\mathcal{O}( \abs{y}^n)
  \end{equation*}	
  as $y\rightarrow0$. Hence when making calculation a choice has to be
  made concerning how fast $|\bar{\partial} \tilde{f}|$ vanishes when
  approaching the real axis. If $f$ is a $C_0^\infty(\R)$ function one
  can find an almost analytic extension $\tilde{f}$ in
  $C_0^\infty(\C)$ such $f(x)=\tilde{f}(x)$ for $x$ in $\R$ and
  \begin{equation*}
    |\bar{\partial} \tilde{f}(x+iy)| \leq C_N  \abs{y}^N, \quad\text{for all } N\geq0.
  \end{equation*}
  without chancing the extension. This type of extension could be
  based on a Fourier transform hence it may not work for a general
  smooth function. For details see \cite[Chapter 8]{MR1735654} or \cite[Chapter 3]{MR2952218}.
\end{remark}
The type of functions for which we can construct a functional calculus
is introduced in the next definition:
\begin{definition}
  For $\rho$ in $\R$ we define the set $S^\rho$ to be the set of
  smooth functions $f:\R\rightarrow\R$ such that
  \begin{equation*}
    |f^{(r)}(x)| \coloneqq \Big|\frac{d^r f}{dx^r}(x) \Big| \leq c_r \lambda(x)^{\rho-r}
  \end{equation*}
  for some $c_r< \infty$, all $x$ in $\R$ and all integers $r\geq0$. Moreover we define
  $\mathcal{W}$ by
  \begin{equation*}
    \mathcal{W} \coloneqq \bigcup_{\rho<0} S^\rho.
  \end{equation*}
\end{definition}
We can now recall the form of the spectral theorem which we will use:
\begin{thm}[The Helffer-Sj\"{o}strand
  formula]\label{B.Helffer-Sjostrand}
  Let $H$ be a self-adjoint operator acting on a Hilbert space
  $\mathscr{H}$ and $f$ a function from $\mathcal{W}$. Moreover let
  $\tilde{f}$ be an almost analytic extension of $f$ with $n$
  terms. Then the bounded operator $f(H)$ is given by the equation
  \begin{equation*}
    f(H) =- \frac{1}{\pi} \int_\C   \bar{\partial }\tilde{f}(z) (z-H)^{-1} \, L(dz),
  \end{equation*}
  where $L(dz)=dxdy$ is the Lebesgue measure on $\C$. The formula
  holds for all numbers $n\geq1$.
\end{thm}
A proof of the above theorem can be found in e.g. \cite{MR1735654} or \cite{MR1349825}. We are now ready to state and prove the functional calculus for a
certain class of rough pseudo-differential operators.
\begin{thm}\label{B.func_calc}
  Let $A_\varepsilon(\hbar)$, for $\hbar$ in $(0,\hbar_0]$, be a
  $\hbar$-$\varepsilon$-admissible operator of regularity
  $\tau \geq 1$ and with symbol
  \begin{equation*}
    a_\varepsilon(\hbar) = \sum_{j\geq 0} \hbar^j a_{\varepsilon,j}.
  \end{equation*}
  Suppose that $A_\varepsilon(h)$ satisfies
  Assumption~\ref{B.self_adj_assumption}. Then for any function $f$
  from $\mathcal{W}$, $f(A_\varepsilon(h))$ is a
  $\hbar$-$\varepsilon$-admissible operator of regularity $\tau$ with
  respect to a constant tempered weight
  function. $f(A_\varepsilon(h))$ has the symbol
  \begin{equation*}
    a_{\varepsilon}^f(\hbar) = \sum_{j\geq 0} \hbar^j a_{\varepsilon,j}^f,
  \end{equation*}
  where
  \begin{equation}\label{B.func_cal_sym} 
  	\begin{aligned}
  	a_{\varepsilon,0}^f &= f(a_{\varepsilon,0}),
	\\
    	a_{\varepsilon,j}^f &= \sum_{k=1}^{2j-1} \frac{(-1)^k}{k!} d_{\varepsilon,j,k} f^{(k)}(a_{\varepsilon,0}) \quad\quad\text{for $j\geq1$},
    \end{aligned}
  \end{equation}
  the symbols $d_{\varepsilon,j,k}$ are the polynomials from
  Lemma~\ref{B.approx_of_resolvent_lemma_1}. Especially we have
  \begin{equation*}
    a_{\varepsilon,1}^f =
    a_{\varepsilon,1}f^{(1)}(a_{\varepsilon,0}).
  \end{equation*}
\end{thm}
The proof is an application of Theorem~\ref{B.Helffer-Sjostrand} and
the fact that the resolvent is a $\hbar$-$\varepsilon$-admissible
operator as well.
\begin{proof}
  By Theorem~\ref{B.self_adjoint_thm_1} the operator
  $A_\varepsilon(\hbar)$ is essentially self-adjoint for sufficiently small $\hbar$. Hence
  Theorem~\ref{B.Helffer-Sjostrand} gives us
  \begin{equation*}
    f(A_\varepsilon(\hbar)) =- \frac{1}{\pi} \int_\C  \bar{\partial }\tilde{f}(z) (z-A_\varepsilon(\hbar))^{-1} \, L(dz),
  \end{equation*}
  where $\tilde{f}$ is an almost analytic extension of $f$. For the
  almost analytic extension of $f$ we will need a suffiently large
  number of terms which we assume to have chosen from the
  start. Theorem~\ref{B.approx_of_resolvent} gives that the resolvent
  is a $\hbar$-$\varepsilon$-admissible operator and the explicit form
  of it as well. Hence
  \begin{equation*}
  \begin{aligned}
   \MoveEqLeft  f(A_\varepsilon(\hbar)) 
    \\
    ={}& \frac{1}{\pi} \int_\C \bar{\partial
    }\tilde{f}(z) \sum_{j=0}^M \hbar^j \OpW(b_{\varepsilon,z,j}) \,
    L(dz)
    - \frac{1}{\pi} \int_\C \bar{\partial }\tilde{f}(z)
    h^{N+1} (z-A_\varepsilon(\hbar))^{-1}\Delta_{z,N+1}(h) \, L(dz),
    \end{aligned}
  \end{equation*}
  where the symbols $b_{\varepsilon,z,j}$ and the operator
  $\Delta_{z,N+1}(h)$ are as defined in
  Theorem~\ref{B.approx_of_resolvent}. If we start by considering the
  error term we have by Theorem~\ref{B.approx_of_resolvent} the
  estimate
  \begin{equation*}
  \begin{aligned}
    \norm{(z-A_\varepsilon(\hbar))^{-1}\Delta_{z,N+1}(h)}_{\mathcal{L}(L^2(\R^d))}
    \leq{}& C \hbar^{\kappa(N)} \frac{1}{\abs{\im(z)}} \left(
      \frac{\abs{z}}{\dist(z,[\zeta_1,\infty))} \right)^{q(N)}
    \\
    \leq{}& C \hbar^{\kappa(N)}
    \frac{\abs{z}^{q(N)}}{\abs{\im(z)}^{q(N)+1} },
    \end{aligned}
  \end{equation*}
  for $N$ sufficiently large and where $q(N)$ is some integer
  dependent of the number $N$. We have
  \begin{equation*}
    |\bar{\partial}\tilde{f}(z)| \leq c_1 \sum_{r=0}^n |\tilde{f}^{(r)}(\re(z))| \lambda(\re(z))^{r-1} \boldsymbol{1}_U(z) + c_2 |\tilde{f}^{(n+1)}(\re(z))| |\im(z)|^n \boldsymbol{1}_V(z),
  \end{equation*}
  where
  \begin{align*}
    U &=\{z \in \C\, |\, \lambda(\re(z)) <
    \abs{\im(z)}<2\lambda(\re(z))\}, \shortintertext{and} V &=\{z \in
    \C \, | \, 0 < \abs{\im(z)}<2\lambda(\re(z))\}.
   \end{align*}
  This estimate follows directly from the definition of
  $\tilde{f}$. By combining these estimates and the definition of the
  class of functions $\mathcal{W}$ we have
  \begin{equation*}
    \Big\lVert\frac{1}{\pi} \int_\C  \bar{\partial }\tilde{f}(z) h^{N+1} (z-A_\varepsilon(\hbar))^{-1}\Delta_{z,N+1}(h) \, L(dz)\Big\rVert_{\mathcal{L}(L^2(\R^d))} \leq C \hbar^{\kappa(N)}.
  \end{equation*}
  What remains to prove the following equality
  \begin{equation}\label{B.sum_symbol_func_proof}
    \sum_{j=0}^M \hbar^j \frac{1}{\pi} \int_\C  \bar{\partial }\tilde{f}(z) \OpW(b_{\varepsilon,z,j}) \, L(dz) =   \sum_{j=0}^M \hbar^j \OpW(a_{\varepsilon,j}^f), 
  \end{equation}
  where the symbols $a_{\varepsilon,j}^f$ are as defined in the
  statement. We will only consider one of the terms as the rest is
  treated analogously. Hence we need to establish the equality
  \begin{equation*}
    \frac{1}{\pi} \int_\C  \bar{\partial }\tilde{f}(z) \OpW(b_{\varepsilon,z,j}) \, L(dz) =   \OpW(a_{\varepsilon,j}^f). 
  \end{equation*}
  As both operators are bounded we need only establish the equality
  weekly for a dense subset of $L^2(\R^d)$. Hence let $\varphi$ and
  $\psi$ be two functions from $C_0^\infty(\R^d)$ and a $j$ be
  given. We have
  \begin{equation}\label{B.func_cal_1}
    \langle \frac{1}{\pi} \int_\C  \bar{\partial }\tilde{f}(z) \OpW(b_{\varepsilon,z,j}) \, L(dz) \varphi , \psi \rangle =  \frac{1}{\pi} \int_\C  \bar{\partial }\tilde{f}(z) \langle \OpW(b_{\varepsilon,z,j}) \varphi , \psi \rangle \, L(dz) ,
  \end{equation}
  where we have
  \begin{equation}\label{B.func_cal_2}
    \begin{aligned}
      \langle \OpW&(b_{\varepsilon,z,j}) \varphi , \psi \rangle
      \\
      ={}& \frac{1}{(2\pi\hbar)^d}\int_{\R^{3d}}
      e^{i\hbar^{-1}\langle x-y,p\rangle}
      b_{\varepsilon,z,j}(\tfrac{x+y}{2},p) \varphi(y)
      \overline{\psi}(x) \, dydpdx
      \\
      ={}& \lim_{\sigma\rightarrow\infty} \frac{1}{(2\pi\hbar)^d}
      \int_{\R^{3d}} e^{i\hbar^{-1}\langle
        x-y,p\rangle} g_\sigma(x,y,p)
      b_{\varepsilon,z,j}(\tfrac{x+y}{2},p) \varphi(y)
      \overline{\psi}(x) \, dydpdx,
    \end{aligned}
  \end{equation}
  where the function $g$ is a positiv Schwartz function bounded by $1$
  and identical $1$ in a neighbourhood of $0$. I the above we have set
  $g_{\sigma}(x,y,p)=g(\tfrac{x}{\sigma},\tfrac{y}{\sigma},\tfrac{p}{\sigma})$. The
  next step in the proof is to apply dominated convergence to move the
  limit outside the integral over $z$.

  We let $\chi$ be in $C_0^\infty(\R^d)$ such that $\chi(p)=1$ for
  $\abs{p}\leq1$ and $\chi(p)=0$ for $\abs{p}\geq2$. With this
  function we have
  \begin{equation}\label{B.func_cal_3}
    \begin{aligned}
      \MoveEqLeft \frac{1}{(2\pi\hbar)^d}\int_{\R^{3d}}
      e^{i\hbar^{-1}\langle x-y,p\rangle} g_\sigma(x,y,p)
      b_{\varepsilon,z,j}(\tfrac{x+y}{2},p) \varphi(y)
      \overline{\psi}(x) \, dydpdx
      \\
      &\phantom{\frac{1}{2}}{}= \frac{1}{(2\pi\hbar)^d}\big[
     \int_{\R^{3d}}e^{i\hbar^{-1}\langle
        x-y,p\rangle} g_\sigma(x,y,p) \chi(p)
      b_{\varepsilon,z,j}(\tfrac{x+y}{2},p) \varphi(y)
      \overline{\psi}(x) \, dydpdx
      \\
      &\phantom{\frac{1}{2}=}{}+ \int_{\R^{3d}}
      e^{i\hbar^{-1}\langle x-y,p\rangle} g_\sigma(x,y,p)(1-\chi(p))
      b_{\varepsilon,z,j}(\tfrac{x+y}{2},p) \varphi(y)
      \overline{\psi}(x) \, dydpdx\big].
    \end{aligned}
  \end{equation}
  By Lemma~\ref{B.approx_of_resolvent_lemma_3} we have
  \begin{equation}\label{B.func_cal_4}
    \begin{aligned}
      \MoveEqLeft \Big|\int_{\R^{3d}}e^{i\hbar^{-1}\langle
        x-y,p\rangle} g_\sigma(x,y,p) \chi(p)
      b_{\varepsilon,z,j}(\tfrac{x+y}{2},p) \varphi(y)
      \overline{\psi}(x) \, dydpdx \Big|
      \\
      &\leq C_j \varepsilon^{-(\tau-j)_{-}} \left(
        \frac{\abs{z-\zeta_1}}{\abs{\im(z)}}\right)^{2j},
    \end{aligned}
  \end{equation}
  where the $\zeta_1$ is the number from
  Assumption~\ref{B.self_adj_assumption}. The factor
  $\varepsilon^{-(\tau-j)_{-}}$ is not an issue as the operator we
  consider has $\hbar^j$ in front. We have just omitted to write this
  factor. This bound is clearly independent of $\sigma$. We now need
  to bound the term with $1-\chi(p)$. Here we use that on the support
  of $1-\chi(p)$ we have $\abs{p}>1$. Hence the operator
  \begin{equation*}
    M= \frac{(-i\hbar)^{2d}}{ \abs{p}^{2d}}(\sum_{k=1}^d \partial_{y_k}^2)^d = \frac{(-i\hbar)^{2d}}{ \abs{p}^{2d}} \sum_{\abs{\alpha}=d} \partial_{y}^{2\alpha},
  \end{equation*}
  is well defined when acting on functions supported in
  $\supp(1-\chi)$. We then obtain that
  \begin{equation*}
     \begin{aligned}
    \MoveEqLeft \int_{\R^{3d}} e^{i\hbar^{-1}\langle
      x-y,p\rangle} g_\sigma(x,y,p) (1-\chi(p))
    b_{\varepsilon,z,j}(\tfrac{x+y}{2},p) \varphi(y)
    \overline{\psi}(x) \, dydpdx 
    \\
    &= \int_{\R^{3d}} (M e^{i\hbar^{-1}\langle
      x-y,p\rangle}) g_\sigma(x,y,p) (1-\chi(p))
    b_{\varepsilon,z,j}(\tfrac{x+y}{2},p) \varphi(y)
    \overline{\psi}(x) \, dydpdx 
    \\
    &= \int_{\R^{3d}} e^{i\hbar^{-1}\langle
      x-y,p\rangle}(1-\chi(p)) M^t(g_\sigma(x,y,p)
    b_{\varepsilon,z,j}(\tfrac{x+y}{2},p)
    \varphi(y))\overline{\psi}(x) \, dydpdx,
     \end{aligned}
  \end{equation*}
  where we have used thet $ M e^{i\hbar^{-1}\langle x-y,p\rangle} = e^{i\hbar^{-1}\langle x-y,p\rangle}$. Just considering the action of $M^t$ inside the integral we have by Leibniz's formula
  \begin{equation*}
  \begin{aligned}
 	|M^tg_\sigma(x,y,p) b_{\varepsilon,z,j}(\tfrac{x+y}{2},p)
    \varphi(y) |
       &= \frac{\hbar^{2d}}{ \abs{p}^{2d}} \Big| \sum_{\abs{\alpha}=d}
    \sum_{\beta\leq
      2\alpha} \partial_{y}^{2\alpha-\beta}(g_\sigma(x,y,p)
    \varphi(y)) \partial_{y}^{\beta}
    b_{\varepsilon,z,j}(\tfrac{x+y}{2},p) \Big|
    \\
    &\leq  C_j \frac{\hbar^{2d}}{ \abs{p}^{2d}}
    \boldsymbol{1}_{\supp(\varphi)}(y) \sum_{\abs{\alpha}=d}
    \sum_{\beta\leq 2\alpha} \varepsilon^{-(\tau-j-\abs{\beta})_{-}}
    \left( \frac{\abs{z-\zeta_1}}{\abs{\im(z)}}\right)^{2j +
      \abs{\beta}}
    \\
    &\leq C \frac{\hbar^{2d}}{ \abs{p}^{2d}}
    \varepsilon^{-(\tau-j-2d)_{-}} \boldsymbol{1}_{\supp(\varphi)}(y)
    \left( 1+\frac{\abs{z-\zeta_1}}{\abs{\im(z)}}\right)^{2j + 2d},
    \end{aligned}
  \end{equation*}
  where we again have used
  Lemma~\ref{B.approx_of_resolvent_lemma_3}. This imply the estimate
  \begin{equation*}
  \begin{aligned}
    \MoveEqLeft\Big|\int_{\R^{3d}} e^{i\hbar^{-1}\langle
      x-y,p\rangle} g_\sigma(x,y,p) (1-\chi(p))
    b_{\varepsilon,z,j}(\tfrac{x+y}{2},p) \varphi(y)
    \overline{\psi}(x) \, dydpdx \Big|
    \\
    &\leq C_j \varepsilon^{-(\tau-j)_{-}} \left(
      1+\frac{\abs{z-\zeta_1}}{\abs{\im(z)}}\right)^{2j + 2d}.
      \end{aligned}
  \end{equation*}
  If we combine this estimate with \eqref{B.func_cal_3} and
  \eqref{B.func_cal_4} we have
  \begin{equation*}
  \begin{aligned}
   \MoveEqLeft\Big| \frac{1}{(2\pi\hbar)^d} \int_{\R^{3d}} 
    e^{i\hbar^{-1}\langle x-y,p\rangle} g_\sigma(x,y,p)
    b_{\varepsilon,z,j}(\tfrac{x+y}{2},p) \varphi(y)
    \overline{\psi}(x) \, dydpdx \Big|
    \\
    &\leq C \varepsilon^{-(\tau-j)_{-}} \left(
      1+\frac{\abs{z-\zeta_1}}{\abs{\im(z)}}\right)^{2j + 2d}.
      \end{aligned}
  \end{equation*}
  As above we have
  \begin{equation*}
   \Big| \int_\C \bar{\partial }\tilde{f}(z) \left(
      1+\frac{\abs{z-\zeta_1}}{\abs{\im(z)}}\right)^{2j + 2d} \,
    L(dz)\Big| < \infty.
  \end{equation*}
  Hence we can apply dominated convergence and by an analogous
  argument we can also apply Fubini's Theorem. This gives
  \begin{equation}\label{B.func_cal_5}
    \begin{aligned}
     \MoveEqLeft  \frac{1}{\pi} \int_\C \bar{\partial }\tilde{f}(z) \langle
      \OpW(b_{\varepsilon,z,j}) \varphi , \psi \rangle \, L(dz)
     \\
      ={}& \lim_{\sigma\rightarrow\infty} \frac{1}{(2\pi\hbar)^d}
      \int_{\R^{3d}} e^{i\hbar^{-1}\langle
        x-y,p\rangle} g_\sigma(x,y,p)
      \frac{1}{\pi} \int_\C \bar{\partial }\tilde{f}(z)
      b_{\varepsilon,z,j}(\tfrac{x+y}{2},p) \, L(dz) \varphi(y)
      \overline{\psi}(x) \, dydpdx
    \end{aligned}
  \end{equation}
  If we only consider the integral over $z$ then we have by a Cauchy
  formula and Lemma~\ref{B.approx_of_resolvent_lemma_1} that
  \begin{equation*}
  \begin{aligned}
    \frac{1}{\pi} \int_\C \bar{\partial }\tilde{f}(z)
    b_{\varepsilon,z,j}(\tfrac{x+y}{2},p) \, L(dz)
    &= \frac{1}{\pi}
    \int_\C \bar{\partial }\tilde{f}(z) \sum_{k=1}^{2j-1}
    d_{\varepsilon, j ,k}(\tfrac{x+y}{2},p)
    b_{\varepsilon,z,0}^{k+1}(\tfrac{x+y}{2},p) \, L(dz)
    \\
    &= \sum_{k=1}^{2j-1} d_{\varepsilon, j ,k}(\tfrac{x+y}{2},p)
    \frac{1}{\pi} \int_\C \bar{\partial }\tilde{f}(z)
    \Big(\frac{1}{a_{\varepsilon,0}(\tfrac{x+y}{2},p)-z}\Big)^{k+1}\,
    L(dz)
    \\
    &= \sum_{k=1}^{2j-1} \frac{(-1)^k}{k!} d_{\varepsilon, j ,k}(\tfrac{x+y}{2},p)
    f^{(k)}( a_{\varepsilon,0}(\tfrac{x+y}{2},p)) =
    a_{\varepsilon,j}^f(\tfrac{x+y}{2},p),
    \end{aligned}
  \end{equation*}
  which is the desired form of $a_{\varepsilon,j}^f$ given in
  \eqref{B.func_cal_sym}. Now combing this with \eqref{B.func_cal_1},
  \eqref{B.func_cal_2} and \eqref{B.func_cal_5} we arrive at
  \begin{equation*}
    \langle \frac{1}{\pi} \int_\C  \bar{\partial }\tilde{f}(z) \OpW(b_{\varepsilon,z,j}) \, L(dz) \varphi , \psi \rangle = \langle \OpW(a_{\varepsilon,j}^f) \varphi , \psi \rangle.
  \end{equation*}
  The remaning $j$' can be treated analogously and hence we obtain the
  equality \eqref{B.sum_symbol_func_proof}. This ends the proof.
\end{proof}
\subsection{Applications of functional calculus}
With the functional calculus we can now prove some useful theorems and lemmas. One of them is an asymptotic expansion of
certain traces.
\begin{thm}\label{B.point.spectrum}
Let $A_\varepsilon(\hbar)$, for $\hbar$ in $(0,\hbar_0]$, be a
  $\hbar$-$\varepsilon$-admissible operator of regularity
  $\tau \geq 1$ and symbol
  \begin{equation*}
    a_\varepsilon(\hbar) = \sum_{j\geq 0} \hbar^j a_{\varepsilon,j}.
  \end{equation*}
  Suppose that $A_\varepsilon(\hbar)$ satisfies
  Assumption~\ref{B.self_adj_assumption}. Let $E_1<E_2$ be two real
  numbers and suppose there exists an $\eta>0$ such
  $a_{\varepsilon,0}^{-1}([E_1 - \eta, E_2+\eta])$ is compact. Then we have
  \begin{equation}
  	\spec(A_\varepsilon(\hbar))\cap[E_1.E_2] \subseteq \spec_{pp}(A_\varepsilon(\hbar)),
  \end{equation}
  for $\hbar$ sufficiently small, where $\spec_{pp}(A_\varepsilon(\hbar))$ is the pure point spectrum of $A_\varepsilon(\hbar)$.
\end{thm}
\begin{proof}
 Let $f$ and $g$ be in
  $C_0^\infty((E_1-\eta,E_2+\eta))$ such $g(t)=1$ for
  $t\in[E_1,E_2]$ and $f(t)=1$ for $t$ in $\supp(g)$. By Theorem~\ref{B.func_calc} we have
  \begin{equation}\label{B.punkt_spec_1}
    f(A_\varepsilon(\hbar)) = A_{\varepsilon,f,N}(\hbar) + \hbar^{N+1}
    R_{N+1,f}(\varepsilon,\hbar),
  \end{equation}
   where the terms $A_{\varepsilon,f,N}(\hbar)$ consists of the first $N$ terms in the
  expansion in $\hbar$ of $ f(A_\varepsilon(\hbar))$. We get by \eqref{B.punkt_spec_1} and the definition of $g$ and $f$ that
\begin{equation*}
	 g(A_\varepsilon(\hbar)) (I - \hbar^{N+1}
    R_{N+1,f}(\varepsilon,\hbar))=  g(A_\varepsilon(\hbar))A_{\varepsilon,f,N}(\hbar). 
\end{equation*}
Hence for $\hbar$ sufficiently small we have
\begin{equation*}
	 g(A_\varepsilon(\hbar)) =  g(A_\varepsilon(\hbar))A_{\varepsilon,f,N}(\hbar) (I - \hbar^{N+1}
    R_{N+1,f}(\varepsilon,\hbar))^{-1},
\end{equation*}
thereby we have the inequality
\begin{equation}\label{B.punkt_spec_2}
	\norm{ g(A_\varepsilon(\hbar))}_{\Tr} \leq c \norm{g}_{\infty} \norm{A_{\varepsilon,f,N}(\hbar)}_{\Tr} \leq C \hbar^{-d},
\end{equation}
where we have used Theorem~\ref{B.thm_est_tr}. Since $\boldsymbol{1}_{[E_1,E_2]}(t)\leq g(t)$ we have that $ \boldsymbol{1}_{[E_1,E_2]}(A_\varepsilon(\hbar))$ is a trace class operator by \eqref{B.punkt_spec_2}. This implies the inclusion
 \begin{equation*}
  	\spec(A_\varepsilon(\hbar))\cap[E_1.E_2] \subseteq \spec_{pp}(A_\varepsilon(\hbar)),
  \end{equation*}
  for $\hbar$ sufficiently small. This ends the proof.
\end{proof}
\begin{thm}\label{B.trace_formula_fkt}
  Let $A_\varepsilon(\hbar)$, for $\hbar$ in $(0,\hbar_0]$, be a
  $\hbar$-$\varepsilon$-admissible operator of regularity
  $\tau \geq 1$ and symbol
  \begin{equation*}
    a_\varepsilon(\hbar) = \sum_{j\geq 0} \hbar^j a_{\varepsilon,j}.
  \end{equation*}
  Suppose that $A_\varepsilon(\hbar)$ satisfies
  Assumption~\ref{B.self_adj_assumption}. Let $E_1<E_2$ be two real
  numbers and suppose there exists an $\eta>0$ such
  $a_{\varepsilon,0}^{-1}([E_1 - \eta, E_2+\eta])$ is compact. Then
  for every $f$ in $C_0^\infty((E_1,E_2))$ and any $N_0$ in
  $\mathbb{N}$ there exists an $N$ in $\mathbb{N}$ such that
  \begin{equation*}
    |\Tr[f(A_\varepsilon(\hbar))] -\frac{1}{(2\pi\hbar)^d} \sum_{j=0}^N \hbar^j T_j(f,A_\varepsilon(\hbar)) | \leq C \hbar^{N_0+1-d}.
  \end{equation*}
  for all sufficiently small $\hbar$, where the $T_j$'s is given by
  \begin{equation*}
    T_j(f,A_\varepsilon(\hbar) ) =  \int_{\R^{2d}} \sum_{k=1}^{2j-1} \frac{(-1)^k}{k!} d_{\varepsilon,j,k} f^{(k)}(a_{\varepsilon,0}) \, dxdp,
  \end{equation*}
  where $d_{\varepsilon,j,k}$ are the polynomials from
  Lemma~\ref{B.approx_of_resolvent_lemma_1}. In particular we have
  \begin{align*} 
    T_0(f,A_\varepsilon(\hbar) ) ={}&  \int_{\R^{2d}}
    f(a_{\varepsilon,0}) \, dxdp
    \\ \shortintertext{and} 
    T_1(f,A_\varepsilon(\hbar) ) = & \int_{\R^{2d}}
    a_{\varepsilon,1} f^{(1)}(a_{\varepsilon,0}) \, dxdp.
   \end{align*}
\end{thm}
The proof is an application of Theorem~\ref{B.func_calc} which gives
the form of the operator $f(A_\varepsilon(\hbar))$ combined with the
trace formula from Theorem~\ref{B.trace_formula} and we use some of the same ideas as in the proof of Theorem~\ref{B.point.spectrum}.
\begin{proof}
  Let $f$ in $C_0^\infty((E_1,E_2))$ be given and fix a $g$ in
  $C_0^\infty((E_1-\eta,E_2+\eta))$ such that $g(t)=1$ for
  $t\in[E_1,E_2]$. By Theorem~\ref{B.func_calc} we have
  \begin{align*}
    f(A_\varepsilon(\hbar)) = A_{\varepsilon,f,N}(\hbar) + \hbar^{N+1}
    R_{N+1,f}(\varepsilon,\hbar),
    \\\shortintertext{and}
    g(A_\varepsilon(\hbar)) = A_{\varepsilon,g,N}(\hbar) + \hbar^{N+1}
    R_{N+1,g}(\varepsilon,\hbar),
   \end{align*}
  where the terms $A_{\varepsilon,f,N}(\hbar)$ and
  $A_{\varepsilon,g,N}(\hbar)$ consist of the first $N$ terms in the
  expansion in $\hbar$ of $ f(A_\varepsilon(\hbar))$ and
  $ g(A_\varepsilon(\hbar))$ respectively.  Since
  $ f(A_\varepsilon(\hbar)) g(A_\varepsilon(\hbar)) =
  f(A_\varepsilon(\hbar))$ we have
  \begin{equation*}
  	\begin{aligned}
    f(A_\varepsilon(\hbar))
    &= (A_{\varepsilon,f,N}(\hbar) + \hbar^{N+1}
    R_{N+1,f}(\varepsilon,\hbar))( A_{\varepsilon,g,N}(\hbar) +
    \hbar^{N+1} R_{N+1,g}(\varepsilon,\hbar))
    \\
    &= A_{\varepsilon,f,N}(\hbar) A_{\varepsilon,g,N}(\hbar)
    +\hbar^{N+1}[ A_{\varepsilon,f,N}(\hbar)
    R_{N+1,g}(\varepsilon,\hbar) +
    R_{N+1,f}(\varepsilon,\hbar)A_{\varepsilon,g,N}(\hbar) ].
    \end{aligned}
  \end{equation*}
  By Theorem~\ref{B.thm_est_tr} we have that
  \begin{equation}\label{B.trace_fct_1}
    \norm{ f(A_\varepsilon(\hbar)) - A_{\varepsilon,f,N}(\hbar) A_{\varepsilon,g,N}(\hbar)  }_{\Tr} \leq C \hbar^{\kappa(N)-d}
  \end{equation}
  as $\hbar\rightarrow 0$. Hence taking $N$ sufficiently large we can
  consider the composition of the operators
  $A_{\varepsilon,f,N}(\hbar) A_{\varepsilon,g,N}(\hbar)$ instead of
  $f(A_\varepsilon(\hbar))$. By the choice of $g$ and
  Theorem~\ref{B.composition-weyl-thm} (composition of operators) we
  have
  \begin{equation}\label{B.trace_fct_2}
  \begin{aligned}
    A_{\varepsilon,f,N}(\hbar) A_{\varepsilon,g,N}(\hbar) &=
    \sum_{j=0}^N \hbar^j \OpW(a_{\varepsilon,j}^f) + \hbar^{N+1}
    R_{N,f,g}(\varepsilon, \hbar,a_\varepsilon)
    \\
    &=A_{\varepsilon,f,N}(\hbar) + \hbar^{N+1} R_{N,f,g}(\varepsilon,
    \hbar,a_\varepsilon).
    \end{aligned}
  \end{equation}
  Hence we have Theorem~\ref{B.thm_est_tr} that
  \begin{equation}\label{B.trace_fct_3}
    \norm{ A_{\varepsilon,f,N}(\hbar) - A_{\varepsilon,f,N}(\hbar) A_{\varepsilon,g,N}(\hbar)  }_{\Tr} \leq C \hbar^{\kappa(N)-d},
  \end{equation}
  where we have used that the error term in \eqref{B.trace_fct_2} is a $\hbar$-pseudo-differential operator, which follows from Theorem~\ref{B.composition-weyl-thm}.  Theorem~\ref{B.trace_formula} now gives
  \begin{equation}\label{B.trace_fct_4}
  	\begin{aligned}
    \Tr[A_{\varepsilon,f,N}(\hbar) ] &= \sum_{j=0}^N \hbar^j
    \Tr[\OpW(a_{\varepsilon,j}^f)]
    \\
    &=\frac{1}{(2\pi\hbar)^d} \sum_{j=0}^N \hbar^j  \int_{\R^{2d}} \sum_{k=1}^{2j-1} \frac{(-1)^k}{k!}
    d_{\varepsilon,j,k} f^{(k)}(a_{\varepsilon,0}) \, dxdp.
    \end{aligned}
  \end{equation}
  By combining \eqref{B.trace_fct_1} and \eqref{B.trace_fct_3} implies that
  \begin{equation}\label{B.trace_fct_5}
  	\abs{\Tr[f(A_\varepsilon(\hbar))]- \Tr[A_{\varepsilon,f,N}(\hbar) ] } \leq C \hbar^{\kappa(N)-2d-1}.
  \end{equation}
   Hence by choosing $N$ sufficiently large and combining \eqref{B.trace_fct_4} and \eqref{B.trace_fct_5} we get the desired estimate.
\end{proof}
The next Lemmas will be use-full in the proof of the Weyl
law. Both of these Lemmas are proven by applying the functional
calculus and the results on compositions of operators.
\begin{lemma}\label{B.Insert_localisation_in_trace}
  Let $A_\varepsilon(\hbar)$, for $\hbar$ in $(0,\hbar_0]$, be a
  $\hbar$-$\varepsilon$-admissible operator of regularity
  $\tau \geq 1$ and symbol
  \begin{equation*}
    a_\varepsilon(\hbar) = \sum_{j\geq 0} \hbar^j a_{\varepsilon,j}.
  \end{equation*}
  Suppose that $A_\varepsilon(h)$ satisfies
  Assumption~\ref{B.self_adj_assumption}. Let $E_1<E_2$ be two real
  numbers and suppose there exists an $\eta>0$ such
  $a_{\varepsilon,0}^{-1}([E_1 - \eta, E_2+\eta])$ is compact. Let $f$
  be in $C_0^\infty((E_1,E_2))$ and suppose $\theta$ is in
  $C_0^\infty(\R^d_x\times\R^d_p)$ such
  $\supp(\theta)\subset a_{\varepsilon,0}^{-1}([E_1 - \eta,
  E_2+\eta])$ and $\theta(x,p)=1$ for all $(x,p)$ in
  $\supp(f(a_{0,\varepsilon}))$. Then we have the bound
  \begin{equation*}
    \norm{(1-\OpW(\theta)) f(A_\varepsilon(\hbar))}_{\Tr} \leq C_N \hbar^N
  \end{equation*}
  for every $N$ in $\N$
\end{lemma}
\begin{proof}
  We choose $g$ in $C_0^\infty((E_1,E_2))$ such $g(t)f(t)=f(t)$. We
  now have
  \begin{equation*}
    \norm{(1-\OpW(\theta)) f(A_\varepsilon(\hbar))}_{\Tr} \leq \norm{(1-\OpW(\theta)) f(A_\varepsilon(\hbar))}_{\mathcal{L}(L^2(\R^d))} \norm{g(A_\varepsilon(\hbar))}_{\Tr}
  \end{equation*}
  By arguing as in the proof of Theorem~\ref{B.trace_formula_fkt} and applying 
  Theorem~\ref{B.thm_est_tr} we get
  \begin{equation*}
    \norm{g(A_\varepsilon(\hbar))}_{\Tr} \leq C_d \hbar^{-d}.
  \end{equation*}
  Form Theorem~\ref{B.func_calc} $f(A_\varepsilon(\hbar))$ is
  $\hbar$-$\varepsilon$-admissible operator with symbols
  \begin{equation*}
    a_{\varepsilon}^f(\hbar) = \sum_{j\geq 0} \hbar^j a_{\varepsilon,j}^f,
  \end{equation*}
  where
  \begin{equation*}
    a_{\varepsilon,j}^f = \sum_{k=1}^{2j-1} \frac{(-1)^k}{k!} d_{\varepsilon,j,k} f^{(k)}(a_{\varepsilon,0}),
  \end{equation*}
  the symbols $d_{\varepsilon,j,k}$ are the polynomials from
  Lemma~\ref{B.approx_of_resolvent_lemma_1}. The support of these
  functions is disjoint from the support of symbol for the operator
  $(1-\OpW(\theta))$. Hence by Theorem~\ref{B.composition-weyl-thm} we
  get the desired estimate.
\end{proof}
\begin{lemma}\label{B.Insert_fct_of_opt_in_trace}
  Let $A_\varepsilon(\hbar)$, for $\hbar$ in $(0,\hbar_0]$, be a
  $\hbar$-$\varepsilon$-admissible operator of regularity
  $\tau \geq 1$ and symbol
  \begin{equation*}
    a_\varepsilon(\hbar) = \sum_{j\geq 0} \hbar^j a_{\varepsilon,j}.
  \end{equation*}
  Suppose that $A_\varepsilon(h)$ satisfies
  Assumption~\ref{B.self_adj_assumption}. Let $E_1<E_2$ be two real
  numbers and suppose there exists an $\eta>0$ such
  $a_{\varepsilon,0}^{-1}([E_1 - \eta, E_2+\eta])$ is compact. Suppose
  $\theta$ is in $C_0^\infty(\R^d_x\times\R^d_p)$ such
  $\supp(\theta)\subset a_{\varepsilon,0}^{-1}((E_1 , E_2 ))$.
  
   Then for every $f$ in $C_0^\infty([E_1-\eta,E_2+\eta])$ such $f(t)=1$ for all
  $t$ in $[E_1-\frac{\eta}{2},E_2+\frac{\eta}{2}]$ the bound
  \begin{equation*}
    \norm{\OpW(\theta)(1- f(A_\varepsilon(\hbar)))}_{\mathcal{L}(L^2(\R^d))} \leq C_N \hbar^N,
  \end{equation*}
  is true for every $N$ in $\N$
\end{lemma}
\begin{proof}
  Theorem~\ref{B.func_calc} gives us that $ f(A_\varepsilon(\hbar))$
  is $\hbar$-$\varepsilon$-admissible operator with symbols
  \begin{equation*}
    a_{\varepsilon}^f(\hbar) = \sum_{j\geq 0} \hbar^j a_{\varepsilon,j}^f,
  \end{equation*}
  where
  \begin{equation*}\label{B.func_cal_sym_lemma}
    a_{\varepsilon,j}^f = \sum_{k=1}^{2j-1} \frac{(-1)^k}{k!} d_{\varepsilon,j,k} f^{(k)}(a_{\varepsilon,0}),
  \end{equation*}
  the symbols $d_{\varepsilon,j,k}$ are the polynomials from
  Lemma~\ref{B.approx_of_resolvent_lemma_1}. Hence we have that the
  principal symbol of $(1- f(A_\varepsilon(\hbar)))$ is
  $1-f(a_{\varepsilon,0})$. By assumption we then have that the
  support of $\theta$ and the support of every symbol in
  $f(A_\varepsilon(\hbar))$ are disjoint. Hence
  Theorem~\ref{B.composition-weyl-thm} implies the desired estimate.
\end{proof}
%
%
%
%%%%%%%%%%%%%%%%%%%%%%%%%%%%%%%%%%%%%%%%%%%%%%%%%%%%%%%%%%%%%%
\section{Microlocal approximation and properties of
  propagators}\label{B.construction_propagator}
In this section we will study the solution to the operator valued
Cauchy problem:
\begin{equation*}
  \begin{cases}
    \hbar \partial_t U(t,\hbar) - i U(t,\hbar) A_\varepsilon(\hbar) = 0 & t \neq0 \\
    U(0,\hbar) = \theta(x,\hbar D) & t=0,
  \end{cases}
\end{equation*}
where $A_\varepsilon$ is self-adjoint and the symbol $\theta$ is in
$C_0^\infty(\R^d_x\times\R^d_p)$. In particular we will only consider
the case where $A_\varepsilon(\hbar)$ is a
$\hbar$-$\varepsilon$-admissible operator of regularity $\tau\geq1$
which satisfies Assumption~\ref{B.self_adj_assumption}. Hence for
sufficiently small $\hbar$ the operator $A_\varepsilon(\hbar)$ is essentially
self-adjoint by Theorem~\ref{B.self_adjoint_thm_1}. It is well-known
that the solution to the operator valued Cauchy problem is the micro
localised propagator
$\theta(x,\hbar D)e^{i\hbar^{-1}tA_\varepsilon(\hbar)}$.

We are interested in the propagators as they turn up in a smoothing of
functions applied to the operators we consider. In this smoothing
procedure we need to understand the behaviour of the propagator for $t$ in a
small interval around zero. Usually this is done by constructing a specific
Fourier integral operator (FIO) as an approximation to propagator. In the ``usual'' construction the phase function is the solution to the Hamilton-Jacobi equations associated to the principal symbol. For
our set up this FIO approximation is not desirable as we can not
control the number of derivatives in the space variables and hence we
can not be certain about how the operator behave. Instead we will use
a different microlocal approximation for times in
$[-\hbar^{1-\frac{\delta}{2}},\hbar^{1-\frac{\delta}{2}}]$.

The construction of the approximation is recursive and inspired by the
construction in the works of Zielinski. If the construction is
compered to the approximation in the works of Ivrii, one can note that
Ivrii's construction is successive.

Our objective is to construct the approximation $U_N(t, \hbar)$ such
that
\begin{equation*}
  \norm{\hbar\partial_t U_N(t, \hbar) - i U_N(t, \hbar) A_\varepsilon}_{\mathcal{L}(L^2(\R^d))}, 
\end{equation*}
is small and the trace of the operator has the \enquote{right} asymptotic
behaviour. The kernel of the
approximation will have the following form
\begin{equation*}
  (x,y) \rightarrow \frac{1}{(2\pi\hbar)^d} \int_{\R^d} e^{i \hbar^{-1} \scp{x-y}{p}} e^{  i t \hbar^{-1} a_\varepsilon(x,p)} \sum_{j=0}^N (it\hbar^{-1})^j u_j(x,p,\hbar,\varepsilon) \, dp,
\end{equation*}
where $N$ is chosen such that the error is of a desired order, and
the $u_j$'s are compactly supported rough functions in $x$ and $p$. A priori these operators looks singular in the sense that for each term in the sum we have an increasing power of the factor $\hbar^{-1}$. What we will see in the following Theorem is that the  since each power of $\hbar^{-1}$ comes with a power of $t$, then for $t$ in $ [-\hbar^{1-\frac{\delta}{2}},\hbar^{1-\frac{\delta}{2}}]$ it has the desired properties. 
\begin{thm}\label{B.existence_of_approximation}
  Let $A_\varepsilon(\hbar)$ be a $\hbar$-$\varepsilon$-admissible
  operator of regularity $\tau\geq1$ with tempered weight $m$ which is
  self-adjoint for all $\hbar$ in $(0,\hbar_0]$, for $\hbar_0>0$ and with
  $\varepsilon\geq\hbar^{1-\delta}$ for a $\delta\in(0,1)$. Let
  $\theta(x,p)$ be a function in
  $C_0^\infty(\R^d_x\times\R^d_p)$. Then for all $N_0\in\N$ there
  exist an operator $U_N(t,\varepsilon,\hbar)$ with integral kernel
  \begin{equation*}
  \begin{aligned}
    K_{U_N}(x,y,t&,\varepsilon,\hbar)
    = \frac{1}{(2\pi\hbar)^d} \int_{\R^d} e^{i \hbar^{-1} \langle
      x-y,p\rangle} e^{ i t \hbar^{-1} a_{\varepsilon,0}(x,p)}
    \sum_{j=0}^N (it\hbar^{-1})^j u_j(x,p,\hbar,\varepsilon) \, dp,
  \end{aligned}
  \end{equation*}
  such that $K_{U_N}(x,y,0,\varepsilon,\hbar)$ is the kernel of the
  operator $\OpN{0}(\theta) = \theta(x,\hbar D)$. The terms in the sum
  satisfies $u_0(x,p,\hbar,\varepsilon)=\theta(x,p)$,
  \begin{equation*}
    u_j(x,p,\hbar,\varepsilon) \in C_0^\infty(\R_x^d\times\R_p^d)
  \end{equation*}
  and they satisfies the bounds
  \begin{equation*}
    \abs{\partial_x^\beta \partial_p^\alpha u_j(x,p,\hbar,\varepsilon)} 
    \leq \begin{cases}
      C_{\alpha\beta} & j=0
      \\
      C_{\alpha\beta} \hbar \varepsilon^{-\abs{\beta}} &j=1
      \\
      C_{\alpha\beta} \hbar^{1+\delta(j-2)} \varepsilon^{-\abs{\beta}} &j\geq2
    \end{cases}
  \end{equation*}
  for all $\alpha$ and $\beta$ in $\N^d$ in the case $\tau=1$. For
  $\tau \geq2$ the $u_j$'s satisfies the bounds
  \begin{equation*}
    \abs{\partial_x^\beta \partial_p^\alpha u_j(x,p,\hbar,\varepsilon)} 
    \leq \begin{cases}
      C_{\alpha\beta} & j=0
      \\
      C_{\alpha\beta} \hbar \varepsilon^{-\abs{\beta}} &j=1,2
      \\
      C_{\alpha\beta} \hbar^{2+\delta(j-3)} \varepsilon^{-\abs{\beta}} &j\geq3
    \end{cases}
  \end{equation*}
  for all $\alpha$ and $\beta$ in $\N^d$. Moreover $U_N$ satisfies the
  following bound:
  \begin{equation*}
    \norm{\hbar\partial_t U_N(t, \hbar) - i U_N(t, \hbar) A_\varepsilon}_{\mathcal{L}(L^2(\R^d))}  \leq C \hbar^{N_0}
  \end{equation*}
  for $\abs{t}\leq \hbar^{1-\frac{\delta}{2}}$.
\end{thm}
\begin{remark}
  If the operator satisfies Assumption~\ref{B.self_adj_assumption},
  then by Theorem~\ref{B.self_adjoint_thm_1} the operator will be essentially
  self-adjoint for all sufficiently small $\hbar$. Hence
  Assumption~\ref{B.self_adj_assumption} would be sufficient but not
  necessary for the above theorem to be true.
  The number $N$ is explicit dependent on $N_0$, $d$ and $\delta$ which follows directly from the proof.
\end{remark}
\begin{proof}
  We start by fixing $N$ such that
  \begin{equation*}
   \min\big( 1+\delta\big(\tfrac{N}{2} -1\big) - d, 2+\delta\big(\tfrac{N}{2} -3\big) - d \big)  \geq N_0.
  \end{equation*}
  By assumption we have for sufficiently large $M$ in $\mathbb{N}$ the
  following form of $A_\varepsilon(\hbar)$
  \begin{equation}\label{B.form_of_A}
    A_\varepsilon(\hbar) = \sum_{j=0}^M \hbar^j \OpW(a_{\varepsilon,j}) + \hbar^{M+1} R_M(\varepsilon,\hbar).
  \end{equation}
  We can choose and fix $M$ such the following estimate is true
  \begin{equation*}
    \hbar^{M+1}\norm{ R_M(\varepsilon,\hbar)}_{\mathcal{L}(L^2(\R^d))} \leq  C_M \hbar^{N_0}.
  \end{equation*}
  With this $M$ we consider the sum in the expression of
  $A_\varepsilon(\hbar)$. By
  Corollary~\ref{B.connection_t_quantisations} there exists a sequence
  $\{\tilde{a}_{\varepsilon,j}\}_{j\in\N}$ of symbols where
  $\tilde{a}_{\varepsilon,j}$ is of regularity $\tau-j$ and a
  $\tilde{M}$ such
  \begin{equation}\label{B.ext_pro_sym_cha}
    \sum_{j=0}^M \hbar^j \OpW(a_{\varepsilon,j}) = \sum_{j=0}^{\tilde{M}} \hbar^j \OpN{1}(\tilde{a}_{\varepsilon,j}) + \hbar^{\tilde{M}+1} \tilde{R}_{\tilde{M}}(\varepsilon,\hbar),
  \end{equation}
  where $a_{\varepsilon,0}=\tilde{a}_{\varepsilon,0}$ and
  \begin{equation*}
    \hbar^{\tilde{M}+1}\norm{ \tilde{R}_{\tilde{M}}(\varepsilon,\hbar)}_{\mathcal{L}(L^2(\R^d))} \leq  C_{\tilde{M}} \hbar^{N_0}.
  \end{equation*}
  We will for the reminder of the proof use the notation
  \begin{equation}\label{B.def_cha_sym}
    \tilde{a}_{\varepsilon}(x,p) =  \sum_{j=0}^{\tilde{M}} \hbar^j \tilde{a}_{\varepsilon,j}(x,p).
  \end{equation}
  The function $\tilde{a}_{\varepsilon}(x,p)$ is a rough function of
  regularity $\tau$.  These choices and definitions will become
  important again at the end of the proof.

  For our fixed $N$ we define the operator
  $\hbar \partial_t - i \mathcal{P}_N : C^\infty(\R_t \times \R^d_x
  \times \R_p^d) \rightarrow C^\infty(\R_t \times \R^d_x \times
  \R_p^d)$, where
  \begin{equation*}
    \mathcal{P}_N b(t,x,p) = \sum_{\abs{\alpha} \leq N} \frac{(-i\hbar)^{\abs{\alpha}}}{\alpha !} \partial_p^\alpha \{b(t,x,p) \partial_x^\alpha \tilde{a}_\varepsilon(x,p)\}
  \end{equation*}
  for a $b\in C^\infty(\R_t\times\R^d_x\times\R^d_p)$. First step is
  to observe how the operator $\hbar \partial_t - i \mathcal{P}_N$
  acts on $ e^{ i t \hbar^{-1} a_{\varepsilon,0}(x,p)}\psi(x,p)$ for
  $\psi\in C_0^\infty(\R^d_x\times\R^d_p)$. We will in the following
  calculation omit the dependence of the variables $x$ and $p$. By
  Leibniz's formula and the chain rule we get
  \begin{equation}\label{B.act_PN}
    \begin{aligned}
      \MoveEqLeft (\hbar \partial_t - i \mathcal{P}_N) e^{it\hbar^{-1}
        a_{\varepsilon,0}} \psi = \hbar \partial_t e^{it\hbar^{-1}
        a_{\varepsilon,0}} \psi - i \sum_{\abs{\alpha} \leq N}
      \frac{(-i\hbar)^{\abs{\alpha}}}{\alpha !} \partial_p^\alpha \{
      e^{it\hbar^{-1} a_{\varepsilon,0}} \psi\partial_x^\alpha
      \tilde{a}_\varepsilon\}
      \\
      ={}&e^{it\hbar^{-1} a_{\varepsilon,0}}\Big[ i \psi
      (a_{\varepsilon,0}-\tilde{a}_\varepsilon) - i
      \sum_{\abs{\alpha}=1}^{N} \frac{(-i\hbar)^{\abs{\alpha}}}{\alpha
        !}  \partial_p^\alpha\{\psi\partial_x^\alpha
      \tilde{a}_\varepsilon\} 
      \\
      &- i \sum_{k=1}^{N } (it\hbar^{-1})^k\sum_{\abs{\alpha} =k }^{N}
      \frac{(-i\hbar)^{\abs{\alpha}}}{\alpha !}
      \sum_{\substack{\beta_1 + \cdots+ \beta_k \leq \alpha \\
          \abs{\beta_j} >0}} c_{\alpha,\beta_1\cdots\beta_k}
      \prod_{j=1}^k \partial_p^{\beta_j}a_{\varepsilon,0} \partial_p^{\alpha-(\beta_1+\cdots+\beta_k
        )} \{\psi\partial_x^\alpha \tilde{a}_\varepsilon\}\Big]
      \\
      ={}& i e^{it\hbar^{-1} a_{\varepsilon,0}} \sum_{k=0}^N
      (it\hbar^{-1})^k q_k(\psi, x,p,\hbar,\varepsilon).
    \end{aligned}
  \end{equation}
  From this we note after acting with
  $\hbar \partial_t - i \mathcal{P}_N$ on
  $ e^{ i t \hbar^{-1} a_{\varepsilon,0}(x,p)}\psi(x,p)$ we get
  $ i e^{it\hbar^{-1} a_{\varepsilon,0}}$ times a polynomial in
  $it\hbar^{-1}$ with coefficients in $C_0^\infty(\R^d_x\times\R^d_p)$
  depending on $\psi$, $\hbar$ and $\varepsilon$. If we consider the
  coefficients in the polynomial we have
  \begin{equation}\label{B.q_est_1}
    \begin{aligned}
      |q_0(\psi, x,p,\hbar,\varepsilon)| 
     &= |\psi (
      a_{\varepsilon,0}-\tilde{a}_\varepsilon) +
      \sum_{\abs{\alpha}=1}^{N} \frac{(-i\hbar)^{\abs{\alpha}}}{\alpha
        !}  \partial_p^\alpha\{\psi\partial_x^\alpha
      \tilde{a}_\varepsilon\}|
      \\
     & \leq c_1\hbar + \sum_{\abs{\alpha}=1}^{N}
      \frac{\hbar^{\abs{\alpha}}}{\alpha
        !}\abs{ \partial_p^\alpha\{\psi\partial_x^\alpha
        \tilde{a}_\varepsilon\}} \leq c_1\hbar +
      \sum_{\abs{\alpha}=1}^{\tau} \frac{\hbar^{\abs{\alpha}}}{\alpha
        !}c_\alpha +\sum_{\abs{\alpha}=\tau+1}^{N}
      \frac{\hbar^{\abs{\alpha}}}{\alpha !}c_\alpha
      \varepsilon^{\tau-\abs{\alpha}}
      \\
     & \leq c_1\hbar + \sum_{\abs{\alpha}=1}^{\tau}
      \frac{\hbar^{\abs{\alpha}}}{\alpha !}c_\alpha
      +\sum_{\abs{\alpha}=\tau+1}^{N}
      \frac{\hbar^{\abs{\alpha}}}{\alpha !}c_\alpha
      \hbar^{(1-\delta)(\tau-\abs{\alpha})} \leq C
      \hbar, %(1+ \sum_{j=1}^\tau \hbar^j + \sum_{j=\tau+1}^N \hbar^{\delta j + \tau(1-\delta)}),
    \end{aligned}
  \end{equation}
  where $C$ depends on the $p$-derivatives of $\psi$ and
  $\partial_x^\alpha a_\varepsilon$ on the support of $\psi$ for
  $\abs{\alpha}\leq N$. For $1\leq k \leq \tau$ we have
  \begin{equation}\label{B.q_est_2}
    \begin{aligned}
      |q_k(\psi, x,p,\hbar,\varepsilon)|
      ={}& |\sum_{\abs{\alpha} =k }^{N}
      \frac{(-i\hbar)^{\abs{\alpha}}}{\alpha !}
      \sum_{\substack{\beta_1 + \cdots+ \beta_k \leq \alpha \\
          \abs{\beta_j} >0}} c_{\alpha,\beta_1\cdots\beta_k}
      \prod_{j=1}^k \partial_p^{\beta_j}a_{\varepsilon,0} \partial_p^{\alpha-(\beta_1+\cdots+\beta_k
        )} \{\psi\partial_x^\alpha \tilde{a}_\varepsilon\}|
      \\
      \leq{}& \sum_{\abs{\alpha} =k }^{\tau} c_\alpha
      \frac{\hbar^{\abs{\alpha}}}{\alpha !}  + \sum_{\abs{\alpha}
        =\tau+1 }^{N} c_\alpha \frac{\hbar^{\abs{\alpha}}}{\alpha !}
      \varepsilon^{\tau-\abs{\alpha}}
      \\
      \leq{}& \sum_{\abs{\alpha} =k }^{\tau} c_\alpha
      \frac{\hbar^{\abs{\alpha}}}{\alpha !}  + \sum_{\abs{\alpha}
        =\tau+1 }^{N} c_\alpha \frac{\hbar^{\abs{\alpha}}}{\alpha !}
      \hbar^{(1-\delta)(\tau-\abs{\alpha})} \leq C \hbar^k,
    \end{aligned}
  \end{equation}
  where $C$ depends on the $p$-derivatives of $\psi$ and
  $\partial_x^\alpha a_\varepsilon$ on the support of $\psi$. For
  $\tau< k \leq N$ we have
  \begin{equation}\label{B.q_est_3}
    \begin{aligned}
      |q_k(\psi, x,p,\hbar,\varepsilon)|
      ={}& |\sum_{\abs{\alpha} =k }^{N}
      \frac{(-i\hbar)^{\abs{\alpha}}}{\alpha !}
      \sum_{\substack{\beta_1 + \cdots+ \beta_k \leq \alpha \\
          \abs{\beta_j} >0}} c_{\alpha,\beta_1\cdots\beta_k}
      \prod_{j=1}^k \partial_p^{\beta_j}a_{\varepsilon,0} \partial_p^{\alpha-(\beta_1+\cdots+\beta_k
        )} \{\psi\partial_x^\alpha \tilde{a}_\varepsilon\}|
      \\
      \leq{}& \sum_{\abs{\alpha} =k }^{N} c_\alpha
      \frac{\hbar^{\abs{\alpha}}}{\alpha !}
      \varepsilon^{\tau-\abs{\alpha}} \leq \ \sum_{\abs{\alpha} =k
      }^{N} c_\alpha \frac{\hbar^{\abs{\alpha}}}{\alpha !}
      \hbar^{(1-\delta)(\tau-\abs{\alpha})} \leq C \hbar^{\tau +
        (k-\tau)\delta},
    \end{aligned}
  \end{equation}
  where $C$ depends on the $p$-derivatives of $\psi$ and
  $\partial_x^\alpha a_\varepsilon$ on the support of $\psi$. It is
  important that the coefficients only depends on derivatives in
  $p$ for the function we apply the operator to. One should also note
  that if $\psi$ had $\hbar$ to some power multiplied to it. Then it should
  be multiplied to the new power obtained. In the reminder of the
  proof we will continue to denote the coefficients obtained by acting
  with $\hbar \partial_t - i \mathcal{P}_N$ by $q_j$ and the exact
  form can be found in \eqref{B.act_PN}.

  We are now ready to start constructing the kernel. We set
  $u_0(x,p,\hbar,\varepsilon)=\theta(x,p)$ which gives the first
  term. In order to find $u_1$ we act with the operator
  $\hbar \partial_t - i \mathcal{P}_N$ on the function
  $ e^{ i t \hbar^{-1} a_{\varepsilon,0}} u_0(x,p,\hbar,\varepsilon)$,
  where we in the reminder of the construction of the approximation
  will omit writing the dependence of the variables $(x,p)$ in the
  exponential. By \eqref{B.act_PN} we get
  \begin{equation*}
    (\hbar \partial_t - i \mathcal{P}_N)e^{  i t \hbar^{-1} a_{\varepsilon,0}} u_0(x,p,\hbar,\varepsilon) =  i  e^{it\hbar^{-1} a_{\varepsilon,0}} \sum_{k=0}^N (it\hbar^{-1})^k q_k(u_0, x,p,\hbar,\varepsilon).
  \end{equation*}
  This would not lead to the desired estimate. So we now take
  \begin{equation*}
   u_1(x,p,\hbar,\varepsilon) = - q_0(u_0,
  x,p,\hbar,\varepsilon).
  \end{equation*}
  We can note by the previous estimates
  \eqref{B.q_est_1} we have
  \begin{equation}\label{B.u:est_1}
    \abs{u_1(x,p,\hbar,\varepsilon)} = \abs{q_0(u_0, x,p,\hbar,\varepsilon)} \leq \hbar C.
  \end{equation}
  Acting with the operator $\hbar \partial_t - i \mathcal{P}_N$ on
  $e^{ i t \hbar^{-1} a_{\varepsilon,0}}(u_0(x,p,\hbar,\varepsilon) +
  i t \hbar^{-1} u_1(x,p,\hbar,\varepsilon))$ we obtain according to
  \eqref{B.act_PN} that
  \begin{equation*}
  \begin{aligned}
   \MoveEqLeft (\hbar \partial_t - i \mathcal{P}_N) (e^{ i t \hbar^{-1}
      a_{\varepsilon,0}}(u_0(x,p,\hbar,\varepsilon) + i t \hbar^{-1}
    u_1(x,p,\hbar,\varepsilon)))
    \\
    ={}& i e^{it\hbar^{-1} a_{\varepsilon,0}} \sum_{k=0}^N
    (it\hbar^{-1})^k q_k(u_0, x,p,\hbar,\varepsilon) + i e^{ i t
      \hbar^{-1} a_{\varepsilon,0}} u_1(x,p,\hbar,\varepsilon)
    \\
    \hbox{}& + i t \hbar^{-1} i e^{it\hbar^{-1} a_{\varepsilon,0}}
    \sum_{k=0}^N (it\hbar^{-1})^k q_k(u_1, x,p,\hbar,\varepsilon)
    \\
    ={}& i e^{it\hbar^{-1} a_{\varepsilon,0}} (\sum_{k=1}^N
    (it\hbar^{-1})^k q_k(u_0, x,p,\hbar,\varepsilon) + \sum_{k=0}^N
    (it\hbar^{-1})^{k+1} q_k(u_1, x,p,\hbar,\varepsilon)).
    \end{aligned}
  \end{equation*}
Now taking 
  $u_2(x,p,\hbar,\varepsilon)= - \frac12
  (q_1(u_0,x,p,\hbar,\varepsilon) + q_0(u_1,x,p,\hbar,\varepsilon))$
  and acting with the operator $\hbar \partial_t - i \mathcal{P}_N$,
  according to \eqref{B.act_PN}, we get that
  \begin{equation*} \begin{aligned}
    (\hbar \partial_t &- i \mathcal{P}_N) e^{ i t \hbar^{-1}
      a_{\varepsilon,0}}\sum_{j=0}^2 (it\hbar^{-1})^j
    u_j(x,p,\hbar,\varepsilon)
    \\
    &= i e^{it\hbar^{-1} a_{\varepsilon,0}} \Big[ \sum_{j=0}^2
    \sum_{k=0}^N (it\hbar^{-1})^{k+j} q_k(u_j, x,p,\hbar,\varepsilon)
    + \sum_{j=1}^2 j (it\hbar^{-1})^{j-1} u_j(x,p,\hbar,\varepsilon)
    \Big]
    \\
    &= i e^{it\hbar^{-1} a_{\varepsilon,0}} \sum_{j=0}^2
    \sum_{k=2-j}^N (it\hbar^{-1})^{k+j} q_k(u_j,
    x,p,\hbar,\varepsilon).
   \end{aligned} \end{equation*}
  We note that the \enquote{lowest} power of $it\hbar^{-1}$ is $2$. Hence it is these
   terms which should be used to construct $u_3$. Moreover, we
  note that by \eqref{B.q_est_1} and \eqref{B.q_est_2} we have
  \begin{equation}\label{B.u:est_2}
    \abs{u_2(x,p,\hbar,\varepsilon)}=\tfrac12 \abs{q_1(u_0,x,p,\hbar,\varepsilon) + q_0(u_1,x,p,\hbar,\varepsilon)} \leq \tfrac12 C (\hbar+\hbar^2) \leq  C\hbar,
  \end{equation}
  and $u_2$ is a smooth compactly supported function in the variables
  $x$ and $p$. Generally for $2\leq j\leq N$ we have
  \begin{equation*}
    u_j(x,p,\hbar,\varepsilon) = - \frac{1}{j} \sum_{k=0}^{j-1} q_{j-1-k}(u_k,x,p,\hbar,\varepsilon).
  \end{equation*}
  We now need estimates for these terms. In the case $\tau=1$ the next
  step will be empty, but for $\tau\geq2$ it is needed. For
  $\tau\geq2$ we have
  \begin{equation*}
    |u_3(x,p,\hbar,\varepsilon)| \leq \frac{1}{3} \sum_{k=0}^{2} |q_{2-k}(u_k,x,p,\hbar,\varepsilon)| \leq C\hbar^2,
  \end{equation*}
  where we have used \eqref{B.q_est_1}, \eqref{B.q_est_2},
  \eqref{B.u:est_1} and \eqref{B.u:est_2}. For the rest of the $u_j$'s
  we split in the two cases $\tau=1$ or $\tau\geq2$. First the cases
  $\tau=1$ for $2\leq j\leq N$ the estimate is
  \begin{equation*}
    \abs{u_j(x,p,\hbar,\varepsilon)} \leq  C \hbar^{1+\delta(j-2)}
  \end{equation*}
  Note that $u_2$ satisfies the above equation hence if we assume it
  okay for $j-1$ between $2$ and $N-1$ we want to show the above
  estimate for $j$. We note that for $j\geq5$
  \begin{equation*}
  \begin{aligned}
    \MoveEqLeft |{u_{j}(x,p,\hbar,\varepsilon)} |
    \\
    \leq{}& \frac{1}{j}
    \sum_{k=0}^{j-1} \abs{q_{j-1-k}(u_k,x,p,\hbar,\varepsilon)}
    \\
    \leq{}& C( \abs{q_{j-1}(u_0,x,p,\hbar,\varepsilon)} +
    \abs{q_{j-2}(u_1,x,p,\hbar,\varepsilon)}
    \\
    &+\sum_{k=2}^{j-3}
    \abs{q_{j-1-k}(u_k,x,p,\hbar,\varepsilon)} +  \abs{q_{1}(u_{j-2},x,p,\hbar,\varepsilon)}+
    \abs{q_{0}(u_{j-1},x,p,\hbar,\varepsilon)} )
    \\
    \leq{}& C(\hbar^{1+\delta(j-2)} + \hbar^{2+\delta(j-3)} + \sum_{k=2}^{j-3}
    \hbar^{1+\delta(j-1-k-1) + 1+\delta(k-2)} +  \hbar^{2+\delta(j-4)}+ \hbar^{2+\delta(j-3)})
    \\
    \leq{}& C(\hbar^{1+\delta(j-2)} + \hbar^{2+\delta(j-3)} +
    \hbar^{2+\delta(j-4)} + \hbar^{2+\delta(j-3)}) \leq C
    \hbar^{1+\delta(j-2)},
    \end{aligned}
  \end{equation*}
  where we
  have used \eqref{B.q_est_1}, \eqref{B.q_est_3} and the induction
  assumption. The cases $j=3$ and $j=4$ are estimated analogously. 

  Now the case $\tau\geq2$ which we will treat as $\tau=2$, here the
  estimate is
  \begin{equation*}
    \abs{u_j(x,p,\hbar,\varepsilon)} \leq  C \hbar^{2+\delta(j-3)}
  \end{equation*}
  for $3\leq j \leq N$. To prove this bound is the same as in the case
  of $\tau=1$. In order to prove the bound with the derivatives as
  stated in the theorem the above arguments are repeated with a number
  of derivatives on the $u_j$'s otherwise it is analogous. 

  What remains is to prove this construction satisfies the bound
  \begin{equation*}
    \norm{\hbar\partial_t U_N(t, \hbar) - i U_N(t, \hbar) A_\varepsilon}_{\mathcal{L}(L^2(\R^d))}  \leq C \hbar^{N_0}
  \end{equation*}
  Here we only consider the case $\tau=1$ as the cases $\tau\geq2$
  will have better estimates. Hence from the above estimates we have
  for $k$ in $\{0,\dots,N\}$ and $\abs{t}\leq \hbar^{1-\frac\delta2}$
  \begin{equation}\label{B.est:on:uer}
    \abs{(it\hbar^{-1})^k u_k(x,p,\hbar,\varepsilon) }  \leq 
    \begin{cases}
      C &k=0
      \\
      C \hbar^{1-\frac\delta2} & k=1
      \\
      C \hbar^{1 + \delta(\frac{k}{2} -2)} & k\geq2.
    \end{cases}
  \end{equation}
  The first step is to apply the operator
  $\hbar \partial_t - i \mathcal{P}_N$ on then \enquote{full} kernel and see
  what error this produces. By construction we have
  \begin{equation*}
    (\hbar \partial_t - i \mathcal{P}_N)e^{ i t \hbar^{-1}
      a_{\varepsilon,0}}\sum_{k=0}^N (it\hbar^{-1})^k
    u_k(x,p,\hbar,\varepsilon)
    = \sum_{j=0}^N \sum_{k=N-j}^N (it\hbar^{-1})^{k+j} q_k(u_j,
    x,p,\hbar,\varepsilon).
  \end{equation*}
  If we start by considering $j$ equal $0$ and $1$ we note that:
  \begin{equation}\label{B.est_kernel_1_01}
  \begin{aligned}
    \MoveEqLeft \Big|\sum_{j=0}^1 \sum_{k=N-j}^N (it\hbar^{-1})^{k+j} q_k(u_j,
    x,p,\hbar,\varepsilon)  \Big| 
    \\
    &\leq C(\hbar^{-\frac{\delta}{2}N} \hbar^{1+\delta(N-1)} +
    \hbar^{-\frac{\delta}{2}N} \hbar^{2+\delta(N-2)} +
    \hbar^{-\frac{\delta}{2}(N+1)} \hbar^{2+\delta(N-1)})
    \\
    &\leq C( \hbar^{1+ \delta(\frac{N}{2} -1)} + \hbar^{2+
      \delta(\frac{N}{2} -2)} + \hbar^{2+ \frac{\delta}{2}(N -3)})
    \\
    &\leq \tilde{C} \hbar^{1+ \delta(\frac{N}{2} -1)} \leq  \tilde{C} \hbar^{N_0+d},
  \end{aligned}
  \end{equation}
  where we have used \eqref{B.est:on:uer} and our choice of $N$. For the rest of the terms we have that
  \begin{equation}\label{B.est_kernel_1_02}
  \begin{aligned}
    \MoveEqLeft \Big| \sum_{j=2}^N \sum_{k=N-j}^N (it\hbar^{-1})^{k+j} q_k(u_j,
    x,p,\hbar,\varepsilon) \Big| 
    \\
    &\leq C \Big[ \sum_{j=2}^N \sum_{k=\max(N-j,1)}^N
    \hbar^{-\frac{\delta}{2}(k+j)} \hbar^{2+\delta(k-1) + \delta(j-2)}
    + \hbar^{-\frac{\delta}{2}N} \hbar^{2+ \delta(N-2)}\Big]
    \\
    &\leq C\Big[ \sum_{j=2}^N \sum_{k=\max(N-j,1)}^N \hbar^{2 +
      \delta(\frac{j+k}{2} -3)} + \hbar^{2+ \delta(\frac{N}{2}-2)}\Big]
     \leq \tilde{C} \hbar^{N_0+d},
    \end{aligned}
  \end{equation}
  where we have used \eqref{B.est:on:uer}, our choice of $N$ and that in the double sum $k+j\geq N$. When \eqref{B.est_kernel_1_01} and \eqref{B.est_kernel_1_02} are combined we have that
  \begin{equation}\label{B.est_kernel_1}
    |(\hbar \partial_t - i \mathcal{P}_N)e^{  i t \hbar^{-1} a_{\varepsilon,0}}\sum_{k=0}^N (it\hbar^{-1})^k u_k(x,p,\hbar,\varepsilon) |
    \leq C  \hbar^{N_0+d}.
  \end{equation}	
  We now let $U_N(t,\hbar)$ be the operator with the integral kernel:
  \begin{equation*}
    K_{U_N}(x,y,t,\varepsilon,\hbar)
    = \frac{1}{(2\pi\hbar)^d} \int_{\R^d} e^{i \hbar^{-1}
      \scp{x-y}{p}} e^{ i t \hbar^{-1} a_{\varepsilon,0}(x,p)}
    \sum_{j=0}^N (it\hbar^{-1})^j u_j(x,p,\hbar,\varepsilon) \, dp,
  \end{equation*}
  which is well defined due to our previous estimates \eqref{B.est:on:uer}. In particular we have that it is a bounded operator
  by the Schur test. We now need to find an expression for
  \begin{equation*}
    \hbar\partial_t U_N(t, \hbar) - i U_N(t, \hbar) A_\varepsilon(\hbar).
  \end{equation*}
  In the start of the proof we wrote the operator
  $A_\varepsilon(\hbar)$ in two different ways \eqref{B.form_of_A} and
  \eqref{B.ext_pro_sym_cha}. If we combine these we have
  \begin{equation*}
  \begin{aligned}
    A_\varepsilon(\hbar) &= \sum_{j=0}^{\tilde{M}} \hbar^j
    \OpN{1}(\tilde{a}_{\varepsilon,j}) + \hbar^{\tilde{M}+1}
    \tilde{R}_{\tilde{M}}(\varepsilon,\hbar) + \hbar^{M+1}
    R_M(\varepsilon,\hbar)
    \\
    &= \OpN{1}(\tilde{a}_{\varepsilon}) + \hbar^{\tilde{M}+1}
    \tilde{R}_{\tilde{M}}(\varepsilon,\hbar) + \hbar^{M+1}
    R_M(\varepsilon,\hbar),
    \end{aligned}
  \end{equation*}
  where the two reminder terms satisfies
  \begin{equation*}
    \norm{ \hbar^{\tilde{M}+1}\tilde{R}_{\tilde{M}}(\varepsilon,\hbar) + \hbar^{M+1} R_M(\varepsilon,\hbar)}_{\mathcal{L}(L^2(\R^d))} \leq  C \hbar^{N_0}.
  \end{equation*}
  If we use this form of $A_\varepsilon(\hbar)$ we have
  \begin{equation*}
  \begin{aligned}
   \MoveEqLeft  \hbar\partial_t U_N(t, \hbar) - i U_N(t, \hbar)
    A_\varepsilon(\hbar) 
    \\
    ={}& \hbar\partial_t U_N(t, \hbar) - i U_N(t,\hbar) \OpN{1}(\tilde{a}_{\varepsilon}) - i\hbar^{\tilde{M}+1} U_N(t, \hbar)
    \tilde{R}_{\tilde{M}}(\varepsilon,\hbar) - i \hbar^{M+1} U_N(t,
    \hbar)R_M(\varepsilon,\hbar).
    \end{aligned}
  \end{equation*}
  When considering the operator norm of the two last terms we have
  \begin{equation}\label{B.est_approx_error}
    \norm{ \hbar^{\tilde{M}+1} U_N(t, \hbar)\tilde{R}_{\tilde{M}}(\varepsilon,\hbar) + \hbar^{M+1} U_N(t, \hbar) R_M(\varepsilon,\hbar)}_{\mathcal{L}(L^2(\R^d))} \leq  C \hbar^{N_0}
  \end{equation}
  as $U_N(t, \hbar)$ is a bounded operator. What remains is the
  expression
  \begin{equation*}
    \hbar\partial_t U_N(t, \hbar) - i U_N(t, \hbar) \OpN{1}(\tilde{a}_{\varepsilon}).
  \end{equation*}
  The rules for composition of kernels gives by a straight forward calculation that the kernel of the
  above expression is
  \begin{equation*}
     K(x,y;\varepsilon,\hbar)  \coloneqq \frac{1}{(2\pi\hbar)^d}
    \int_{\R^d} e^{i \hbar^{-1} \scp{x-y}{p}} (\hbar\partial_t -
    i\tilde{a}_\varepsilon(y,p)) e^{ i t \hbar^{-1}
      a_{\varepsilon,0}(x,p)} \sum_{j=0}^N (it\hbar^{-1})^j u_j(x,p,\hbar,\varepsilon) \, dp.
  \end{equation*}
  By preforming 
  a Taylor expansion of $\tilde{a}_\varepsilon$ in the variable $y$
  centred at $x$ we obtain that 
  \begin{equation*}
  \begin{aligned}
    \tilde{a}_\varepsilon(y,p) 
   ={}& \sum_{\abs{\alpha}\leq N}
    \frac{(y-x)^\alpha}{\alpha !} \partial_x^\alpha
    \tilde{a}_\varepsilon(x,p)
    \\
    &+ \sum_{\abs{\alpha}=N+1} (N+1) \frac{(y-x)^\alpha}{\alpha !}
    \int_0^1 (1-s)^N \partial_x^\alpha \tilde{a}_\varepsilon(x +
    s(y-x),p) \, ds.
    \end{aligned}
  \end{equation*}
  We replace $\tilde{a}_\varepsilon(y,p)$ by the above Taylor
  expansion in the kernel and start by considering the part of the
  kernel with the first sum. Here we have
  \begin{equation*}
  \begin{aligned}
    \MoveEqLeft \int_{\R^d} e^{i \hbar^{-1} \langle
      x-y,p \rangle} \big[ \hbar\partial_t  -i\sum_{\abs{\alpha}\leq N}
    \frac{(y-x)^\alpha}{\alpha !} \partial_x^\alpha
    \tilde{a}_\varepsilon(x,p)  \big] e^{ i t
      \hbar^{-1}a_{\varepsilon,0}(x,p)}\sum_{j=0}^N (it\hbar^{-1})^j
    u_j(x,p,\hbar,\varepsilon) \big] \, dp
    \\
    &=  \int_{\R^d} e^{i \hbar^{-1} \langle
      x-y,p \rangle}[ \hbar\partial_t e^{ i t \hbar^{-1}
      a_{\varepsilon,0}(x,p)}\sum_{j=0}^N (it\hbar^{-1})^j
    u_j(x,p,\hbar,\varepsilon)
    \\
    &\phantom{= \frac12 }{} -i\sum_{\abs{\alpha}\leq N}
    \frac{(-i\hbar)^{\abs{\alpha}}}{\alpha
      !} \partial_p^\alpha[ \partial_x^\alpha
    \tilde{a}_\varepsilon(x,p) e^{ i t
      \hbar^{-1}a_{\varepsilon,0}(x,p)}\sum_{j=0}^N (it\hbar^{-1})^j
    u_j(x,p,\hbar,\varepsilon)]] \, dp
    \\
    &=  \int_{\R^d} e^{i \hbar^{-1} \langle
      x-y,p \rangle} ( \hbar\partial_t- i \mathcal{P}_N) [ e^{ i t
      \hbar^{-1} a_\varepsilon(x,p)} \sum_{j=0}^N (it\hbar^{-1})^j
    u_j(x,p,\hbar,\varepsilon)] \, dp,
    \end{aligned}
  \end{equation*}
  where we have used the identity $(y-x)^\alpha e^{i \hbar^{-1} \langle x-y,p \rangle}= (-i\hbar)^\alpha \partial_p^\alpha e^{i \hbar^{-1} \langle x-y,p \rangle}$, integration by parts and omitted the pre-factor $(2\pi\hbar)^{-d}$. When considering the part of the kernel
  with the error term we have
  \begin{equation*}
  \begin{aligned}
    \MoveEqLeft \frac{-i}{(2\pi\hbar)^d} \int_{\R^d} e^{i \hbar^{-1} \langle
      x-y,p \rangle} e^{ i t \hbar^{-1} a_\varepsilon(x,p)}
    \sum_{j=0}^N (it\hbar^{-1})^j u_j(x,p,\hbar,\varepsilon)
    \\
    &\times \sum_{\abs{\alpha}=N+1} (N+1)
    \frac{(y-x)^\alpha}{\alpha !} \int_0^1 (1-s)^N \partial_x^\alpha
    \tilde{a}_\varepsilon(x + s(y-x),p) \, ds dp
    \\
   ={}&\frac{-i}{(2\pi\hbar)^d} \int_{\R^d} e^{i \hbar^{-1} \langle
      x-y,p \rangle} \sum_{\abs{\alpha}=N+1} (N+1)
    \frac{(-i\hbar)^\alpha}{\alpha !} \partial_p^\alpha [ e^{ - i t
      \hbar^{-1} a_\varepsilon(x,p)} \sum_{j=0}^N (-it\hbar^{-1})^j
    \\
    &\times u_j(x,p,\hbar,\varepsilon) \int_0^1
    (1-s)^N \partial_x^\alpha \tilde{a}_\varepsilon(x + s(y-x),p) \,
    ds] dp,
    \end{aligned}
  \end{equation*}
  where we again have used the above identity and
  integration by parts. Combing the two expressions we get
  \begin{equation}\label{B.Kernel-form}
    \begin{aligned}
      K(x,y;\varepsilon,\hbar)
      ={}& \frac{1}{(2\pi\hbar)^d} \int_{\R^d} e^{i \hbar^{-1} \langle
        x-y,p \rangle} ( \hbar\partial_t- i \mathcal{P}_N) [ e^{ i t
        \hbar^{-1} a_\varepsilon(x,p)} \sum_{j=0}^N (it\hbar^{-1})^j
      u_j(x,p,\hbar,\varepsilon)] \, dp
      \\
      & + \frac{-i}{(2\pi\hbar)^d} \int_{\R^d} e^{i
        \hbar^{-1} \langle x-y,p \rangle} \sum_{\abs{\alpha}=N+1}
      (N+1) \frac{(-i\hbar)^\alpha}{\alpha !} \partial_p^\alpha [ e^{
        i t \hbar^{-1} a_\varepsilon(x,p)} \sum_{j=0}^N
      (it\hbar^{-1})^j
      \\
      &\times u_j(x,p,\hbar,\varepsilon) \int_0^1
      (1-s)^N \partial_x^\alpha \tilde{a}_\varepsilon(x + s(y-x),p) \,
      ds] dp.
    \end{aligned}
  \end{equation}
  In order to estimate the operator norm we will divide the kernel
  into two parts. We do this by considering a part localised in $y$
  and the reminder. To localise in $y$ we let $\psi$ be a smooth
  function on $\R^d$ such $\psi(y) = 1$ on the set
  $\{y\in\R^d \,|\, \dist[y,\supp_x(\theta)]\leq 1\}$ and supported in
  the set $\{y\in\R^d \,|\, \dist[y,\supp_x(\theta)]\leq 2\}$. With
  this function our kernel can be written as
  \begin{equation}\label{B.Kernel_split}
    K(x,y;\varepsilon,\hbar)  =  K(x,y;\varepsilon,\hbar) \psi(y) + K(x,y;\varepsilon,\hbar) (1-\psi(y))
  \end{equation}
  If we consider the part multiplied by $\psi(y)$ then this part has
  the form as in \eqref{B.Kernel-form} but each term is multiplied by
  $\psi(y)$. By the estimate in \eqref{B.est_kernel_1} we have for the
  first part of $K(x,y;\varepsilon,\hbar) \psi(y) $ the following
  estimate
  \begin{equation}\label{B.est_kernel_2}
    \begin{aligned}
      \MoveEqLeft \Big| \int_{\R^d} e^{i \hbar^{-1} \langle x-y,p \rangle} (
      \hbar\partial_t- i \mathcal{P}_N) [ e^{ i t \hbar^{-1}
        a_{\varepsilon,0}(x,p)} \sum_{j=0}^N (it\hbar^{-1})^j
      u_j(x,p,\hbar,\varepsilon)]\psi(y) \, dp \Big|
      \\
      &\leq  \psi(x)\psi(y) C \hbar^{N_0+d}.
    \end{aligned}
  \end{equation}
  For the second part of $K(x,y;\varepsilon,\hbar) \psi(y) $ we have
  by Leibniz's formula and Fa\`a di Bruno formula
  (Theorem~\ref{B.Faa_di_bruno_x}) for each term in the sum
  over $\alpha$
  \begin{equation*}
    \begin{aligned}
      \MoveEqLeft \frac{(-i\hbar)^{\abs{\alpha}}}{\alpha !} \partial_p^\alpha
      [e^{ i t \hbar^{-1} a_{\varepsilon,0}(x,p)}  \sum_{j=0}^N
      (it\hbar^{-1})^j u_j(x,p,\hbar,\varepsilon) \int_0^1
      (1-s)^N \partial_x^\alpha \tilde{a}_\varepsilon(x + s(y-x),p) \,
      ds]\psi(y)
      \\
      ={}&\sum_{j=0}^N \frac{(-i\hbar)^{\abs{\alpha}}}{\alpha !} e^{ i t
        \hbar^{-1} a_{\varepsilon,0}(x,p)} \sum_{k=0}^{N+1}
      (it\hbar^{-1})^{k+j} \quad\sum_{\mathclap{\substack{\beta_1 +
            \cdots+ \beta_k \leq \alpha \\ \abs{\beta_j} >0}}} \quad
      c_{\alpha,\beta_1\cdots\beta_k}
      \prod_{n=1}^k \partial_p^{\beta_n}a_{\varepsilon,0}(x,p)
      \\
      &
      \times \partial_p^{\alpha-(\beta_1+\cdots+\beta_k )}
      [u_j(x,p,\hbar,\varepsilon) \int_0^1 (1-s)^N \partial_x^\alpha
      \tilde{a}_\varepsilon(x + s(y-x),p) \, ds]\psi(y).
    \end{aligned}
  \end{equation*}
  We note that for $j$ equal $0$ we have an estimate of the following
  form:
  \begin{equation}\label{B.kerne.est.1}
  	\begin{aligned}
    \MoveEqLeft \Big|\frac{(-i\hbar)^{\abs{\alpha}}}{\alpha !} e^{ i t \hbar^{-1}
      a_{\varepsilon,0}(x,p)} \sum_{k=0}^{N+1} (it\hbar^{-1})^k
    \sum_{\substack{\beta_1 + \cdots+ \beta_k \leq \alpha \\
        \abs{\beta_j} >0}} c_{\alpha,\beta_1\cdots\beta_k}
    \prod_{n=1}^k \partial_p^{\beta_n}a_{\varepsilon,0}(x,p)
    \\
    & \times \partial_p^{\alpha-(\beta_1+\cdots+\beta_k
      )} [u_0(x,p,\hbar,\varepsilon) \int_0^1
    (1-s)^N \partial_x^\alpha \tilde{a}_\varepsilon(x + s(y-x),p) \,
    ds]\psi(y)\Big|
    \\
    \leq{}& C \sum_{k=0}^{N+1} \hbar^{N+1} \hbar^{-\frac{\delta}{2}k}
    \varepsilon^{-N} \psi(x)\psi(y) \leq C
    \hbar^{N+1}\hbar^{-\frac{\delta}{2}(N+1)} \hbar^{-N +\delta
      N}\psi(x)\psi(y)
    \\
    \leq{}& C \hbar^{N_0+d}\psi(x)\psi(y).
    	\end{aligned}
  \end{equation}
  We note that for $j$ equal $1$ we have an error of the following
  form:
  \begin{equation}\label{B.kerne.est.2}
  \begin{aligned}
    \MoveEqLeft \Big|\frac{(-i\hbar)^{\abs{\alpha}}}{\alpha !} e^{ i t \hbar^{-1}
      a_{\varepsilon,0}(x,p)} \sum_{k=0}^{N+1} (it\hbar^{-1})^{k+1}
    \sum_{\substack{\beta_1 + \cdots+ \beta_k \leq \alpha \\
        \abs{\beta_j} >0}} c_{\alpha,\beta_1\cdots\beta_k}
    \prod_{n=1}^k \partial_p^{\beta_n}a_{\varepsilon,0}(x,p)
    \\
    & \times \partial_p^{\alpha-(\beta_1+\cdots+\beta_k )}
    [u_1(x,p,\hbar,\varepsilon) \int_0^1 (1-s)^N \partial_x^\alpha
    \tilde{a}_\varepsilon(x + s(y-x),p) \, ds]\psi(y)\Big|
    \\
    &\leq C \sum_{k=0}^{N+1} \hbar^{N+1}
    \hbar^{-\frac{\delta}{2}(k+1)} \hbar
    \varepsilon^{-N}\psi(x)\psi(y) \leq C
    \hbar^{N+2}\hbar^{-\frac{\delta}{2}(N+2)} \hbar^{-N +\delta
      N}\psi(x)\psi(y)
    \\
    &\leq C \hbar^{N_0+d}\psi(x)\psi(y),
    \end{aligned}
  \end{equation}
  where we have used the estimate
  $|u_1(x,p,\hbar,\varepsilon)|\leq c\hbar$.  We note that for $j$
  greater than or equal to $2$, we have an error of the following form:
  \begin{equation}\label{B.kerne.est.3}
  	\begin{aligned}
    \MoveEqLeft \Big|\frac{(-i\hbar)^{\abs{\alpha}}}{\alpha !} e^{ i t \hbar^{-1}
      a_{\varepsilon,0}(x,p)} \sum_{k=0}^{N+1} (it\hbar^{-1})^{k+j}
    \sum_{\substack{\beta_1 + \cdots+ \beta_k \leq \alpha \\
        \abs{\beta_j} >0}} c_{\alpha,\beta_1\cdots\beta_k}
    \prod_{n=1}^k \partial_p^{\beta_n}a_{\varepsilon,0}(x,p)
    \\
    &\phantom{=}{} \times \partial_p^{\alpha-(\beta_1+\cdots+\beta_k
      )} [u_j(x,p,\hbar,\varepsilon) \int_0^1
    (1-s)^N \partial_x^\alpha \tilde{a}_\varepsilon(x + s(y-x),p) \,
    ds]\psi(y) \Big|
    \\
    &\leq C \sum_{k=0}^{N+1} \hbar^{N+1}
    \hbar^{-\frac{\delta}{2}(k+j)} \hbar^{1+\delta(j-2)}
    \varepsilon^{-N} \psi(x)\psi(y)
    \\
    &\leq C \hbar^{N+1}
    \hbar^{-\frac{\delta}{2}(N+1+j)}\hbar^{1+\delta(j-2)}
    \hbar^{-N+\delta N}\psi(x)\psi(y) \leq C \hbar^{2 +
      \frac{\delta}{2}(N + j -1) -2\delta}\psi(x)\psi(y)
    \\
    &\leq C \hbar^{N_0+d }\psi(x)\psi(y),
    \end{aligned}
  \end{equation}
  where we have used the estimate
  $|u_j(x,p,\hbar,\varepsilon)|\leq c \hbar^{1+\delta(j-2)}$. Now by
  combining \eqref{B.kerne.est.1}, \eqref{B.kerne.est.2} and \eqref{B.kerne.est.3} we arrive at
  \begin{equation*}
  	\begin{aligned}
    \MoveEqLeft \Bigl| \int_{\R^d} e^{i \hbar^{-1} \langle x-y,p
      \rangle} \sum_{\abs{\alpha}=N+1} (N+1)
    \frac{(-i\hbar)^\alpha}{\alpha !} \partial_p^\alpha [ e^{ i t
      \hbar^{-1} a_\varepsilon(x,p)} \sum_{j=0}^N (it\hbar^{-1})^j
    \\
   &\times u_j(x,p,\hbar,\varepsilon) \int_0^1
    (1-s)^N \partial_x^\alpha \tilde{a}_\varepsilon(x + s(y-x),p) \,
    ds]\psi(y) \,dp\Bigr|
    \leq C \hbar^{N_0+d}\psi(x)\psi(y).
    \end{aligned}
  \end{equation*}
  Combining this estimate with \eqref{B.est_kernel_2} we have
  \begin{equation}\label{B.est_kernel_3}
    | K(x,y;\varepsilon,\hbar) \psi(y)| \leq   C \hbar^{N_0}\psi(x)\psi(y),
  \end{equation}
  Now we turn
  to the term $K(x,y;\varepsilon,\hbar)(1- \psi(y))$. On the support
  of this kernel we have
  \begin{equation*}
    1\leq \abs{x-y}
  \end{equation*}
  due to the definition of $\psi$. This imply we can divide by the
  difference between $x$ and $y$ where the kernel is supported. The
  idea is now to multiply the kernel with
  $\frac{\abs{x-y}}{\abs{x-y}}$ to an appropriate power $\eta$. We
  take $\eta$ such
  \begin{equation*}
    \frac{m(x+s(y-x),p)}{\abs{x-y}^{2\eta}} \leq \frac{C}{\abs{x-y}^{d+1}} \quad\text{for } (x,p) \in \supp(\theta),
  \end{equation*}
  where $m$ is the tempered weight function associated to our operator. The
  existence of such a $\eta$ is ensured by the definition of the tempered
  weight. By \eqref{B.Kernel-form} the kernel
  $K(x,y;\varepsilon,\hbar)(1- \psi(y))$ is of the form
  \begin{equation*} \begin{aligned}
    \frac{1}{(2\pi\hbar)^d} \int_{\R^d} e^{i \hbar^{-1} \langle x-y,p
      \rangle} \varphi(x,y,p;\hbar,\varepsilon) (1- \psi(y))dp,
   \end{aligned} \end{equation*}
  where the exact form of $\varphi$ is not
  important at the moment. Now for our choice of $\eta$ we have
  \begin{equation*}
  \begin{aligned}
   \MoveEqLeft  \int_{\R^d} e^{i \hbar^{-1} \langle x-y,p \rangle}
    \varphi(x,y,p;\hbar,\varepsilon) (1- \psi(y))dp
    \\
    &= \int_{\R^d} e^{i \hbar^{-1} \langle x-y,p \rangle}
    \frac{\abs{x-y}^{2\eta}}{\abs{x-y}^{2\eta}}
    \varphi(x,y,p;\hbar,\varepsilon) (1- \psi(y))dp
    \\
    &= \int_{\R^d} (-i\hbar)^{2\eta}
    \sum_{\abs{\gamma}=\eta} c_\gamma \partial_p^{2\gamma} (e^{i \hbar^{-1}
      \langle x-y,p \rangle}) \frac{1}{\abs{x-y}^{2\eta}}
    \varphi(x,y,p;\hbar,\varepsilon) (1- \psi(y))dp
    \\
    &= \int_{\R^d} e^{i \hbar^{-1} \langle x-y,p \rangle} \frac{1-
      \psi(y)}{\abs{x-y}^{2\eta}}\sum_{\abs{\gamma}=\eta} c_\gamma
    (i\hbar)^{2\eta} \partial_p^{2\gamma}\varphi(x,y,p;\hbar,\varepsilon)
    dp.
    \end{aligned}
  \end{equation*}
  By analogous estimates to the estimate used above we have
  \begin{equation*}
    \Big|\frac{1-
        \psi(y)}{\abs{x-y}^{2\eta}}\sum_{\abs{\gamma}=\eta}
      (i\hbar)^{2\eta} \partial_p^{2\gamma}\varphi(x,y,p;\hbar,\varepsilon)\Big|
    \leq C \hbar^{2\eta(1-\frac{\delta}{2})+1 +\frac{ \delta}{2}(N
      -1)} \frac{1- \psi(y)}{\abs{x-y}^{d+1}}
    \boldsymbol{1}_{\supp(\theta)}(x,p),
  \end{equation*}
  where the term $\hbar^{-\eta\delta}$ is due to the exponentials $e^{ i t
      \hbar^{-1} a_\varepsilon(x,p)}$ in $\varphi$ which gives $it\hbar^{-1}$ when we take a derivative with respect to $p_j$ for all $j$ in $\{1,\dots,d\}$ and that $|t|\leq \hbar^{1-\frac{\delta}{2}}$. The rest of the powers in $\hbar$ is found analogous to above. Hence we have
  \begin{equation*}
  \begin{aligned}
    |K(x,y;\varepsilon,\hbar)(1- \psi(y))| \leq {}& C
    \hbar^{2\eta(1-\frac{\delta}{2})+1 +\frac{ \delta}{2}(N -1) -d}
    \boldsymbol{1}_{\supp_x(\theta)}(x) \frac{1-
      \psi(y)}{\abs{x-y}^{d+1}}
      \\
    \leq{} & C \hbar^{N_0} \psi(x) \frac{1- \psi(y)}{\abs{x-y}^{d+1}}.
  \end{aligned}
  \end{equation*}
  By combining this with \eqref{B.Kernel_split} and
  \eqref{B.est_kernel_3} we have
  \begin{equation}\label{B.est_kernel_4}
    |K(x,y;\varepsilon,\hbar)| \leq C \hbar^{N_0}\Big[\psi(x)\psi(y) +  \boldsymbol{1}_{\supp_x(\theta)}(x) \frac{1- \psi(y)}{\abs{x-y}^{d+1}}\Big].
  \end{equation}
  We have by definition of $\psi$ the estimates
  \begin{equation*} \begin{aligned}
    &\sup_{x\in\R^d} \int_{\R^d} |\psi(x)\psi(y)+
    \boldsymbol{1}_{\supp_x(\theta)}(x) \frac{1-
      \psi(y)}{\abs{x-y}^{d+1}}| \, dy \leq c + \int_{\abs{y} \geq1}
    \frac{1}{\abs{y}^{d+1}} \, dy \leq C_1
    \\
    &\sup_{y\in\R^d} \int_{\R^d} |\psi(x)\psi(y)+
    \boldsymbol{1}_{\supp_x(\theta)}(x) \frac{1-
      \psi(y)}{\abs{x-y}^{d+1}}| \, dx \leq C_2
   \end{aligned} \end{equation*}
  These estimates combined with the Schur test,
  \eqref{B.est_approx_error} and \eqref{B.est_kernel_4} gives
  \begin{equation*}
    \norm{\hbar\partial_t U_N(t, \hbar) - i U_N(t, \hbar) A_\varepsilon}_{\mathcal{L}(L^2(\R^d))}  \leq C \hbar^{N_0}
  \end{equation*}
  for $\abs{t}\leq \hbar^{1-\frac{\delta}{2}}$. This is the desired
  estimate which ends the proof.
\end{proof}
In the previous proof we constructed a microlocal approximation for
the propagator for short times dependent on $\hbar$. It would be
preferable to not have this dependence of $\hbar$ in the time. In the
following Lemma we prove that under a non-critical condition on the
principal symbol a localised trace of the approximation becomes
negligible.
\begin{lemma}\label{B.neg_part}
  Let $A_\varepsilon(\hbar)$ be a $\hbar$-$\varepsilon$-admissible
  operator of regularity $\tau\geq1$ with tempered weight $m$ which is
  self-adjoint for all $\hbar$ in $(0,\hbar_0]$ and with
  $\varepsilon\geq\hbar^{1-\delta}$ for a $\delta\in(0,1)$. Let
  $\theta(x,p)$ be a function in
  $C_0^\infty(\R^d_x\times\R^d_p)$. Suppose
  \begin{equation*}
    \abs{\nabla_p a_{\varepsilon,0}(x,p)}\geq c>0 \quad\text{for all } (x,p)\in\supp(\theta),
  \end{equation*}
  where $a_{\varepsilon,0}$ is the principal symbol of
  $A_\varepsilon(\hbar)$. Moreover let the operator $U_N(t,\hbar)$ be
  the one constructed in Theorem~\ref{B.existence_of_approximation}
  with the function $\theta$.  Then for
  $\abs{t}\in[\frac{1}{2}\hbar^{1-\frac{\delta}{2}},1]$ and every
  $N_0$ in $\N$ it holds 
  \begin{equation*}
    \abs{\Tr[U_N(t,\hbar) \OpN{1}(\theta)]} \leq C \hbar^{N_0}
  \end{equation*}
  for a constant $C>0$, which depends on the constant from the
  non-critical condition.
\end{lemma}
\begin{proof}
  Recall that the kernel of $U_N(t,\hbar)$ is given by
  \begin{equation*}
    K_{U_N}(x,y,t,\varepsilon,\hbar)
    = \frac{1}{(2\pi\hbar)^d} \int_{\R^d} e^{i \hbar^{-1} \langle
      x-y,p\rangle} e^{ i t \hbar^{-1} a_\varepsilon(x,p)}
   u_N(x,p,t,\hbar,\varepsilon) \, dp,
  \end{equation*}
  where 
    \begin{equation*}
     u_N(x,p,t,\hbar,\varepsilon) =
    \sum_{j=0}^N (it\hbar^{-1})^j u_j(x,p,\hbar,\varepsilon).
  \end{equation*}
  From Theorem~\ref{B.existence_of_approximation} we have the
  estimate
  \begin{equation}\label{B.neg_appro_kernel_est_1}
    \sup_{x,p} \sup_{\abs{t}\leq1}| \partial_x^\alpha \partial_p^\beta u_N(x,p,t,\hbar,\varepsilon) | \leq C_{\alpha\beta} \hbar^{(\delta-1)N} \varepsilon^{-\abs{\alpha}}.
  \end{equation}
  This initial estimate is a priori not desirable as it implies the trace is of order $ \hbar^{(\delta-1)N-d} $. Due to the form of the kernels for the operators $U_N(t,\hbar)$ and $\OpN{1}(\theta)$ it immediate follows that the kernel for the composition is given by   
    \begin{equation*}
    (x,y) \mapsto
     \frac{1}{(2\pi\hbar)^d} \int_{\R^d} e^{i \hbar^{-1} \langle
      x-y,p\rangle} e^{ i t \hbar^{-1} a_\varepsilon(x,p)}
   u_N(x,p,t,\hbar,\varepsilon) \theta(y,p) \, dp.
  \end{equation*}
  With this expression for the kernel we get that the trace is given by
    \begin{equation*}
 	\Tr[U_N(t,\hbar) \OpN{1}(\theta)] =  \frac{1}{(2\pi\hbar)^d} \int_{\R^{2d}} e^{ i t \hbar^{-1} a_\varepsilon(x,p)}
   u_N(x,p,t,\hbar,\varepsilon) \theta(x,p) \, dxdp
  \end{equation*}
    Since we suppose $\abs{\nabla_p a_{\varepsilon,0}}\geq c>0$ on
  the support of $\theta(x,p)$ we have
  \begin{equation}\label{B.neg_appro_kernel_est_2}
  \begin{aligned}
   \MoveEqLeft \Tr[U_N(t,\hbar) \OpN{1}(\theta)] 
   \\
    ={}& \frac{1}{(2\pi\hbar)^d} \int_{\R^{2d}}  \frac{\sum_{j=1}^d
      (\partial_{p_j}a_{\varepsilon,0}(x,p))^2}{\abs{\nabla_p
        a_{\varepsilon,0}(x,p)}^2} e^{ i t \hbar^{-1} a_\varepsilon(x,p)}
   u_N(x,p,t,\hbar,\varepsilon) \theta(x,p) \, dxdp
    \\
    ={}&\frac{-i\hbar t^{-1}}{(2\pi\hbar)^d} \sum_{j=1}^d \int_{\R^{2d}} \partial_{p_j}e^{ i t \hbar^{-1} a_\varepsilon(x,p)}
    \frac{\partial_{p_j}a_{\varepsilon,0}(x,p)}{\abs{\nabla_pa_{\varepsilon,0}(x,p)}^2}
   u_N(x,p,t,\hbar,\varepsilon) \theta(x,p) \, dxdp
    \\
    ={}& \frac{i\hbar t^{-1}}{(2\pi\hbar)^d} \sum_{j=1}^d \int_{\R^{2d}}
    e^{ i t \hbar^{-1} a_{\varepsilon,0}(x,p)} \partial_{p_j} \Big[   \frac{\partial_{p_j}a_{\varepsilon,0}(x,p)}{\abs{\nabla_pa_{\varepsilon,0}(x,p)}^2}
   u_N(x,p,t,\hbar,\varepsilon) \theta(x,p)\Big ]\, dxdp.
    \end{aligned}
  \end{equation}
  Combining \eqref{B.neg_appro_kernel_est_1}, \eqref{B.neg_appro_kernel_est_2} and our assumptions on $t$ we see that we have gained $\hbar^{\frac{\delta}{2}}$ compared to our naive first estimate. To obtain the desired estimate we iterate the argument in \eqref{B.neg_appro_kernel_est_2} until an error of the desired order has been obtained. This concludes the proof.
\end{proof}
The previous Lemma showed that under a non-critical assumption on the
principal symbol a localised trace of our approximation becomes
negligible. But we would also need a result similar to this for the
true propagator. Actually this can be proven in a setting for which we
will need it, which is the content of the next Thoerem. An
observation of this type was first made by Ivrii (see
\cite{MR1807155}). Here we will follow the proof of such a statement
as made by Dimassi and Sj\"{o}strand in \cite{MR1735654}. The
statement is:
\begin{thm}\label{B.neg_propagator}
  Let $A_\varepsilon(\hbar)$ be a $\hbar$-$\varepsilon$-admissible
  operator of regularity $\tau\geq1$ which satisfies
  Assumption~\ref{B.self_adj_assumption}, has a bounded principal
  symbol and suppose there exists a $\delta$ in $(0,1)$ such that
  $\varepsilon\geq\hbar^{1-\delta}$. Furthermore, suppose there exists a number
  $\eta>0$ such $a_{\varepsilon,0}^{-1}([-2\eta,2\eta])$ is compact
  and a constant $c>0$ such
  \begin{equation*}
    \abs{\nabla_p a_{\varepsilon,0}(x,p)} \geq c \quad \text{for all } (x,p) \in a_{\varepsilon,0}^{-1}([-2\eta,2\eta]),
  \end{equation*}
  where $a_{\varepsilon,0}$ is the principal symbol of
  $A_\varepsilon(\hbar)$. Let $f$ be in $C_0^\infty((-\eta,\eta))$ and
  $\theta$ be in $C_0^\infty(\R^d_x\times\R^d_p)$ such that
  $\supp(\theta)\subset a_{\varepsilon,0}^{-1}((-\eta,\eta))$.
  
  Then  there
  exists a constant $T_0>0$ such that if $\chi$ is in
  $C_0^\infty((\frac12 \hbar^{1-\gamma},T_0))$ for a $\gamma$ in
  $(0,\delta]$, then for every $N$ in $\mathbb{N}$, we have
  \begin{equation*}
    \abs{\Tr[\OpW(\theta)f(A_\varepsilon(\hbar)) \mathcal{F}_\hbar^{-1}[\chi](s-A_\varepsilon(\hbar))\OpW(\theta)]} \leq C_N \hbar^N
  \end{equation*}
  uniformly for $s$ in $(-\eta,\eta)$.
\end{thm}
\begin{remark} 
Theorems of this type for non-regular operators can be found in the works of Ivrii see \cite{ivrii2019microlocal1} and Zielinski see \cite{MR2343462,MR2952218}. In both cases the proof of such theorems is different from the one we present here. The techniques used by both is based on propagation of singularities. The propagation of singularities is not directly present in the proof presented here but hidden in the techniques used. 
 
In both \cite{MR1735654} and  \cite{ivrii2019microlocal1} they assume the symbol to microhyperbolic in some direction. It might also be possible to extend the Theorem here to a general microhyperbolic assumption instead of the non-critical assumption. The challenge in this will be that for the proof to work under a general microhyperbolic assumption the symbol should be change such that microhyperbolic assumption similar to the non-critical assumption is achieved. This change might be problematic to do since it could mix the rough and non rough variables. 

The localising operators $\OpW(\theta)$ could be omitted if the first step of the proof is change to introducing them by applying Lemma~\ref{B.Insert_localisation_in_trace}. We have chosen to state the theorem with them since when we will apply the theorem we have the localisations. 
\end{remark}
\begin{proof}
  We start by remarking that it suffices to show the estimate with a
  function $\chi_\xi(t) = \chi(\frac{t}{\xi})$, where $\chi$ is in
  $C_0^\infty((\frac12,1))$ uniformly for $\xi$ in
  $[\hbar^{1-\gamma},T_0]$. Indeed assume such an estimate has been
  prove. We can split the interval $(\frac12 \hbar^{1-\gamma},T_0)$ in
  $\frac{2T_0}{\hbar^{1-\gamma}}$ intervals of size
  $\frac12\hbar^{1-\gamma}$ and make a partition of unity which
  members is supported in each of these intervals. By linearity
  of the inverse Fourier transform and trace we would have
  \begin{equation*} \begin{aligned}
    |\Tr[\OpW&(\theta)f(A_\varepsilon(\hbar))
    \mathcal{F}_\hbar^{-1}[\chi](s-A_\varepsilon(\hbar))\OpW(\theta)]|
    \\
    \leq{}& \sum_{j=1}^{M(\hbar)}
    \abs{\Tr[\OpW(\theta)f(A_\varepsilon(\hbar))
      \mathcal{F}_\hbar^{-1}[\chi_{\xi_j}](s-A_\varepsilon(\hbar))\OpW(\theta)]}
    \leq \tilde{C}_N \hbar^{N-1+\delta}.
   \end{aligned} \end{equation*}
  Hence we will consider the trace
  \begin{equation*}
    \Tr[\OpW(\theta)f(A_\varepsilon(\hbar)) \mathcal{F}_\hbar^{-1}[\chi_\xi](s-A_\varepsilon(\hbar))\OpW(\theta)],
  \end{equation*}
  with $\chi_\xi(t) = \chi(\frac{t}{\xi})$, where $\chi$ is in
  $C_0^\infty((\frac12,1))$ and $\xi$ in $[\hbar^{1-\gamma},T_0]$. For
  the rest of the proof we let a $N$ in $\mathbb{N}$ be given as the
  error we desire.

  Without loss of generality we can assume $\theta=\sum_k \theta_k$,
  where the $\theta_k$'s satisfies that if
  $\supp(\theta_k)\cap\supp(\theta_l)\neq\emptyset$ then there exists
  $j$ in $\{1,\dots,d\}$ such
  $|\partial_{p_j} a_{\varepsilon,0}(x,p)|>\tilde{c}$ on the set
  $\supp(\theta_k)\cup\supp(\theta_l)$. With this splitting of
  $\theta$ we have
  \begin{equation*} \begin{aligned}
    \Tr[\OpW(\theta)&f(A_\varepsilon(\hbar))
    \mathcal{F}_\hbar^{-1}[\chi_\xi](s-A_\varepsilon(\hbar))\OpW(\theta)]
    \\
    &= \sum_{k} \sum_{l} \Tr[\OpW(\theta_k)f(A_\varepsilon(\hbar))
    \mathcal{F}_\hbar^{-1}[\chi_\xi](s-A_\varepsilon(\hbar))\OpW(\theta_l)].
   \end{aligned} \end{equation*}
  By the cyclicity of the trace and the formulas for composition of
  pseudo-differential operators we observe if
  $\supp(\theta_k)\cap\supp(\theta_l)=\emptyset$ then the term is
  negligible. Hence what remains is the terms with
  $\supp(\theta_k)\cap\supp(\theta_l)\neq\emptyset$. All terms of the
  form are estimated with analogous techniques but some different
  indexes. Hence we will suppose we have a pair $\theta_k$ and $\theta_l$  of sets such that
  $\supp(\theta_k)\cap\supp(\theta_l)\neq\emptyset$ and
  $|\partial_{p_1} a_{\varepsilon,0}(x,p)|>\tilde{c}$ on the set
  $\supp(\theta_k)\cup\supp(\theta_l)$. This imply we either have
  $\partial_{p_1} a_{\varepsilon,0}(x,p)>\tilde{c}$ or
  $-\partial_{p_1} a_{\varepsilon,0}(x,p)>\tilde{c}$. We suppose we
  are in the first case. The other case is treated in the same manner
  but with a change of some signs.

  To sum up we have reduced to the case where we consider
  \begin{equation*}
    \Tr[\OpW(\theta_k)f(A_\varepsilon(\hbar)) \mathcal{F}_\hbar^{-1}[\chi_\xi](s-A_\varepsilon(\hbar))\OpW(\theta_l)],
  \end{equation*}
  with $\partial_{p_1} a_{\varepsilon,0}(x,p)>\tilde{c}$ on the the
  set $\supp(\theta_k)\cup\supp(\theta_l)$. The next step is to change
  the principal symbol of our operator such it becomes global
  microhyperbolic in the direction $(\boldsymbol{0};(1,0,\dots,0))$, where $\boldsymbol{0}$ is the $d$-dimensional vector with only zeros.

  We let $\varphi_2$ be a function in $C^\infty_0(\R_x^d\times\R_p^d)$
  such $\varphi_2(x,p)=1$ on a small neighbourhood of
  $\supp(\theta_k)\cup\supp(\theta_l)$ and have support contained in
  the set
  \begin{equation*}
  \big\{(x,p)\in\R^{2d}\,|\, |\partial_{p_1} a_{\varepsilon,0}(x,p)|>\tfrac{\tilde{c}}{2} \big\}.
  \end{equation*}
   Moreover we let
  $\varphi_1$ be a function in $C^\infty_0(\R_x^d\times\R_p^d)$ such
  $\varphi_1(x,p)=1$ on $\supp(\varphi_2)$ and such that
  \begin{equation}
  \supp(\varphi_1) \subseteq \big\{(x,p)\in\R^{2d}\,|\, |\partial_{p_1}
  a_{\varepsilon,0}(x,p)|>\tfrac{\tilde{c}}{4} \big\}.
  \end{equation}
   With these functions we define the symbol
  \begin{equation*}
    \tilde{a}_{\varepsilon,0} (x,p) = a_{\varepsilon,0} (x,p) \varphi_1(x,p) + C(1-\varphi_2(x,p)), 
  \end{equation*}
  where the constant $C$ is chosen such
  $\tilde{a}_{\varepsilon,0} (x,p)\geq 1 + \eta$ outside the support
  of $\varphi_2(x,p)$. We have
  \begin{equation*}
    \partial_{p_1}\tilde{a}_{\varepsilon,0} (x,p) = (\partial_{p_1} a_{\varepsilon,0}) (x,p) \varphi_1(x,p) + a_{\varepsilon,0} (x,p) (\partial_{p_1}\varphi_1)(x,p) - C\partial_{p_1}\varphi_2(x,p).
  \end{equation*}
  Hence there exist constants $c_0$ and $c_1$ such
  \begin{equation}\label{B.micro_hyp}
    \partial_{p_1}\tilde{a}_{\varepsilon,0} (x,p) \geq c_0 - c_1(\tilde{a}_{\varepsilon,0} (x,p))^2,
  \end{equation}
  for all $(x,p)$ in $\R^{2d}$. To see this we observe that on
  $\supp(\theta_k)\cup\supp(\theta_l)$ we have the inequality
  \begin{equation*}
    \partial_{p_1}\tilde{a}_{\varepsilon,0} (x,p)\geq \tilde{c}.
  \end{equation*}
  By continuity there exists an open neighbourhood $\Omega$ of
  $\supp(\theta_k)\cup\supp(\theta_l)$ such
  $\partial_{p_1}\tilde{a}_{\varepsilon,0} (x,p)\geq
  \frac{\tilde{c}}{3}$ and $(1-\varphi_2)\neq0$ on $\Omega^c$. Hence
  outside $\Omega$ we get the the bound
  \begin{equation*}
    \partial_{p_1}\tilde{a}_{\varepsilon,0} (x,p) \geq c_0 - c_1(\tilde{a}_{\varepsilon,0} (x,p))^2,
  \end{equation*}
  by choosing $c_1$ sufficiently large.
  This estimates is that our new symbol is global microhyperbolic in
  the direction $(\boldsymbol{0};(1,0,\dots,0))$.
 
  Our assumptions on the operator ${A}_\varepsilon(\hbar)$ imply the
  form
  \begin{equation*}
    A_\varepsilon(\hbar) = \sum_{j=0}^{N_0} \hbar^j \OpW(a_{\varepsilon,j}) + \hbar^{N_0+1} R_{N_0}(\hbar,\varepsilon),
  \end{equation*}
  where $N_0$ is chosen such
  \begin{equation*}
    \hbar^{N_0+1}\norm{R_{N_0}(\hbar,\varepsilon)}_{\mathcal{L}(L^2(\R^d))} \leq C \hbar^{N+d}.
  \end{equation*}
  By $\tilde{A}_\varepsilon(\hbar)$ we denote the operator obtained by
  taking the $N_0$ first terms of $A_\varepsilon(\hbar)$ and
  exchanging the principal symbol $a_{\varepsilon,0}$ of
  $A_\varepsilon(\hbar)$ by $\tilde{a}_{\varepsilon,0}$ . Note that
  the operator $\tilde{A}_\varepsilon(\hbar)$ still satisfies
  Assumption~\ref{B.self_adj_assumption} as the original symbols where
  assumed to be bounded. We have
  \begin{equation*}
    A_\varepsilon(\hbar)-\tilde{A}_\varepsilon(\hbar) = \OpW(a_{\varepsilon,0}  -\tilde{a}_{\varepsilon,0}) + \hbar^{N_0+1} R_{N_0}(\hbar,\varepsilon),
  \end{equation*}
  and by construction is
  $a_{\varepsilon,0} -\tilde{a}_{\varepsilon,0}$ supported away from
  $\supp(\theta_k)\cup\supp(\theta_l)$. Hence if we apply the resolvent identity
  \begin{equation*}
    (z-A_\varepsilon(\hbar))^{-1} - (z-\tilde{A}_\varepsilon(\hbar))^{-1} = (z-A_\varepsilon(\hbar))^{-1}(A_\varepsilon(\hbar) - \tilde{A}_\varepsilon(\hbar)) (z-\tilde{A}_\varepsilon(\hbar))^{-1},
  \end{equation*}
   and use the formula for composition of operators we get for
  $N_1\geq\tau$ the estimate
  \begin{equation}\label{B.neg_pro_01}
    \begin{aligned}
      \lVert\OpW(\theta_k)&((z-A_\varepsilon(\hbar))^{-1} -
      (z-\tilde{A}_\varepsilon(\hbar))^{-1})\OpW(\theta_l)\rVert_{\Tr}
      \\
      ={}&C\hbar^{-d}\norm{\OpW(\theta_k)[(z-A_\varepsilon(\hbar))^{-1}(A_\varepsilon(\hbar)
        - \tilde{A}_\varepsilon(\hbar))
        (z-\tilde{A}_\varepsilon(\hbar))^{-1}]}_{\mathcal{L}(L^2(\R^d))}
      \\
      \leq{}& C_{N_1} \frac{\hbar^{(N_1-\tau)\delta +
          \tau-d}}{\abs{\im(z)}^{N+2}} + C_{N_0}
      \frac{\hbar^N}{\abs{\im(z)}^2},
    \end{aligned}
  \end{equation}
  for $z$ in $\C$ with $|\im(z)|>0$. In order to use the above estimate we will use
  Theorem~\ref{B.Helffer-Sjostrand} and hence we need to make an
  almost analytic extensions of the function $f$. Let $\tilde{f}$ be
  an almost analytic extension of $f$, such $\tilde{f}$ is in
  $C_0^\infty(\C)$ and
  \begin{equation*} \begin{aligned}
    \tilde{f}(x) &=f(x) \text{ for all } x\in\R
    \\
    \bar{\partial}\tilde{f}(z) &\leq C_N \abs{\im(z)}^N \text{ for all
    } N\in\mathbb{N}.
   \end{aligned} \end{equation*}
  Such an extension exists according to
  Remark~\ref{B.almost_analytic_ex_remark}. As
  $ \mathcal{F}_\hbar^{-1}[\chi_\xi](s-z)$ is an analytic function in
  $z$ we have by Theorem~\ref{B.Helffer-Sjostrand} the identity
  \begin{equation}\label{B.neg_pro_02}
    \begin{aligned}
      \MoveEqLeft \Tr[\OpW(\theta_k)f(A_\varepsilon(\hbar))
      \mathcal{F}_\hbar^{-1}[\chi_\xi](s-A_\varepsilon(\hbar))\OpW(\theta_l)]
      \\
      &= -\frac{1}{\pi} \int_\C \bar{\partial}(\tilde{f})(z)
      \mathcal{F}_\hbar^{-1}[\chi_{\xi}](s-z)
      \Tr[\OpW(\theta_k)(z-A_\varepsilon(\hbar))^{-1}\OpW(\theta_l)]
      \, L(dz).
    \end{aligned}
  \end{equation}
  This identity is also true for $\tilde{A}_\varepsilon(\hbar)$. On
  the support of $\tilde{f}$ we have
  \begin{equation*}
    \mathcal{F}_\hbar^{-1}[\chi_{\xi}](s-z) = \frac{1}{2\pi\hbar} \int_\R e^{it\hbar^{-1}(s-z)} \chi_\xi(t)\,dt \leq C \frac{\xi}{\hbar}.
  \end{equation*}
  Now by the properties of $\tilde{f}$, \eqref{B.neg_pro_01},
  \eqref{B.neg_pro_02} and the above estimate we have for $N_1\geq\tau$
  \begin{equation*} \begin{aligned}
    |\Tr&[\OpW(\theta_k)f(A_\varepsilon(\hbar))
    \mathcal{F}_\hbar^{-1}[\chi_\xi](s-A_\varepsilon(\hbar))\OpW(\theta_l)]
    \\
    \hbox{} &- \Tr[\OpW(\theta_k)f(\tilde{A}_\varepsilon(\hbar))
    \mathcal{F}_\hbar^{-1}[\chi_\xi](s-\tilde{A}_\varepsilon(\hbar))\OpW(\theta_l)]|
    \\
    &\leq \frac{1}{\pi} \int_\C |\bar{\partial}(\tilde{f})(z)
    \mathcal{F}_\hbar^{-1}[\chi_{\xi}](s-z)
    \\
    &\phantom{ \frac{1}{\pi}| \int_\C \bar{\partial}(\tilde{f})(z)
      \mathcal{F}}{} \times
    \Tr[\OpW(\theta_k)((z-A_\varepsilon(\hbar))^{-1} -
    (z-\tilde{A}_\varepsilon(\hbar))^{-1})\OpW(\theta_l)]| \, L(dz)
    \\
    &\leq C_{N_1} \hbar^{(N_1-\tau)\delta + \tau-d-1} + C \hbar^N.
   \end{aligned} \end{equation*}
  Hence by choosing $N_1$ sufficiently large we can change the
  principal symbol. Note that the constant $C_{N_1}$ also depends on
  the symbols.

  For the reminder of the proof we will omit the tilde on the operator
  and principal symbol but instead assume the principal symbol to be
  global micro-hyperbolic in the direction $(\boldsymbol{0};(1,0,\dots,0))$
  (\eqref{B.micro_hyp} without the tildes).

  In order to estimate
  \begin{equation*}
    \Tr[\OpW(\theta_k)f(A_\varepsilon(\hbar)) \mathcal{F}_\hbar^{-1}[\chi_\xi](s-A_\varepsilon(\hbar))\OpW(\theta_l)],
  \end{equation*}
  we will need an auxiliary function. Let $\psi$ be in $C^\infty(\R)$
  such $\psi(t)=1$ for $t\leq1$ and $\psi(t)=0$ for $t\geq2$. Moreover
  let $M$ be a sufficiently large constant which will be fixed later
  and put
  \begin{equation*}
    \psi_{\mu_1}(z) = \psi\Big(\frac{\im(z)}{\mu_1}\Big),
  \end{equation*}
  where $\mu_1=\frac{M\hbar}{\xi}\log(\frac{1}{\hbar})$. With this
  function we have
  \begin{equation}\label{B.neg_pro_1}
    |\bar{\partial}(\tilde{f}\psi_{\mu_1})| \leq \begin{cases}
      C_N \abs{\im(z)}^N, & \text{if } \im(z)<0
      \\
      C_N \psi_{\mu_1}(z) \abs{\im(z)}^N + \mu_1^{-1} \boldsymbol{1}_{[1,2]}(\frac{\im(z)}{\mu_1}) , & \text{if } \im(z)\geq0,
    \end{cases}
  \end{equation}
  for any $N$ in $\mathbb{N}$. Again we can use
  Theorem~\ref{B.Helffer-Sjostrand} for the operator
  $A_\varepsilon(\hbar)$ on the function
  $(\tilde{f}\psi_{\mu_1})(z)
  \mathcal{F}_\hbar^{-1}[\chi_\xi](s-z)$. This gives
  \begin{equation*} \begin{aligned}
    (\tilde{f}\psi_{\mu_1})(A_\varepsilon(\hbar))&
    \mathcal{F}_\hbar^{-1}[\chi_\xi](s-A_\varepsilon(\hbar))
    \\
    &= -\frac{1}{\pi} \int_\C \bar{\partial}(\tilde{f}\psi_{\mu_1})(z)
    \mathcal{F}_\hbar^{-1}[\chi_{\xi}](s-z)
    (z-A_\varepsilon(\hbar))^{-1} \, L(dz),
   \end{aligned} \end{equation*}
  where we have used that $\mathcal{F}_\hbar^{-1}[\chi_\xi](s-z)$ is
  an analytic function in $z$.  Hence the trace we consider is
  \begin{equation}\label{B.neg_pro_2}
    \begin{aligned}
      &\Tr[\OpW(\theta_k)f(A_\varepsilon(\hbar))
      \mathcal{F}_\hbar^{-1}[\chi_\xi](s-A_\varepsilon(\hbar))\OpW(\theta_l)]
      \\
      &= -\frac{1}{\pi} \int_\C
      \bar{\partial}(\tilde{f}\psi_{\mu_1})(z)
      \mathcal{F}_\hbar^{-1}[\chi_{\xi}](s-z)
      \Tr[\OpW(\theta_k)(z-A_\varepsilon(\hbar))^{-1}
      \OpW(\theta_l)]\, L(dz)
      \\
      &= -\frac{1}{\pi} \int_{\im(z)<0} \cdots\, L(dz) -\frac{1}{\pi}
      \int_{\im(z)\geq0} \cdots\, L(dz)
    \end{aligned}
  \end{equation}
  If we shortly investigate each of the integrals. Firstly we note the
  bound
  \begin{equation*}
    \abs{\Tr[\OpW(\theta_k)(z-A_\varepsilon(\hbar))^{-1} \OpW(\theta_l)]} \leq \frac{C}{\hbar^d \abs{\im(z)}}.
  \end{equation*}
  If we consider the integral over the negative imaginary part we have
  \begin{equation*} \begin{aligned}
    \Big|\frac{1}{\pi} \int_{\im(z)<0}&
    \bar{\partial}(\tilde{f}\psi_{\mu_1})(z)
    \mathcal{F}_\hbar^{-1}[\chi_{\xi}](s-z)
    \Tr[\OpW(\theta_k)(z-A_\varepsilon(\hbar))^{-1} \OpW(\theta_l)]\,
    L(dz)\Big|
    \\
    &\leq \frac{C_{2N}\xi}{\pi\hbar^{d+1}} \int_{\im(z)<0}\frac{
      \im(z)^{2N}}{|\im(z)|} e^{\frac{\xi\im(z)}{2\hbar}}\, d\im(z)
    \\
    &\leq \frac{C_{2N}}{\pi\hbar^d} (\frac{\hbar}{\xi})^{2N-2} \leq
    \tilde{C} \hbar^{(2N-2)\gamma - d},
   \end{aligned} \end{equation*}
  for any $N$ in $\mathbb{N}$. We have in the above calculation used
  integration by parts and the estimate
  \begin{equation*}
    \abs{ \mathcal{F}_\hbar^{-1}[\chi_{\xi}](s-z)} \leq C \frac{\xi}{\hbar} e^{\frac{\xi\im(z)}{2\hbar}}.
  \end{equation*}
  The above estimate imply that the contribution to the trace from the
  negative integral is negligible.  If we split the integral over
  positive imaginary part up according to $\mu_1$ we have by
  \eqref{B.neg_pro_1} the estimate
  \begin{equation*} \begin{aligned}
    \Big|\frac{1}{\pi}& \int_{0\leq\im(z)\leq \mu_1}
    \bar{\partial}(\tilde{f}\psi_{\mu_1})(z)
    \mathcal{F}_\hbar^{-1}[\chi_{\xi}](s-z)
    \Tr[\OpW(\theta_k)(z-A_\varepsilon(\hbar))^{-1} \OpW(\theta_l)]\,
    L(dz)\Big|
    \\
    &\leq \frac{C_{2N}\xi}{\pi\hbar^{d+1}} \int_{0\leq\im(z)\leq
      \mu_1} \psi_{\mu_1}(z) \abs{\im(z)}^N
    e^{\frac{\xi\im(z)}{2\hbar}}\, d\im(z)
    \\
    &\leq \tilde{C} \frac{\xi}{\hbar^{d+1}} \mu_1^{N+1} \leq \tilde{C}
    \frac{\xi}{\hbar^{d+1}}
    \frac{M\hbar}{\xi}\log(\frac{1}{\hbar})^{N+1} \leq \tilde{C} M
    \hbar^{(N+1)\gamma - d-1} ,
   \end{aligned} \end{equation*}
  for any $N$ in $\mathbb{N}$. Hence this terms also becomes
  negligible. What remains from \eqref{B.neg_pro_2} is the expression
  \begin{equation}\label{B.neg_pro_reminder}
    -\frac{1}{\pi} \int_{\im(z)>\mu_1} \bar{\partial}(\tilde{f}\psi_{\mu_1})(z)  \mathcal{F}_\hbar^{-1}[\chi_{\xi}](s-z) \Tr[\OpW(\theta_k)(z-A_\varepsilon(\hbar))^{-1} \OpW(\theta_l)]\, L(dz).
  \end{equation}
  In order to estimates this we will need to change all our
  operators. This is done by introducing an auxiliary variable in the
  symbols and make an almost analytic extension in this
  variable. Recall we have change the operator
  $A_{\varepsilon}(\hbar)$ such it is a sum of Weyl quantised pseudo-differential operators.  In the following we let $q(x,p)$ be
  one of our symbols and we let
  $q_t(x,p) = q(x,(p_1+t,p_2,\dots,p_d))$. We now take $t$ be complex
  and make an almost analytic extension $\tilde{q}_t$ of $q_t$ in t
  according to Definition~\ref{B.def_almost_analytic_ex} for
  $|\im(t)| <1$. The form of $\tilde{q}_t$ is
  \begin{equation*}
    \tilde{q}_t(x,p) = \sum_{r=0}^{n} (\partial_{p_1}^r q)(x,(p_1+\re(t),p_2,\dots,p_d)) \frac{(i\im(t))^r}{r!},
  \end{equation*}
  Recalling the identity
  \begin{equation*}
    \OpW(q_{\re(t)}) = e^{-i\re(t)\hbar^{-1}x_1} \OpW(q) e^{i\re(t)\hbar^{-1}x_1},
  \end{equation*}
  we have
  \begin{equation}\label{B.husk.form.tilde}
    \OpW(\tilde{q}_t) = \sum_{r=0}^{n}   \frac{(i\im(t))^r}{r!} e^{-i\re(t)\hbar^{-1}x_1} \OpW(\partial_{p_1}^r q) e^{i\re(t)\hbar^{-1}x_1}.
  \end{equation}
  If we take derivatives with respect to $\re(t)$ and $\im(t)$ in
  operator sense we see
  \begin{align*}
    \frac{\partial}{\partial\re(t)} \OpW(\tilde{q}_t) &=
    -\frac{i}{\hbar} \sum_{r=0}^{n} \frac{(i\im(t))^r}{r!}
    e^{-i\re(t)\hbar^{-1}x_1} [x_1,\OpW(\partial_{p_1}^r q)]
    e^{i\re(t)\hbar^{-1}x_1}
    \\
    &= \sum_{r=0}^{n} \frac{(i\im(t))^r}{r!} e^{-i\re(t)\hbar^{-1}x_1}
    \OpW(\partial_{p_1}^{r+1} q) e^{i\re(t)\hbar^{-1}x_1},
    \shortintertext{and}
    \\
    \frac{\partial}{\partial\im(t)} \OpW(\tilde{q}_t) &= i
    \sum_{r=1}^{n} \frac{(i\im(t))^{r-1}}{(r-1)!}
    e^{-i\re(t)\hbar^{-1}x_1} \OpW(\partial_{p_1}^r q)
    e^{i\re(t)\hbar^{-1}x_1}
    \\
    &= i\sum_{r=0}^{n-1} \frac{(i\im(t))^{r}}{r!}
    e^{-i\re(t)\hbar^{-1}x_1} \OpW(\partial_{p_1}^{r+1} q)
    e^{i\re(t)\hbar^{-1}x_1}.
   \end{align*}
  In the above calculation the unbounded operator $x_1$ appear, but for all the
  symbols we consider the commutator $[x_1,\OpW(\partial_{p_1}^r q)] $
  will be bounded. This calculation gives
  \begin{equation*}
    \Big(\frac{\partial}{\partial\re(t)} + i \frac{\partial}{\partial\im(t)}\Big) \OpW(\tilde{q}_t) =  \frac{(i\im(t))^{n}}{n!} e^{-i\re(t)\hbar^{-1}x_1} \OpW(\partial_{p_1}^{n+1} q) e^{i\re(t)\hbar^{-1}x_1}
  \end{equation*}
  This imply
  \begin{equation}\label{B.neg_pro_3}
    \begin{aligned}
      &\big\lVert\frac{\partial}{\partial \bar{t}}
        \OpW(\tilde{\theta}_{j,t})\big\rVert_{\Tr} \leq C_n\hbar^{-d}
      \abs{\im(t)}^n \quad\quad\text{for $j=k,l$}
      \\
      &\big\lVert\frac{\partial}{\partial
          \bar{t}}\tilde{A}_\varepsilon(\hbar;t)\big\rVert_{\mathcal{L}(L^2(\R^d))}
      \leq C_n \abs{\im(t)}^n,
    \end{aligned}
  \end{equation}
  for any $n$ in $\N$ by choosing an almost analytic expansion of this
  order. The operator $\tilde{A}_\varepsilon(\hbar;t)$ is the operator
  where we have made the above construction for each symbol in the
  expansion of the operator. Moreover we have by the construction of
  $\tilde{A}_\varepsilon(\hbar;t)$
  \begin{equation*}
    \tilde{A}_\varepsilon(\hbar;t) =e^{-i\re(t)\hbar^{-1}x_1} A_\varepsilon(\hbar) e^{i\re(t)\hbar^{-1}x_1} + i\im(t) B_\varepsilon(\hbar;t)
  \end{equation*}
  where $B_\varepsilon(\hbar;t)$ is a bounded operator this form is obtained from \eqref{B.husk.form.tilde} with $q$ replaced by the symbol of $A_\varepsilon(\hbar)$. This gives
  \begin{equation*}
    z-\tilde{A}_\varepsilon(\hbar;t) =(z-U^{*} A_\varepsilon(\hbar) U)[I +(z-U^{*} A_\varepsilon(\hbar) U)^{-1}  i\im(t) B_\varepsilon(\hbar;t)],
  \end{equation*}
  where $U= e^{i\re(t)\hbar^{-1}x_1}$. Hence if
  $\abs{\im(t)} \leq\tfrac{ \abs{\im(z)}}{ C_1}$ the operator
  $z-\tilde{A}_\varepsilon(\hbar;t)$ has an inverse where
  $C_1\geq\norm{B_\varepsilon(\hbar;t)}_{\mathcal{L}(L^2(\R^d))} +1$.
  This imply that the following function
  \begin{equation*}
    \eta(t,z) = \Tr[\OpW(\tilde{\theta}_{k,t})(z- \tilde{A}_\varepsilon(\hbar;t))^{-1} \OpW(\tilde{\theta}_{l,t})]
  \end{equation*}
  is well defined for $\abs{\im(t)} \leq\tfrac{ \abs{\im(z)}}{
    C_1}$. The function have by construction the properties
  \begin{equation*} \begin{aligned}
    \abs{\eta(t,z)} &\leq \frac{c}{\hbar^d\abs{\im(z)}}
    \\
    |\Big(\frac{\partial}{\partial\re(t)} + i
    \frac{\partial}{\partial\im(t)}\Big) \eta(t,z) | &\leq
    \frac{c_n\abs{\im(t)}^n}{\hbar^d\abs{\im(z)}^2}.
   \end{aligned} \end{equation*}
  for $n$ in $\N$. But by cyclicity of the trace the function
  $\eta(t,z)$ is independent of $\re(t)$. Hence we have
  \begin{equation*}
    \abs{\eta(\pm i \im(t),z) - \eta(0,z)} \leq \frac{c_N\abs{\im(t)}^n}{\hbar^d\abs{\im(z)}^2}
  \end{equation*}
  by the fundamental theorem of calculus. The construction of $\eta$
  gives us that
  \begin{equation*}
    \eta(0,z) =  \Tr[\OpW(\theta_{k})(z- A_\varepsilon(\hbar))^{-1} \OpW(\theta_{l})].
  \end{equation*}
  Hence we can exchange the trace in \eqref{B.neg_pro_reminder} by
  $\eta(-i\frac{\mu_1}{C_1},z)$ with an error of the order
  $\hbar^{\gamma n-d}$. This is due to our choice of
  $\mu_1=\frac{M\hbar}{\xi} \log(\frac{1}{\hbar})$ in the start of the
  proof and that the integral is only over a compact region where
  $\abs{\im(z)}>\frac{\mu_1}{C_1}$ due to the definition of $\psi_{\mu_1}$.  It now remains to estimate the
  term
  \begin{equation}\label{B.neg_pro_reminder_2}
    -\frac{1}{\pi} \int_{\im(z)>\mu_1} \bar{\partial}(\tilde{f}\psi_{\mu_1})(z)  \mathcal{F}_\hbar^{-1}[\chi_{\xi}](s-z) \eta(-i\mu_2,z)\, L(dz),
  \end{equation}
  where
  \begin{equation*}
    \eta(-i\mu_2,z) = \Tr[ \OpW(\tilde{\theta}_{k,-i\mu_2})(z- \tilde{A}_\varepsilon(\hbar;-i\mu_2))^{-1} \OpW(\tilde{\theta}_{l,-i\mu_2}) ],
  \end{equation*}
  and $\mu_2=\frac{\mu_1}{C_1}$. From the construction of the almost
  analytic extension we have the following form of the principal
  symbol of $z- \tilde{A}_\varepsilon(\hbar;-i\mu_2)$
  \begin{equation*}
    z- \tilde{a}_{\varepsilon,0}(x,p;-i\mu_2) = z - (a_{\varepsilon,0}(x,p) - i\mu_2 (\partial_{p_1}a_{\varepsilon,0})(x,p) + \mathcal{O}(\mu_2^2)).
  \end{equation*}
  For $-\frac{c_0\mu_2}{4} <\im(z)<0$, where $c_0$ is the constant
  from the global micro-hyperbolicity \eqref{B.micro_hyp}. We have by
  the global micro-hyperbolicity for $\abs{\re(z)}<\eta$ and $\hbar$
  sufficiently small
  \begin{equation*}
    \im(z- \tilde{a}_{\varepsilon,0}(x,p;-i\mu_2)) \geq c_0\mu_2 +\im(z) - C \mu_2(\re(z)-a_{\varepsilon,0}(x,p))^2.
  \end{equation*}
  To see this recall how the principal symbol was changed and that if
  $\re(z)-a_{\varepsilon,0}(x,p)$ is zero or small then is
  $ (\partial_{p_1}a_{\varepsilon,0})(x,p)>2c_0$ hence we have to
  assume $\hbar$ sufficiently small. This implies there exists a $C_2$
  such we have the inequality
  \begin{equation*} \begin{aligned}
    \im(z- \tilde{a}_{\varepsilon,0}(x,p;-i\mu_2)& + C_2 \mu_2
    (\overline{z-\tilde{a}_{\varepsilon,0}(x,p;-i\mu_2)})(z-\tilde{a}_{\varepsilon,0}(x,p;-i\mu_2))
    \\
    &\geq \frac{c_0}{2}\mu_2 +\im(z),
   \end{aligned} \end{equation*}
  Where we again assume $\hbar$ sufficiently small and that all terms
  from the product in the above equation which is not
  $(\re(z)-a_{\varepsilon,0}(x,p))^2$ comes with at least one extra
  $\mu_2$.  Now by Theorem~\ref{B.sharp_gaar} we have for every $g$ in
  $L^2(\R^d)$
  \begin{equation*} \begin{aligned}
    \im&(\langle \OpW(z- \tilde{a}_{\varepsilon,0}(-i\mu_2)) g,g
    \rangle) + C_2 \mu_2
    \norm{\OpW(z-\tilde{a}_{\varepsilon,0}(-i\mu_2))g}_{L^2(\R^d)}^2
    \\
    &\geq \langle \OpW(\im(z- \tilde{a}_{\varepsilon,0}(-i\mu_2)) +
    C_2 \mu_2
    (\overline{z-\tilde{a}_{\varepsilon,0}(-i\mu_2)})(z-\tilde{a}_{\varepsilon,0}(-i\mu_2))
    ) g,g \rangle
    \\
    &\phantom{\geq \langle \OpW(\im(z-
      \tilde{a}_{\varepsilon,0}(-i\mu_2)) }{}- c \mu_2\hbar^\delta
    \norm{g}_{L^2(\R^d)}^2
    \\
    &\geq ( \frac{c_0\mu_2}{2} +\im(z) ) \norm{g}_{L^2(\R^d)}^2 -
    \tilde{c} ( \hbar^\delta + \mu_2 \hbar^\delta
    )\norm{g}_{L^2(\R^d)}^2 \geq \frac{c_0\mu_2}{6}
    \norm{g}_{L^2(\R^d)}^2,
   \end{aligned} \end{equation*}
  for $\hbar$ sufficiently small. Moreover, by a H\"older inequality we have
  \begin{equation*} \begin{aligned}
    \frac{c_0\mu_2}{6} & \norm{g}_{L^2(\R^d)}^2
    \\
    \leq{}& \abs{\langle \OpW(z- \tilde{a}_{\varepsilon,0}(-i\mu_2)) g,g
      \rangle} + C_2 \mu_2
    \norm{\OpW(z-\tilde{a}_{\varepsilon,0}(-i\mu_2))g}_{L^2(\R^d)}^2
    \\
    \leq{}& \frac{c_0\mu_2}{12} \norm{g}_{L^2(\R^d)}^2 + (
    \frac{6}{2c_0\mu_2} +C_2 \mu_2)
    \norm{\OpW(z-\tilde{a}_{\varepsilon,0}(-i\mu_2))g}_{L^2(\R^d)}^2.
   \end{aligned} \end{equation*}
  This shows that there exists a constant $C$ such
  \begin{equation*}
    \frac{c_0\mu_2}{C}   \norm{g}_{L^2(\R^d)} \leq \norm{\OpW(z-\tilde{a}_{\varepsilon,0}(-i\mu_2))g}_{L^2(\R^d)},
  \end{equation*}
  for all $g$ in $L^2(\R^d)$. Since
  $\OpW(z-\tilde{a}_{\varepsilon,0}(-i\mu_2))$ is the principal part
  of $ \tilde{A}_\varepsilon(\hbar;-i\mu_2)$ and the rest comes with
  an extra $\hbar$ in front as we have assumed regularity $\tau\geq1$
  the above estimate imply
  \begin{equation*}
    \frac{c_0\mu_2}{2C}   \norm{g}_{L^2(\R^d)} \leq \norm{z-\tilde{A}_\varepsilon(\hbar;-i\mu_2)g}_{L^2(\R^d)},
  \end{equation*}
  for $\hbar$ sufficiently small. We can do the above argument again
  for $\im(z)\geq0$ and obtain the same result. The estimate implies
  that the set $\{ z\in\C \,|\, \im(z)>-\frac{c_0\mu_2}{4}\}$ is in
  the regularity set of $\tilde{A}_\varepsilon(\hbar;-i\mu_2)$. Since
  $\{ z\in\C \,|\, \im(z)>-\frac{c_0\mu_2}{4}\}$ is connected we have
  that this is a subset of the resolvent set if just one point of the
  set is in the resolvent set. For a $z$ in $\C$ with positive
  imaginary part and
  $\abs{z}\geq 2\norm{\tilde{A}_\varepsilon(\hbar;-i\mu_2)}$ we have
  existence of $(z-\tilde{A}_\varepsilon(\hbar;-i\mu_2))^{-1}$ as a Neumann
  series. Hence we can conclude that
  $(z-\tilde{A}_\varepsilon(\hbar;-i\mu_2))^{-1}$ extends to a
  holomorphic function for $z$ in $\C$ such
  $\im(z) \geq -\frac{c_0\mu_2}{4 C_1}$. This implies
  \begin{equation*} \begin{aligned}
    0 &=-\frac{1}{\pi}\int_\C (\tilde{f}\psi_{\mu_1}
    \psi_{-\frac{c_0\mu_2}{4C_1}})(z)
    \mathcal{F}_\hbar^{-1}[\chi_{\xi}](s-z) \bar{\partial}
    \eta(-i\mu_2,z) \, L(dz)
    \\
    &=\frac{1}{\pi}\int_\C \bar{\partial}(\tilde{f}\psi_{\mu_1}
    \psi_{-\frac{c_0\mu_2}{4C_1}})(z)
    \mathcal{F}_\hbar^{-1}[\chi_{\xi}](s-z) \eta(-i\mu_2,z) \, L(dz)
    \\
    & = \frac{1}{\pi}\int_{\im(z)\geq0}
    \bar{\partial}(\tilde{f}\psi_{\mu_1})(z)
    \mathcal{F}_\hbar^{-1}[\chi_{\xi}](s-z) \eta(-i\mu_2,z) \, L(dz)
    \\
    &\phantom{ = \frac{1}{\pi}}{}+ \frac{1}{\pi}\int_{\im(z)<0}
    \bar{\partial}(\tilde{f}\psi_{\mu_1}
    \psi_{-\frac{c_0\mu_2}{4C_1}})(z)
    \mathcal{F}_\hbar^{-1}[\chi_{\xi}](s-z) \eta(-i\mu_2,z) \, L(dz),
   \end{aligned} \end{equation*}
  where we have used that $\psi_{-\frac{c_0\mu_2}{4C_1}}(z)=1$ for all
  $z$ in $\C$ with $\im(z)\geq0$. This equality gives us the following
  rewriting of \eqref{B.neg_pro_reminder_2}
  \begin{equation}\label{B.neg_pro_reminder_3}
    \begin{aligned}
      -\frac{1}{\pi} &\int_{\im(z)>\mu_1}
      \bar{\partial}(\tilde{f}\psi_{\mu_1})(z)
      \mathcal{F}_\hbar^{-1}[\chi_{\xi}](s-z) \eta(-i\mu_2,z)\, L(dz)
      \\
      &= \frac{1}{\pi}\int_{\im(z)<0}
      \bar{\partial}(\tilde{f}\psi_{\mu_1}
      \psi_{-\frac{c_0\mu_2}{4C_1}})(z)
      \mathcal{F}_\hbar^{-1}[\chi_{\xi}](s-z) \eta(-i\mu_2,z) \, L(dz)
      + \mathcal{O}(\hbar^{N_0}),
    \end{aligned}
  \end{equation}
  for any $N_0$ in $\N$. We have
  \begin{equation*}
    \bar{\partial}(\tilde{f}\psi_{\mu_1} \psi_{-\frac{c_0\mu_2}{4C_1}})(z) =  \bar{\partial}(\tilde{f})(z) (\psi_{\mu_1} \psi_{-\frac{c_0\mu_2}{4C_1}})(z) + \tilde{f}\psi_{\mu_1}(z) \bar{\partial} \psi_{-\frac{c_0\mu_2}{4C_1}}(z),
  \end{equation*}
  for $\im(z)<0$, where we have used that $\psi_{\mu_1}(z) = 1$ for
  $\im(z)\leq1$. The part of the integral on the right hand side of
  \eqref{B.neg_pro_reminder_3} with the derivative on $\tilde{f}$ will
  be small due to the same argumentation as previously in the
  proof. What remains is the part where the derivative is on
  $ \psi_{-\frac{c_0\mu_2}{4C_1}}$. For this part we have
  \begin{equation*} \begin{aligned}
    \frac{1}{\pi} \Big|&\int_{\im(z)<0} \tilde{f}(z) \bar{\partial}
    \psi_{-\frac{c_0\mu_2}{4C_1}}(z)
    \mathcal{F}_\hbar^{-1}[\chi_{\xi}](s-z) \eta(-i\mu_2,z) \, L(dz) \Big|
    \\
    &\leq \frac{C}{\hbar^d \left(\frac{M\hbar}{C_1\xi}
        \log(\frac{1}{\hbar})\right)^2} \int_{\substack{-\eta <\re(z)
        <\eta \\ -\frac{M\hbar
          c_0}{2C_1^2\xi}\log(\frac{1}{\hbar})<\im(z)< -\frac{M\hbar
          c_0}{4C_1^2\xi}\log(\frac{1}{\hbar})}} \frac{\xi}{\hbar}
    e^{\frac{\xi\im(z)}{2\hbar}} \,L(dz)
    \\
    &= \frac{C}{\hbar^d \left(\frac{M\hbar}{C_1\xi}
        \log(\frac{1}{\hbar})\right)^2)}
    e^{-\frac{c_oM}{2C_1^2}\log(\frac{1}{\hbar})} = \frac{\tilde{C}
      \xi^2}{\hbar^{d+2} M^2 \log(\frac{1}{\hbar})^2}
    \hbar^{\frac{c_o}{2C_1^2}M}.
   \end{aligned} \end{equation*}
  Hence by choosing $M$ sufficiently large we can make the above
  expression smaller than $\hbar^N$ for any $N$ in $\N$. This
  concludes the proof.
\end{proof}
This Theorem actually imply a stronger version of it self, where
the assumption of boundedness is not needed.
\begin{corollary}\label{B.neg_propagator_cor}
  Let $A_\varepsilon(\hbar)$ be a strongly
  $\hbar$-$\varepsilon$-admissible operator of regularity $\tau\geq1$
  which satisfies Assumption~\ref{B.self_adj_assumption} and assume there
  exists a $\delta$ in $(0,1)$ such that
  $\varepsilon\geq\hbar^{1-\delta}$. Suppose there exists a number
  $\eta>0$ such $a_{\varepsilon,0}^{-1}([-2\eta,2\eta])$ is compact
  and a constant $c>0$ such
  \begin{equation*}
    \abs{\nabla_p a_{\varepsilon,0}(x,p)} \geq c \quad \text{for all } (x,p) \in a_{\varepsilon,0}^{-1}([-2\eta,2\eta]),
  \end{equation*}
  where $a_{\varepsilon,0}$ is the principal symbol of
  $A_\varepsilon(\hbar)$. Let $f$ be in $C_0^\infty((-\eta,\eta))$ and
  $\theta$ be in $C_0^\infty(\R^d_x\times\R^d_p)$ such that
  $\supp(\theta)\subset a_{\varepsilon,0}^{-1}((-\eta,\eta))$. There
  exists a constant $T_0>0$ such that if $\chi$ is in
  $C_0^\infty((\frac12 \hbar^{1-\gamma},T_0))$ for a $\gamma$ in
  $(0,\delta]$, then for every $N$ in $\mathbb{N}$, we have
  \begin{equation*}
    \abs{\Tr[\OpW(\theta)f(A_\varepsilon(\hbar)) \mathcal{F}_\hbar^{-1}[\chi](s-A_\varepsilon(\hbar))\OpW(\theta)]} \leq C_N \hbar^N,
  \end{equation*}
  uniformly for $s$ in $(-\eta,\eta)$.
\end{corollary}
\begin{proof}
  The operator $A_\varepsilon(\hbar)$ satisfies the assumptions of
  Theorem~\ref{B.func_calc}. This gives us the functional calculus for
  the pseudo-differential operator for functions in the set
  $\mathcal{W}$ which contains all functions from $C^\infty_0(\R)$. It
  can be remarked that the function
  $f(t) \mathcal{F}_\hbar^{-1}[\chi](s-t)$ is a $C^\infty_0(\R)$ in
  $t$ and both imaginary and real part is also in $C^\infty_0(\R)$ just with real values. By $g(t)$ we denote either the real or imaginary part of
  $f(t) \mathcal{F}_\hbar^{-1}[\chi](s-t)$. Theorem~\ref{B.func_calc}
  gives
  \begin{equation*}
    g(A_\varepsilon(h)) = \sum_{j\geq 0} \hbar^j \OpW(a_{\varepsilon,j}^g),
  \end{equation*}
  where
  \begin{equation}\label{B.cor_neg_trace}
  \begin{aligned}
  a_{\varepsilon,0}^g &= g(a_{\varepsilon,0})
  \\
    a_{\varepsilon,j}^g &= \sum_{k=1}^{2j-1} \frac{(-1)^k}{k!} d_{\varepsilon,j,k} g^{(k)}(a_{\varepsilon,0}) \quad\quad\text{for $j\geq1$},
  \end{aligned}
  \end{equation}
  the symbols $d_{\varepsilon,j,k}$ are the polynomials from
  Lemma~\ref{B.approx_of_resolvent_lemma_1}. Now Let $\varphi$ be in
  $C_0^\infty(\R^d_x\times\R^d_p)$ such $\varphi(x,p)=1$ on a
  neighbourhood of
  $\supp(f(a_{\varepsilon,0})
  \mathcal{F}_\hbar^{-1}[\chi](s-a_{\varepsilon,0}))$. Then if we
  define the operator $\tilde{A}_\varepsilon(\hbar)$ as the operator
  with symbol
  \begin{equation*}
    \tilde{a}_\varepsilon(\hbar) =  \sum_{j\geq 0} \hbar^j \varphi a_{\varepsilon,j}.	
  \end{equation*}
  This operator satisfies the assumptions in
  Theorem~\ref{B.neg_propagator}. Hence we have
  \begin{equation}\label{B.cor_neg_trace_2}
    \abs{\Tr[\OpW(\theta)f(\tilde{A}_\varepsilon(\hbar)) \mathcal{F}_\hbar^{-1}[\chi](s-\tilde{A}_\varepsilon(\hbar))\OpW(\theta)]} \leq C_N \hbar^N
  \end{equation}
  But by construction we have
  \begin{equation*}
    \norm{g(A_\varepsilon(h)) - g(\tilde{A}_\varepsilon(h))} \leq C_n \hbar^n,
  \end{equation*}
  for every $n$ in $\N$, where we have used the form of the symbols
  given in \eqref{B.cor_neg_trace}. Combining this with
  \eqref{B.cor_neg_trace_2} we achieve the desired estimate.
\end{proof}
%
%
%
%%%%%%%%%%%%%%%%%%%%%%%%%%%%%%%%%%%%%%%%%%%%%%%%%%%%%%%%%%%%%

\section{Weyl law for rough pseudo-differential operators}

In this section we will prove a Weyl law for rough pseudo-differential
operators and we will do it with the approach used in
\cite{MR897108}. Hence we will first consider some asymptotic
expansions of certain integrals.

\begin{thm}\label{B.Expansion_of_trace}
  Let $A_\varepsilon(\hbar)$ be a $\hbar$-$\varepsilon$-admissible
  operator of regularity $\tau\geq1$ which satisfies
  Assumption~\ref{B.self_adj_assumption} and assume there exists a $\delta$
  in $(0,1)$ such that $\varepsilon\geq\hbar^{1-\delta}$. Suppose
  there exists $\eta>0$ such that
  $a_{0,\varepsilon}^{-1}([-2\eta,2\eta])$ is compact and every value in the interval $[-2\eta,2\eta]$ is non critical for $a_{0,\varepsilon}$, where
  $a_{\varepsilon,0}$ is the principal symbol of
  $A_\varepsilon(\hbar)$. Let $\chi$ be in $C^\infty_0((-T_0,T_0))$
  and $\chi=1$ in a neighbourhood of $0$, where $T_0$ is the number
  from Corollary~\ref{B.neg_propagator_cor}. Then for every $f$ in
  $C_0^\infty((-\eta,\eta))$ we have
  \begin{equation*}
    \int_\R \Tr[f(A_\varepsilon(\hbar)) e^{it\hbar^{-1}A_\varepsilon(\hbar)}]e^{-its\hbar^{-1} } \chi(t) \,dt = (2\pi\hbar)^{1-d} \Big[\sum_{j=0}^{N_0} \hbar^j \xi_j(s) + \mathcal{O}(\hbar^{N})\Big].
  \end{equation*}
  where the error term is uniform with respect to
  $s \in (-\eta,\eta)$ and the number $N_0$ depends on the desired error. The functions $\xi_j(s)$ are smooth functions in $s$ and are given by
  \begin{equation*}
    \xi_0(s) = f(s)  \int_{\{a_{\varepsilon,0}=s\}} \frac{1}{\abs{\nabla{a_{\varepsilon,0}}}} \,dS_s,
  \end{equation*}
  \begin{equation*}
    \xi_j(s) =  \sum_{k=1}^{2j-1}\frac{1}{k!} f(s)  \partial^k_s \int_{\{a_{\varepsilon,0}=s\}} \frac{d_{\varepsilon,j,k}}{\abs{\nabla{a_{\varepsilon,0}}}}
    \,dS_s,
  \end{equation*}
where the symbols $d_{\varepsilon,j,k}$ are the polynomials from Lemma~\ref{B.approx_of_resolvent_lemma_1}. In particular we have
  \begin{equation*} \begin{aligned}
    \xi_1 (s) = -f(s) \partial_s
    \int_{\{a_{\varepsilon,0}=s\}}
    \frac{a_{\varepsilon,1}}{\abs{\nabla{a_{\varepsilon,0}}}}
    \,dS_s.
   \end{aligned} \end{equation*}
  Moreover, we have the a priori bounds
  \begin{equation}
  |\hbar^j  \xi_j(s)  | \leq 
  \begin{cases}
  	C & \text{if $j=0$}
	\\
	C\hbar & \text{if $j=1$}
	\\
	C\hbar^{1+\delta(j-2)} & \text{if $j=1$}.
  \end{cases}
  \end{equation}
\end{thm}
\begin{remark}\label{B.Expansion_of_trace_remark}
Suppose we are in the setting of Theorem~\ref{B.Expansion_of_trace}. The statement of the theorem can be rephrased in terms of convolution of measures and a function. To see this let $f$ be in $C_0^\infty((-\eta,\eta))$, for this function we can define the function
\begin{equation*}
	\begin{aligned}
	M_f^0(\omega;\hbar) &=\Tr[f(A_\varepsilon(\hbar)) \boldsymbol{1}_{(-\infty,\omega]}(A_\varepsilon(\hbar))] 
	\\
	&= \Tr[f(A_\varepsilon(\hbar)) \boldsymbol{1}_{[0,\infty)}(\omega-A_\varepsilon(\hbar))] =  \Tr[f(A_\varepsilon(\hbar)) (\omega-A_\varepsilon(\hbar))_{+}^0],
	\end{aligned}
\end{equation*}
where $t_+=\max(0,t)$. We have that $M_f^0(\omega;\hbar)$ is a monotonic increasing function hence it defines a measure in the natural way. If we consider the function
\begin{equation*}
	\hat{\chi}_\hbar(t) = \frac{1}{2\pi\hbar} \int_\R e^{it\hbar^{-1}s} \chi(s) \, ds,
\end{equation*}
then we have
\begin{equation}
	\begin{aligned}
	\MoveEqLeft\hat{\chi}_\hbar*dM_f^0(\cdot;\hbar) (s) = \sum_{e_j(\hbar)\in\mathcal{P}} \hat{\chi}_\hbar(s-e_j(\hbar))f(e_j(\hbar))
	=\Tr[\hat{\chi}_\hbar(s-A_\varepsilon(\hbar))f(A_\varepsilon(\hbar))]
	\\
	&=\frac{1}{2\pi\hbar}\int_\R \Tr[f(A_\varepsilon(\hbar)) e^{it\hbar^{-1}A_\varepsilon(\hbar)}]e^{-its\hbar^{-1} } \chi(t) \,dt 
	= \frac{1}{(2\pi\hbar)^{d}} \Big[\sum_{j=0}^{N_0} \hbar^j \xi_j(s) + \mathcal{O}(\hbar^{N})\Big].
	\end{aligned}
\end{equation}
This formulation of the theorem will prove useful when we consider Riesz means.
\end{remark}
The proof of the theorem is split in two parts. First is the existence
of the expansion proven by a stationary phase theorem. Next is the
form of the coefficients found by application of the functional
calculus developed earlier.

\begin{proof}
  In order to be in a situation, where we can apply the stationary
  phase theorem we need to exchange the propagator with the
  approximation of it constructed in
  Section~\ref{B.construction_propagator}. As the construction
  required auxiliary localisation we need to introduce these. Let
  $\theta$ be in $C_0^\infty(\R_x^d\times\R^d_p)$ such
  $\supp(\theta)\subset a_{\varepsilon,0}^{-1}((-\eta,\eta))$ and
  $\theta(x,p) = 1$ for all $(x,p)$ in
  $\supp(f(a_{\varepsilon,0}))$. Now by
  Lemma~\ref{B.Insert_localisation_in_trace} we have
  \begin{equation}\label{B.expansion_change}
    \norm{(1-\OpW(\theta))f(A_\varepsilon(\hbar)) e^{it\hbar^{-1}A_\varepsilon(\hbar)} }_{\Tr} \leq  \norm{(1-\OpW(\theta))f(A_\varepsilon(\hbar))}_{\Tr} \leq C_N\hbar^N,
  \end{equation}
  for every $N$ in $\mathbb{N}$. Hence we have
  \begin{equation*}
    |\Tr[f(A_\varepsilon(\hbar))e^{it\hbar^{-1}A_\varepsilon(\hbar)}]-\Tr[\OpW(\theta)f(A_\varepsilon(\hbar))e^{it\hbar^{-1}A_\varepsilon(\hbar)}]| \leq C_N\hbar^N,
  \end{equation*}
  for any $N$ in $\mathbb{N}$. This implies the identity
  \begin{equation}\label{B.expansion_change_0}
    \begin{aligned}
      \MoveEqLeft \int_\R \Tr[f(A_\varepsilon(\hbar))
      e^{it\hbar^{-1}A_\varepsilon(\hbar)}]e^{-its\hbar } \chi(t) \,dt
      \\
      &= \int_\R \Tr[ \OpW(\theta) f(A_\varepsilon(\hbar))
      e^{it\hbar^{-1}A_\varepsilon(\hbar)}]e^{-its\hbar } \chi(t) \,dt
      +\mathcal{O}(\hbar^N).
    \end{aligned}
  \end{equation}
  In order to use the results of
  Section~\ref{B.construction_propagator} we need also to localise in
  time. To do this we let $\chi_2$ be in $C_0^\infty(\R)$ such
  $\chi_2(t)=1$ for $t$ in
  $[-\frac12 \hbar^{1-\frac{\delta}{2}}, \frac12
  \hbar^{1-\frac{\delta}{2}}]$ and
  $\supp(\chi_2) \subset
  [-\hbar^{1-\frac{\delta}{2}},\hbar^{1-\frac{\delta}{2}}]$. With this
  function we have
  \begin{equation}\label{B.expansion_change_1}
    \begin{aligned}
     \MoveEqLeft \int_\R \Tr[ \OpW(\theta) f(A_\varepsilon(\hbar))
      e^{it\hbar^{-1}A_\varepsilon(\hbar)}]e^{-its\hbar } \chi(t) \,dt
      \\
      &= \int_\R \Tr[ \OpW(\theta) f(A_\varepsilon(\hbar))
      e^{it\hbar^{-1}A_\varepsilon(\hbar)}]e^{-its\hbar } \chi_2(t)
      \chi(t) \,dt
      \\
      &\phantom{=}{}+ \int_\R \Tr[ \OpW(\theta)
      f(A_\varepsilon(\hbar))
      e^{it\hbar^{-1}A_\varepsilon(\hbar)}]e^{-its\hbar }(1-\chi_2(t))
      \chi(t)\,dt.
    \end{aligned}
  \end{equation}
  We will use the notation $\tilde{\chi}(t)=(1-\chi_2(-t))\chi(-t)$ in
  the following. We start by considering the second term. Here we introduce an extra localisation analogous to how we introduced the first. Using the 
  estimate in \eqref{B.expansion_change} and cyclicity of the trace
  again we have
  \begin{equation*}
    \Tr[ \OpW(\theta) f(A_\varepsilon(\hbar)) e^{it\hbar^{-1}A_\varepsilon(\hbar)}] =  \Tr[ \OpW(\theta) f(A_\varepsilon(\hbar)) e^{it\hbar^{-1}A_\varepsilon(\hbar)}\OpW(\theta)] + C_N \hbar^N.
  \end{equation*}
  Now by Corollary~\ref{B.neg_propagator_cor} we have
  \begin{equation}\label{B.expansion_change_2}
    \begin{aligned}
      \MoveEqLeft \int_\R \Tr[ \OpW(\theta) f(A_\varepsilon(\hbar))
      e^{it\hbar^{-1}A_\varepsilon(\hbar)}\OpW(\theta)]e^{-its\hbar
      }\tilde{\chi}(-t)\,dt
      \\
      &= 2\pi\hbar \Tr[ \OpW(\theta) f(A_\varepsilon(\hbar))
      \mathcal{F}_\hbar^{-1}[\tilde{\chi}](s-A_\varepsilon(\hbar))
      \OpW(\theta)]
      \leq \tilde{C}_N \hbar^N,
    \end{aligned}
  \end{equation}
  uniformly in $s$ in $[-\eta,\eta]$ and any $N$ in $\N$. What remains
  in  \eqref{B.expansion_change_1} is the first term. For this term we change the quantisation of the
  localisation. By Corollary~\ref{B.connection_t_quantisations} we
 obtain for any $N$ in $\N$
  \begin{equation*}
    \OpW(\theta) = \OpN{0}(\theta_0^N) + \hbar^{N+1} R_N(\hbar),
  \end{equation*}
  where $R_N$ is a bounded operator uniformly in $\hbar$ since
  $\theta$ is a non-rough symbol. Moreover we have
  \begin{equation*}
    \theta_0^N(x,p) = \sum_{j=0}^N \frac{\hbar^j}{j!}  \big(-\frac{1}{2}\big)^j (\nabla_x D_p)^j \theta(x,p).
  \end{equation*}
  If we choose $N$ sufficiently large (greater than or equal to $2$)
  we can exchange $ \OpW(\theta)$ by $\OpN{0}(\theta_0^N)$ plus a
  negligible error. We will in the following omit the $N$ on
  $\theta_0^N$. For the first term on the right hand side in
  \eqref{B.expansion_change_1} we have
  $\abs{t}\leq\hbar^{1-\frac{\delta}{2}}$. Now by
  Theorem~\ref{B.existence_of_approximation} there exists
  $U_N(t,\varepsilon,\hbar)$ with integral kernel
  \begin{equation*} \begin{aligned}
     K_{U_N}(x,y,t,\varepsilon,\hbar)
    = \frac{1}{(2\pi\hbar)^d} \int_{\R^d} e^{i \hbar^{-1} \langle
      x-y,p\rangle} e^{ i t \hbar^{-1} a_{\varepsilon,0}(x,p)}
    \sum_{j=0}^N (it\hbar^{-1})^j u_j(x,p,\hbar,\varepsilon) \, dp,
   \end{aligned} \end{equation*}
  such that
  \begin{equation}\label{B.expansion_change_3}
    \norm{\hbar\partial_t U_N(t,\varepsilon, \hbar) - i U_N(t,\varepsilon, \hbar) A_\varepsilon(\hbar)}_{\mathcal{L}(L^2(\R^d))}  \leq C \hbar^{N_0},
  \end{equation}
  for $\abs{t}\leq \hbar^{1-\frac{\delta}{2}}$ and
  $U_N(0,\varepsilon,\hbar) = \OpN{0}(\theta_0)$. We emphasise that
  the number $N$ in the operator $U_N$ is dependent on the error
  $N_0$.  We observe that
  \begin{equation*}
  	\begin{aligned}
     \MoveEqLeft |\Tr[ \OpN{0}(\theta_0) e^{it\hbar^{-1}A_\varepsilon(\hbar)}
    f(A_\varepsilon(\hbar))] - \Tr[ U_N(t,\varepsilon, \hbar)
    f(A_\varepsilon(\hbar))] |
    \\
    &=| \Tr[ \int_0^t \partial_s( U_N(t-s,\varepsilon, \hbar)
    e^{is\hbar^{-1}A_\varepsilon(\hbar)} f(A_\varepsilon(\hbar))
    \,ds]|
    \\
    &=| \Tr[ \int_0^t (-(\partial_tU_N)(t-s,\varepsilon, \hbar) +
    i\hbar^{-1} U_N(t-s,\varepsilon, \hbar) A_\varepsilon(\hbar) )
    e^{is\hbar^{-1}A_\varepsilon(\hbar)} f(A_\varepsilon(\hbar))
    \,ds]|
    \\
    &\leq \hbar^{-1} \int_0^t \norm{\hbar\partial_sU_N(s,\varepsilon,
      \hbar) - i U_N(s,\varepsilon, \hbar)
      A_\varepsilon(\hbar)}_{\mathcal{L}(L^2(\R^d))}
      \norm{e^{i(t-s)\hbar^{-1}A_\varepsilon(\hbar)}
      f(A_\varepsilon(\hbar))}_{\Tr} \,ds
    \\
    &\leq C_N \hbar^{N_0-d},
  \end{aligned}
  \end{equation*}
  where we have used \eqref{B.expansion_change_3}. By combining this
  with \eqref{B.expansion_change_1} and \eqref{B.expansion_change_2}
  we have
  \begin{equation}\label{B.expansion_change_4}
    \begin{aligned}
      \int_\R &\Tr[ \OpW(\theta) f(A_\varepsilon(\hbar))
      e^{it\hbar^{-1}A_\varepsilon(\hbar)}]e^{-its\hbar } \chi_2(t)
      \chi(t) \,dt
      \\
      &= \int_\R \Tr[ U_N(t,\varepsilon,
      \hbar)f(A_\varepsilon(\hbar))]e^{-its\hbar }\chi_2(t) \chi(t)
      \,dt + \mathcal{O}(\hbar^N).
    \end{aligned}
  \end{equation}
  Before we proceed we will change the quantisation of
  $f(A_\varepsilon(\hbar))$. From Theorem~\ref{B.func_calc} we have
  \begin{equation*}
    f(A_\varepsilon(\hbar)) = \sum_{j\geq 0} \hbar^j  \OpW(a_{\varepsilon,j}^f),
  \end{equation*}
  where
  \begin{equation}\label{B.expansion_cal_sym}
    a_{\varepsilon,j}^f = \sum_{k=1}^{2j-1} \frac{(-1)^k}{k!} d_{\varepsilon,j,k} f^{(k)}(a_{\varepsilon,0}),
  \end{equation}
  the symbols $d_{\varepsilon,j,k}$ are the polynomials from
  Lemma~\ref{B.approx_of_resolvent_lemma_1}. We choose a sufficiently
  large $N$ and consider the first $N$ terms of the operator
  $f(A_\varepsilon(\hbar))$. For each of these terms we can use
  Corollary~\ref{B.connection_t_quantisations} and this yields
  \begin{equation*}
    \OpW(a_{\varepsilon,j}^f) =\OpN{1}(a_{\varepsilon,j}^{f,M}) + \hbar^{M+1} R_M,
  \end{equation*}
  where $\hbar^{M+1} R_M$ is a bounded by $C_M \hbar^N$ in operator
  norm. The symbol $a_{\varepsilon,j}^{f,M}$ is given by
  \begin{equation*}
    a_{\varepsilon,j}^{f,M} = \sum_{j=0}^M \frac{\hbar^j}{j!}  \big(\frac{1}{2}\big)^j (\nabla_x D_p)^j a_{\varepsilon,j}^{f}.
  \end{equation*}
  By chossing $N$ sufficiently large we can exchange
  $f(A_\varepsilon(\hbar))$ by
  \begin{equation*}
    \OpN{1}(\tilde{a}_{\varepsilon}^{f,M}) \coloneqq \sum_{j= 0}^N \hbar^j \OpN{1}(a_{\varepsilon,j}^{f,M}), 
  \end{equation*}
  plus a negligible error as $U_N(t,\varepsilon, \hbar)$ is trace
  class. We will omit the $M$ when writing
  $\tilde{a}_{\varepsilon}^{f,M}$. Hence we have the equality
  \begin{equation*} \begin{aligned}
    \int_\R &\Tr[ \OpW(\theta) f(A_\varepsilon(\hbar))
    e^{it\hbar^{-1}A_\varepsilon(\hbar)}]e^{-its\hbar } \chi_2(t)
    \chi(t) \,dt
    \\
    &= \int_\R \Tr[ U_N(t,\varepsilon, \hbar)
    \OpN{1}(\tilde{a}_{\varepsilon}^{f,M}) ]e^{-its\hbar } \chi_2(t)
    \chi(t) \,dt + \mathcal{O}(\hbar^N).
   \end{aligned} \end{equation*}
  As we have the non-critical assumption
  Lemma~\ref{B.neg_part} gives us that the trace in the above expression is
  negligible for $\frac12 \hbar^{1-\frac{\delta}{2}} \leq|t|\leq T_0$.
  Hence we can omit the $ \chi_2(t)$ in the expression and then
  we have
  \begin{equation}\label{B.expansion_change_5}
    \begin{aligned}
      \MoveEqLeft \int_\R \Tr[ \OpW(\theta) f(A_\varepsilon(\hbar))
      e^{it\hbar^{-1}A_\varepsilon(\hbar)}]e^{-its\hbar } \chi_2(t)
      \chi(t) \,dt
      \\
      &= \int_\R \Tr[ U_N(t,\varepsilon, \hbar)
      \OpN{1}(\tilde{a}_{\varepsilon}^{f,M}) ]e^{-its\hbar } \chi(t)
      \,dt + \mathcal{O}(\hbar^N).
    \end{aligned}
  \end{equation}
  The two operators $U_N(t,\varepsilon, \hbar)$ and
  $\OpN{1}(\tilde{a}_{\varepsilon}^{f,M})$ are both given by kernels
  and the composition of the operators has the kernel
  \begin{equation*} \begin{aligned}
    \MoveEqLeft K_{U_N(t,\varepsilon,
      \hbar)\OpN{1}(\tilde{a}_{\varepsilon}^{f,M})}(x,y)
    \\
    & = \frac{1}{(2\pi\hbar)^d} \int_{\R^d} e^{i
      \hbar^{-1} \langle x-y,p\rangle} e^{ i t \hbar^{-1}
      a_{\varepsilon,0}(x,p)} \sum_{j=0}^N (it\hbar^{-1})^j
    u_j(x,p,\hbar,\varepsilon) \tilde{a}_{\varepsilon}^{f,M} (y,p) \,
    dp.
   \end{aligned} \end{equation*}
  We can now calculate the trace and we get
  \begin{equation}\label{B.expansion_change_6}
    \begin{aligned}
      \MoveEqLeft \int_\R \Tr[ U_N(t,\varepsilon, \hbar)
      \OpN{1}(\tilde{a}_{\varepsilon}^{f,M}) ]e^{-its\hbar } \chi(t)
      \,dt
      \\
      &= \frac{1}{(2\pi\hbar)^{d}} \int_{\R^{2d+1}}
      \chi(t) e^{ i t \hbar^{-1} (a_{\varepsilon,0}(x,p)-s)}
      u(x,p,t,\hbar,\varepsilon) \tilde{a}_{\varepsilon}^{f,M} (x,p)
      \,dxdpdt,
    \end{aligned}
  \end{equation}
  where
  \begin{equation*}
    u(x,p,t,\hbar,\varepsilon) = \sum_{j=0}^N (it\hbar^{-1})^j u_j(x,p,\hbar,\varepsilon).
  \end{equation*}
  In order to evaluate the integral we will a stationary phase argument. We will use the theorem in $t$ and one of the $p$
  coordinates after using a partition of unity according to
  $p$.  By assumption we have that $\abs{\nabla_p a_\varepsilon,0}>c$ on
  the support of $\theta$. Hence we can make a partition $\Omega_j$
  such that $\partial_{p_j} a_\varepsilon\neq 0$ on $\Omega_j$ and
  with loss of generality we can assume that $\Omega_j$ is
  connected. To this partition we choose a partition of the unit
  supported on each of the sets $\Omega_j$. When we have localised to
  each of these sets the calculation will be identical with some
  indicies changed. Hence we assume that
  $\partial_{p_1} a_\varepsilon\neq 0$ on the entire support of the
  integrant. We will now make a change of variables in the integral in
  the following way:
  \begin{equation*}
    F:(x,p) \rightarrow (X,P)=(x_1,\dots,x_d,a_{\varepsilon,0}(x,p),p_2,\dots,p_d).
  \end{equation*}
  This transformation has the following jacobian matrix
  \begin{equation*}
    DF = 
    \begin{pmatrix}
      I_d & 0_{d\times d} \\
      \nabla_x a_{\varepsilon,0}^t & \nabla_p a_{\varepsilon,0}^t \\
      0_{d-1 \times d+1} & I_{d-1}
    \end{pmatrix},
  \end{equation*}
  where $I_d$ is the $d$-dimensional identity matrix,
  $\nabla_x a_\varepsilon^t$ and $\nabla_p a_\varepsilon^t$ are the
  transposed of the respective gradients and the zeros are
  corresponding matrices with only zeroes and the dimensions indicated
  in the subscript. We note that
  \begin{equation*}
    \det(DF)= \partial_{p_1}a_{\varepsilon,0},
  \end{equation*}
  which is non zero by our assumptions. Hence the inverse map exists
  and we will denote it by $F^{-1}$. For the inverse
  we denote the part that gives $p$ as a function of $(X,P)$ by
  $F_2^{-1}$. By this change of variables we have
  \begin{equation*} \begin{aligned}
    & \int_{\R^{2d+1}}\chi(t) e^{ i t \hbar^{-1}
      (a_{\varepsilon,0}(x,p)-s)} u(x,p,t,\hbar,\varepsilon)
    \tilde{a}_{\varepsilon}^{f,M} (x,p) \,dxdpdt
    \\
    &= \int_{\R^{2d+1}} \chi(t) e^{ i t \hbar^{-1}
      (P_1-s)} \frac{u(X,F_2^{-1}(X,P),t,\hbar,\varepsilon)
      \tilde{a}_{\varepsilon}^{f,M}
      (X,F_2^{-1}(X,P))}{\partial_{p_1}a_{\varepsilon,0}(X,F_2^{-1}(X,P))}
    \,dXdPdt,
   \end{aligned} \end{equation*}
  where we have omitted the prefactor $(2\pi\hbar)^{-d}$. If we preform the change of variables $\tilde{P}_1 = P_1-s$ we arrive at a situation we we
  can apply quadratic stationary phase. Hence by stationary phase in the variables $\tilde{P}_1$ and $t$,
  \eqref{B.expansion_change_0}, \eqref{B.expansion_change_1},
  \eqref{B.expansion_change_4}, \eqref{B.expansion_change_5} and
  \eqref{B.expansion_change_6} we get
  \begin{equation}\label{B.expansion_change_7}
    \int_\R \Tr[f(A_\varepsilon(\hbar)) e^{it\hbar^{-1}A_\varepsilon(\hbar)}]e^{-its\hbar } \chi(t) \,dt = (2\pi\hbar)^{1-d} \Big[\sum_{j=0}^{N_0} \hbar^j \xi_j(s) + \mathcal{O}(\hbar^{N})\Big],
  \end{equation}
  uniformly for $s$ in $(-\eta,\eta)$. The a priori bounds given in the Theorem also follows directly from this application of stationary phase combined with the estimates on $u_j(x,p,\hbar,\varepsilon)$ from Theorem~\ref{B.existence_of_approximation}. This ends the proof of the
  existence of the expansion. 
  
  From the above expression we have that
  $\xi_j(s)$ are smooth functions in $s$ hence the above expression
  defines a distribution on $C^\infty_0((-\eta,\eta))$. So in order to
  find the expressions of the $\xi_j(s)$'s we consider the action of
  the distribution. We let $\varphi$ be in
  $C_0^\infty((-\eta,\eta))$ and consider the expresion
  \begin{equation}\label{B.find_xi_0}
    \int_{\R^2} \Tr[f(A_\varepsilon(\hbar)) e^{it\hbar^{-1}A_\varepsilon(\hbar)}]e^{-its\hbar } \chi(t) \varphi(s) \,dt ds.
  \end{equation}
 Using that $f$ is supported in the pure point spectrum of
  $A_\varepsilon(\hbar)$ we have
  \begin{equation}\label{B.find_xi_1}
    \begin{aligned}
      \MoveEqLeft \int_{\R^2} \Tr[f(A_\varepsilon(\hbar))
      e^{it\hbar^{-1}A_\varepsilon(\hbar)}]e^{-its\hbar } \chi(t)
      \varphi(s) \,dt ds
      \\
      &= \Tr[f(A_\varepsilon(\hbar)) \int_\R
      \mathcal{F}_1[\chi](\tfrac{s}{\hbar})
      \varphi(A_\varepsilon(\hbar)-s) \, ds ],
    \end{aligned}
  \end{equation}
  where we have used Fubini's theorem. That $f$ is supported in the pure point spectrum follows from Theorem~\ref{B.point.spectrum}. If we consider the integral in
  the right hand side of \eqref{B.find_xi_1} and let $\psi$ be in
  $C_0^\infty((-2,2))$ such that $\psi(t)=1$ for $\abs{t}\leq 1$ we
  have
  \begin{equation}\label{B.find_xi_2}
    \begin{aligned}
      \int_\R \mathcal{F}_1[\chi](\tfrac{s}{\hbar})
      \varphi(A_\varepsilon(\hbar)-s) \, ds
      % \\
      ={}& \int_{\R^2} e^{-it s\hbar^{-1}} \chi(t) \psi(s)
      \varphi(A_\varepsilon(\hbar)-s) \, ds dt
      \\
      &+ \int_{\R^2} e^{-it s\hbar^{-1}} \chi(t)(1- \psi(s))
      \varphi(A_\varepsilon(\hbar)-s) \, ds dt.
    \end{aligned}
  \end{equation}
 From the identity
  \begin{equation*}
    \Big(\frac{i\hbar}{s}\Big)^n \partial_t^n e^{-it s\hbar^{-1}} =e^{-it s\hbar^{-1}},
  \end{equation*}
  integration by parts, the spectral theorem and that the function
  $(1- \psi(s))$ is support on $\abs{s}\geq1$, we have that
  \begin{equation}\label{B.find_xi_3}
    \Big\lVert \int_{\R^2} e^{-it s\hbar^{-1}} \chi(t)(1- \psi(s)) \varphi(A_\varepsilon(\hbar)-s) \, ds dt \Big\rVert_{\mathcal{L}(L^2(\R^d))}= C_N \hbar^N, 
  \end{equation}
  for any $N$ in $\mathbb{N}$.  Now for the first integral in the
  right hand side of \eqref{B.find_xi_2} we have by
  Proposition~\ref{B.quad_stationary_phase} (Quadratic stationary phase)
  \begin{equation}\label{B.find_xi_4}
    \begin{aligned}
      \int_{\R^2}  &e^{-it s\hbar^{-1}} \chi(t) \psi(s)
      \varphi(A_\varepsilon(\hbar)-s) \, ds dt
      \\
      &= 2\pi\hbar \sum_{j=0}^N \hbar^j \frac{(i)^j}{j!}
      \chi^{(j)}(0) \varphi^{(j)} (A_\varepsilon(\hbar)) +
      \hbar^{N+1} R_{N+1}(\hbar),
    \end{aligned}
  \end{equation}
  where we have used that $\psi(0)=1$ and $\psi^{(j)}(0)=0$ for all
  $j\in\mathbb{N}$. Moreover, we have from
  Theorem~\ref{B.quad_stationary_phase} the estimate
  \begin{equation*}
    \abs{R_{N+1}(\hbar)} \leq c \sum_{l+k=2} \int_{\R^2}  | \chi^{(N+1+l)}(t) \partial_s^{N+1+k} \psi(s) \varphi (A_\varepsilon(\hbar)-s)|  \,dtds.
  \end{equation*}
  As the integrants are supported on a compact set the integral will
  be convergent and since $\varphi$ is $C_0^\infty(\R)$ we have by the
  spectral theorem
  \begin{equation}\label{B.find_xi_5}
    \norm{R_{N+1}(\hbar)}_{\mathcal{L}(L^2(\R^d))} \leq C.
  \end{equation}
  If we now use that $\chi$ is $1$ in a neighbourhood of $0$ and
  combine \eqref{B.find_xi_1}--\eqref{B.find_xi_5} we have
  \begin{equation}\label{B.find_xi_6}
    \begin{aligned}
     \MoveEqLeft \int_{\R^2} \Tr[f(A_\varepsilon(\hbar))
      e^{it\hbar^{-1}A_\varepsilon(\hbar)}]e^{-its\hbar } \chi(t)
      \varphi(s) \,dt ds
      \\
      &= 2\pi\hbar \Tr[f(A_\varepsilon(\hbar))
      \varphi(A_\varepsilon(\hbar)) ] + C_N \hbar^N.
    \end{aligned}
  \end{equation}
  Since both $f$ and $\varphi$ are $C_0^\infty((-\eta,\eta))$
  functions we have by Theorem~\ref{B.trace_formula_fkt} the identity
  \begin{equation}\label{B.find_xi_7}
    \Tr[f(A_\varepsilon(\hbar)) \varphi(A_\varepsilon(\hbar)) ]  = \frac{1}{(2\pi\hbar)^d} \sum_{j=0}^N \hbar^j T_j(f \varphi,A_\varepsilon(\hbar)) + \mathcal{O}(\hbar^{N+1-d}).
  \end{equation}
  From
  Theorem~\ref{B.trace_formula_fkt} we have the exact form of the
  terms $T_j(f \varphi,A_\varepsilon(\hbar))$, which is given
  by
  \begin{equation*}
    T_j(f \varphi,A_\varepsilon(\hbar) ) =
    \begin{cases} 
     \int_{\R^{2d}} (f\varphi)(a_{\varepsilon,0}) \, dxdp & j=0
     \\
    \int_{\R^{2d}} a_{\varepsilon,1}  (f\varphi)^{(1)}(a_{\varepsilon,0}) \, dxdp & j=1
     \\
    \int_{\R^{2d}} \sum_{k=1}^{2j-1} \frac{(-1)^k}{k!} d_{\varepsilon,j,k} (f\varphi)^{(k)}(a_{\varepsilon,0}) \, dxdp & j\geq2,
    \end{cases}
  \end{equation*}
 where the symbols $d_{\varepsilon,j,k}$ are the polynomials from Lemma~\ref{B.approx_of_resolvent_lemma_1}. If we combine \eqref{B.expansion_change_7}, \eqref{B.find_xi_6} and
  \eqref{B.find_xi_7} we get
  \begin{equation*}
    \int_\R \xi_j(s) \varphi(s) \, ds =  T_j(f \varphi,A_\varepsilon(\hbar) )  .
  \end{equation*}
  If we consider $T_0(f \varphi,A_\varepsilon(\hbar) )$ we
  have
  \begin{equation*} \begin{aligned}
    T_0 (f \varphi,A_\varepsilon(\hbar) ) &=
   \int_{\R^{2d}} (f\varphi)(a_{\varepsilon,0}) \,
    dxdp
    \\
    &= \int_\R f(\omega) \varphi(\omega)
    \int_{\{a_{\varepsilon,0}=\omega\}}
    \frac{1}{\abs{\nabla{a_{\varepsilon,0}}}} \,dS_\omega d\omega,
   \end{aligned} \end{equation*}
  where $S_\omega$ is the euclidian surface measure on the surface in
  $\R_x^d\times\R_p^d$ given by the equation
  $a_{\varepsilon,0}(x,p)=\omega$. If we now consider
  $T_j(f \check{\varphi},A_\varepsilon(\hbar) )$ we have
\begin{equation}
	\begin{aligned}
    	T_j(f \varphi,A_\varepsilon(\hbar) )& =\int_{\R^{2d}}\sum_{k=1}^{2j-1} \frac{(-1)^k}{k!} d_{\varepsilon,j,k} (f\varphi)^{(k)}(a_{\varepsilon,0}) \, dxdp
	\\
	&=  \sum_{k=1}^{2j-1}\frac{(-1)^k}{k!} \int_{\R} (f\varphi)^{(k)}(\omega)  \int_{\{a_{\varepsilon,0}=\omega\}} \frac{d_{\varepsilon,j,k}}{\abs{\nabla{a_{\varepsilon,0}}}}
    \,dS_\omega d\omega
    \\
    &=  \sum_{k=1}^{2j-1}\frac{1}{k!} \int_{\R} (f\varphi)(\omega)  \partial^k_\omega \int_{\{a_{\varepsilon,0}=\omega\}} \frac{d_{\varepsilon,j,k}}{\abs{\nabla{a_{\varepsilon,0}}}}
    \,dS_\omega d\omega,
	\end{aligned}
\end{equation}
  where we in the last equality have made integration by parts. These equalities implies the stated form of the functions $\xi_j$.
\end{proof}

We will now fixing some notation which will be useful for the rest of
this section. We let $\chi$ be a function in $C_0^\infty((-T_0,T_0))$,
where the $T_0$ is a sufficiently small number. The number will be the
number $T_0$ from Corollary~\ref{B.neg_propagator_cor}. We suppose
$\chi$ is even and $\chi(t)=1$ for $|t|\leq\frac{T_0}{2}$. We then set
\begin{equation*}
  \hat{\chi}_1(s) = \frac{1}{2\pi} \int_{\R} \chi(t) e^{-its} \, dt.
\end{equation*}
We assume $\hat{\chi}_1\geq0$ and that there is a $c>0$ such that
$\hat{\chi}_1(t)\geq c$ in a small interval around $0$. These
assumptions can be guaranteed by replacing $\chi$ by $\chi*\chi$. With
this we set
\begin{equation*}
  \hat{\chi}_\hbar(s) =\frac{1}{\hbar}\hat{\chi}_1(\tfrac{s}{\hbar}) = \frac{1}{2\pi\hbar} \int_{\R}\chi(t) e^{its\hbar^{-1}} \, dt.
\end{equation*}
\begin{lemma}\label{B.bound_on_distri_L1}
Assume we are in the setting of Theorem~\ref{B.Expansion_of_trace} and let $g$ be in $L^1(\R)$ with $\supp(g)\subset(-\eta,\eta)$. Then for any $j\in \N_0$
  \begin{equation*}
    \Big| \int_{\R} g(s) \xi_j(s) \,ds \Big| \leq  C\varepsilon^{(\tau-j)_{-}} \norm{g}_{L^1(\R)}.
  \end{equation*}
\end{lemma}
\begin{proof}
From Theorem~\ref{B.Expansion_of_trace} we have that $ \xi_j(s) $ is bounded for all $j$. Hence by standard approximations it is sufficient to prove the statement for $g\in C_0^\infty ((-\eta,\eta))$ which we assume from here on. 

That $ \xi_j(s) $ is bounded for all $j$ immediately give us the estimate for $j=0$. So we may assume $j\geq1$. By definition of $\xi_j(s)$ we have that 
  \begin{equation*}
  \begin{aligned}
    \int_{\R} g(s) \xi_j(s) \,ds  
    &= \sum_{k=1}^{2j-1}\frac{1}{k!} \int_{\R} g(s)  f(s)  \partial^k_s \int_{\{a_{\varepsilon,0}=s\}} \frac{d_{\varepsilon,j,k}}{\abs{\nabla{a_{\varepsilon,0}}}}
    \,dS_s ds
    \\
    & = \sum_{k=1}^{2j-1}\frac{(-1)^k}{k!} \int_{\R^{2d}} (gf)^{(k)}( a_{\varepsilon,0}(x,p)) d_{\varepsilon,j,k}(x,p) \,dx dp
    \end{aligned}
  \end{equation*}
On the support of $fg(a_{\varepsilon,0}(x,p))$ we have that $|\nabla_p a_{\varepsilon,0}(x,p)|>c$. Hence we can make a partition $\{\Omega_j\}_{j=1}^d$ such that $\partial_{p_j} a_\varepsilon\neq 0$ on $\Omega_j$ and with loss of generality we can assume that $\Omega_j$ is connected. As in the proof of Theorem~\ref{B.Expansion_of_trace} we now choose a partition of unity supported on each $\Omega_j$ and split the integral accordingly. Hence without loss of generality we can assume $\partial_{p_1} a_\varepsilon\neq 0$ on the whole support. With the same notation as in the proof of Theorem~\ref{B.Expansion_of_trace} we preform the change of variables   
\begin{equation*}
    F:(x,p) \rightarrow (X,P)=(x_1,\dots,x_d,a_{\varepsilon,0}(x,p),p_2,\dots,p_d).
  \end{equation*}
  This gives us that
    \begin{equation*}
  \begin{aligned}
    \int_{\R} g(s) \xi_j(s) \,ds  
    &=  \sum_{k=1}^{2j-1}\frac{1}{k!} \int_{\R^{2d}} g( P_1) f( P_1) \partial_{P_1}^k \frac{d_{\varepsilon,j,k}(X,F_2^{-1}(X,P))}{\partial_{p_1}a_{\varepsilon,0}(X,F_2^{-1}(X,P))} \,dX dP,
    \end{aligned}
  \end{equation*}
where we after the change of variables have preformed integration by parts. We have from the definition of the polynomials $d_{\varepsilon,j,k}$ that they are of regularity $\tau-j$ and since we are integrating over a compact subset of $\R^{2d}$ we obtain the estimate
  \begin{equation*}
    \Big| \int_{\R} g(s) \xi_j(s) \,ds \Big| \leq  C\varepsilon^{(\tau-j)_{-}} \norm{g}_{L^1(\R)}.
  \end{equation*}
This concludes the proof.
\end{proof}
\subsection{Weyl law for rough pseudo-differential operators}
Before we state and prove the Weyl law for rough pseudo-differential operators we recall a Tauberian theorem from
\cite[Theorem V--13]{MR897108}.
\begin{thm}\label{B.tauberian}
  Let $\tau_1<\tau_2$ and $\sigma_\hbar:\R\rightarrow\R$ be a family
  of increasing functions, where $\hbar$ is in $(0,1]$. Suppose that
  \begin{enumerate}[label={$(\roman*)$}]
  \item $\sigma_\hbar(\tau) =0$ for every $\tau\leq\tau_1$.
  \item $\sigma_\hbar(\tau)$ is constant for $\tau\geq\tau_2$.
  \item $\sigma_\hbar(\tau) = \mathcal{O}(\hbar^{-n})$ as
    $\hbar\rightarrow0$, $n\geq1$ and uniformly with respect to $\tau$
    in $\R$.
  \item
    $ \partial_\tau \sigma_\hbar * \hat{\chi}_\hbar (\tau) =
    \mathcal{O}(\hbar^{-n})$ as $\hbar\rightarrow0$, with the same $n$
    as above and uniformly with respect to $\tau$ in $\R$.
  \end{enumerate}
  where $\hat{\chi}_\hbar$ is defined as above. Then we have
  \begin{equation*}
    \sigma_\hbar(\tau)= \sigma_\hbar * \hat{\chi}_\hbar (\tau) + \mathcal{O}(\hbar^{1-d}),
  \end{equation*}
  as $\hbar\rightarrow0$ and uniformly with respect to $\tau$ in $\R$.
\end{thm}
\begin{thm}\label{B.Weyl.law.rough_1}
  Let $A_\varepsilon(\hbar)$ be a strongly
  $\hbar$-$\varepsilon$-admissible operator of regularity $\tau\geq1$
  which satisfies Assumption~\ref{B.self_adj_assumption} and assume there
  exists a $\delta$ in $(0,1)$ such that
  $\varepsilon\geq\hbar^{1-\delta}$. Suppose there exists a $\eta>0$
  such $a_{\varepsilon,0}^{-1}((-\infty,\eta])$ is compact, where
  $a_{\varepsilon,0}$ is the principal symbol of
  $A_\varepsilon(\hbar)$. Moreover we suppose
  \begin{equation}\label{B.non-critical-weyl}
    |\nabla_p a_{\varepsilon,0}(x,p)| \geq c \quad\text{for all } (x,p)\in a_{\varepsilon,0}^{-1}(\{ 0 \}).
  \end{equation}
  Then we have
  \begin{equation*}
    |\Tr[\boldsymbol{1}_{(-\infty,0]}(A_\varepsilon(\hbar))] - \frac{1}{(2\pi\hbar)^d} \int_{\R^{2d}}\boldsymbol{1}_{(-\infty,0]}( a_{\varepsilon,0}(x,p)) \,dx dp | \leq C \hbar^{1-d},
  \end{equation*}
  for all sufficiently small $\hbar$.
\end{thm}
\begin{proof}
  By the assumption in \eqref{B.non-critical-weyl} there exists a
  $\nu>0$ such
  \begin{equation*}
    |\nabla_p a_{\varepsilon,0}(x,p)| \geq \frac{c}{2} \quad\text{for all } (x,p)\in a_{\varepsilon,0}^{-1}([-2\nu,2\nu]).
  \end{equation*}
  More over by Theorem~\ref{B.self_adjoint_thm_1} we have that the
  spectrum of $A_\varepsilon(\hbar)$ is bounded from below uniformly
  in $\hbar$ and let $E$ denote a number with distance $1$ to the
  bottom of the spectrums. We now take two functions $f_1$ and $f_2$
  in $C_0^\infty(\R)$ such
  \begin{equation*}
    f_1(t) + f_2(t)  = 1,
  \end{equation*}
  for every $t$ in $[E,0]$,
  $\supp(f_2)\subset[-\frac{\nu}{4},\frac{\nu}{4}]$, $f_2(t)=1$ for
  $t$ in $[-\frac{\nu}{8},\frac{\nu}{8}]$ and $f_2(t)=f_2(-t)$ for all
  $t$. With these functions we have
  \begin{equation}\label{B.Weyl_proof_1}
    \Tr[\boldsymbol{1}_{(-\infty,0]}(A_\varepsilon(\hbar)) ]= \Tr[ f_1(A_\varepsilon(\hbar))] + \Tr[ f_2(A_\varepsilon(\hbar))\boldsymbol{1}_{(-\infty,0]}(A_\varepsilon(\hbar))].
  \end{equation}
  For the first term on the right hand in the above equality we have
  by Theorem~\ref{B.trace_formula_fkt} that
  \begin{equation}\label{B.Weyl_proof_2}
    \Tr[ f_1(A_\varepsilon(\hbar))] = \frac{1}{(2\pi\hbar)^d} \int_{\R^{2d}} f_1(a_{\varepsilon,0}(x,p)) \,dxdp + \mathcal{O}(\hbar^{1-d}).
  \end{equation}
  In order to calculate the second term on the right hand side in
  \eqref{B.Weyl_proof_1} we will study the function
  \begin{equation}\label{B.Weyl_proof_3}
    \omega \rightarrow M(\omega;\hbar) = \Tr[ f_2(A_\varepsilon(\hbar))\boldsymbol{1}_{(-\infty,\omega]}(A_\varepsilon(\hbar))].
  \end{equation}
  We have that $M(\omega;\hbar)$ satisfies the three first conditions
  in Theorem~\ref{B.tauberian}. We will in what follows use the notation
  \begin{equation*}
    \mathcal{P}=\supp(f_2)\cap\spec(A_\varepsilon(\hbar)),
  \end{equation*}
  where $\spec(A_\varepsilon(\hbar))$ is the spectrum of the operator
  $A_\varepsilon(\hbar)$. The function $M$ can be written in the
  following form
  \begin{equation*}
    M(\omega;\hbar) = \sum_{e_j \in \mathcal{P}} f_2(e_j) \boldsymbol{1}_{[e_j,\infty)}(\omega),
  \end{equation*}
  since $f_2$ is supported in the pure point spectrum of
  $A_\varepsilon(\hbar)$. This follows from Theorem~\ref{B.point.spectrum}. Let $\hat{\chi}_\hbar$ be defined as
  above. Then we will consider the convolution
  \begin{equation*}
    (M(\cdot;\hbar)*\hat{\chi}_\hbar)(\omega) = \int_\R M(s;\hbar) \hat{\chi}_\hbar(\omega-s) \,ds =  \sum_{e_j \in\mathcal{P}} f_2(e_j) \int_{e_j}^\infty \hat{\chi}_\hbar(\omega-s) \,ds.
  \end{equation*}
  If we take a derivative with respect to $\omega$ we get
  \begin{equation*} \begin{aligned}
    \partial_\omega (M(\cdot;\hbar)*\hat{\chi}_\hbar)(\omega) &=
    \sum_{e_j \in\mathcal{P}} f_2(e_j) \hat{\chi}_\hbar(\omega-e_j)
    \\
    &= \frac{1}{2\pi\hbar} \int_\R
    \Tr[f_2(A_\varepsilon(\hbar))e^{it\hbar^{-1}
      A_\varepsilon(\hbar)}] e^{-it\omega\hbar^{-1}} \chi(t) \,dt,
   \end{aligned} \end{equation*}
  by the definition of $\hat{\chi}_\hbar$. We get now by
  Theorem~\ref{B.Expansion_of_trace} the identity
  \begin{equation}
    \partial_\omega (M(\cdot;\hbar)*\hat{\chi}_\hbar)(\omega) = \frac{1}{(2\pi\hbar)^d} f_2(\omega) \int_{\{a_{\varepsilon,0}=\omega\}} \frac{1}{\abs{\nabla a_{\varepsilon,0}}} \,dS_\omega + \mathcal{O}(\hbar^{1-d}),
  \end{equation}
  where the size of the error follows from the a priori estimates given in the Theorem. This verifies the fourth condition in Theorem~\ref{B.tauberian} for
  $M(\cdot;\hbar)$, hence the Theorem gives the identity
  \begin{equation}\label{B.Weyl_proof_4}
    \begin{aligned}
      \Tr[
      f_2&(A_\varepsilon(\hbar))\boldsymbol{1}_{(-\infty,0]}(A_\varepsilon(\hbar))]
      \\
      &= \frac{1}{(2\pi\hbar)^d} \int_{-\infty}^0 f_2(\omega)
      \int_{\{a_{\varepsilon,0}=\omega\}} \frac{1}{\abs{\nabla
          a_{\varepsilon,0}}} \,dS_\omega d\omega +
      \mathcal{O}(\hbar^{1-d})
      \\
      &=\frac{1}{(2\pi\hbar)^d} \int_{\R^{2d}}
      f_2(a_{\varepsilon,0}(x,p))
      \boldsymbol{1}_{(-\infty,0]}(a_{\varepsilon,0}(x,p)) \,dxdp +
      \mathcal{O}(\hbar^{1-d}).
    \end{aligned}
  \end{equation}
  By combining \eqref{B.Weyl_proof_1}, \eqref{B.Weyl_proof_2} and
  \eqref{B.Weyl_proof_4} we get
  \begin{equation*}
    \Tr[\boldsymbol{1}_{(-\infty,0]}(A_\varepsilon(\hbar)) ] = \frac{1}{(2\pi\hbar)^d} \int_{\R^{2d}}  \boldsymbol{1}_{(-\infty,0]}(a_{\varepsilon,0}(x,p)) \,dxdp + \mathcal{O}(\hbar^{1-d}).
  \end{equation*}
  This is the desired estimate and this ends the proof.
\end{proof}
\subsection{Riesz means for rough pseudo-differential operators}
\begin{thm}\label{B.Weyl.law.rough_2}
  Let $A_\varepsilon(\hbar)$ be a strongly
  $\hbar$-$\varepsilon$-admissible operator of regularity $\tau\geq2$
  which satisfies Assumption~\ref{B.self_adj_assumption} and assume there
  exists a $\delta$ in $(0,1)$ such that
  $\varepsilon\geq\hbar^{1-\delta}$. Suppose there exists a $\eta>0$
  such $a_{\varepsilon,0}^{-1}((-\infty,\eta])$ is compact, where
  $a_{\varepsilon,0}$ is the principal symbol of
  $A_\varepsilon(\hbar)$. Moreover we suppose
  \begin{equation}\label{B.non-critical-weyl_2}
    |\nabla_p a_{\varepsilon,0}(x,p)| \geq c \quad\text{for all } (x,p)\in a_{\varepsilon,0}^{-1}(\{ 0 \}).
  \end{equation}
  Then for $1\geq\gamma>0$ we have
  \begin{equation*}
    \big|\Tr[(A_\varepsilon(\hbar))_{-}^\gamma] - \frac{1}{(2\pi\hbar)^d}[ \Psi_0(\gamma,A_\varepsilon) + \hbar\Psi_1(\gamma,A_\varepsilon)] \big| \leq C \hbar^{1+\gamma-d},
  \end{equation*}
  for all sufficiently small $\hbar$. The numbers $\Psi_j(\gamma,A_\varepsilon)$ are given by
  \begin{equation}
  	\begin{aligned}
  \Psi_0(\gamma,A_\varepsilon) &= \int_{\R^{2d}} (a_{\varepsilon,0}(x,p))_{-}^\gamma \,dxdp
	\\
  \Psi_1(\gamma,A_\varepsilon) &=  \gamma \int_{\R^{2d}}  a_{\varepsilon,1}(x,p) (a_{\varepsilon,0}(x,p))_{-}^{\gamma-1}  \,dxdp.
  	\end{aligned}
  \end{equation}
\end{thm}
\begin{remark}\label{Estimate_error_reisz_1}
	The proof of this theorem is valid for any $\gamma>0$. For the case where $\gamma>1$ the expansion will have additional terms. These additional terms can also be found and calculated explicitly. But note that to ensure the error is as stated here one will need to impose the restriction $\gamma\leq\tau-1$ on $\gamma$ and the regularity $\tau$ in genereal.
	
	If we had assumed $\tau=1$ the error would, in terms of the semiclassical parameter, have been $\max(\hbar^{1+\delta-d},\hbar^{1+\gamma-d})$. Under the assumption that $\delta\geq \gamma$ we would get the desired error of $\hbar^{1+\gamma-d}$.  This choice would be a possibility, however if $\tau=1$ and we have such a choice of $\delta$ will only be able to obtain the desired estimate, when comparing the phase space integrals in the proof of Theorem~\ref{B.Riesz_means_thm_irr_cof},  in the case $\gamma\leq\frac{1}{3}$. In this case we would need to assume $\mu\geq \frac{2\gamma}{1-\gamma}$. For $\gamma>\frac{1}{3}$ we will not be able to obtain the desired error when comparing the phase space integrals in the proof of Theorem~\ref{B.Riesz_means_thm_irr_cof} due to the fraction  $\frac{2\gamma}{1-\gamma}$ being greater than one.
\end{remark}	
\begin{proof}
  By the assumption in \eqref{B.non-critical-weyl_2} there exists a
  $\nu>0$ such
  \begin{equation*}
    |\nabla_p a_{\varepsilon,0}(x,p)| \geq \frac{c}{2} \quad\text{for all } (x,p)\in a_{\varepsilon,0}^{-1}([-2\nu,2\nu]).
  \end{equation*}
By Theorem~\ref{B.self_adjoint_thm_1} we have that the
  spectrum of $A_\varepsilon(\hbar)$ is bounded from below uniformly
  in $\hbar$. We will in the following let $E$ denote a number with distance $1$ to the
  bottom of the spectrum. We now take two functions $f_1$ and $f_2$
  in $C_0^\infty(\R)$ such
  \begin{equation*}
    f_1(t) + f_2(t)  = 1,
  \end{equation*}
  for every $t$ in $[E,0]$,
  $\supp(f_2)\subset[-\frac{\nu}{4},\frac{\nu}{4}]$, $f_2(t)=1$ for
  $t$ in $[-\frac{\nu}{8},\frac{\nu}{8}]$ and $f_2(t)=f_2(-t)$ for all
  $t$. We can now write 
  \begin{equation}\label{B.splitting.gamma1}
  	\Tr[(A_\varepsilon(\hbar))_{-}^\gamma] = \Tr[f_1(A_\varepsilon(\hbar))(A_\varepsilon(\hbar))_{-}^\gamma] 
	+\Tr[f_2(A_\varepsilon(\hbar))(A_\varepsilon(\hbar))_{-}^\gamma]
  \end{equation}
  The first term in \eqref{B.splitting.gamma1} is now the trace of a smooth compactly supported function of our operator. Hence we can calculate the asymptotic using Theorem~\ref{B.trace_formula_fkt}. Before applying this theorem we will study the second term in \eqref{B.splitting.gamma1}. Before we proceed we note that
  \begin{equation*}
  	(A_\varepsilon(\hbar))_{-}^\gamma = (0-A_\varepsilon(\hbar))_{+}^\gamma 
  \end{equation*}
  This form will be slightly more convient to work with and we will introduce the notation $\varphi^\gamma(t)=(t)_+^\gamma$. 
  As we will use a smoothing procedure we will consider the expression
    \begin{equation}\label{B.general_weyl_0} 
    	\begin{aligned}
	M_{f_2}^\gamma(\omega;\hbar)&=\Tr[f_2(A_\varepsilon(\hbar))\varphi^\gamma(\omega-A_\varepsilon(\hbar))] 
	\\
	&= \sum_{e_j(\hbar)\leq\omega} f_2(e_j(\hbar))\varphi^\gamma(\omega-e_j(\hbar))  =(\varphi^\gamma * dM_{f_2}^0(\cdot;\hbar))(\omega),  
	\end{aligned}
  \end{equation}
  where $dM_{f_2}^0$\footnote{The measure is just a sum of delta measures in the eigenvalues.} is the measure induced by the function
  \begin{equation*}
	M_{f_2}^0(\omega;\hbar) =  \Tr[f_2(A_\varepsilon(\hbar)) (\omega-A_\varepsilon(\hbar))_{+}^0],
\end{equation*}
as in Remark~\ref{B.Expansion_of_trace_remark}. Let $\hat{\chi}_\hbar(t)$ be as above and consider the convolution $\hat{\chi}_\hbar*M_{f_2}^\gamma$. Then we have
\begin{equation*}
	\begin{aligned}
	\hat{\chi}_\hbar*M_{f_2}^\gamma(\omega) =
	[  \hat{\chi}_\hbar*\varphi^\gamma *dM_{f_2}^0(\cdot;\hbar)] (\omega)=[ \varphi^\gamma * \hat{\chi}_\hbar*dM_{f_2}^0(\cdot;\hbar)] (\omega).
	\end{aligned}
\end{equation*}
From Remark~\ref{B.Expansion_of_trace_remark} we have an asymptotic expansion of $\hat{\chi}_\hbar*dM_{f_2}^0(\cdot;\hbar)(\omega) $. In order to see that we can use this expansion let $g$ be in $C_0^\infty([-\nu,\nu])$ such that $g(t)=1$ for all $t$ in $[-\frac{\nu}{2},\frac{\nu}{2}]$. Then we have
\begin{equation}\label{B.general_weyl_1}
	\begin{aligned}
	\MoveEqLeft {} [\varphi^\gamma* \hat{\chi}_\hbar*dM_{f_2}^0(\cdot;\hbar)] (\omega)
	\\
	={}&\int_\R \varphi^\gamma(\omega-t) g(t) [\hat{\chi}_\hbar*dM_{f_2}^0(\cdot;\hbar)](t) \,dt
	+ \int_\R \varphi^\gamma(\omega-t) (1-g(t)) [\hat{\chi}_\hbar*dM_{f_2}^0(\cdot;\hbar)](t) \,dt
	\end{aligned}
\end{equation}
For the second term we have
\begin{equation*}
	 (1-g(t)) [\hat{\chi}_\hbar*dM_{f_2}^0(\cdot;\hbar)](t) = \frac{ 1-g(t) }{2\pi\hbar} \sum_{e_j(\hbar)} f_2(e_j(\hbar)) \hat{\chi}\left(\tfrac{e_j(\hbar)-t}{\hbar}\right),
\end{equation*}
Since $\hat{\chi}$ is Schwarz class and $\abs{t-e_j(\hbar)}\geq\frac{\nu}{4}$ on the support of $ (1-g(t)) $ by the support properties of $f_2$ we have that
\begin{equation}\label{B.general_weyl_2}
	\left| \int_\R \varphi^\gamma(\omega-t) (1-g(t)) [\hat{\chi}_\hbar*dM_{f_2}^0(\cdot;\hbar)](t) \,dt\right| \leq C_N\hbar^N
\end{equation}
for all $N$ in $\N$. We can use the expansion from Remark~\ref{B.Expansion_of_trace_remark} on the first term in \eqref{B.general_weyl_1} due to the support properties of $g$ and the estimate obtained in \eqref{B.general_weyl_2} for the second. This yields the estimate
\begin{equation}\label{B.general_weyl_3}
	\begin{aligned}
	\MoveEqLeft \Big| [\varphi^\gamma* \hat{\chi}_\hbar*dM_{f_2}^0(\cdot;\hbar)] (\omega) -   \frac{1}{(2\pi\hbar)^{d}}\sum_{j=0}^{1} \hbar^j  \int_\R \varphi^\gamma(\omega-t) g(t) \xi_j(t)  \,dt  \Big| \leq C\hbar^{2-d},
	\end{aligned}
\end{equation}
where we have also used Lemma~\ref{B.bound_on_distri_L1} and the assumption that $\tau\geq2$ to obtain the estimate. Combining \eqref{B.splitting.gamma1}, \eqref{B.general_weyl_0} and \eqref{B.general_weyl_3} we obtain that   
  \begin{equation*}
  	\begin{aligned}
     \big|\Tr[(A_\varepsilon(\hbar))_{-}^\gamma] - \frac{1}{(2\pi\hbar)^d} &[ \Psi_0(\gamma,A_\varepsilon) + \hbar\Psi_1(\gamma,A_\varepsilon)] \big| 
    \\
    &\leq |M_{f_2}^\gamma(\omega;\hbar) - [\hat{\chi}_\hbar*M_{f_2}^\gamma(\cdot;\hbar)](\omega) |+ C \hbar^{2-d}.
    \end{aligned}
  \end{equation*}
What remains is an Tauberian argument to prove that
\begin{equation*}
	|M_{f_2}^\gamma(\omega;\hbar) - [\hat{\chi}_\hbar*M_{f_2}^\gamma(\cdot;\hbar)](\omega) | \leq C \hbar^{1+\gamma-d}
\end{equation*}
uniform in $\omega$. We have by definition of $\hat{\chi}_\hbar$ that
\begin{equation*}
	M_{f_2}^\gamma(\omega;\hbar) - [\hat{\chi}_\hbar*M_{f_2}^\gamma(\cdot;\hbar)](\omega)  = \frac{1}{2\pi} \int_\R (M_{f_2}^\gamma(\omega;\hbar) - M_{f_2}^\gamma(\omega-t\hbar;\hbar))\hat{\chi}(t) \,dt,
\end{equation*}
where we have used that
\begin{equation*}
	\frac{1}{2\pi}\int_\R \hat{\chi}(t) \,dt = 1.
\end{equation*}
In order to proceed we will need the Fourier transform of $M_{f_2}^\gamma(\omega;\hbar)$. This function exists as a limit of a complex Fourier transform. We have that
\begin{equation}
	\begin{aligned}
	\mathcal{F}[M_{f_2}^\gamma(\cdot;\hbar)] (t - i\kappa) 
	&=\int_\R e^{-i(t-i\kappa)s}\sum_{e_j(\hbar)} f_2(e_j(\hbar))\varphi^\gamma(s-e_j(\hbar))  \,ds
	\\
	&=\int_\R e^{-i(t-i\kappa)s}\varphi^\gamma(s)  \,ds  \sum_{e_j(\hbar)} e^{-i(t-i\kappa)e_j(\hbar)} f_2(e_j(\hbar))
	\\
	&= \int_0^\infty e^{(-it-\kappa)s}s^\gamma \, ds \sum_{e_j(\hbar)} e^{-i(t-i\kappa)e_j(\hbar)} f_2(e_j(\hbar))
	\\
	&=  \frac{\Gamma(\gamma+1)}{(t-i\kappa)^{\gamma+1}} i^{\gamma+1}  \sum_{e_j(\hbar)} e^{-i(t-i\kappa)e_j(\hbar)} f_2(e_j(\hbar)),
	\end{aligned}
\end{equation}
Where $\Gamma$ is the gamma-function. By taking $\kappa$ to zero we get that
\begin{equation}
	\mathcal{F}[M_{f_2}^\gamma(\cdot;\hbar)] (t ) =  \frac{\Gamma(\gamma+1)}{(t+i0)^{\gamma+1}} e^{\frac{i\pi}{2}(\gamma+1)}  \sum_{e_j(\hbar)} e^{-ite_j(\hbar)} f_2(e_j(\hbar))
\end{equation}
With this we have
\begin{equation*}
	 M_{f_2}^\gamma(\omega;\hbar) - M_{f_2}^\gamma(\omega-t\hbar;\hbar) = \frac{1}{2\pi} \int_\R e^{is\omega} \mathcal{F}[M_{f_2}^\gamma(\cdot;\hbar)] (s )(1-e^{-it\hbar s}) \,ds
\end{equation*}
Combining these things we have
\begin{equation*}
	\begin{aligned}
	M_{f_2}^\gamma(\omega;\hbar) - [\hat{\chi}_\hbar*M_{f_2}^\gamma(\cdot;\hbar)](\omega)  
	&= \frac{1}{2\pi} \int_\R (M_{f_2}^\gamma(\omega;\hbar) - M_{f_2}^\gamma(\omega-t\hbar;\hbar))\hat{\chi}(t) \,dt
	\\
	&=\frac{1}{4\pi^2}  \int_\R \int_\R e^{is\omega} \mathcal{F}[M_{f_2}^\gamma(\cdot;\hbar)] (s )(1-e^{-it\hbar s})\hat{\chi}(t) \, ds dt
	\end{aligned}
\end{equation*}
By reintroducing $\kappa$ and then take $\kappa$ to zero we get by Fubini and the form of the Fourier transform that
\begin{equation*}
	\begin{aligned}
	\MoveEqLeft M_{f_2}^\gamma(\omega;\hbar) - [\hat{\chi}_\hbar*M_{f_2}^\gamma(\cdot;\hbar)](\omega)  
	\\
	&=\lim_{\kappa\rightarrow0} \frac{1}{4\pi^2}  \int_{\R^2} e^{is\omega}  \frac{\Gamma(\gamma+1)}{(s-i\kappa)^{\gamma+1}} {i^{\gamma+1}}  \sum_{e_j(\hbar)} e^{-i(s-i\kappa)e_j(\hbar)} f_2(e_j(\hbar)))(1-e^{-it\hbar s})\hat{\chi}(t) \, ds dt
	\\
	&=\lim_{\kappa\rightarrow0} \frac{1}{2\pi}  \int_\R e^{is\omega}  \frac{\Gamma(\gamma+1)}{(s-i\kappa)^{\gamma+1}} {i^{\gamma+1}}  \sum_{e_j(\hbar)} e^{-i(s-i\kappa)e_j(\hbar)} f_2(e_j(\hbar))(1-\chi(-\hbar s)) \, ds
	\\
	&= C(\gamma)  \sum_{e_j(\hbar)} f_2(e_j(\hbar)) \hbar^{-1} \lim_{\kappa\rightarrow0} e^{-i\kappa e_j(\hbar)} \int_\R e^{is\hbar^{-1}(\omega-e_j(\hbar))}  \frac{(1-\chi( s))}{(s\hbar^{-1}-i\kappa)^{\gamma+1}}  \, ds
	\\
	&= C(\gamma)  \sum_{e_j(\hbar)} f_2(e_j(\hbar)) \hbar^{\gamma} \int_\R e^{-is\hbar^{-1}(e_j(\hbar)-\omega)}  \frac{(1-\chi( s))}{s^{\gamma+1}}  \, ds
	\\
	&=C(\gamma) \hbar^{\gamma}  \sum_{e_j(\hbar)} f_2(e_j(\hbar))\hat{\psi}\left( \frac{e_j(\hbar)-\omega}{\hbar} \right),
	\end{aligned}
\end{equation*}
where $\psi(s) = (1-\chi( s))s^{-\gamma-1}$ which is in $L^1(\R)$ by the definition of $\chi$ and that $\gamma>0$. The function $\hat{\psi}$ will be continuous and smooth except at $0$, moreover it will decay at any polynomial order at infinity. Now we just need to estimate the sum such that we get one power of $\hbar$ and loos $d$. Too do too note that
\begin{equation*}
	\begin{aligned}
	\Big|\sum_{e_j(\hbar)} f_2(e_j(\hbar))\hat{\psi}\left( \frac{e_j(\hbar)-\omega}{\hbar} \right) \Big|
	& = \left| \sum_{m\in\Z} \sum_{m\hbar< e_j(\hbar)-\omega \leq (m+1)\hbar}f_2(e_j(\hbar))\hat{\psi}\left( \frac{e_j(\hbar)-\omega}{\hbar} \right) \right|
	\\
	&\leq \sum_{m\in\Z} C \hbar^{1-d} (1+m^2) \sup_{m<t\leq m+1} |\hat{\psi}(t)| \leq C \hbar^{1-d}
	\end{aligned}
\end{equation*}
where we really need that $f_2$ is supported in a non-critical interval. Hence we get a Tauberian error of the desired order.
\end{proof}

%%%%%%%%%%%%%%%%%%%%%%%%%%%%%%%%%%%%%%%%%%%%%%%%%%%%%%%%%%%%%
\section{Proof of main theorems}
\begin{lemma}\label{lemma_weyl_q_framin_form}
Let $\mathcal{A}_\hbar$ be a sesquilinear  form which satisfies Assumption~\ref{ses_assump_gen} with the numbers $(\tau,\mu)$ in $\N\times[0,1]$. Suppose $\mathcal{A}_{\hbar,\varepsilon}^{\pm}$ is the framing sesquilinear forms constructed in Proposition~\ref{B.construc_framing_op}. Then for all $\varepsilon$ sufficiently small there exists two $\hbar$-$\varepsilon$-admissible operators of regularity $\tau$ $A^{\pm}_\varepsilon(\hbar)$ such that they satisfies Assumption~\ref{B.self_adj_assumption}. In particular they will be lower semi-bounded, selfajoint for all $\hbar$ sufficiently small and satisfies that     
\begin{equation}
	A^{\pm}_\varepsilon(\hbar)[\varphi,\psi] = \mathcal{A}_{\hbar,\varepsilon}^{\pm}[\varphi,\psi] \qquad\text{for all $\varphi,\psi \in \mathcal{D}( \mathcal{A}_{\hbar,\varepsilon}^{\pm})$ }.
\end{equation}
\end{lemma}
\begin{remark}
Assume we are in the setting of Lemma~\ref{lemma_weyl_q_framin_form}. Since we have that $A^{\pm}_\varepsilon(\hbar)$ both are  $\hbar$-$\varepsilon$-admissible of regularity $\tau$ it follows from Theorem~\ref{B.connection_quantisations} that we can write the operators as a sum Weyl quantised symbols that is
\begin{equation}
 	A^{\pm}_\varepsilon(\hbar) = \sum_{j=0}^{2m} \hbar^j \OpW(a^{\pm}_{\varepsilon,j}).
\end{equation}
Written in this form the operator still satisfies  Assumption~\ref{B.self_adj_assumption}. In particular we have that 
\begin{equation}
	a^{\pm}_{\varepsilon,0}(x,p) = \sum_{|\alpha|, |\beta|\leq m}a_{\alpha\beta}^\varepsilon (x) p^{\alpha+\beta}   \pm C_1 \varepsilon^{k+\mu}\sum_{|\alpha|\leq m} p^{2\alpha}
\end{equation}
and 
\begin{equation}
	a^{\pm}_{\varepsilon,1}(x,p) =  i \sum_{|\alpha|, |\beta|\leq m}\sum_{j=1}^d \frac{\beta_j-\alpha_j}{2} \partial_{x_j} a_{\alpha\beta}^\varepsilon (x) p^{\alpha+\beta-\eta_j}, 
\end{equation}
where $\eta_j$ is the multi index with all zeroes except on the $j$'th coordinate where it is 1. Note that from the assumption $a_{\alpha\beta}^\varepsilon(x)=\overline{a_{\beta\alpha}^\varepsilon(x)}$ it follows that $a^{\pm}_{\varepsilon,1}(x,p)$ is real valued.

\end{remark}
\begin{proof}
Assume  $\varphi$ and $\psi $ are Schwartz functions and recall that the sesquilinear forms  $\mathcal{A}_{\hbar,\varepsilon}^{\pm}$ are given by
\begin{equation}\label{B.sym_on_S_approx_form}
 	\begin{aligned}
   \mathcal{A}_{\hbar,\varepsilon}^{\pm}[\varphi,\psi] = {}& \sum_{\abs{\alpha},\abs{\beta}\leq m}  \int_{\R^d} a_{\alpha\beta}^\varepsilon (x) (\hbar D_x)^\beta\varphi(x) \overline{(\hbar D_x)^\alpha\psi(x)} \, dx 
   \\
   & \pm C_1 \varepsilon^{k+\mu}  \sum_{\abs{\alpha}\leq m}  \int_{\R^d} (\hbar D_x)^\alpha\varphi(x) \overline{(\hbar D_x)^\alpha\psi(x)} \, dx
   \\
   ={}& \sum_{\abs{\alpha},\abs{\beta}\leq m} \sum_{\alpha_1\leq \alpha} (-i\hbar)^{|\alpha_1|} \binom{\alpha}{\alpha_1}   \int_{\R^d} [\partial_x^{\alpha_1}a_{\alpha\beta}^\varepsilon] (x) (\hbar D_x)^{\alpha+\beta-\alpha_1}\varphi(x) \overline{\psi(x)} \, dx 
   \\
   & \pm C_1 \varepsilon^{k+\mu}  \sum_{\abs{\alpha}\leq m}  \int_{\R^d} (\hbar D_x)^{2\alpha} \varphi(x) \overline{\psi(x)} \, dx,
  	\end{aligned}
\end{equation}
where $a_{\alpha\beta}^\varepsilon (x)$ are smooth functions and in the identity we have made integration by parts. With this in mind we define the symbol $b_\varepsilon(x,p;\hbar)$ to be
\begin{equation}
	\begin{aligned}
	b_\varepsilon(x,p) &= \sum_{j=0}^m (-i\hbar)^j \sum_{\substack{j\leq |\alpha| \leq m \\ |\beta|\leq m}} \sum_{\substack{\alpha_1\leq \alpha \\ |\alpha_1|= j} }\binom{\alpha}{\alpha_1} [\partial_x^{\alpha_1}a_{\alpha\beta}^\varepsilon] (x) p^{\alpha+\beta -\alpha_1 }  \pm C_1 \varepsilon^{k+\mu}\sum_{|\alpha|\leq m} p^{2\alpha}
	\\
	&= \sum_{j=0}^m \hbar^j b_{\varepsilon,j}(x,p).
	\end{aligned}
\end{equation}
Next we verify that there exists a $\zeta<0$ such that $b_{\varepsilon,0}(x,p)-\zeta$ is a tempered weight and that $b_{\varepsilon,j}$ is in $\Gamma^{b_{\varepsilon,0}-\zeta,\tau-j}_{0,\varepsilon}(\R_x^{d}\times\R_p^{d})$ for all $j$.  

We have that the forms are elliptic for all $\varepsilon$ sufficiently small and by assumption all coefficients are bounded below by a number $-\zeta_0$. This implies that 
\begin{equation}
	\begin{aligned}
	b_{\varepsilon,0}(x,p) &=  \sum_{|\alpha|, |\beta|\leq m}a_{\alpha\beta}^\varepsilon (x) p^{\alpha+\beta}   \pm C_1 \varepsilon^{k+\mu}\sum_{|\alpha|\leq m} p^{2\alpha}
	\\
	&\geq C |p|^{2m} -\zeta_0  \sum_{|\alpha|, |\beta|< m}| p^{\alpha+\beta}|   - C_1 \varepsilon^{k+\mu}\sum_{|\alpha|< m} |p^{2\alpha}| \geq -C
	\end{aligned}
\end{equation}
for some positive constant $C$ and for all $(x,p)$. Here we have used that $a_{\alpha\beta}^\varepsilon(x)=\overline{a_{\beta\alpha}^\varepsilon(x)}$ which implies that $b_{\varepsilon,0}(x,p) $ is real for all $(x,p)$. Hence we can find $\zeta<0$ such that $b_{\varepsilon,0}(x,p)-\zeta>0$ for all $(x,p)$. 

To see that this is a tempered weight we first recall how the new coefficients $a_{\alpha\beta}^\varepsilon(x)$ was constructed. They are given by
  \begin{equation*}
    a_{\alpha\beta}^\varepsilon(x) = \int_{\R^d} a_{\alpha\beta}(x-\varepsilon y)\omega(y) \,dy, 
  \end{equation*}
  where $\omega(y)$ is a real Schwarz function integrating to 1. By assumption there exists a constant $c$ such that $a_{\alpha\beta}^\varepsilon(x)+ c>0$ and we have
  \begin{equation}\label{B.Weyl_with_irre_1}
    \begin{aligned}
    \MoveEqLeft   a_{\alpha\beta}^\varepsilon(x)+ c = \int_{\R^d}
      (a_{\alpha\beta}(x-\varepsilon y)+c)\omega(y) \,dy
      \leq C_1(a_{\alpha\beta}(z)+c)
      \int_{\R^d}(1+|x-\varepsilon y-z|)^M \omega(y) \,dy
      \\
      &\leq C (a_{\alpha\beta}^\varepsilon(z) + a_{\alpha\beta}(z) -
      a_{\alpha\beta}^\varepsilon(z) +c) (1+|x-z|)^M
      \leq \tilde{C} (a_{\alpha\beta}^\varepsilon(z) + c)
      (1+|x-z|)^M,
    \end{aligned}
  \end{equation}
where we have used that $\varepsilon\leq1$, some of the assumptions on $a_{\alpha\beta}(x)$  and that
  \begin{equation}\label{B.Weyl_with_irre_2}
    |a_{\alpha\beta}(z) - a_{\alpha\beta}^\varepsilon(z)| \leq c\varepsilon^{1+\mu},
  \end{equation} 
  by Proposition~\ref{B.smoothning_of_func}. This gives us that the coefficients  $a_{\alpha\beta}^\varepsilon(x)$ have at most polynomial growth. Due to the polynomial structure of $b_{\varepsilon,0}(x,p)$ and the at most polynomial growth of  $a_{\alpha\beta}^\varepsilon(x)$ there exists constants $C$ and $N$ independent of $\varepsilon$ in $(0,1]$ such that for all $(x,p)$ in $\R^{2d}$
\begin{equation}
	b_{\varepsilon,0}(x,p) - \zeta \leq C (b_{\varepsilon,0}(y,q) - \zeta)\lambda(x-y,p-q)^N \qquad \forall  (x,p) \in \R^{2d}.
\end{equation}
This shows that $b_{\varepsilon,0}(x,p)$ is a tempered weight. To see that $b_{\varepsilon,j}$ is in $\Gamma^{b_{\varepsilon,0}-\zeta,\tau-j}_{0,\varepsilon}(\R_x^{d}\times\R_p^{d})$ for all $j$ we first note that  for $\eta$ in $\N^d_0$ with $|\eta|\leq \tau$ we have
  \begin{equation}\label{B.Weyl_with_irre_3}
    \begin{aligned}
      \abs{\partial_{x}^\eta a_{\alpha\beta}^\varepsilon(x)} &= |
      \int_{\R^d} \partial_{x}^\eta a_{\alpha\beta}(x-\varepsilon
      y)\omega(y) \,dy| \leq \int_{\R^d}
      | \partial_{x}^\eta a_{\alpha\beta}(x-\varepsilon y)\omega(y) | \,dy
      \\
      &\leq \int_{\R^d} c_\eta (a_{\alpha\beta}(x-\varepsilon y) +
      c) |\omega(y) | \,dy
      \leq C(a_{\alpha\beta}(x) + c) \int_{\R^d}
      (1+\varepsilon|y|)^M |\omega(y)| \,dy
      \\
      &\leq C_\eta (a_{\alpha\beta}^\varepsilon(x) + c),
    \end{aligned}
  \end{equation}
  where we again have used \eqref{B.Weyl_with_irre_2}. The calculation also shows that $C_j$ is uniform for $\varepsilon$ in $(0,1]$. For any $\eta$
  in $\mathbb{N}_0^d$ with $\abs{\eta}\geq\tau$ we have
  \begin{equation}\label{B.Weyl_with_irre_4}
    \abs{\partial_x^\eta a_{\alpha\beta}^\varepsilon(x)} \leq c \varepsilon^{\tau+\mu - \abs{\eta}} \leq C_\eta \varepsilon^{\tau-\abs{\eta}}(a_{\alpha\beta}^\varepsilon(x) + c),
  \end{equation}
  by Proposition~\ref{B.smoothning_of_func} with a constant which is uniform for $\varepsilon$ in $(0,1]$. From combing \eqref{B.Weyl_with_irre_3} and \eqref{B.Weyl_with_irre_4} it follows that we for all $\eta$ and $\delta$ in $\N_0^d$ can find constants $C_{\eta\delta}$ such that
  \begin{equation}
  	| \partial_x^\eta \partial_p^\delta b_{\varepsilon,j}(x,p) | \leq C_{\eta\delta} \varepsilon^{-(\tau-j-|\eta|)_{-}} (b_{\varepsilon,0}(x,p)-\zeta) \qquad  \forall  (x,p) \in \R^{2d}.
  \end{equation}
This estimate gives us that $b_{\varepsilon,j}$ is in $\Gamma^{b_{\varepsilon,0}-\zeta,\tau-j}_{0,\varepsilon}(\R_x^{d}\times\R_p^{d})$ for all $j$.

Define the operators $A^{\pm}_\varepsilon(\hbar)$ to be
\begin{equation}
	A^{\pm}_\varepsilon(\hbar) = \OpN{0}(b_\varepsilon(\hbar))
\end{equation}
From the above estimates we have that $A^{\pm}_\varepsilon(\hbar)$  is a $\hbar$-$\varepsilon$-admissible of regularity $\tau$. That the regularity is $\tau$ follows directly from Proposition~\ref{B.smoothning_of_func}. Further more we already have that the operator  $A^{\pm}_\varepsilon(\hbar) $ satisfies \ref{B.H.2} and  \ref{B.H.3} from Assumption~\ref{B.self_adj_assumption}. From \eqref{B.sym_on_S_approx_form} we have that $A^{\pm}_\varepsilon(\hbar) $ also satisfies \ref{B.H.1} of Assumption~\ref{B.self_adj_assumption}. Theorem~\ref{B.self_adjoint_thm_1} now gives us that $A^{\pm}_\varepsilon(\hbar) $ are lower semi-bounded and essentially self-adjointfor all $\hbar$ sufficiently small. Denote also by $A^{\pm}_\varepsilon(\hbar)$ the self-adjointexstension. From \eqref{B.sym_on_S_approx_form} we also get that
\begin{equation}
	A^{\pm}_\varepsilon(\hbar)[\varphi,\psi] = \mathcal{A}_{\hbar,\varepsilon}^{\pm}[\varphi,\psi] \qquad\text{for all $\varphi,\psi \in \mathcal{D}( \mathcal{A}_{\hbar,\varepsilon}^{\pm})$ }.
\end{equation}
This concludes the proof.
\end{proof}

\begin{lemma}\label{B_eks_sf_lsb_op}
Let $\mathcal{A}_\hbar$ be a sesquilinear  form which satisfies Assumption~\ref{ses_assump_gen} with the numbers $(\tau,\mu)$ in $\N\times[0,1]$.  Then the form is symmetric, lower semi-bounded, densely defined and closed. 
\end{lemma}
\begin{proof}
Recall that the sesquilinear  form  $\mathcal{A}_\hbar$ is given by
  \begin{equation}
   \mathcal{A}_\hbar[\varphi,\psi] =  \sum_{\abs{\alpha},\abs{\beta}\leq m}  \int_{\R^d} a_{\alpha\beta}(x) (\hbar D_x)^\beta\varphi(x) \overline{(\hbar D_x)^\alpha\psi(x)} \, dx, \qquad \varphi,\psi \in \mathcal{D}(   \mathcal{A}_\hbar).
  \end{equation}
We have by assumption that $a_{\alpha\beta}(x)=\overline{a_{\beta\alpha}(x)}$ for all $\alpha$ and $\beta$. This gives us that the form is symmetric. Moreover, we also have that  the coefficients  $a_{\alpha\beta}(x)$ at most grow polynomially, hence it follows that the Schwartz functions are contained in the form domain. This gives that the form is densely defined.

From Proposition~\ref{B.construc_framing_op} we get the two forms $\mathcal{A}_{\hbar,\varepsilon}^{\pm}$ such that
 \begin{equation}
 	\mathcal{A}_{\hbar,\varepsilon}^{-}[\varphi,\varphi]\leq \mathcal{A}_{\hbar}[\varphi,\varphi]\leq\mathcal{A}_{\hbar,\varepsilon}^{+}[\varphi,\varphi]
 \end{equation}
  for all $\varphi$ in $\mathcal{D}(   \mathcal{A}_\hbar)$. Moreover, the results in Lemma~\ref{lemma_weyl_q_framin_form} implies that the forms $\mathcal{A}_{\hbar,\varepsilon}^{\pm}$ are closed. Hence $\mathcal{A}_{\hbar}$ will also be closed. This concludes the proof.
\end{proof}
\subsection{Proof of Theorem~\ref{B.Weyl_law_thm_irr_cof}}
Recall that we assumes $\mathcal{A}_\hbar$ is  sesquilinear form  which satisfies Assumption~\ref{ses_assump_gen} with the numbers $(1,\mu)$ with $\mu>0$. Then for the Friedrichs extension of $\mathcal{A}_\hbar$ we have
  \begin{equation*}
   \big |\Tr[\boldsymbol{1}_{(-\infty,0]}(A(\hbar))] - \frac{1}{(2\pi\hbar)^d} \int_{\R^{2d}} \boldsymbol{1}_{(-\infty,0]}( a_{0}(x,p)) \,dx dp \big| \leq C \hbar^{1-d},
  \end{equation*}
\begin{proof}[Proof of Theorem~\ref{B.Weyl_law_thm_irr_cof}]
From Lemma~\ref{B_eks_sf_lsb_op} we get the existence of the operator $A(\hbar)$ as a Friedrichs extension. Combing Proposition~\ref{B.construc_framing_op} and Lemma~\ref{lemma_weyl_q_framin_form} we get the existence of two  $\hbar$-$\varepsilon$-admissible operators $A^{\pm}_\varepsilon(\hbar)$ of regularity $1$ satisfying Assumption~\ref{B.self_adj_assumption} for all $\varepsilon$ sufficiently small such that
  \begin{equation*}
    A_\varepsilon^{-}(\hbar) \leq A(\hbar)  \leq A_\varepsilon^{+}(\hbar),
  \end{equation*}
in the sense of quadratic forms. Moreover the operators $A^{\pm}_\varepsilon(\hbar)$ have the following expansions
\begin{equation}
 	A^{\pm}_\varepsilon(\hbar) = \sum_{j=0}^{2m} \hbar^j \OpW(a^{\pm}_{\varepsilon,j}).
\end{equation}
In particular we have that 
\begin{equation}
	a^{\pm}_{\varepsilon,0}(x,p) = \sum_{|\alpha|, |\beta|\leq m}a_{\alpha\beta}^\varepsilon (x) p^{\alpha+\beta}   \pm C_1 \varepsilon^{k+\mu}\sum_{|\alpha|\leq m} p^{2\alpha}
\end{equation}
Moreover from Proposition~\ref{B.construc_framing_op} we also get that $0$ is a non-critical value for the operators $A^{\pm}_\varepsilon(\hbar)$.
What remains in order to be able to apply
  Theorem~\ref{B.Weyl.law.rough_1} for the two operators $A^{\pm}_\varepsilon(\hbar)$  is the existence of a
  $\tilde{ \nu}>0$ such that the preimage of $(-\infty,\tilde{\nu}]$
  under $a^{\pm}_{\varepsilon,0}$ is compact. By the
  ellipticity of the operators we have that the preimage is compact in $p$. Hence if we
  choose $\tilde{\nu}=\frac{\nu}{2}$ and note that as in the proof of
  Proposition~\ref{B.construc_framing_op} we have the estimate
  \begin{equation}\label{B.weyl_compar_sym}
    | \sum_{\abs{\alpha},\abs{\beta}\leq m} (a_{\alpha\beta}^\varepsilon(x) - a_{\alpha\beta}(x))p^{\alpha+\beta}  \pm C_1 \varepsilon^{1+\mu}(1+p^2)^m |\leq C \varepsilon^{1+\mu} ,
  \end{equation}
  since we can assume $p$ to be in a compact set. This implies the
  inclusion
  \begin{equation*}
    \{ (x,p) \in\R^{2d} \, |\,  a_{\varepsilon,0}^{\pm}(x,p) \leq \frac{\nu}{2} \} \subseteq \{ (x,p) \in\R^{2d} \, |\,  a_{0}(x,p) \leq \frac{\nu}{2} + C \varepsilon^{1+\mu} \}. 
  \end{equation*}
  Hence for a sufficiently small $\varepsilon$ we have that
  $\{ (x,p) \in\R^{2d} \, |\, a_{\varepsilon,0}^{\pm}(x,p) \leq
  \frac{\nu}{2} \}$ is compact due to our assumptions. Now by
  Theorem~\ref{B.Weyl.law.rough_1} we get for sufficiently small $\hbar$
  and $\varepsilon\geq\hbar^{1-\delta}$ for a positive
  $\delta<1$ that
  \begin{equation}\label{Weyl_est_1}
    |\Tr[\boldsymbol{1}_{(-\infty,0]}(A_\varepsilon^{\pm}(\hbar))] - \frac{1}{(2\pi\hbar)^d} \int_{\R^{2d}} \boldsymbol{1}_{(-\infty,0]}( a_{\varepsilon,0}^\pm(x,p)) \,dx dp | \leq C \hbar^{1-d}.
  \end{equation}
  Here we choose $\delta=\frac{\mu}{1+\mu}$. Now if we consider the
  following difference between integrals
  \begin{equation}\label{Weyl_est_2}
    \begin{aligned}
     \MoveEqLeft  | \int_{\R^{2d}} \boldsymbol{1}_{(-\infty,0]}(
      a_{\varepsilon,0}^\pm(x,p)) \,dx dp - \int_{\R^{2d}}
      \boldsymbol{1}_{(-\infty,0]}( a_{0}(x,p)) \,dx dp |
      \\
      &\leq \int_{\R^{2d}}
      \boldsymbol{1}_{[-C\varepsilon^{1+\mu},C\varepsilon^{1+\mu}]}(
      a_{\varepsilon,0}(x,p)) \,dx dp \leq
      \tilde{C}\varepsilon^{1+\mu},
    \end{aligned}
  \end{equation}
  for $\varepsilon$ and hence $\hbar$ sufficiently small. Where we in
  the last inequality have used the non-critical condition. By
  combining \eqref{Weyl_est_1} and \eqref{Weyl_est_2} we get
  \begin{equation}\label{Weyl_est_3}
    |\Tr[\boldsymbol{1}_{(-\infty,0]}(A_\varepsilon^{\pm}(\hbar))] - \frac{1}{(2\pi\hbar)^d} \int_{\R^{2d}} \boldsymbol{1}_{(-\infty,0]}( a_{\varepsilon,0}(x,p)) \,dx dp | \leq C \hbar^{1-d} + \tilde{C}\varepsilon^{1+\mu} \hbar^{-d}.
  \end{equation}
  If we take $\varepsilon = \hbar^{1-\delta}$ we have that
  \begin{equation*}
    \varepsilon^{1+\mu} = \hbar^{(1+\mu)(1-\delta)}=\hbar.
  \end{equation*}
  Hence \eqref{Weyl_est_3} with this choice of $\delta$ and
  $\varepsilon$ gives the estimate
  \begin{equation}\label{Weyl_est_4}
    |\Tr[\boldsymbol{1}_{(-\infty,0]}(A_\varepsilon^{\pm}(\hbar))] - \frac{1}{(2\pi\hbar)^d} \int_{\R^{2d}} \boldsymbol{1}_{(-\infty,0]}( a_{\varepsilon,0}(x,p)) \,dx dp | \leq C \hbar^{1-d}.
  \end{equation}
  Now as the framing operators satisfied the relation
  \begin{equation*}
    A_\varepsilon^{-}(\hbar) \leq A(\hbar)  \leq A_\varepsilon^{+}(\hbar)
  \end{equation*}
 in the sense of quadratic forms we get by the min-max-theorem the relation
  \begin{equation*}
    \Tr[\boldsymbol{1}_{(-\infty,0]}(A_\varepsilon^{+}(\hbar))] \leq \Tr[\boldsymbol{1}_{(-\infty,0]}(A(\hbar)) ] \leq  \Tr[\boldsymbol{1}_{(-\infty,0]}(A_\varepsilon^{-}(\hbar))].
  \end{equation*}
  Combining this with \eqref{Weyl_est_4} we get the estimate
  \begin{equation}
    |\Tr[\boldsymbol{1}_{(-\infty,0]}(A(\hbar))] - \frac{1}{(2\pi\hbar)^d}\int_{\R^{2d}} \boldsymbol{1}_{(-\infty,0]}( a_{\varepsilon,0}(x,p)) \,dx dp | \leq C \hbar^{1-d}.
  \end{equation}
  Which is the desired estimate and this ends the proof.
\end{proof}
\subsection{Proof of Theorem~\ref{B.Riesz_means_thm_irr_cof}}
Recall that we a given a number $\gamma$ in $(0,1]$ and a sesquilinear form  $\mathcal{A}_\hbar$  which satisfies Assumption~\ref{ses_assump_gen} with the numbers $(2,\mu)$, where we suppose $\mu=0$ if $\gamma<1$ and $\mu>0$ if $\gamma=1$. Then for the Friedrichs extension of $\mathcal{A}_\hbar$ we have
  \begin{equation*}
   \big |\Tr[(A(\hbar))^\gamma_{-}] - \frac{1}{(2\pi\hbar)^d} \int_{\R^{2d}} ( a_{0}(x,p))^{\gamma}_{-} +\hbar \gamma a_1(x,p) ( a_{0}(x,p))^{\gamma-1}_{-} \,dx dp \big| \leq C \hbar^{1+\gamma-d},
  \end{equation*}
  for all sufficiently small $\hbar$ where the symbol $a_1(x,p)$ is defined as
  \begin{equation}
	a_{1}(x,p) =  i \sum_{|\alpha|, |\beta|\leq m}\sum_{j=1}^d \frac{\beta_j-\alpha_j}{2} \partial_{x_j} a_{\alpha\beta} (x) p^{\alpha+\beta-\eta_j}, 
\end{equation}
  where $\eta_j$ is the multi index with all entries zero except the j'th which is one. 
\begin{proof}[Proof of Theorem~\ref{B.Riesz_means_thm_irr_cof}]
The start of the proof is analogous to the proof of Theorem~\ref{B.Weyl_law_thm_irr_cof}. Again from Lemma~\ref{B_eks_sf_lsb_op} we get the existence of the operator $A(\hbar)$ and by combing Proposition~\ref{B.construc_framing_op} and Lemma~\ref{lemma_weyl_q_framin_form} we get the existence of two  $\hbar$-$\varepsilon$-admissible operators $A^{\pm}_\varepsilon(\hbar)$ of regularity $2$ satisfying Assumption~\ref{B.self_adj_assumption} for all $\varepsilon$ sufficiently small such that
  \begin{equation*}
    A_\varepsilon^{-}(\hbar) \leq A(\hbar)  \leq A_\varepsilon^{+}(\hbar),
  \end{equation*}
in the sense of quadratic forms. Moreover the operators $A^{\pm}_\varepsilon(\hbar)$ have the following expansions
\begin{equation}
 	A^{\pm}_\varepsilon(\hbar) = \sum_{j=0}^{2m} \hbar^j \OpW(a^{\pm}_{\varepsilon,j}).
\end{equation}
In particular we have that 
\begin{equation}
	a^{\pm}_{\varepsilon,0}(x,p) = \sum_{|\alpha|, |\beta|\leq m}a_{\alpha\beta}^\varepsilon (x) p^{\alpha+\beta}   \pm C_1 \varepsilon^{2+\mu}\sum_{|\alpha|\leq m} p^{2\alpha}
\end{equation}
and
\begin{equation}
	a^{\pm}_{\varepsilon,1}(x,p) =  i \sum_{|\alpha|, |\beta|\leq m}\sum_{j=1}^d \frac{\beta_j-\alpha_j}{2} \partial_{x_j} a_{\alpha\beta}^\varepsilon (x) p^{\alpha+\beta-\eta_j}.
\end{equation}
Moreover from Proposition~\ref{B.construc_framing_op} we also get that $0$ is a non-critical value for the operators $A^{\pm}_\varepsilon(\hbar)$. Just as before we also get that there exist $\tilde{\nu}$ such that $a_{\varepsilon,0}^{-1}((-\infty,\tilde{\nu}])$ is compact for all sufficiently small $\varepsilon$. Hence we are in a situation where we can apply Theorem~\ref{B.Weyl.law.rough_2} for the operators $A^{\pm}_\varepsilon(\hbar) $. This gives us that
  \begin{equation}\label{B.Riesz_means_est_1}
   \big |\Tr[(A^{\pm}_\varepsilon(\hbar))^\gamma_{-}] - \frac{1}{(2\pi\hbar)^d} \int_{\R^{2d}} ( a_{\varepsilon,0}^{\pm}(x,p))^{\gamma}_{-} +\hbar \gamma a_{\varepsilon,1}^{\pm}(x,p) ( a_{\varepsilon,0}^{\pm}(x,p))^{\gamma-1}_{-} \,dx dp \big| \leq C \hbar^{1+\gamma-d}.
  \end{equation}
What remains is to compare the phase space integrals. We let $\kappa>0$ be a number such that $a_{0}(x,p)$ is non-critical on  $a_{0}^{-1}(-2\kappa,2\kappa)$. We then do the following splitting of integrals
\begin{equation}\label{B.Riesz_means_est_2}
	\begin{aligned}
		\MoveEqLeft \big| \int_{\R^{2d}} ( a_{\varepsilon,0}^{\pm}(x,p))^{\gamma}_{-} \,dxdx -  \int_{\R^{2d}} ( a_{0}(x,p))^{\gamma}_{-} \,dxdp\big| 
		\\
		\leq{}& \big| \int_{\R^{2d}}\boldsymbol{1}_{(-c,-\kappa]}(a_{0}(x,p)) \big[ ( a_{\varepsilon,0}^{\pm}(x,p))^{\gamma}_{-} -   ( a_{0}(x,p))^{\gamma}_{-}\big] \,dxdp\big|  
		\\
		& + \big| \int_{\R^{2d}}\boldsymbol{1}_{(-\kappa,0]}(a_{0}(x,p)) \big[ ( a_{\varepsilon,0}^{\pm}(x,p))^{\gamma}_{-} -   ( a_{0}(x,p))^{\gamma}_{-}\big] \,dxdp\big|,
	\end{aligned}
\end{equation}
where $c>0$ is a number such that $a_{0}(x,p)-1 \geq -c$ for all $(x,p)$. On the preimage $a_{0}^{-1}((-\infty,-\kappa])$ we have that $a_{0,\varepsilon}^{\pm}$ will be bounded from above by $-\frac{\kappa}{2}$ for all $\varepsilon$ sufficiently small. Hence using that the map $t\mapsto (t)^\gamma_{-}$ is globally Lipschitz continuous when restricted to any interval of the form $(-\infty,-\frac{\kappa}{2})$ we get that
\begin{equation}\label{B.Riesz_means_est_3}
	\begin{aligned}
		\MoveEqLeft  \big| \int_{\R^{2d}}\boldsymbol{1}_{(-c,-\kappa]}(a_{0}(x,p)) \big[ ( a_{\varepsilon,0}^{\pm}(x,p))^{\gamma}_{-} -   ( a_{0}(x,p))^{\gamma}_{-}\big] \,dxdp\big|  
		\\
		& \leq C   \int_{\R^{2d}}\boldsymbol{1}_{(-c,-\kappa]}(a_{0}(x,p)) | a_{\varepsilon,0}^{\pm}(x,p) -    a_{0}(x,p) | \,dxdp \leq C \varepsilon^{2+\mu},
	\end{aligned}
\end{equation}
where we in the last inequality have used estimates similar to the ones used in the proof of Proposition~\ref{B.construc_framing_op}. Before we estimate the second term in \eqref{B.Riesz_means_est_2} we note that for all $\gamma$ in $(0,1]$ and $t\leq0$ we have that
\begin{equation*}
	(t)_{-}^\gamma = \gamma \int_{t}^0 (-s)^{\gamma-1} \, ds.
\end{equation*}
Using this observation we get that
\begin{equation}\label{B.Riesz_means_est_4}
	\begin{aligned}
		\MoveEqLeft  \big| \int_{\R^{2d}}\boldsymbol{1}_{(-\kappa,0]}(a_{0}(x,p)) \big[ ( a_{\varepsilon,0}^{\pm}(x,p))^{\gamma}_{-} -   ( a_{0}(x,p))^{\gamma}_{-}\big] \,dxdp\big|
		\\
		&= \big| \gamma \int_{-\frac{3\kappa}{2}}^0 (-s)^{\gamma-1}  \int_{\R^{2d}} \boldsymbol{1}_{(-\kappa,0]}(a_{0}(x,p))\big[ \boldsymbol{1}_{( a_{\varepsilon,0}^{\pm}(x,p),0]}(s)  -   \boldsymbol{1}_{( a_{0}(x,p),0]}(s)\big] \,dxdp ds\big| 
		\\
		&\leq  \gamma \int_{-\frac{3\kappa}{2}}^0 (-s)^{\gamma-1}  \int_{\R^{2d}} \boldsymbol{1}_{(-\kappa,0]}(a_{0}(x,p)) \boldsymbol{1}_{[s-\tilde{c}\varepsilon^{2+\mu},s+\tilde{c}\varepsilon^{2+\mu}]}(a_{0}(x,p))   \,dxdp ds \leq C \varepsilon^{2+\mu}
	\end{aligned}
\end{equation}
where we in the last inequality have used the non-critical assumption. Combining \eqref{B.Riesz_means_est_2}, \eqref{B.Riesz_means_est_3} and \eqref{B.Riesz_means_est_4} we get
\begin{equation}\label{B.Riesz_means_est_5}
	\begin{aligned}
		 \big| \int_{\R^{2d}} ( a_{\varepsilon,0}^{\pm}(x,p))^{\gamma}_{-} \,dxdx -  \int_{\R^{2d}} ( a_{0}(x,p))^{\gamma}_{-} \,dxdp\big| 
		\leq C \varepsilon^{2+\mu}.
	\end{aligned}
\end{equation}
Next we consider the difference 
  \begin{equation}\label{B.Riesz_means_est_6}
   \Big| \int_{\R^{2d}}  a_{\varepsilon,1}^{\pm}(x,p) ( a_{\varepsilon,0}^{\pm}(x,p))^{\gamma-1}_{-} \,dx dp - \int_{\R^{2d}}  a_{1}(x,p) ( a_{0}(x,p))^{\gamma-1}_{-} \,dx dp\Big|
  \end{equation}
Firstly we notice that from Proposition~\ref{B.construc_framing_op} we get that 
  \begin{equation}\label{B.Riesz_means_est_7}
   \Big| \int_{\R^{2d}}  a_{\varepsilon,1}^{\pm}(x,p) ( a_{\varepsilon,0}^{\pm}(x,p))^{\gamma-1}_{-} \,dx dp - \int_{\R^{2d}}  a_{1}(x,p) ( a_{\varepsilon,0}^{\pm}(x,p))^{\gamma-1}_{-} \,dx dp\Big| \leq C\varepsilon^{1+\mu}.
  \end{equation}
Moreover, we have that
  \begin{equation}\label{B.Riesz_means_est_8}
  	\begin{aligned}
   	\MoveEqLeft \Big|  \int_{\R^{2d}}  a_{1}(x,p) ( a_{\varepsilon,0}^{\pm}(x,p))^{\gamma-1}_{-} \,dx dp -  \int_{\R^{2d}}  a_{1}(x,p) ( a_{0}(x,p))^{\gamma-1}_{-} \,dx dp\Big| 
	\\
	\leq{}& \Big|  \int_{\R^{2d}}  a_{1}(x,p) \big[ \boldsymbol{1}_{(-\kappa,0]}(a_{\varepsilon,0}^{\pm}(x,p)) (- a_{\varepsilon,0}^{\pm}(x,p))^{\gamma-1} - \boldsymbol{1}_{(-\kappa,0]}(a_{0}(x,p))  ( -a_{0}(x,p))^{\gamma-1} \big] \,dx dp \Big| 
	\\
	&+C\varepsilon^{2+\mu},
   	\end{aligned}
  \end{equation}
where the number $\kappa>0$ is the same as above. In this estimate we have used that
  \begin{equation}
  	\begin{aligned}
   	\MoveEqLeft \Big|  \int_{\R^{2d}}  a_{1}(x,p) \big[ \boldsymbol{1}_{(-c,-\kappa]}(a_{\varepsilon,0}^{\pm}(x,p)) (- a_{\varepsilon,0}^{\pm}(x,p))^{\gamma-1} - \boldsymbol{1}_{(-c,-\kappa]}(a_{0}(x,p))  ( -a_{0}(x,p))^{\gamma-1} \big] \,dx dp \Big|
	\\
	& \leq C\varepsilon^{2+\mu},
   	\end{aligned}
  \end{equation}
where  the arguments are analogous to the ones used for estimating the first difference of phase space integrals after recalling that  the map $t\rightarrow (t)^{\gamma-1}_{-}$ is Lipschitz continuous when restricted to a bounded interval of the form $(-a,-b)$, where $0<b<a$. 

On the pre-image $a_{0}^{-1}((-\kappa,0])$ we have that 
\begin{equation}
	|\nabla_p a_{0}(x,p) | \geq c_0 >0.
\end{equation}
Then by using a partition of unity we may assume that $|\partial_{p_1}a_{0}(x,p)|\geq c>0$ on $a_{0}^{-1}((-\kappa,0])$. From the construction of the approximating symbols it follows that for $\varepsilon$ sufficiently small we get that $|\partial_{p_1}a_{\varepsilon,0}^{\pm}(x,p)|\geq \tilde{c}$ when $a_{\varepsilon,0}^{\pm}(x,p)\in(-\kappa,0]$. 
As in the proof of Theorem~\ref{B.Expansion_of_trace} we now define the maps $F^{\pm}$ and $F$ by
  \begin{equation*}
    F^\pm:(x,p) \rightarrow (x_1,\dots,x_d,a_{\varepsilon,0}^{\pm}(x,p),p_2,\dots,p_d)
  \end{equation*}
and the map $F$ is defined similarly with $a_{\varepsilon,0}^{\pm}(x,p)$ replaced by $a_{0}(x,p)$. The determinants for the Jacobian matrices are given by
  \begin{equation*}
    \det(DF^{\pm})= \partial_{p_1}a_{\varepsilon,0}^{\pm} \quad\text{and}\quad  \det(DF)= \partial_{p_1}a_{0}.
  \end{equation*}
This gives us that the inverse maps ($G^{\pm}$, $G$) exists and they will be close in the sense that    
\begin{equation}
	|G^{\pm}(y,q) - G(y,q)| \leq C\varepsilon^{2+\mu}
\end{equation}
on the domain, where both functions are defined. To see this note that they will only be different in one coordinate and note that these coordinates will be determined by the equation $a_0(y,(p_1,q_2,\dots,q_d))=q_1$ and similarly for the two other functions. With a change of variables we get that   
  \begin{equation}\label{B.Riesz_means_est_9}
  	\begin{aligned}
   	\MoveEqLeft  \Big|  \int_{\R^{2d}}  a_{1}(x,p) \big[ \boldsymbol{1}_{(-\kappa,0]}(a_{\varepsilon,0}^{\pm}(x,p)) (- a_{\varepsilon,0}^{\pm}(x,p))^{\gamma-1} - \boldsymbol{1}_{(-\kappa,0]}(a_{0}(x,p))  ( -a_{0}(x,p))^{\gamma-1} \big] \,dx dp \Big| 
	\\
	\leq{}& \Big|  \int_{\Omega}  \frac{a_{1}( G^{\pm}(x,p)) \boldsymbol{1}_{(-\kappa,0]}(p_1) (- p_1)^{\gamma-1}}{\partial_{p_1}a_{\varepsilon,0}^{\pm}(G^{\pm}(x,p))} -\frac{a_{1}( G(x,p)) \boldsymbol{1}_{(-\kappa,0]}(p_1) (- p_1)^{\gamma-1}}{\partial_{p_1}a_{0}(G(x,p))} \,dx dp \Big| 
	\\
	\leq{}& \int_{\Omega} \Big| \frac{a_{1}( G^{\pm}(x,p)) }{\partial_{p_1}a_{\varepsilon,0}^{\pm}(G^{\pm}(x,p))} -\frac{a_{1}( G(x,p))}{\partial_{p_1}a_{0}(G(x,p))} \Big| \boldsymbol{1}_{(-\kappa,0]}(p_1) (- p_1)^{\gamma-1} \,dx dp
	\leq C \varepsilon^{1+\mu},
   	\end{aligned}
  \end{equation}
where $\Omega$ is a compact subset of $\R^{2d}$. The last estimate follows similarly to the previous estimates. From combining \eqref{B.Riesz_means_est_6}, \eqref{B.Riesz_means_est_7}, \eqref{B.Riesz_means_est_8} and \eqref{B.Riesz_means_est_9} we get that
  \begin{equation}\label{B.Riesz_means_est_10}
   \Big| \int_{\R^{2d}}  a_{\varepsilon,1}^{\pm}(x,p) ( a_{\varepsilon,0}^{\pm}(x,p))^{\gamma-1}_{-} \,dx dp - \int_{\R^{2d}}  a_{1}(x,p) ( a_{0}(x,p))^{\gamma-1}_{-} \,dx dp\Big| \leq C\varepsilon^{1+\mu}
  \end{equation}
Now by choosing $\delta= 1- \frac{1+\gamma}{2+\mu}$ and combining \eqref{B.Riesz_means_est_1}, \eqref{B.Riesz_means_est_5} \eqref{B.Riesz_means_est_10} we obtain that
  \begin{equation}\label{B.Riesz_means_est_11}
  	\begin{aligned}
   \MoveEqLeft \big |\Tr[(A^{\pm}_\varepsilon(\hbar))^\gamma_{-}] - \frac{1}{(2\pi\hbar)^d} \int_{\R^{2d}} ( a_{0}(x,p))^{\gamma}_{-} +\hbar \gamma a_{1}(x,p) ( a_{0}(x,p))^{\gamma-1}_{-} \,dx dp \big| 
   \\
   &\leq C_1 \hbar^{1+\gamma-d} + C_2\hbar^{-d}\varepsilon^{2+\mu} +C_3\hbar^{1-d}\varepsilon^{1+\mu} \leq C\hbar^{1+\gamma-d}
   \end{aligned}
  \end{equation}
Now as the framing operators satisfied the relation
  \begin{equation*}
    A_\varepsilon^{-}(\hbar) \leq A(\hbar)  \leq A_\varepsilon^{+}(\hbar)
  \end{equation*}
 in the sense of quadratic forms we get by the min-max-theorem the relation
  \begin{equation*}
    \Tr[(A_\varepsilon^{+}(\hbar))^\gamma_{-}] \leq \Tr[(A(\hbar) )^\gamma_{-}] \leq  \Tr[(A_\varepsilon^{-}(\hbar))^\gamma_{-}].
  \end{equation*}
These inequalities combined with \eqref{B.Riesz_means_est_11} gives the desired estimates and this concludes the proof.
\end{proof}
\subsection{Proof of Theorem~\ref{B.Weyl_law_thm_adop_plus_irr_cof} and Theorem~\ref{B.Riesz_means_thm_adop_plus_irr_cof}}
Recall that we assume $B(\hbar)$ is an admissible operator and that $B(\hbar)$ satisfies Assumption~\ref{admis_op_assump} with $\hbar$ in $(0,\hbar_0]$. Moreover we assume that  $\mathcal{A}_\hbar$ be a sesquilinear form which satisfies Assumption~\ref{ses_assump_gen} with the numbers $(k,\mu)$, where the value of the two numbers $k$ and $\mu$ depend on which Theorem one consider. We defined the symbol $\tilde{b}_0(x,p) $ to be
\begin{equation}
	\tilde{b}_0(x,p) = b_0(x,p) +  \sum_{\abs{\alpha},\abs{\beta}\leq m} a_{\alpha\beta}(x)p^{\alpha+\beta}.
\end{equation}
  Finally we supposed that there is a $\nu$ such that $\tilde{b}_0^{-1}((-\infty,\nu])$ is compact and there is $c>0$ such that
  \begin{equation}
    |\nabla_p \tilde{b}_0(x,p)| \geq c \quad\text{for all } (x,p)\in \tilde{b}_0^{-1}(\{0\}).
  \end{equation}
\begin{proof}[Proof of Theorem~\ref{B.Weyl_law_thm_adop_plus_irr_cof} and Theorem~\ref{B.Riesz_means_thm_adop_plus_irr_cof}]
From Lemma~\ref{B_eks_sf_lsb_op} we get the existence of the self-adjointoperator $A(\hbar)$ as the Friedrichs extension $\mathcal{A}_\hbar$. Since the Schwartz functions are in the domain of the operators $A(\hbar)$ and $B(\hbar)$ we can define their form sum $\tilde{B}(\hbar) =B(\hbar)+ A(\hbar)$ as a selfadjoint, lower semi-bounded operator see e.g. \cite[Proposition 10.22]{MR2953553}. Using Proposition~\ref{B.construc_framing_op} and by arguing as in Lemma~\ref{lemma_weyl_q_framin_form} we get the existence of two  $\hbar$-$\varepsilon$-admissible operators $\tilde{B}^{\pm}_\varepsilon(\hbar)$ of regularity $1$ or $2$ satisfying Assumption~\ref{B.self_adj_assumption} for all $\varepsilon$ sufficiently small such that
  \begin{equation*}
    \tilde{B}_\varepsilon^{-}(\hbar) \leq \tilde{B}(\hbar)  \leq \tilde{B}_\varepsilon^{+}(\hbar).
  \end{equation*}
In the sense of quadratic forms. In fact we will have that $\tilde{B}^{\pm}_\varepsilon(\hbar) = B(\hbar)+ A^{\pm}_\varepsilon(\hbar)$, where $ A^{\pm}_\varepsilon(\hbar)$ is the operator from Lemma~\ref{lemma_weyl_q_framin_form}. From here the proof is either analogous to that of Theorem~\ref{B.Weyl_law_thm_irr_cof} if it is the proof of Theorem~\ref{B.Weyl_law_thm_adop_plus_irr_cof}. If it is the proof of Theorem~\ref{B.Riesz_means_thm_adop_plus_irr_cof} it will be analogous to the proof of Theorem~\ref{B.Riesz_means_thm_irr_cof} from here.
\end{proof}

\subsection*{Acknowledgements} The author was supported by EPSRC grant EP/S024948/1 and Sapere Aude Grant DFF--4181-00221 from the Independent Research Fund Denmark. Part of this work was carried out while the author visited the Mittag-Leffler Institute in Stockholm, Sweden. The Author is grateful for useful comments on previous versions of the work by Clotilde Fermanian Kammerer, S\o{}ren Fournais and San V\~u Ng\d{o}c.

%%%%%%%%%%%%%%%%%%%%%%%%%%%%%%%%%%%%%%%%%%%%%%%%%%%%%%%%%%%%%%

\appendix
  
\section{{F}a\`a di {B}runo formula}
 
In this appendix we will recall some results about multivariate
differentiation. In particular we will recall the {F}a\`a di {B}runo formula, which gives a multivariate chain rule for any number of derivatives:
\begin{thm}[{F}a\`a di {B}runo formula]\label{B.Faa_di_bruno_x}
  Let $f$ be a function from $C^\infty(\R)$ and $g$ a function from
  $C^\infty(\R^d)$. Then for all multi indices $\alpha$ with
  $\abs{\alpha}\geq1$ the following formula holds:
  \begin{equation*}
    \partial_x^\alpha f(g(x)) = \sum_{k=1}^{\abs{\alpha}} f^{(k)}(g(x)) \sum_{\substack{\alpha_1 + \cdots+ \alpha_k = \alpha \\ \abs{\alpha_j} >0}} c_{\alpha_1\cdots\alpha_k}\partial_x^{\alpha_1} g(x) \cdots \partial_x^{\alpha_k}g(x),
  \end{equation*}
  where $f^{(k)}$ is the $k$'th derivative of f. The second sum should
  be understod as a sum over all ways to split the multi index
  $\alpha$ in $k$ non-trivial parts. The numbers
  $c_{\alpha_1\cdots\alpha_k}$'s are combinatorial constants
  independent of the functions.
\end{thm}
A proof of the Fa\`a di Bruno formula can be found in
\cite{MR1325915}, where they prove the formula in greater generality
then stated here. It is also possible to find the constants from their
proof, but for our purpose here the exact values of the constants are
not important. The next Corollary is the Fa\`a di Bruno formula in the
case of $\R^{2d}$ instead of just $\R^d$. But we need to control
the exact number of derivatives in the first $d$ components hence it
is stated separately.
\begin{corollary}\label{B.Faa_di_bruno_xp}
  Let $f$ be a function from $C^\infty(\R)$ and $a(x,p)$ a function from
  $C^\infty(\R_x^d\times\R^d_p)$. Then for all multi indices $\alpha$
  and $\beta$ with $\abs{\alpha}+\abs{\beta}\geq1$ the following
  formula holds:
  \begin{equation*}
    \partial_p^\beta \partial_x^\alpha f(a(x,p)) = \sum_{k=1}^{\abs{\alpha}+\abs{\beta}} f^{(k)} (a(x,p)) \sum_{\mathcal{I}_k(\alpha,\beta)}c_{\alpha_1\cdots\alpha_k}^{\beta_1\cdots\beta_k}  \partial_p^{\beta_1} \partial_x^{\alpha_1} a(x,p) \cdots \partial_p^{\beta_k} \partial_x^{\alpha_k} a(x,p),
  \end{equation*}
  where the set $\mathcal{I}_k(\alpha,\beta)$ is defined by
  \begin{equation*} \begin{aligned}
    \mathcal{I}_k(\alpha,\beta) =
    \{(\alpha_1,\dots,\alpha_k,&\beta_1,\dots,\beta_k) \in \N^{2kd} \
    \\
    & | \ \sum_{l=1}^k \alpha_l=\alpha, \, \sum_{l=1}^k \beta_l=\beta,
    \, \max(\abs{\alpha_l},\abs{\beta_l}) \geq1 \, \forall l \}.
   \end{aligned} \end{equation*}
  The second sum is a sum over all elements in the set
  $\mathcal{I}_k(\alpha,\beta)$, the constants
  $c_{\alpha_1\cdots\alpha_k}^{\beta_1\cdots\beta_k}$ are
  combinatorial constants independent of the functions and $f^{(k)}$
  is the $k$'th derivative of the function $f$.
\end{corollary}
We will only give a short sketch of the proof of this corollary.
%
%\vspace{\baselineskip}
%
\begin{proof}[Sketch]
We have by the {F}a\`a di {B}runo formula (Theoram~\ref{B.Faa_di_bruno_x}) the identity
\begin{equation}\label{B.cor_fa_proof_1}
	 \partial_x^\alpha f(a(x,p)) =  \sum_{k=1}^{\abs{\alpha}} f^{(k)}(a(x,p)) \sum_{\substack{\alpha_1 + \cdots+ \alpha_k = \alpha \\ \abs{\alpha_j} >0}} c_{\alpha_1\cdots\alpha_k}\partial_x^{\alpha_1} a(x,p) \cdots \partial_x^{\alpha_k}a(x,p)
\end{equation}
By Leibniz's formula we have
\begin{equation}\label{B.cor_fa_proof_2}
	\begin{aligned}
	   \partial_p^\beta &\partial_x^\alpha f(a(x,p))
	     \\
	     & = \sum_{\gamma\leq\beta} \binom{\beta}{\gamma}  \sum_{k=1}^{\abs{\alpha}}  \partial_p^{\gamma} f^{(k)}(a(x,p)) \sum_{\substack{\alpha_1 + \cdots+ \alpha_k = \alpha \\ \abs{\alpha_j} >0}}  c_{\alpha_1\cdots\alpha_k}  \partial_p^{\beta-\gamma} [\partial_x^{\alpha_1} a(x,p) \cdots \partial_x^{\alpha_k}a(x,p)].
	 \end{aligned}
\end{equation}
In order to obtain the form stated in the corollary we need to use the {F}a\`a di {B}runo formula on the terms 
\begin{equation}\label{B.cor_fa_proof_3}
  \partial_p^{\gamma} f^{(k)}(a(x,p)),
\end{equation}
and we need to use Leibniz's formula (multiple times) on the terms
\begin{equation}\label{B.cor_fa_proof_4}
	  \partial_p^{\beta-\gamma} [\partial_x^{\alpha_1} a(x,p) \cdots \partial_x^{\alpha_k}a(x,p)].
\end{equation}
If this is done, then by using some algebra the stated form can be obtained. The particular form of the index set $\mathcal{I}_k(\alpha,\beta)$ also follows from this algebra.
\end{proof}

\bibliographystyle{plain} \bibliography{Bib_paperB.bib}

\end{document}